\documentclass[reqno,makeidx,11pt]{amsart}
\usepackage{amssymb,amsfonts,dsfont,amsmath,graphicx, euscript
}
\usepackage{mathrsfs}
\usepackage[toc]{appendix}
\usepackage[all]{xy}
\usepackage[T1]{fontenc}
\usepackage{extarrows}
\usepackage{centernot}
\usepackage{hyperref}

\makeatletter

\@addtoreset{equation}{section}
\makeatother

\def\C{{\mathbb{C}}}

\def\F{{\mathbb{F}}}

\def\N{{\mathbb{N}}}
\def\0{{\mathbb{O}}}

\def\Q{{\mathbb{Q}}}
\def\R{{\mathbb{R}}}

\def\Z{{\mathbb{Z}}}

\def\wh{\widehat}
\def\wt{\widetilde}

\def\cA{{\mathcal A}}
\def\cB{{\mathcal B}}
\def\cC{{\mathcal C}}

\def\cE{{\mathcal E}}
\def\cF{{\mathcal F}}
\def\cG{{\mathcal G}}
\def\cH{{\mathcal H}}
\def\cI{{\mathcal I}}
\def\cK{{\mathcal K}}
\def\cL{{\mathcal L}}

\def\cS{{\mathcal S}}
\def\cT{{\mathcal T}}

\def\cV{{\mathcal V}}

\def\Ker{{\rm Ker \, }}
\def\Id{{\rm Id \, }}

\def \supp {\hbox{Supp}}

\newcommand{\norm}[1]{{\left\|{#1}\right\|}}

\newcommand{\abs}[1]{{\left|{#1}\right|}}
\newcommand{\scal}[1]{{\left\langle{#1}\right\rangle}}
\newcommand{\set}[1]{{\left\{{#1}\right\}}}

\newcommand{\actson}{\curvearrowright}
\newcommand{\rd}{{\,\mathrm d}}
\newcommand{\Ra}{{\Rightarrow}}

\def\fig{ \centerline{Fig. \the\count200\cGlobal\advance\count200 by 1}}
\newtheorem{thm}{Theorem}[section]
\newtheorem{cor}[thm]{Corollary}
\newtheorem{lem}[thm]{Lemma}
\newtheorem{prop}[thm]{Proposition}

\newtheorem*{thm*}{Theorem}
\newtheorem*{lem*}{Lemma}
\newtheorem*{cor*}{Corollary}

\theoremstyle{definition}
\newtheorem{defn}[thm]{Definition}
\newtheorem{ex}[thm]{Example}
\newtheorem{exs}[thm]{Examples}

\theoremstyle{remark}
\newtheorem{rem}[thm]{Remark}
\newtheorem{rems}[thm]{Remarks}

\makeindex

\makeatletter
\@namedef{subjclassname@2020}{\textup{2020} Mathematics Subject Classification}
\makeatother

\title[Groupoids exactness and the weak containment problem]{Groupoids exactness and the weak containment problem}

\author{Claire Anantharaman-Delaroche}
\address{Institut Denis Poisson\newline
\indent Universit\'e d'Orl\'eans, Universit\'e de Tours et CNRS-UMR 7013,\newline
\indent Route de Chartres, B.P. 6759; 
F-45067 Orl\'eans Cedex 2, France}
\email{claire.anantharaman@univ-orleans.fr}



\begin{document}

\begin{abstract} 
Our purpose is to study  in the setting of locally compact grou\-poids the analogues of the well-known equivalent definitions of exactness for discrete groups. Our best results are obtained for a class of \'etale groupoids that we call  inner amenable. For locally compact groups this notion coincides with a classical notion of inner amenability. We give examples of such groupoids. Whether all \'etale groupoids have this property is still unknown. For  inner amenable \'etale groupoids we extend what is known for discrete  groups in proving the equivalence of six natural notions of exactness: (1) strong amenability at infinity; (2) amenability at infinity; (3) nuclearity of the uniform  algebra of the groupoid; (4) exactness of this $C^*$-algebra; (5) exactness of the groupoid in the sense of Kirchberg-Wassermann; (6) exactness of the reduced $C^*$-algebra of the groupoid. We   give several illustrations of these results.

 One of our motivations for this study of exactness is that it plays a crucial role in examining the relationships between the amenability of a groupoid and the fact that its full and reduced $C^*$-algebras coincide. This  is highlighted in the review we give of the results obtained by several authors on this subject.
  We end our monograph with open questions and an appendix on fibrewise compactifications whose study is needed because  our work requires to extend from discrete groups to any \'etale groupoid $\cG$ the notion of Stone-\v Cech compactification on which $\cG$ acts.

 \end{abstract}
\maketitle

\tableofcontents
\renewcommand{\sectionmark}[1]{}

\section*{\textbf{\textsc{Introduction}}}

This monograph aims to present the state of current knowledge on the two problems below concerning locally compact groupoids:
\begin{itemize}
\item [(A)] compare the different notions of exactness for them;
\item [(B)] study the weak containment problem, that is, whether the equality of the full and reduced $C^*$-algebras of a locally compact groupoid implies its amenability.
\end{itemize}

A groupoid whose  full and reduced $C^*$-algebras coincide is said to have the {\it weak containment property} (WCP).
It has been discovered over the last 10 years that the notion of exactness was involved in the study of problem (B).  

For a locally compact group $G$, these problems have now a long history. The positive answer to the question (B) dates back to the 1960s  \cite{Hul}. The study of the first question was initiated by Kirchberg at the end of the 1970s \cite{Kir78} and continued in the 1990s in a series of papers \cite{KW95, KW99, KW99bis}.  Motivated by the study of the continuity of fibrewise crossed product $C^*$-bundles \cite{Rie}, Kirchberg and Wassermann introduced in \cite{KW99} the notion of exact locally compact group, that we call KW-{\it exactness}\footnote{KW stands of course for Kirchberg and Wassermann.}: roughly speaking, this  means that  the reduced crossed  product functor preserves $G$-equivariant short exact sequences of $C^*$-algebras (see Definition \ref{def:exact1}). In \cite{KW99} it is proved, among many other results, that a discrete group $\Gamma$ is KW-exact if and only if the corresponding reduced $C^*$-algebra $C^*_{r}(\Gamma)$ is exact. Soon after it was proved, still for a discrete group, that the exactness of $C^*_{r}(\Gamma)$  (called $C^*$-{\it exactness} of the group) is equivalent to the nuclearity of the uniform Roe $C^*$-algebra $C^*_{u}(\Gamma)$ (known to be isomorphic  to the reduced crossed product $\ell^\infty(\Gamma)\rtimes_r \Gamma$) and also to the fact (called {\it amenability at infinity}) that the group $\Gamma$ admits an amenable action on a compact space (see \cite{G-K, GKAd, Oza}). This was extended in \cite{AD02} to the case of any  locally compact group $G$ which is {\it inner amenable} in the sense of  Definition \ref{def:IA}.  Recently, it was shown \cite{BCL, OS20} that KW-exactness implies amenability at infinity for any locally compact group $G$. This  last property does not always imply inner amenability: for instance connected non-amenable groups give a counterexample (see \cite{LR87}, \cite[Remark 5.10]{AD02}). The only remaining open problem for locally compact groups is whether $C^*$-exactness implies KW-exactness.
Current knowledge is summarized by the diagram below:

$$\xymatrix{ \hbox{amen. at } \infty\ar@{=>}[rd]|{/}\ar@{<=>}[d]\ar@{<=}[r]&\hbox{amenable}\ar@{<=>}[r]\ar@{=>}[d] &\hbox{(WCP)}\\
\hbox{KW-exact}\ar@{=>}[d]&\hbox{inner amenable}&\\
\hbox{$C^*$-exact}\ar@/_2pc/@{=>}[u]_?&}$$

We can speak without ambiguity of an exact discrete group, since all definitions are equivalent in this case. These groups are uniformly embeddable into a Hilbert space and therefore the Novikov higher signature conjecture holds for them \cite{Yu, HR, H00, STY}.

The notion of amenability, the Baum-Connes and the Novikov conjectures have been generalized to locally compact groupoids (see \cite{AD-R} as regards to amenability and \cite{Tu99, Tu99bis, TuSurvey} for the conjectures).  Concerning the extension of the notion of exactness to groupoids the situation is not so clear. We presented some results on this subject in a talk given at the MSRI in 2000 \cite{AD00}, but we hadn't written down all the details at that time.  Following the publication of Willett's article \cite{Wil15}, it became clear that the notion of exactness for groupoids plays an important role in order to have a positive answer to the problem (B) above.  This encouraged us to post on arXiv in 2016 \cite{AD16} a first version of our text.  We offer here a revised version which takes into account recent advances on the subject, some extracts of which have been summarized in the article \cite{AD23}.

For locally compact groupoids one can define, as for groups, the notions of {\it amenability}, {\it inner amenability}, KW-{\it exactness, $C^*$-exactness} (see Subsections \ref{subsec:amenable}, \ref{subsection:IA}, \ref{subsec:exact}). 

Concerning {\it ame\-nability at infinity} there is a subtlety which led us to also introduce the notion of {\it strong amenability of infinity }(Definition \ref{def:sai}). In addition, another notion of exactness, that we called {\it inner exactness} in \cite{AnD16} comes into play here (see Definition \ref{def:inamen}). The relations between these notions that always hold are the following ones:

$$\xymatrix{&\hbox{strongly amen. at } \infty\ar@{<=}[rd]\ar@{=>}[d]_b&&\\
& \hbox{amen. at } \infty\ar@{=>}[rd]|{/}\ar@{=>}[d]_c\ar@{<=}[r]&\hbox{amenable}\ar@{=>}[r]^a\ar@{=>}[d] &\hbox{(WCP)}\\
\hbox{inner exact}\ar@{<=}[r]&\hbox{KW-exact}\ar@{=>}[d]_d&\hbox{inner amenable}&\\
&\hbox{$C^*$-exact}&&}$$

Our main concern is  to study conditions that are sufficient to reverse the arrows a,b,c,d and to give examples.

 Let us give now some details on the content of this text. The analogue of a discrete group is an \'etale groupoid (see Section \ref{subset:etaleg}). We are especially interested in these groupoids.  Therefore, in this introduction, to make easier the discussion, $\cG$  will be a second countable \'etale groupoid, although the fact that $\cG$ is \'etale is not always necessary, as specified in the text.

As already said, a discrete group $\Gamma$ is  amenable at infinity if  and only if it has an amenable action on a compact space $Y$. This is equivalent to the fact that the natural action of $\Gamma$ on its Stone-\v Cech compactification $\beta \Gamma$ is amenable. An action of $\cG$  involves a fibre space $(Y,p)$ over the space $X= \cG^{(0)}$ of units of $\cG$ (see Definition \ref{def:Gspace}).  The analogue of compactness is the property for $p$ to be proper, in which case we say that $(Y,p)$ is {\it fibrewise compact}, following the terminology of \cite[Definition 3.1]{James}. We say that $\cG$ is {\it amenable at infinity}  if it has an amenable action on a fibrewise compact fibre space. Like for genuine locally compact spaces, a fibre space $(Y,p)$  has a greatest fibrewise compactification called the {\it Stone-\v Cech fibrewise compactification} of $(Y,p)$ (see Remark \ref{rem:f_c}). The groupoid $\cG$ being \'etale (an important assumption here), we show in Proposition \ref{prop:max_min} that this compactification carries a natural $\cG$-action ({\it i.e.}, is a $\cG$-space). We view $\cG$ as left $\cG$-space fibred on $X$ by the range map $r$ and we denote  by $(\beta_r \cG, r_\beta)$ its Stone-\v Cech fibrewise compactification\footnote{We point out that in our definition of a fibre space $(Y,p)$ we do not require $p$ to be open since in general $r_\beta : \beta_r \cG \to \cG^{(0)}$ is not open (see the remark \ref{rem:compG}). In this respect, we follow the convention of \cite[Definition 2.1, Remark 2.2]{Will19}.}. We do not know whether the amenability at infinity of $\cG$ implies that its natural action on $(\beta_r \cG, r_\beta)$ is amenable. This holds if and only if $\cG$ has an amenable action on a fibrewise compact fibre space $(Y,p)$ such that $p$ admits a continuous section. If this is the case we say that $\cG$ is {\it strongly amenable at infinity}. For groups, this makes no difference. Although $\beta_r\cG$ is a rather ugly space, the nice feature of strong amenability at infinity is that it has an intrinsic characterization,  in term of positive definite kernels on the subspace $\cG *_r\cG$ of pairs  $(\gamma,\gamma_1)\in \cG\times \cG$ with the same range (Theorem \ref{prop:amen_inf}).

 The reduced $C^*$-algebra  $C^*_{r}(\beta_r\cG \rtimes \cG)$ of the semi-direct  product groupoid $\beta_r\cG \rtimes \cG$ is canonically isomorphic  to  the groupoid analogue $C^*_{u}(\cG)$ of the uniform Roe algebra  generated by the finite propagation operators in case of a discrete group (Theorem \ref{prop:Roe}). This $C^*$-algebra $C^*_{u}(\cG)$ has the advantage over $C^*_{r}(\beta_r\cG \rtimes \cG)$ to be more elementarily defined.

   We show that $\cG$ is KW-exact whenever it is amenable at infinity and that, if $\cG$ is KW-exact, then its reduced $C^*$-algebra $C^*_{r}(\cG)$ is exact (Proposition \ref{prop:equiv})\footnote{See \cite[Theorem 6.14]{Lal15} for a more general result.}.   This property is called {\it $C^*$-exactness} of $\cG$. The proofs are easy adaptations of the known proof for  groups (see \cite[Theorem 7.2]{AD02} for instance). 
  The links between the different notions of exactness that we have introduced are stated in Theorem \ref{cor:ai_ex}. To establish that all these notions are equivalent, it would be necessary to be able to prove that  $C^*$-exactness implies (strong) amenability at infinity. 

 Let us give a quick idea of the proof of this fact for a discrete group $\Gamma$, which    uses the following easy observation:  if $\Phi : C^*_{r}(\Gamma)\to \cB(\ell^2(\Gamma))$ is a completely positive map such that $\Phi(\delta_s) = 0$  except for $s$ in a finite subset $F$ of $\Gamma$ (where $\delta_s$ is the Dirac mass at $s\in \Gamma$), then the kernel $k: (s,t)\in \Gamma\times \Gamma \to \C$ defined by $k(s,t) = \scal{\delta_{s^{-1}}, \Phi(\delta_{s^{-1}t}) \delta_{t^{-1}}}$ is positive definite on $\Gamma\times\Gamma$ and has its support in the tube $\set{(s,t)\in \Gamma\times \Gamma: s^{-1}t\in F}$. The first important ingredient of the proof,   essentially contained in \cite{G-K}, is that the  $C^*$-exactness of $\Gamma$ allows the construction of nets of such kernels that converges to 1 uniformly on tubes. The second ingredient of the proof is the characteri\-zation of the amenability at infinity of $\Gamma$ in term of the existence of such nets \cite{HR,Oza,AD02}. 
 
As already mentioned, this characterization  has an analogue in the setting of \'etale groupoids (Theorem \ref{prop:amen_inf}) in term of positive definite kernels on   $\cG *_r\cG$.
On the other hand, in the groupoid case, the defi\-nition of a  positive definite kernel $k$ on $\cG *_r\cG$, starting from a completely positive map $\Phi$ from $C^*_{r}(\cG)$ into  the $C^*$-algebra (analogous to $\cB(\ell^2(\Gamma))$) containing the regular representation $\Lambda$ of $C^*_{r}(\cG)$  defined in Section \ref{sec:full}, raises difficulties. Indeed, for $\gamma$, $\gamma_1$ with the same range, we need an analogue $f_{\gamma^{-1}\gamma_1} \in C_{r}^*(\cG)$ of $\delta_{s^{-1}t}$. We observe that if $\Gamma$ is a discrete group and if $g$ denotes the characteristic function of the diagonal of $\Gamma\times \Gamma$ then $\delta_{s^{-1}t}$ is the function $y\mapsto g(s^{-1}t,y)$ defined on $\Gamma$. The function $g$ is (of course) continuous and positive definite on the product group $\Gamma\times \Gamma$. Moreover this function is properly supported in the sense that for every compact ({\it i.e.}, finite) subset $F$ of $\Gamma$, the intersections of the support of $g$ with $F\times \Gamma$ and with $\Gamma \times F$ are compact. Following an idea of Jean Renault developed in an unpublished note, we introduce, for a locally compact groupoid, a property similar to the existence of $g$,  that we call {\it inner amenability} (see Definition \ref{def:wia}). The definition is justified by the fact that, for locally compact groups, it coincides with inner amenability in the sense of \cite{LP}.  This includes every discrete group. On the other hand a connected  locally compact group is  inner amenable if and only if it is amenable \cite{LR87}. 
Inner amenability   insures the existence of functions on the groupoid product $\cG\times \cG$ that behave like the characteristic function of the diagonal for $\Gamma\times \Gamma$.

 Our main result states that if $\cG$ is  a  second countable inner amenable \'etale groupoid, then  the following six conditions are equivalent: (1) strong amenability at infinity; (2) amenability at infinity; (3) nuclearity of the uniform  algebra of the groupoid; (4) exactness of this $C^*$-algebra; (5) KW-exactness of $\cG$; (6) exactness of the reduced $C^*$-algebra  (Theorem \ref{thm:equiv}).  This is applied to show (Corollary \ref{cor:main1}) that if $\cG$ is such that there exists a locally proper (see Definition \ref{defn:lpropre}) continuous homomorphism from itself into an exact discrete group, then the above six  conditions hold. We also prove that for any locally compact groupoid which is equivalent to an \'etale inner amenable groupoid, the conditions (2), (5) and (6) are equivalent (Corollary \ref{cor:equiv}).
 
 One very important feature of countable discrete exact groups is that, for them, the Baum-Connes assembly map with coefficients is split injective \cite{H00}, and therefore they satisfy the Novikov conjecture. More generally, this holds true for second countable strongly amenable at infinity \'etale groupoids \cite{Bon20, BP24, Mao}.

Non-exact discrete groups are quite exotic. The first examples were exhibited by Gromov in \cite{Gro}, often named Gromov monster groups. New examples have been constructed by Osajda \cite{Osa, Osa18}. But still, there are very few examples, whereas in the setting of \'etale groupoids, it is easy to find examples of KW-exact and of non KW-exact groupoids.
Let us consider  the case of an  \'etale groupoid $\cG$ which is a groupoid group bundle $(\cG(x))_{x\in X}$  (Definition \ref{def:bundle_groupoid_cont}). Recall that $\cG$ is amenable if and only if each group $\cG(x)$ is amenable \cite{AD-R, Ren_13}, and in this case $C_{r}^*(\cG)$ is a continuous field of $C^*$-algebras, with fibres $C^*_{r}(\cG(x))$, $x\in X$ \cite{Ram, LR}. The discussion relative to exactness is more subtle. If $\cG$ is KW-exact, then $C_{r}^*(\cG)$ is a still a continuous field of $C^*$-algebras (see Corollary \ref{cor:cont_field}).  Since $C^*_{r}(\cG)$ is then exact,  each quotient $C^*$-algebra $C^{*}_r(\cG(x))$ is exact and therefore the groups $\cG(x)$ are exact. On the other hand, the exactness of the groups $\cG(x)$ is not sufficient to ensure that $C^*_{r}(\cG)$ is exact. A particularly simple example is given by the construction due to  Higson, Lafforgue and Skandalis \cite{HLS} of group bundle groupoids (called HLS-groupoids) associated with the data of an approximated  discrete group $(\Gamma, N_k)$  as defined  in \ref{ex:HLS}. If $\Gamma$ is a non-amenable exact group, the corresponding HLS-groupoid is a bundle of exact groups whose reduced $C^*$-algebra is not exact (see Proposition \ref{prop:HLSbis}).

  A  locally compact groupoid has the {\it weak containment property}, in short the (WCP), if its full and reduced $C^*$-algebras coincide. The  HLS-groupoids  gave the first examples of groupoids  having the (WCP) without being amenable \cite{Wil15}, thus solving by the negative a long-standing open question. 
 It is the lack of a weak form of exactness ({\it inner exactness}, see Definition \ref{def:inamen})  that prevents an HLS-groupoid  having the (WCP) from being amenable (see Proposition \ref{prop:HLS}).

 In the recent years, a lot of work has been carried out to understand the influence of exactness in the study of the problem (B), starting with the paper \cite{Mat}. Very recently, it was established that the amenability of an \'etale groupoid is equivalent to  the (WCP) added with strong amenability at infinity \cite{Kra}. In case the orbit space relative to the action of the groupoid on its space of units is $T_0$, one can weaken the strong amenability at infinity assumption and replace it by inner exactness. This holds not only for \'etale groupoids but   for any locally compact groupoid with $T_0$-orbit space (see \cite{Bon} and Theorem \ref{thm:bon} in this text) and extends the well-known Hulanicki's result for locally compact groups \cite{Hul}.

Since the partial transformation groupoid $\Gamma\ltimes X$ associated with a partial action of an exact discrete group $\Gamma$ on a locally compact space $X$ is strongly amenable at infinity (Proposition \ref{prop:partial}), we see that  $\Gamma\ltimes X$ is amenable if and only if it has the (WCP). This extends Matsumura's result \cite{Mat} and give another proof of the same fact obtained in \cite{BFS}. 

An aspect that has not been considered in this paper is the behaviour of crossed products under the action of an exact groupoid on a $C^*$-algebra, for which we refer to \cite{Lal15}. For additional results about KW-exact groupoids we also refer to \cite{Lal17}. We point out that KW-exact groupoids are simply called exact groupoids in \cite{Lal15, Lal14, Lal17}.

This monograph is mainly organized in two parts. Part 1 (Sections 1 to 6) concerns exclusively  groupoids. Their $C^*$-algebras only appear in Part 2 (Section 7 to 12). We end this introduction by a brief summary of the content of this text.

In the first section we provide the definitions and the notation we need in the rest of the text. We develop in particular the notions of action and partial action of groupoids and of semi-direct products (subsections \ref{subsec-Action} to  \ref{subsec:semi-direct}).  Then we introduce in Subsection \ref{subset:etaleg} the notion of \'etale groupoid, for which inverse semigroup actions provide interesting examples (subsection \ref{subsec:ISA}). Groupoid group bundles  defined in Subsection \ref{subsubsect:gbg} are also an important family of examples. This is followed in subsection \ref{subsec:ext-groupoids} by a detailed analysis of some subtleties in the definition  of groupoid extensions.   In all this examples  we focus on the description of their classical Haar systems. Finally, in Subsection \ref{subsec:gen-mor} the notion of homomorphism between groupoids is extended to that of generalized morphism, and in particular faithful and locally proper such morphisms.

Section  \ref{sec:EquGr} discusses two notions  of  equivalence for groupoids: similarity and the most important one, simply called equivalence of groupoids.  The purpose of this section is to later obtain some informations on the invariance of the various pro\-per\-ties of groupoids studied  in this text. The latter notion includes many interesting examples as shown in Subsection \ref{subsec:eqgroupoid}. These two notions don't coincide in general, but this holds true for ample groupoids.  

 Section \ref{sec:2} is devoted to the generalization of the $\Gamma$-equivariant  compactifications of  a locally compact space on which the discrete group $\Gamma$ acts by homeomorphisms, when replacing $\Gamma$ by an \'etale groupoid $\cG$. The analogue for $\cG$ of the Stone-\v Cech  compactification of $\Gamma$ on which $\Gamma$ acts by translations will play a very important role in the sequel.
 
 The notion of amenable groupoid has been widely studied in several publications. However, since amenable actions are so present in our text, we devote Section \ref{sec:Amen} to giving some informations on this subject. In Subsections \ref{subsec:amenable} and \ref{subsec:stable} we recall several equivalent definitions of amenability and then point out some of its permanence properties. Amenability of extensions is studied in detail in Subsection \ref{subsec:amenext} with an application to the amenability of partial actions of discrete groups.
 
 Inner amenability is introduced and studied in Section \ref{section:IA}. After having first established  several permanence properties of this notion in Subsection \ref{subsection:IA}, we discuss in Subsection \ref{subsec:aboutIA} its invariance  under two forms of equivalence. Despite positive examples given in this section, we don't know  whether all \'etale groupoids are inner amenable.
 
 Section \ref{sec:ameninf} deals with the first notion of exactness that interests us in this text, namely amenability at infinity.  Various properties  of this structure are studied in Subsection \ref{def:ameninf}, followed by the study of the case of extensions in the subsection \ref{subsec:iue}. The two last subsections \ref{subsec:aietale} and \ref{subsec:charai} are limited to the case of \'etale groupoids, ending with a handy characterization of amenability at infinity,.

 Section \ref{groupoid-alg} is devoted to the description of the various $C^*$-algebras associated with a groupoid once a left Haar system is fixed: full and reduced $C^*$-algebras in Subsection \ref{sec:full}, the uniform  $C^*$-algebra in Subsection \ref{subsec:Roe}, and crossed products relative to actions on $C^*$-algebras in Subsection \ref{subsect:Gaction}.

  Section \ref{sect:ext} contains our first approach to various definitions of exactness, the classical ones in Subsection \ref{subsec:exact}, and a new one in Subsection \ref{subsec:inex}. We called it inner exactness in \cite{AnD16} since it only take  into account the action of a groupoid on its unit space. It is a weak notion of exactness that holds for instance for every locally compact group. However, for a HLS-groupoid it is equivalent to its amenability.

In Section \ref{sec:comparison}  we are starting the comparison between the different forms of exactness previously introduced. Generalities are presented for locally compact groupoids with Haar system in Subsection \ref{subsec:comparison}. The case of  étale groupoids is treated in the subsection \ref{subsec:compar-etale} where a chain of implications, going from amenability to infinity to the weakest notion of $C^*$-exactness is established. This is followed by several examples in Subsection \ref{subsec:examples}.

In Section \ref{sect:eiai} we show the equivalence, for \'etale inner amenable groupoids, of all the notions of exactness studied in the previous section.

 In Section \ref{sec:WC} the role of exactness is highlighted in the study of the weak containment problem: first, inner exactness for locally compact groupoids with a $T_0$ orbit space (subsection \ref{subsec:ia}), then strong amenability at infinity for \'etale groupoids (subsection \ref{subsec:tg}).
 
 Open questions and comments are discussed in Section \ref{sec:open}.

The appendix  is a purely topological part concerning the various fibrewise compactifications  of a fibre space $(Y,p)$. It can be read independently of the rest of the paper. Several different notions of fibrewise compactifications have been studied in the past (see \cite[\S 8]{James}).   Here we consider  a more restrictive si\-tuation since we want our fibrewise compactifications to be locally compact  and to contain $Y$ as a dense open subspace. The compactifications by adding appropriate ultrafilters that are constructed in \cite{James} are not convenient for our purpose. Our compactifications are obtained as Gelfand spectra of well-chosen abelian $C^*$-algebras that we characterize.  In particular, there is a smallest one and a greatest one. They give rise respectively to the Alexandroff fibrewise compactification $(Y^{+}_p, p^+)$ and to the Stone-\v{C}ech fibrewise compactification $(\beta_p Y, p_\beta)$ of $(Y,p)$. 

We  point out that there are still several important points that remain to be clarified\footnote{For a more precise explanation of the issues, we refer to the section \ref{sec:open}.}:
\begin{itemize}
\item Is KW-exactness equivalent to (strong) amenability at infinity for locally compact groupoids, as it is the case for locally compact groups?
\item Is $C^*$-exactness equivalent to (strong) amenability at infinity for \'etale grou\-poids, as it is the case for discrete groups?
\item Is it true that amenability is equivalent to the (WCP) + inner exactness for any locally compact grou\-poid (thus generalizing Hulanicki's theorem \cite{Hul})?
\item At least, is  it true that amenability is equivalent to the (WCP) + (strong) amenability at infinity for any locally compact groupoid?
\end{itemize}

\noindent{\bf Conventions.} In order to exclude pathological cases, in this text  locally compact spaces are always assumed to be Hausdorff and second countable even if it is not always expli\-citly mentioned and needed, with some obvious exceptions, like Stone-\v{C}ech compactifications. By a measure on a locally compact space, we always mean a Radon measure.

Throughout, we only consider locally compact groupoids and we assume that their range and source maps   are  open. In Part 2, from Section 7 to Section 12, the groupoids are implicitly endowed with a left Haar system.

For us, a homomorphism between $C^*$-algebras preserves the algebraic operations and also the involution.

\noindent{\bf Acknowledgments.} In this monograph, many  results are due  to nu\-me\-rous researchers, besides  some new results from me. I have tried to cite the original sources to the best of my knowledge. I apologize for the inevitable omissions.

 I thank Christian B\"onicke  and Jean Renault who authorized me to include some of their private results pertaining to the subject of this monograph. I am also grateful to Julian Kranz for interesting discussions on the topic.

\newpage

\part{}

\section{\textbf{\textsc{Preliminaries: definitions and examples}}}

In this section we recall some more or less basic facts on groupoids, as well as some useful constructions and examples.

\subsection{Groupoids} We assume that the reader is familiar with the basic definitions about groupoids. For  details we refer to \cite{Ren_book}, \cite{Pat}, \cite{Will19}. Let us recall some  notation and terminology.  A {\it groupoid} consists of a set $\cG$, a subset $\cG^{(0)}$ called the set of {\it units}, two maps $r,s :\cG \to \cG^{(0)}$ called respectively the {\it range} and {\it source} maps,  a {\it composition law} $(\gamma_1,\gamma_2) \in \cG^{(2)} \mapsto \gamma_1\gamma_2\in \cG$, where
$$\cG^{(2)} = \set{(\gamma_1,\gamma_2)\in \cG\times \cG : s(\gamma_1) = r(\gamma_2)},$$
and an {\it inverse} map $\gamma\mapsto \gamma^{-1}$. These operations satisfy obvious rules, such as the facts   that the composition law ({\it i.e.}, product) is associative, that the elements of $\cG^{(0)}$ act as units ({\it i.e.}, $r(\gamma)\gamma = \gamma = \gamma s(\gamma)$), that $\gamma\gamma^{-1} = r(\gamma)$, $\gamma^{-1}\gamma = s(\gamma)$, and so on (see \cite[Definition 1.1]{Ren_book}). Given a subset $E$ of $\cG^{(0)}$ we set $\cG^E = r^{-1}(E)$, $\cG_E = s^{-1}(E)$\index{$\cG^E$, $\cG_E$} and $\cG(E) = \cG^E\cap \cG_E$\index{$\cG(E)$}. When $E= \set{x}$ we rather write $\cG^x$, $\cG_x$ \index{$\cG^x$, $\cG_x$} and $\cG(x)$ respectively\index{$\cG(x)$}. Usually, $X$ will denote the set of units of $\cG$.

 A subgroupoid $\cH$ of  a groupoid $\cG$ is a subset of $\cG$ which is stable under product and inverse. For instance,  let $E$ be a subset of  $\cG^{(0)}$. Then $\cG(E)$ is a  subgroupoid of $\cG$ called the {\it reduction of $\cG$ by $E$}\index{reduction of a groupoid}. When $E$ is reduced to a single element $x$, then $\cG(x)$ is a  group called the {\it isotropy group of $\cG$ at $x$}\index{isotropy group}.

A {\it locally compact groupoid}\index{locally compact groupoid} is a groupoid $\cG$ equipped with a  locally compact topology such  that the structure maps are continuous, where $\cG^{(2)}$ has the topology induced by $\cG\times\cG$ and $\cG^{(0)}$ has the topology induced by $\cG$. Since $\cG$ is Hausdorff, $\cG^{(0)}$ is a closed subset of $\cG$.  Furthermore, {\it we assume that the range (and therefore the source) map is open}.

We say that a locally compact groupoid $\cG$ is {\it proper} \index{proper groupoid}if the map $r\times s: \cG \to \cG^{(0)}\times \cG^{(0)}$ is proper.

 A  {\it locally compact subgroupoid}\index{subgroupoid (locally compact)} $\cH$ of $\cG$ is a subgroupoid of $\cG$ that is locally compact with respect to the induced topology ({\it i.e.}, if and only if it is locally closed) and such that the restrictions of the range and source maps to $\cH$ are open.  
 For instance, $\cG^{(0)}$ is a closed subgroupoid of $\cG$. Note also that if $E$ is an {\it invariant locally compact subset} \index{invariant subset of units} of $\cG^{(0)}$ (that is, $r(\gamma)\in E$ if and only if $s(\gamma)\in E$), then the reduction $\cG(E)$ is a locally compact subgroupoid. 
 
 \begin{rem} Let us recall that a locally compact subgroup of a locally compact group is closed. This is not the case in the more general setting of groupoids. In particular, an open subgroupoid is not necessarily closed, even if it has the same units as the ambient groupoid.
  \end{rem}

 Let $\cG$ and $\cH$ be two locally compact groupoids. A {\it homomorphism} \index{homomorphism of groupoids} $c:\cG\to\cH$ is a continuous map such that for every $(\gamma,\eta)\in \cG^{(2)}$ we have $(c(\gamma),c(\eta))\in \cH^{(2)}$ and $c(\gamma\eta)=c(\gamma)c(\eta)$. It follows that $c(\cG^{(0)}) \subset\cH^{(0)}$ and we denote by $c^{(0)}:\cG^{(0)}\to \cH^{(0)}$ \index{$c^{(0)}$} the restriction of $c$ to $\cG^{(0)}$. For $\gamma\in \cG$, we have $c^{(0)}(r(\gamma)) = r(c(\gamma))$ and $c^{(0)}(s(\gamma)) = s(c(\gamma))$.

 Given a locally compact space $Y$, we denote by $\cC_b(Y)$ \index{$\cC_b(Y)$} the algebra of continuous bounded complex valued functions on $Y$, by $\cC_0(Y)$ \index{$\cC_0(Y)$: $Y$ locally compact space} its subalgebra of functions vanishing at infinity, and by $\cC_c(Y)$ \index{$\cC_c(Y)$} the subalgebra of continuous functions with compact support. The support of $f \in \cC_b(Y)$ will be denoted by $\supp(f)$\index{$\supp(f)$}.

 \begin{defn}\label{def:Haar} Let $\cG$ be a locally compact groupoid. A (left) {\it Haar system}\index{Haar system} on $\cG$ is a family $ \lambda=(\lambda^x)_{x\in X}$ of measures on $\cG$, indexed by the set $X= \cG^{(0)}$ of units, satisfying the following conditions:
 \begin{itemize}
 \item {\it Support}: $\lambda^x$ has exactly $\cG^x$ as support, for every $x\in X$;
 \item {\it Continuity}: for every $f\in \cC_c(\cG)$, the function $x\mapsto  \lambda(f)(x)=\int_{\cG }f\rd\lambda^{x}$ is continuous;
\item {\it Invariance}: for $\gamma\in \cG$ and $f\in \cC_c(\cG)$, we have 
$$\int_{\cG} f(\gamma\gamma_1) \rd\lambda^{s(\gamma)}(\gamma_1) =  \int_{\cG} f(\gamma_1) \rd\lambda^{r(\gamma)}(\gamma_1).$$
\end{itemize}
\end{defn}

Let $\cG$ be a locally compact groupoid which has a Haar system $\lambda$. If $E$ is an  invariant locally compact subset  of $\cG^{(0)}$,  then the reduction $\cG(E)$ has a Haar system, namely the restriction of $x\mapsto \lambda^x$ to $E$.   Note that for a Haar system to exist,  the range (and therefore the source) map of $\cG$ must be open \cite[Proposition 2.4, Chapter 1]{Ren_book}.

 The most basic examples of locally compact groupoids are locally compact groups, {\it locally compact spaces} (or diagonal equivalence relations) and {\it trivial equiva\-lence relations} (or {\it pair groupoids})\index{trivial equivalence relation}\index{pairgroupoid}. These examples admit Haar systems. For a locally compact group it is its Haar measure, unique up to a multiplicative positive constant. For a locally compact space $\cG = X$, we have $\cG^{(0)} = X$ and $\cG^x = \set{x}$ for every $x\in X$. Then any map $x\mapsto \lambda^x = f(x)\delta_x$, where $f$ is any positive continuous function and $\delta_x$ is the Dirac mass\index{$\delta_x$: Dirac mass at $x$} at $x$, is a Haar system. For the trivial equivalence relation $\cG= X\times X$ on $X$ we have $\cG^{(0)} = X$ and $\cG^x = \set{x}\times X$. Then $x\mapsto \lambda^x = \delta_x\times \mu$, where $\mu$ is any measure on $X$ whose support is $X$, is a Haar system.

Given two locally compact groupoids $\cG$ and $\cH$, the {\it product} $\cG\times \cH$ is a locally compact groupoid \index{product groupoid} under the product topology and the coordinatewise operations. Its space of units is identified with $\cG^{(0)}\times\cH^{(0)}$. If $\cG$ and $\cH$ have a  Haar system, then $\cG\times \cH$ has a Haar system, namely the product of the two Haar systems.

Among the examples that we are going to describe, a very important class is that of semi-direct pro\-ducts. For their study, we first need to introduce actions and partial actions of groupoids.

\subsection{Actions of groupoids on locally compact spaces}\label{subsec-Action} Let $X$ be a locally compact space. A {\it fibre space}\index{fibre space} over $X$ is a pair $(Y,p)$ where $Y$ is a locally compact space  and $p$ is  a continuous  map from  $Y$ into $X$. For $x\in X$ we denote by $Y^x$ the {\it fibre} $p^{-1}(x)$. 

Let $ (Y_i,p_i)$, $i=1,2$, be two fibre spaces over $X$. We denote by $Y_{1}   
 \,_{p_1}\!\!*_{p_2} Y_2$ (or $Y_1*Y_2$ when there is no ambiguity) the {\it fibred product}\index{fibred product} 
 $$\set{(y_1,y_2)\in Y_1\times Y_2: p_1(y_1) = p_2(y_2)}$$
  equipped  with the topology induced by the product topology. For subsets $A_1$ and $A_2$ of $Y_1$ and $Y_2$ respectively, we use similarly the notation $A_{1}   \,_{p_1}\!\!*_{p_2} A_2$.

 A {\it morphism} from $(Y_1, p_1)$\index{morphism of fibre spaces} into $(Y_2,p_2)$ is a continuous map $f: Y_1\to Y_2$ such that $p_2\circ f = p_1$.

The following observation will be useful later.
 
\begin{lem}\label{lem:easy} Let Let $ (Y_i,p_i)$, $i=1,2$ as above. We denote by $\pi_1: Y_{1} \,_{p_1}\!\!*_{p_2} Y_2 \to Y_1$ the first projection.
\begin{itemize}
\item[(i)] If $p_2$ is open, then $\pi_1$ is open;
\item[(ii)] If $p_2$ is proper, then $\pi_1$ is proper.
\end{itemize}
\end{lem}

\begin{proof} Assume first that $p_2$ is open and let us show that the  map $\pi_1$ is  open.  Let $\Omega$ be an open subset of $ Y_{1}\, _{p_1}\!\!*_{p_2} Y_2 \to Y_1$. Let $y_0 \in \pi_1(\Omega)$ and let $z_0\in Y_2$ be such that $(y_0,z_0)\in \Omega$. There exist open neighborhoods $U$ of $y_0$ and $V$ of $z_0$ such that $U\,_{p_1}\!*_{p_2} V\subset \Omega$. Then we have $y_0\in p_{1}^{-1}(p_2(V))\cap U\subset \pi_1(\Omega)$ and $p_{1}^{-1}(p_2(V))\cap U$ is an open neighborhood of  $y_0$ since $p_2$ is open.

Assume now that $p_2$ is proper. Then $\pi_1$ is proper since for every compact subset of $Y_1$ we have $\pi_1^{-1}(K) = K\,_{p_1}\!\!*_{p_2} p_2^{-1}(p_1(K))$.
\end{proof}

 \begin{defn}\label{def:Gspace} Let $\cG$ be a locally compact groupoid. A {\it left} $\cG$-{\it space} \index{left $\cG$-space} is a fibre space $(Y,p)$ over $X = \cG^{(0)}$, equipped with a continuous map $(\gamma, y) \mapsto \gamma y$ from $\cG\,_s\!*_pY$ into $Y$, satisfying the following conditions:
 \begin{itemize}
 \item $p(\gamma y) = r(\gamma)$ for $(\gamma, y) \in \cG\,_s\!*_pY$, and $p(y)y = y$ for $y\in Y$;
 \item if $(\gamma_1, y) \in \cG\,_s\!*_pY$ and $(\gamma_2,\gamma_1)\in \cG^{(2)}$, then $(\gamma_2\gamma_1)y = \gamma_2(\gamma_1 y)$.
 \end{itemize}
 \end{defn}
  The map $p$ will be called the {\it moment map}\index{moment map} of the $\cG$-space. We will also say that the map $(\gamma, y) \mapsto \gamma y$ is a {\it $\cG$-action on} $Y$ and we will use the notation $\cG\actson(Y,p)$ or simply $\cG\actson Y$\index{$\cG\actson Y$}.
  
  \begin{defn}\label{def:principal} One says that the $\cG$-action on $Y$ is {\it free}\index{free action} if $\gamma y = y$ implies  $\gamma = p(y)$, and that it is {\it proper}\index{proper action} if the map $(\gamma,y) \mapsto ( \gamma y, y)$ from $\cG\,_s\!*_pY$ to $Y\times Y$ is proper. When $\cG\actson Y$ is free and proper with an open surjective moment map, one says that $Y$ is a {\it principal $\cG$-bundle}. \index{principal  $\cG$-bundle}
\end{defn}
  
 A {\it  $\cG$-equivariant morphism}  from a left $\cG$-space $(Y_1,p_1)$ to a left $\cG$-space $(Y_2,p_2)$ is a continuous morphism $f: (Y_1,p_1) \to (Y_2,p_2)$ such that $f(\gamma y) = \gamma f(y)$ for every $(\gamma, y) \in \cG\,_s\!*_{p_1}Y$.
  
 {\it Right $\cG$-spaces}\index{right $\cG$-space} are defined similarly. A right $\cG$-space $(Y,p)$ can be viewed as a left $\cG$-space by setting $\gamma y = y \gamma^{-1}$ if $r(\gamma) = p(y)$.  This allows left and right actions  to be treated in the same way. {\it Without further precisions, a $\cG$-space will be a left $\cG$-space} but we will also deal with right $\cG$-spaces from time to time. 
 
 Let us observe that $(\cG,r)$ is a left $\cG$-space in an obvious way, as well as $X= \cG^{(0)}$. It this latter case, the action of $\gamma\in s^{-1}(x)$ onto $x\in X$ will be denoted by $\gamma\cdot x$, \index{$\gamma\cdot x$}in order to distinguish it from $\gamma x = \gamma$.  By definition, we have $\gamma\cdot x = r(\gamma)$. We also note that if $(Y,p)$ is a left $\cG$-space, then $p$ is $\cG$-equivariant: $p(\gamma y) = \gamma\cdot p(y)$, and therefore $p(Y)$ is an invariant subset of $X$. 
 
 \begin{rem} We warn the reader that our definition of $\cG$-space is different from the  definition sometimes found in the literature where it is required that $p$ is an open map. We do not assume this property because the moment map of the action of a locally compact groupoid on its Stone-\v Cech fibrewise compactification is not always open (see  remarks \ref{rem:notopen},   \ref{rem:compG}). Moreover, it is a fact that, even if $p$ is not open, the corresponding semi-direct product groupoid defined later has an open range (see Lemma \ref{lem:open}).
 
 When it is not necessary we do not assume that $p$ is surjective. \end{rem}
 
 The following lemma will be useful later. Let $(Y,p)$ be a left $\cG$-space. For $g\in \cC_c(\cG)$ and $f\in \cC_b(Y)$ we define the convolution product  $g*f$ \index{$g*f$} by
 $$(g*f)(y) = \int_{\cG} g(\gamma) f(\gamma^{-1}y) \rd \lambda^{p(y)}(\gamma).$$
 
 \begin{lem}\label{lem:convol} If $f\in \cC_c(Y)$ $($resp. $f\in \cC_0(Y)$$)$ we have $g*f\in \cC_c(Y)$ $($resp. $g*f\in \cC_0(Y)$$)$.
 \end{lem}
 
 \begin{proof} By approximation, it suffices to consider the case $f\in \cC_c(Y)$. The main point is to show the continuity of $g*f$. We follow the proof given in \cite[Proposition 1.33]{Will19}. We first observe that the function $h:(\gamma,y)\mapsto  g(\gamma) f(\gamma^{-1}y)$ is continuous with compact support on $\cG \,_r\!*_p Y$. Since $\cG\times Y$ is second countable, hence normal, we can extend $h$ to an element of $\cC_b(\cG\times Y)$. After multiplication by a continuous function with compact support that is identically $1$  on the support of $h$  we get 
 a function $h'\in \cC_c(\cG\times Y)$ such that $h'(\gamma,y) = h(\gamma,y)$ for all $(\gamma,y) \in \cG \,_r\!*_p Y$. Therefore we may assume that $h \in \cC_c(\cG\times Y)$ and finally that $h$ is of the form $(\gamma, y) \mapsto g'(\gamma)f'(y)$ with $g'\in \cC_c(\cG)$ and $f'\in\cC_c(Y)$ since the algebra of finite sums of such functions is dense in $\cC_c(\cG\times Y)$. The conclusion follows immediately.
  \end{proof}
 
\subsection{Partial actions of groupoids on locally compact spaces}\label{subsec:PartialAction}  Partial actions of groups were introduced in the 1990s. They encompass many interesting examples (see \cite{Exel17}). In particular, we will see in the subsection \ref{subsec:ISA} that many semigroups provide partial transformation groupoids. We will need in the proof of Proposition \ref{prop:partial} the more general notion of partial action of a locally compact groupoid. 

\begin{defn}\label{def:part top}$($\cite[Definition 5.1]{AD20}$)$ Let $\cG$ be a locally compact groupoid with  $\cG^{(0)} = X$. Let $Y$ be a locally compact space. A {\it partial} (left) {\it action of $\cG$ on }\index{partial action} $Y$ is the data of a  fibre space $p: Y\to X$ equipped with a pair $\boldsymbol{\beta} = ((\beta_\gamma)_{\gamma\in \cG}, (Y_\gamma)_{\gamma\in \cG})$, where for $\gamma\in \cG$, $Y_\gamma$ is an open subset of $Y^{r(\gamma)} = p^{-1}(r(\gamma))$ and $\beta_\gamma$ is a homeomorphism from $Y_{\gamma^{-1}}$ onto $Y_\gamma$, such that
\begin{itemize}
\item[(i)] $Y_x = Y^x$ and $\beta_x = \Id_{Y_x}$ for $x\in p(Y)$;
\item[(ii)] $\beta_{\gamma\gamma_1}$ is an extension of $\beta_{\gamma}\circ \beta_{\gamma_1}$ if $s(\gamma) = r(\gamma_1)$;
\item[(iii)] $C = \set{(y,\gamma)\in Y_p \!*_r \cG: y\in Y_\gamma}$ is open in $Y_p\! *_r \cG$;
\item[(iv)] the map $(y,\gamma) \mapsto \beta_{\gamma^{-1}}(y)$ is continuous from $C$ into $Y$.
\end{itemize}
\end{defn}

\begin{rem} (see \cite[Lemma 1.2]{QR}) Let $\boldsymbol{\beta} : \cG \actson Y$ be a partial action. Then we have
\begin{itemize}
\item[(1)] $\beta_{\gamma^{-1}} = (\beta_{\gamma})^{-1}$ for every $\gamma\in \cG$.
\item[(2)] $\beta_{\gamma}(Y_{\gamma^{-1}} \cap Y_{\gamma_1}) = Y_\gamma\cap Y\gamma\gamma_1$ if $s(\gamma) = r(\gamma_1)$;
\item[(3)] $\beta_\gamma\circ \beta_{\gamma_1}$  is a homeomorphism from $Y_{\gamma_{1}^{-1}} \cap Y_{\gamma_{1}^{-1}\gamma^{-1}}$ onto $Y_\gamma \cap Y_{\gamma\gamma_1}$ if $s(\gamma) = r(\gamma_1)$, and $\beta_\gamma\circ \beta_{\gamma_1}(y)  = \beta_{\gamma\gamma_1}(y)$ if $y\in Y_{\gamma_{1}^{-1}} \cap Y_{\gamma_{1}^{-1}\gamma^{-1}}$.
\end{itemize}
\end{rem}

\begin{rem} The map $(y,\gamma)\mapsto (\gamma^{-1},y)$ is a homeomorphism from $Y_p \!*_r \cG$ onto $\cG_s \! *_p Y$ sending $C$ onto $C' = \set{(\gamma,y)\in \cG_s \! *_p Y: y\in Y_{\gamma^{-1}}}$.

Condition (iv) means that $(y,\gamma) \mapsto (\gamma,\beta_{\gamma^{-1}}(y))$ is continuous from $C$ onto $C'$. It is a homeomorphism with inverse map $(\gamma,y) \mapsto (\beta_\gamma(y), \gamma)$.
\end{rem}

Of course, a $\cG$-action $\cG\actson (Y,p)$ is a particular example of partial action, where $Y_\gamma = Y^{r(\gamma)}$ for all $\gamma\in \cG$.

\begin{ex}\label{ex:reducpartial} Let $\alpha: G\actson Y$ be a continuous left action of a locally compact group $G$ on a locally compact space $Y$. Let $X$ be an open subset of $Y$. For $g\in G$ we set $X_g = X\cap \alpha_g(X)$ and we denote by $\beta_g$ the restriction of $\alpha_g$ to $X_{g^{-1}}$. Then $\boldsymbol{\beta } = (\set{X_g}_{g\in G}, \set{\beta_g}_{g\in G})$  is a partial action of $G$ on $X$. We will say that this action is the {\it reduction of $\alpha$ to} $X$. We observe that, for every $g\in G$ the graph $\set{(x,\beta_g(x)) : x\in X_{g^{-1}}}$ is closed into $X\times X$.

Let us consider the group $\Z_2 = \set{0,1}$, $X= ]-1,+\infty[$ and $X_1= ]-1,1[$. The partial action of $\Z_2$ on $X$ such that $\beta_1 = \Id_{X_1}$ is not a reduction to $X$, whereas the partial action such that $\beta_1(x) = -x$ for $x\in X_1$ is the reduction to $X$ of the global action of $\Z_2$ on $\R$ such that $\beta_1(x) = -x$ for all $x\in \R$.

The study of partial group actions that are reductions of a global action  will be developed in the example \ref{ex:Abd}.
\end{ex}

\subsection{Semi-direct product groupoids associated with partial actions}\label{subsec:semi-direct} 
Let $\boldsymbol{\beta }: \cG \actson Y$ be a partial action.  We define as follows a new locally compact groupoid $\cG_\beta$, called the {\it semi-direct pro\-duct groupoid} or the {\it transformation groupoid}\footnote{when we want to emphasize that it is a  semi-direct product groupoid relative to an action on a space. Compare with Subsection \ref{subsec:gg}.}\index{semi-direct product groupoid} \index{transformation groupoid} associated to $\boldsymbol{\beta }: \cG \actson Y$. It is the topological subspace 
$$\cG_\beta= \set{(x,\gamma,y): \gamma\in \cG, y\in Y_{\gamma^{-1}}, x= \beta_{\gamma}(y)}$$
of $Y\times \cG\times Y$ equipped with the following operations\footnote{We use the notation ${\underline r}$, ${\underline s}$ when we want to distinguish the range and source maps of $\cG_\beta$ from those of $\cG$.}:  ${\underline r}(x,\gamma,y) = x$ (where $x$ is identified with $(x,p(x),x)$), ${\underline s}(x,\gamma,y) = y$, \index{$\underline r$}\index{$\underline s$}the composition law is given by $(x,\gamma,y)(y,\eta,z) = (x,\gamma\eta,y)$ and the inverse is given by $(x,\gamma,y)^{-1} = (y,\gamma^{-1},x)$.
This construction dates back to \cite{Abad04} for partial group actions.

Note that the map $(x,\gamma,y) \mapsto (\gamma,y)$ is a homeomorphism from this space $\cG_\beta$ onto the space $C'= \set{(\gamma,y)\in \cG\times Y: y \in Y_{\gamma^{-1}}}$ (viewed as a topological subspace of $\cG\times Y$) on which one can read the structure of groupoid in another familiar way (after having identified $Y$ with  $\set{(p(y),y):y\in Y}$): 
$${\underline r}(\gamma,y) =\beta_\gamma(y),\quad {\underline s}(\gamma,y) = y,\quad (\gamma,y)^{-1} = (\gamma^{-1}, \beta_\gamma(y)),\quad (\gamma,\beta_\eta(y))(\eta,y) = (\gamma\eta,y).$$
We denote by $\cG\ltimes_\beta Y$ \index{$\cG\ltimes_\beta Y$} this version of $\cG_\beta$.

Of course, we can equivalently consider the isomorphism $(x,\gamma,y) \mapsto (x,\gamma)$, defi\-ning the isomorphic version $C=Y\rtimes_\beta \cG$\index{$Y\rtimes_\beta \cG$} of the groupoid $\cG_\beta$.  Again we identify   $(Y\rtimes_\beta \cG)^{(0)}$ with $Y$.      The range of $(y,\gamma)$ is $y$ and its source is $\beta_{\gamma^{-1}}(y)$.  The product is given by
 $$(y,\gamma)(\beta_{\gamma^{-1}}(y),\gamma_1) = (y, \gamma\gamma_1).$$
 The inverse is given by
 $$(y,\gamma)^{-1} = (\beta_{\gamma^{-1}}(y),\gamma^{-1} ).$$
 
 In particular, the map $(\gamma,y)\mapsto (\beta_\gamma(y), \gamma)$ is an isomorphism from the groupoid $\cG\ltimes_\beta Y$ onto the groupoid $Y \rtimes_\beta \cG$. 
  
  Let us observe that the groupoid $Y\rtimes_\beta \cG$ is  locally compact  since $Y\!_p\!*_r \cG$ is closed into $Y\times \cG$ and $C$ is open into $Y\!_p\!*_r \cG$. Moreover, as shown below, its range and source maps are open The same observation applies for $\cG\ltimes_\beta Y$.

\begin{rem} Usually, we will drop the index $\beta$ in the notation of  semi-direct product groupoids, when there is no ambiguity. We will use one or the other version of the semi-direct product depending on the context. We will often write $\gamma y$ instead  of $\beta_\gamma(y)$, and we write $r$ (resp. $s$) instead of $\underline r$ (resp. $\underline s$),  for the range (resp. source) map on semi-direct product groupoids when there is no risk of confusion.
\end{rem}

\begin{rem}\label{rem:right} All this applies in particular to the case of a global left action $\beta:\cG\actson Y$.  Recall that as topological spaces $Y\rtimes_\beta  \cG= Y   
 \,_p\!*_r \cG$  and $\cG\ltimes_\beta Y = \cG \,_s\!*_p Y$. 
 
 Whereas we only work with left partial actions, sometimes, right global actions will be considered. To   a {\it right} global action $\alpha : \cG\actson Y$  the corresponding the semi-direct product groupoid $Y\rtimes_\alpha \cG$ \index{$Y\rtimes_\alpha \cG$} is defined as  $Y \,_p\!*_r \cG$ equipped with the following structure:
$$r(y,\gamma) = y; \,\,s(y,\gamma) = y\gamma; \,\,(y,\gamma)(y\gamma,\gamma_1) = (y,\gamma\gamma_1); \,\,(y,\gamma)^{-1} = (y\gamma,\gamma^{-1}).$$
Then $Y\rtimes_\alpha \cG$ is nothing else than $Y\rtimes_{\beta} \cG$, where $\beta: (\gamma,y)\mapsto \gamma y = y\gamma^{-1}$ is  the left global action associated with $\alpha$.
\end{rem}

\begin{lem}\label{lem:open} Let $\boldsymbol{\beta }: \cG \actson (Y,p)$ be a partial action.The range and source maps of the groupoid $Y\rtimes_\beta \cG$ are open. \end{lem}

\begin{proof}  This follows from Lemma \ref{lem:easy}, since $r$ is open and $C = Y\rtimes_\beta \cG$ is open in $Y\!_p\!*_r\cG$.
\end{proof}

Let $\boldsymbol{\beta }: \cG \actson (Y,p)$ be a partial action and let us  assume that $\cG$ is given with a  Haar system $\lambda =(\lambda^x)_{x\in X}$ where $X = \cG^{(0)}$. We are  going to associate a canonical   Haar system for $Y\rtimes_\beta \cG$.
We set $C_y = \set{\gamma\in \cG: (y,\gamma) \in C} = \set{\gamma\in \cG: y\in Y_\gamma}$. Then we have $(Y\rtimes_\beta \cG)^y = \set{y} \times C_y$. Since $C$ is an open subset of $Y\!_p\!*_r\cG$, we see that $C_y$ is an open subset of $\cG^{p(y)}$.

When $B = \set{y}\times B_1$ is a Borel subset of $\set{y} \times C_y$, we set $\tilde{\lambda}^y(B) = \lambda^{p(y)}(B_1)$.

\begin{prop}$($\cite[Proposition 2.2]{Abad04}$)$\label{prop:Abad2.2} The family $\tilde\lambda =(\tilde{\lambda}^y)_{y\in Y}$ is a Haar system on $Y\rtimes_\beta \cG$.
\end{prop}

\begin{proof} Since  $C_y$ is open in $\cG^{p(y)}$ we see that the support of $\tilde{\lambda}^y$ is $ \set{y} \times C_y$. 
Let us show the left invariance, meaning that for every $f\in \cC_c(Y\rtimes \cG)$ and for every $(y,\gamma) \in Y\rtimes_\beta \cG$ we have
\begin{equation}\label{eq:leftinv}
\int_{(Y\rtimes\cG)^{\gamma^{-1}y}} f((y,\gamma)(y_1,\gamma_1))\rd \tilde{\lambda}^{\gamma^{-1}y}(y_1,\gamma_1) = \int_{(Y\rtimes\cG)^{y}} f(y_1,\gamma_1) \rd \tilde{\lambda}^y(y_1,\gamma_1).
\end{equation}
The right-hand side is
$$\int_{(Y\rtimes\cG)^{y}} f(y_1,\gamma_1) \rd \tilde{\lambda}^y(y_1,\gamma_1) = \int_{\cG^{r(\gamma)}} \mathds{1}_{C_y}(\gamma_1)f(y,\gamma_1) \rd {\lambda}^{r(\gamma)}(\gamma_1).$$

As for the left-hand side   in \eqref{eq:leftinv}, we integrate on the set of $(y_1,\gamma_1)$ such that $y_1 = \gamma^{-1} y\in Y_{\gamma_1}$, that is on the set of $( \gamma^{-1} y, \gamma_1)$ with 
$$y\in \gamma (Y_{\gamma^{-1}}\cap Y_{\gamma_1}) = Y_\gamma\cap Y_{\gamma\gamma_1}.$$
It follows that
\begin{align*}\int_{(Y\rtimes\cG)^{\gamma^{-1}y}} f((y,\gamma)(y_1,\gamma_1))\rd \tilde{\lambda}^{\gamma^{-1}y}(y_1,\gamma_1)& =
\int_{\cG^{s(\gamma)} } \mathds{1}_{C_{\gamma^{-1}y}}(\gamma_1)f(y,\gamma\gamma_1) \rd {\lambda}^{s(\gamma)}(\gamma_1)\\
&=\int_{\cG^{s(\gamma)} }\mathds{1}_{C_y}(\gamma\gamma_1)f(y,\gamma\gamma_1) \rd {\lambda}^{s(\gamma)}(\gamma_1).
\end{align*}
This proves the invariance.

It remains to show that for $f\in \cC_c(Y\rtimes_\beta \cG)$ the map 
$$y\mapsto \int_{(Y\rtimes_\beta\cG)^{y}} f(y,\gamma) \rd \tilde{\lambda}^y(y,\gamma)= \int_{\cG^{p(y)}} \mathds{1}_{C_y}(\gamma)f(y,\gamma) \rd {\lambda}^{p(y)}(\gamma) $$ is continuous.

We have $$\int_{\cG^{p(y)}} \mathds{1}_{C_y}(\gamma)f(y,\gamma) \rd {\lambda}^{p(y)}(\gamma) = \int_{\cG^{p(y)}}f(y,\gamma) \rd \lambda^{p(y)}(\gamma)$$ since the support of $f$ is contained in $C$. 

We will show that for every $f\in \cC_c(Y\times_\beta\cG)$, the function 
$$y\mapsto \int_{\cG^{p(y)}}f(y,\gamma) \rd \lambda^{p(y)}(\gamma)$$
is continuous. This will conclude the proof, after having observed  that  every function $f$ in $\cC_c(C)$ can be extended as an element of $\cC_c(Y\times \cG)$. Indeed, $C$ being open into $Y\!_p\!*_r\cG$ we first extend $f$ as an element of $\cC_c(Y\!_p\!*_r\cG)$ and then we use   the Tietze extension theorem \cite[Lemma 1.42]{Will}).

Let $f$ be in $\cC_c(Y\times \cG)$ and let $\Omega_1\times \Omega_2$ be an open neighborhood of the support of $f$, where $\Omega_1$ and $\Omega_2$ are relatively compact. We have 
$$M=\sup_{y\in \Omega_1}\lambda^{p(y)}(\Omega_2) < +\infty.$$
 Let us consider the algebra $\cC$ linearly generated by the functions of the form 
$$(y,\gamma)\in Y\times\cG \mapsto  (h\otimes k)(y,\gamma) = h(y)k(\gamma)$$
 where $h,k$ are continuous with compact support in $\Omega_1$ and $\Omega_2$ respectively. Using the Stone-Weierstrass theorem, given $\varepsilon >0$ there is $g\in \cC$ such that $\norm{f-g}_\infty <\varepsilon$. It follows that
 $$\abs{\int_{\cG^{p(y)}}f(y,\gamma)\rd \lambda^{p(y)}(\gamma) -  \int_{\cG^{p(y)}}g(y,\gamma)\rd \lambda^{p(y)}(\gamma)} \leq M\varepsilon$$
 for every $y\in Y$.
 Therefore, it suffices to show that 
 $$y\mapsto \int_{\cG^{p(y)}} h(y)k(\gamma) \rd \lambda^{p(y)}(\gamma)$$
 is continuous for $h\in \cC_c(Y)$ and $k\in \cC_c(\cG)$. But this is obvious.
 \end{proof}
 
Note that the proof is straightforward for a global action $\cG\actson Y$, because we have $C_y = \cG^{p(y)}$  and $\tilde{\lambda}^y= \delta_y\times \lambda^{p(y)}$ for all $y\in Y$.

For more details on partial actions see for instance \cite{Exel17}, \cite{McC}, \cite{Abad04}.

 \begin{rem}\label{rem:partial?}   Let $\alpha: G\actson Y$ be a continuous left action of a locally compact group $G$ on a locally compact space $Y$. Let $X$ be an open subset of $Y$. The reduction $\cG(X)$ of the groupoid $\cG= G \ltimes_\alpha Y $ is the semi-direct product groupoid $G\ltimes_\beta X$ associated to the partial action $\boldsymbol{\beta } = (\set{X_g}_{g\in G}, \set{\beta_g}_{g\in G})$ where, for $g\in G$, we set $X_g= X\cap \alpha_g(X)$ and where we denote by $\beta_g$ the restriction of $\alpha_g$ to $X_{g^{-1}}$. 
 \end{rem}

We will use later on several  observations  about semi-direct product groupoids. As a first basic observation, given a locally compact  groupoid $\cH$, viewed as acting to the left on $\cH^{(0)}$, we have $\cH^{(0)}\rtimes \cH \simeq \cH$. This applies in particular to any semi-direct product $\cH= Z\rtimes \cG$ for a global action $\cG\actson Z$ and we have $Z \rtimes (Z\rtimes \cG) \simeq Z\rtimes \cG$. We will make this identification when necessary. More generally we have:

\begin{lem}\label{lem:idengr} Let $(Z,q)$ be a left $\cG$-space  and let $Y$ be locally compact space.  
\begin{itemize}
\item[(i)] Assume that  $Y$ is a left $\cG$-space and that there is a $\cG$-equivariant map $p: Y\to Z$. Then $Y$ has a canonical structure  of left $Z\rtimes \cG$-space, with moment map $p$,  defined by $(z,\gamma)y = \gamma y$ when $\gamma^{-1}z = p(y)$. Moreover $(y,\gamma) \mapsto (y, (p(y),\gamma))$ is an isomorphism from the groupoid $Y\rtimes \cG$ onto the groupoid $Y\rtimes (Z\rtimes \cG)$ and $p$ is $Z\rtimes \cG$-equivariant.
\item[(ii)]  Conversely, assume that $(Y,p)$ is a left  $(Z\rtimes \cG)$-space.  Then $(Y, q\circ p)$  is a left $\cG$-space where the action is defined by $(\gamma,y)\mapsto (\gamma p(y),\gamma)y$ when $s(\gamma) = q\circ p(y)$. Moreover, $p$ is $\cG$-equivariant and $(y,\gamma) \mapsto (y, (p(y),\gamma))$ is an isomorphism between the groupoids $Y\rtimes \cG$ and $Y\rtimes (Z\rtimes \cG)$.
\end{itemize}
\end{lem}

\begin{proof}  In case (i) observe that the moment map  of $Y$ as a $\cG$-space is $q\circ p$.  The verifications are the assertion (i) are straightforward, as they are also for (ii).
 \end{proof}

\begin{lem}\label{lem:equiv} Let $\cG\actson Z$ and $\cG \actson Y$ be two left actions, with moment maps $p_Z$ and $p_Y$ respectively, of a locally compact groupoid $\cG$, and let $q:Z\to Y$ be an equivariant map.  To any  left action $Y\rtimes \cG$ on a locally compact space $V$, with moment map $p_V$,  is canonically associated a left action from $Z\rtimes \cG$ on $V \,_{p_V}* _q Z$ and a canonical homomorphism $f$ from the groupoid $(V \,_{p_V}* _q Z)\rtimes (Z\rtimes \cG)$ into the groupoid $V\rtimes(Y\rtimes \cG)$. Moreover the homomorphism $f$ is proper whenever $q$ is proper.
\end{lem}

\begin{proof} For $(v,z)\in V \,_{p_V}* _q Z$ we set $p'(v,z) = z$. We let $Z\rtimes \cG$ act on $V \,_{p_V}* _q Z$, with moment map  $p'$, as follows:
$$(z,\gamma)(v, \gamma^{-1}z) = \big((q(z),\gamma)v,z\big)$$
if $p_Z(z) = r(\gamma)$ and $\gamma^{-1}(q(z)) = p_V(v)$.

We have 
\begin{align*}
(V \,_{p_V}* _q Z)\rtimes (Z\rtimes \cG) &= \set{\big((v,z),(z,\gamma)\big): p_Z(z) = r(\gamma), p_V(v) = q(z)}\\
& = \set{(v,z,\gamma): p_Z(z) = r(\gamma), p_V(v) = q(z)}.
\end{align*}
The homomorphism is defined by $f(v,z,\gamma) = (v,(q(z),\gamma))$. 
\end{proof}

\subsection{Etale groupoids}\label{subset:etaleg} Apart from semi-direct product groupoids another important class  of examples is given by \'etale groupoids.

An {\it \'etale groupoid} \index{etale groupoid@étale groupoid} is a locally compact groupoid $\cG$ such  that the range (and therefore the source) map is a local homeomorphism from $\cG$ into $\cG^{(0)}$. In this situation,  the fibres $\cG^x = r^{-1}(x)$ with their induced topology are discrete and  $\cG^{(0)}$ is open in $\cG$.  The family of counting measures $\lambda^x$ on $\cG^x$ forms a Haar system (see \cite[Proposition 2.8]{Ren_book}), which will be implicit in the sequel.  

Examples of \'etale groupoids are plentiful. Let us mention groupoids associated with discrete group actions, local homeomorphisms, pseudo-groups of partial homeomorphisms, topological Markov shifts, graphs, inverse semigroups, metric spaces with bounded geometry,... (see \cite{Ren_book}, \cite{Dea}, \cite{KPRR}, \cite{Ren_1997},  \cite{Ren_2000}, \cite{Pat}, \cite{KS02} for a non exhaustive list). For a brief account on the notion of \'etale groupoid see also \cite[Section 5.6]{BO}. 

A {\it bisection} \index{bisection} is a subset $S$ of $\cG$ such that the restriction of $r$ and $s$ to $S$ is injective. Given an open bisection $S$, we will denote by $r_{S}^{-1}$ \index{$r_{S}^{-1}$} the inverse map, defined on the open subset $r(S)$ of $\cG^{(0)}$, of the restriction of $r$ to $S$. Note that $r_{S}^{-1}$ is continuous.

 An \'etale  groupoid $\cG$ has a cover by open bisections. These open bisections form an inverse semigroup (see the definition \ref{def:ISA} below) with composition law
$$ST = \set{\gamma_1\gamma_2: (\gamma_1,\gamma_2)\in (S\times T) \cap \cG^{(2)}},$$
the inverse $S^{-1}$ of $S$ being the image of $S$ under the inverse map (see \cite[Proposition 2.2.4]{Pat}). 
A compact subset $K$ of $\cG$ is covered by a finite number of open bisections. Thus, using partitions of unity, we see that every element of $\cC_c(\cG)$ is a finite sum of continuous functions whose compact support  is contained in some open bisection.

Semi-direct product groupoids relative to partial actions of \'etale groupoids are themselves \'etale.

\begin{prop}\label{prop:etale} Let $\cG$ be an \'etale groupoid  and let $\cG\actson (Y,p)$ be a partial action. Then the groupoid $Y\rtimes \cG$ is \'etale.
\end{prop}

\begin{proof}
Let us show  that every  $(y_0,\gamma_0) \in Y\rtimes \cG$ has an open neighborhood $\Omega$ such that the restriction of ${\underline r}$  to $\Omega$ is a homeomorphism onto its image, which is open in $Y$.  We first choose an open bisection $S$ of $\cG$ which contains $\gamma_0$ and we set $W= p^{-1}(r(S))$.  It is an open subset of  $Y$.
We set $\Omega = (W \times S)\cap(Y\rtimes \cG)$. It is an open subset of $Y\rtimes \cG$ whose image  ${\underline r}(\Omega)$  is open in $(Y\rtimes \cG)^{(0)}$ by Lemma \ref{lem:open}. Obviously, ${\underline r}$ is a continuous bijection from $\Omega$ onto its image. Its inverse map is $y \mapsto (y,r_S^{-1}(p(y))$ which is continuous on ${\underline r}(\Omega)$.
\end{proof}

 \begin{rem}\label{rem:nota}Let $\cG$ and $(Y,p)$ be as in the previous proposition. An open bisection $S$ of $\cG$ defines a homeomorphism $\alpha_S$ from $\cup_{\gamma\in S} Y_{\gamma^{-1}}$ onto $\cup_{\gamma\in S} Y_{\gamma}$, sending $y$ onto $\gamma y$, where $\gamma$ is the unique element of $S$ such that $s(\gamma) = p(y)$. For simplicity of notation, we usually write $Sy$ instead of $\alpha_S(y)$. When  $(Y,p) = (\cG^{(0)},\Id)$, we rather use the notation $x\mapsto S\cdot x$. \index{$Sy$, $S\cdot x$} Note that $p(Sy) = S\cdot p(y)$ for $y\in p^{-1}(s(S))$ and therefore, for every subset $W$ of $p^{-1}(s(S))$, we have
\begin{equation}\label{eqn:equiv}
p(SW) =  S\cdot p(W).
\end{equation}
\end{rem}

\subsection{Inverse semigroup actions}\label{subsec:ISA}  We first recall a few definitions. Our main references for this topic are \cite{Pat, Law, Exel}. A semigroup $S$ is an {\it inverse semigroup} \index{inverse semigroup} if for each $s\in S$ there exists a unique $s^*\in S$ (called the inverse of $s$) such that $s=s s^* s$ and $s^* = s^* s s^*$. The set $E_S$ of idempotents of $S$ (also denoted by $E$ when there is no ambiguity) plays a crucial role. It is an abelian sub-semigroup of $S$. Inverse semigroups where $E$ is a singleton are nothing else that the discrete groups. 

Let $X$ be a locally compact space. We denote by $\cI(X)$ the inverse semigroup formed by all homeomorphisms between open subsets  of $X$, under the operation given by composition of functions in the largest domain under which the composition may be defined. The inverse of $f\in \cI(X)$ is $f^{-1}$, the idempotents are the identity maps of the open subsets  of $X$, that we identify with the corresponding open set. This semigroup has a zero (the empty set) and a unit (the space $X$).

\begin{defn}\label{def:ISA} An {\it action of an inverse semigroup $S$ on a locally compact space} $X$ \index{inverse semigroup action} is a semigroup homomorphism $\theta: S \to \cI(X)$ such that $X$ is the union of the domains of all the $\theta_s$. If $S$ has a zero we require $\theta_0$ to be the empty set and if $S$ has a unit $1$ we require $\theta_1$ to be $X$.
\end{defn}

For every idempotent $e$ we denote by $D_e$ its domain. Then for $s\in S$ the domain of $\theta_s$ is $D_{s^*s}$ and $\theta_s$ is a homeomorphism from $D_{s^*s}$ onto $D_{ss^*}$.

Let us now describe the {\it groupoid of germs} $\cG(S,\theta)$ \index{$\cG(S,\theta)$} constructed from $\theta$ (for more details see \cite[\S 4]{Exel}). We set $\Xi = \set{(t,x) \in S\times X : x\in D_{t^*t}}$.
On $\Xi$ we define the equivalence relation $(t,x)\sim (t_1,x_1)$ if $x = x_1$ and there exists $e\in E$ such that $x\in D_e$ and $te = t_1e$. Then $\cG(S, \theta)$ is the quotient of $\Xi$ with respect to this equivalence relation, equipped with the quotient topology. The composition law is given by $[t,x][t', x'] = [tt',x']$ if $\theta_{t'}(x')= x$ and $[t,x]^{-1} = [t^*,\theta_t(x]$.  We identify $[e,x]$ with $x$, where $e$ is any idempotent such that $x\in D_e$ and in this way the space of units of $\cG(S,\theta)$ is identified with $X$. The range of the class $[t,x]$ of $(t,x)$ is $\theta_t(x)$ and its source is $x$.  The topology of $\cG(S,\theta)$ has the sets $\set{[t,x]: x\in U}$ as basis of open subsets, where $t\in S$ and $U$ ranges over the open subsets of $D_{t^*t}$. 

For $t\in S$ we set $\cS_t = \set{[t,x]: x\in D_{t^*t}}$. Then $\cS_t$ is an open bisection. Although $X$ is Hausdorff, it is not the case for $\cG(S,\theta)$ in general. When $\cG(S,\theta)$ is Hausdorff, it is therefore an \'etale groupoid.

There is a natural partial order on $S$ defined by $s\leq t$ if there exists an idempotent $e\in E$ such that $s =  t e$ or, equivalenty, if $s = ts^*s$. Given $s\in S$ we denote by $\cF_s$ the union of the $D_e$ where $e$ ranges on the set of idempotents such that $e\leq s$.

\begin{prop}$($\cite[Theorem 3.15]{EP}$)$\label{prop:GHaus} Let $\theta$  be an action of $S$ on $X$. Then $\cG(S,\theta)$ is Hausdorff if and only if for every $s\in S$ the set  $\cF_s$ is closed into  $D_{s^*s}$. 
\end{prop} 

 If $S$ is a discrete group $\Gamma$, an action $\theta$ of $S$ is an usual action of $\Gamma$ and we have $\cG(S,\theta) = \Gamma\ltimes X$.

 Among the inverse semigroups  closest to discrete groups are the $E$-unitary  and $E^*$-unitary inverse semigroups.

\begin{defn}\label{def:idpure1} An inverse semigroup $S$ is said to be {\it $E$-unitary} \index{E-unitary, E*-unitary inverse semigroup} if  every element greater than an idempotent is an idempotent. When $S$ has a zero, this means that $S= E$. In this case, we rather introduce {\it $E^*$-unitary} inverse semigroups\footnote{also called $0$-$E$-unitary inverse semigroups}, defined by the fact that every $s$ greater than a non-zero idempotent is an idempotent.
\end{defn}

Assume that $S$ is $E$-unitary or $E^*$-unitary. Then if $s\in S$ is not an idempotent, we have $\cF_s = \emptyset$, and we have $\cF_s = D_s = D_{s^*s}$ if $s$ is an idempotent.  Therefore $\cG(S,\theta)$ is Hausdorff for any action $\theta$ of $S$. Numerous examples of such inverse semigroups are discussed in \cite{Law}.

 \subsection{Group bundle groupoids}\label{subsubsect:gbg} A {\it group bundle groupoid}\index{group bundle groupoid} is a locally compact groupoid $\cG$ such that $r= s$. Since $r$ is open,   for $x\in \cG^{(0)}$ we can choose a Haar measure $\lambda(x)$ on the group $\cG(x)$ in such a way that $(\lambda(x))_{x\in \cG^{(0)}}$ is a Haar system for $\cG$ (see \cite[Lemma 1.3]{Ren91} or \cite[Theorem 6.9]{Will19}).

\begin{ex}\label{ex:HLS}The following class of group bundle groupoids (that we call HLS-groupoids)\index{HLS-groupoid} was introduced by  Higson, Lafforgue and Skandalis \cite{HLS},
 in order to provide examples of groupoids for which the Baum-Connes conjecture fails. We consider an infinite discrete group $\Gamma$ and a decreasing sequence $(N_k)_{k\in \N}$ of normal subgroups of $\Gamma$ of finite index. We set $\Gamma_\infty = \Gamma$, and $\Gamma_k = \Gamma/N_k$ and we denote by $q_k : \Gamma \to \Gamma_k$ the quotient homomorphism for $k$ in the  Alexandroff compactification $\N^+$ of $\N$. Let  $\cG$ be the quotient of $  \N^+ \times\Gamma$ with respect to the equivalence relation
$$(k,t)\sim (l,u) \,\,\,\hbox{if} \,\,\, k=l \,\,\,\hbox{and}\,\,\, q_k(t) = q_k(u).$$ 
Then $\cG$ is the bundle of groups $k\mapsto \Gamma_k$ over $\N^+$. The range and source maps are given by $r([k,t]) = s([k,t]) = k$, where $[k,t] = (k,q_k(t))$ is the equivalence class of $(k,t)$.  We endow $\cG$ with the quotient topology. Then $r^{-1}(\N)$ is a discrete open subset of $\cG$. A basis of neighborhoods of $(\infty,s)$ is formed by the subsets $\set{(k,q_k(s)) : k\geq k_0}$, where $k_0$ runs over $\N$. It $s\not = t$, we observe that $(\infty,s)$ and $(\infty, t)$ have disjoint neighborhoods if and only if there exists an integer $k_0$ such that $q_{k_0}(s)\not=q_{k_0}(t)$. Therefore, $\cG$ is Hausdorff  if and only if  for every $s\not= 1$ there exists $k_0$ such that $s\notin N_{k_0}$  (hence, $\Gamma$ is residually finite).  Note that $\cG$ is then an \'etale   groupoid.

In this case, following \cite[Definition 2.1]{Wil15} we will say that the pair  $(\Gamma, (N_k))$ is {\it an approximated group}\index{approximated group}. Examples are provided by taking $\Gamma = \hbox{SL}_n(\Z)$ and $\Gamma_k = \hbox{SL}_n(\Z/k\Z)$, for $k\geq 2$.
\end{ex}

\subsection{Extensions of groupoids}\label{subsec:ext-groupoids} Let $\boldsymbol{\beta }: \cG \actson Y$ be a partial action. Then $c: (y,\gamma)\in Y\rtimes \cG \mapsto \gamma\in \cG$ is a continuous homomorphism which is surjective if and only if $Y_\gamma \not=\emptyset$ for all $\gamma\in \cG$. We have 
$$\Ker(c) = \set{(y,\gamma): c(y,\gamma) \in \cG^{(0)}} = \set{(y,p(y)): y\in Y} \equiv Y.$$
When $c$ is surjective, one can think of $\cL =Y\rtimes \cG$ as being an extension of $\cG$ by $Y$ viewed as a groupoid. However,  the notion of surjectivity of c is ambiguous and assuming that $c(\cL) = \cG$ is sometimes too weak. In the litterature it is a stronger notion that is considered, being more tractable and including many interesting examples. In \cite[Definition 5.3.7]{AD-R}, the notion of strongly surjective homomorphism is defined in the Borel setting. In the topological case we adopt the following definition.

\begin{defn}\label{def:strong-surj} Let $c:\cL\to \cG$ be a continuous homomorphism from a locally compact groupoid  $\cL$ to a locally compact groupoid  $\cG$. We say that $c$ is {\it strongly surjective}\index{strongly surjective homomorphism} if 
\begin{itemize}
\item[(i)] $c^{(0)}: \cL^{(0)} \to \cG^{(0)}$ is surjective;
\item[(ii)] for every $y\in \cL^{(0)}$ we have $c(\cL^y) = \cG^{c^{(0)}(y)}$;
\item[(iii)] the map $\phi: \gamma\in \cL \mapsto (r(\gamma), c(\gamma))$ from $\cL$ to $\cL^{(0)} \!_{c^{(0)}}\!\!*_r \cG$ is open.
\end{itemize}
\end{defn}
Note the the condition (ii) is equivalent to the surjectivity of $\phi$. In particular, $\phi$ is a groupoid fibration in the sense of \cite[Definition 2.1]{BM}. When $\cL$ and $\cG$ are locally compact groups, strong surjectivity is the same as surjectivity.

Note that given any $\cG$-space $(Y,p)$ where $p: Y\to \cG^{(0)}$ is surjective, the canonical homomorphism $c: Y\rtimes \cG \to \cG$ is strongly surjective, but for partial actions that are not global, Condition (ii) is not satisfied. 

\begin{prop}\label{prop:exten} We keep the assumption of the previous definition.
\begin{itemize}
\item[(i)] $\cH = \Ker(c)$ is a subgroupoid of $\cL$.
\item[(ii)] $\phi$ induces an homeomorphism $\dot{\phi}$ from $\cL/\cH$ onto $\cL^{(0)} \!_{c^{(0)}}\!\!*_r \cG$, by passing to quotient.
\item[(iii)] $c^{(0)}$ is open if and only if $c$ is open.
\end{itemize}
\end{prop}

\begin{proof} For (i) we have to show that the range map of $\cH$ is open. Note that $\cH^{(0)} = \cL^{(0)}$. We denote by $\psi$ the map $y\mapsto (y, c^{(0)}(y))$ from $\cH^{(0)}$ to $\cL^{(0)} \!_{c^{(0)}}\!\!*_r \cG$. Then $\gamma\mapsto (\gamma, r(\gamma))$ is a homeomorphism from $\cH$ onto $\cL\,_\phi\!*_\psi \cH^{(0)}$ and the range map of $\cH$  corresponds to the second projection  $\cL\,_\phi\!*_\psi \cH^{(0)}\to \cH^{(0)}$. The conclusion follows from Lemma \ref{lem:easy} (i), since $\phi$ is open.

Let us check (ii). Obviously, $\dot{\phi}$ is injective, and it is surjective because for every $y\in \cL^{(0)}$ we have $c(\cL^y) = \cG^{c^{(0)}(y)}$. Moreover $\dot{\phi}$ is a homeomorphism since $\phi$ is continuous and open.

In order to prove (iii) we first observe that $c = \pi_2\circ \phi$, where $\pi_2$ is the projection on the second variable.  Since $\phi$ is surjective and open, we see that $c$ is open if and only if $\pi_2$ is open. If $c^{(0)}$ is open then $\pi_2$ is open,  using  the lemma \ref{lem:easy} (i),  and therefore so is $c$.
Conversely assume that $c$ is open. Then $c^{(0)}$ is open since for every open subset $U$ of $\cL^{(0)}$ we have $c^{(0)}(U) = r\big(\pi_2(U \!_{c^{(0)}}\!\!*_r \cG)\big)$
\end{proof}

\begin{defn}\label{def:extension} When the properties of the definition \ref{def:strong-surj} are satisfied, we will say that $\cL$ is an {\it extension of $\cG$ by $\cH = \Ker(c)$}.\index{extension of groupoids}  
\end{defn}

\begin{prop}\label{prop:Haar} $($\cite[Theorem 5.1]{BM}$)$ Let us assume that $\cH$ and $\cG$ have Haar systems, denoted respectively by $\lambda$ and $\mu$.  Given $f\in \cC_c(\cL)$ and  $\gamma \in \cL$, the integral $\int_{\cH^{s(\gamma)} }f(\gamma h) \rd \lambda^{s(\gamma)}(h)$  only depends on $\gamma\cH$.  After passing to quotient, the map $\gamma \mapsto   \int_{\cH^{s(\gamma)} }f(\gamma h) \rd \lambda^{s(\gamma)}(h)$ defines a continuous function with compact support on $\cL/\cH$ identified with  $\cL^{(0)} \!_{c^{(0)}}\!\!*_r \cG$. We denote it by $(y,g) \mapsto F(y,g)$. 
We set $\int_\cL f \rd \nu^y = \int_\cG  F(y, g) \rd \mu^{c^{(0)}(y)}(g)$. Then $(\nu^y)_{y\in \cL^{(0)}}$ is a Haar system on $\cL$.
\end{prop}

\begin{proof} First we note that $\cL$ is a principal $\cH$-bundle for the right action of $\cH$ on $\cL$. Then the Haar system  $\lambda$ induces a continuous family of measures $(\dot \lambda^{\dot\gamma})$  on the fibres of the quotient map $\gamma \mapsto \dot\gamma = \gamma\cH$ from $\cL$ onto $\cL/\cH$ (see \cite{Ren87}). For $f\in \cC_c(\cL)$  we have 
$$\dot \lambda^{\dot\gamma}(f) = \int_{\cH }f(\gamma h) \rd \lambda^{s(\gamma)}(h).$$
We denote by $T_\cH(f)$ this continuous function with compact support on $\cL/\cH$ and we set $F(y,g) = T_\cH(f)\circ \dot \phi^{-1}(y,g)$ for $(y,g)\in  \cL^{(0)} \!_{c^{(0)}}\!\!*_r \cG$.
Then we set 
\begin{equation}\label{eq:Haar}
\int_\cL f \rd \nu^y = \int_\cG  F(y, g) \rd \mu^{c^{(0)}(y)}(g).
\end{equation}
We define in this way a continuous family of measures $y \mapsto \nu^y$ on $\cL$ where the support of $\nu^y$ is $\cL^y$. Let us check that $\gamma\nu^{s(\gamma)} = \nu^y$ where $y = r(\gamma)$. If $r(\gamma_1) = s(\gamma)$ we have
$$\int_{\cH} f(\gamma\gamma_1 h)\rd\lambda^{s(\gamma_1)}(h) = T_\cH(f)(\gamma\dot\gamma_1) = F(y,c(\gamma)g))$$
where $g=c(\gamma_1)$. It follows that
$$\int_\cL f\rd \gamma\nu^{s(\gamma)} =  \int_\cG F(y,c(\gamma)g)\rd\mu^{s(c(\gamma))}(g) = \int_\cG F(y,g)\rd\mu^{r(c(\gamma))}(g)= \int_\cL f \rd \nu^y,$$
since $r(c(\gamma)) = c^{(0}(y)$.
Thus $\nu$ is a Haar system for $\cL$. For more details see \cite[Theorem 5.1]{BM}.
\end{proof}

 As a particular case of this notion of extension, we have unit fixing extensions. 

\subsubsection{Unit fixing groupoid extensions}\index{unit fixing groupoid extension} They are defined by a strongly surjective map $c: \cL\to \cG$ such that  $c^{(0)}$ is a homeomorphism fron $\cL^{(0)}$ onto $\cG^{(0)}$. We set $\cH= \Ker(c)$, and we identify the unit spaces of these three groupoids\footnote{ We keep the terminology of \cite[page 236]{Will19}.}. Note that $\cL^{(0)} \!_{c^{(0)}}\!\!*_r \cG$ can be canonically identified with $\cG$ and then $c = \phi$  an open surjection. We have $r_H = s_H$, and so $\cH$ is a bundle of locally compact groups.  Moreover $\cH$ is normal in $\cL$, that is, for every $h\in \cH$ and every $\gamma\in \cL$ with $s(\gamma) = r_H(h)$ we have $\gamma h \gamma^{-1} \in \cH$.

 We choose a Haar system $\lambda$ on $\cH$ as explained in the subsection \ref{subsubsect:gbg}. Assume that $\cG$ has a Haar system $\mu$.  Then the corresponding Haar system $\nu$ on $\cL$ is defined  as follows:
 $$\int_\cL f\rd \nu^y = \int_\cG \Big(\int_\cH f(\gamma h)\rd \lambda^{s(\gamma)}(h)\Big)\rd \mu^{y}(\dot \gamma),$$
 where $r(\gamma) = y$ and $\dot\gamma = c(\gamma)$.

When $\cL$ and $\cG$ are locally compact groups, extensions and unit fixing extensions are of course the same, and they are the classical notion of group extension.

Let us describe now the  case of semi-direct product groupoids, where we will use the following notation: given an extension $\cL$ of $\cG$ by $\cH$, for more clarity, we reserve the letters $r,s$ for the range and source maps of $\cL$; those of $\cH$ (resp. $\cG$) will be denoted by $r_H, s_H$ (resp. $r_G, s_G$).

\subsubsection{Case of a groupoid acting on another groupoid}\label{subsec:gg} Let $\cG$ and $\cH$ be two locally compact groupoids. As in \cite[Definition 5.3.24]{AD-R} we consider the case $\alpha: \cG \actson \cH$ of a right action. It is in particular   an action on $\cH$ viewed as a locally compact space. We denote by $p: \cH\to \cG^{(0)}$ its moment map and we assume here that $p$ is {\it surjective}. Moreover, the definition requires that 
$$p = p\circ r_H = p\circ s_H.$$
 .  

The right action $\alpha: (h,g)\mapsto h\cdot g\in \cH$ is defined on $\cH\, _p\!*_r \cG$ and $r_H, s_H$ are $\cG$-equivariant:
$$r_H(h\cdot g) = r_H(h)\cdot g,\quad  s_H(h\cdot g) = s_H(h)\cdot g$$
(note that $p(h\cdot g) = s_G(g)$). Then for $h_1, h_2\in \cH$ such that $s_H(h_1) = r_H(h_2)$, the pair  $h_1\cdot g, h_2\cdot g$ is composable whenever it is defined (that is $p(h_i) = r_G(g)$, $i= 1,2$) and we have
$$(h_1 h_2)\cdot g = (h_1\cdot g)(h_2\cdot g).$$

Let us now describe the groupoid structure of 
$$\cL = \cH\rtimes_\alpha \cG = \set{(h,g)\in \cH\times\cG: p(h) = r_G(g)}.$$
 Its unit space is $\cH^{(0)}\, _p *\cG^{(0)}$ that we identify to $\cH^{(0)}$ via the map $(y,p(y)) \mapsto y$. We have
$$r(h,g) = (r_H(h), r_G(g)\equiv r_H(h), \quad s(h,g) = (s_H(h)\cdot g, s_G(g) \equiv s_H(h)\cdot g.$$
The product is given by
$$(h_1,g_1)\cdot (h_2,g_2) = (h_1(h_2\cdot g_1^{-1}), g_1g_2)$$
when $s_H(h_1)\cdot g_1 = r_H(h_2)$.

We leave it to the reader to check that $r,s$ are open, since the range and source maps of $\cH$ and $\cG$ are open. We denote by $c:\cL\to \cG$ the map $(h,g)\mapsto g$.
We have 
$$\Ker(c) = \set{(h,p(h)): h\in \cH} \equiv \cH,$$
$$ \cL^{(0)} = \set{ (y,x) \in \cH^{(0)}\times \cG^{(0)} : x=p(y)} \equiv \cH^{(0)}.$$
Here $\phi(h,g) = (r_H(h),g)$ from $\cH\rtimes_\alpha \cG \to \cH^{(0)}\,_p*_r \cG$ is an open continuous surjection, and $c^{(0)}(y) = p(y)$  so $c^{(0)} : \cL^{(0)} \to \cG^{(0)}$ is surjective. It follows that $\cL =  \cH\rtimes_\alpha \cG$ is an extension of $\cG$ by $\cH$.  We call it the {\it semi-direct product of $\cG$ acting on $\cH$}. Note that  $\cH^{(0)}\,_p*_r \cG$ is a closed subset of  $\cL$ and that $\dot{\phi}^{-1}(r_H(h),g) = (r_H(h),g)\cH$.

Let us assume now that $\cH$ and $\cG$ have  
Haar systems, denoted respectively by $\lambda$ and $\mu$ and   let $\nu$ be the corresponding Haar system on $\cL$. For 
$$\gamma = (h,g)\in \cL^y = \cH^y \times \cG^{p(y)}$$
 where $r(\gamma) = r_H(h) = y$,  and for $f\in \cC_c(\cL)$, we have (with the notation of the proof of the previous proposition)
\begin{align*}
&\int f\big((h,g)\cdot h_1\big) \rd \lambda^{s_H(h)\cdot g}(h_1) = \int f\Big(\big((h\cdot g)\cdot h_1)\cdot g^{-1}, g\big)\Big)\rd \lambda^{s_H(h)\cdot g}(h_1)\\
& = \int f(h_1\cdot g^{-1}, g) \rd \lambda^{y\cdot g}(h_1) = T_\cH(f)\big((h,g)\cH\big).
\end{align*}

Observe that $(h,g)\cH = (r_H(h),g)\cH$ and therefore $T_\cH(f)\big((h,g)\cH\big) =  F(y,g)$. It follows from the formula \ref{eq:Haar} that 
$$\int_\cL f\rd \nu^y = \int_{\cG^{p(y)} }\Big(\int_{\cH^{y\cdot g}} f(h\cdot g^{-1},g)\rd \lambda^{y\cdot g}(h)\Big)\rd \mu^{p(y)}(g).$$
When $\lambda$ is invariant under the $\cG$-action we find the formula given in \cite[Proposition 5.3.27]{AD-R}. 

As a particular example, we have of course the case where $\cH$ is the locally compact space $Z$ and we find the semi-direct product groupoid $Z\rtimes_\alpha \cG$. Another important example is the case where $\cH$ and $\cG$ are locally compact groups $H$ and $G$  where we get the usual semi-direct product group $H\rtimes_\alpha G$.

 \subsection{Generalized morphisms between groupoids}\label{subsec:gen-mor}
 \subsubsection{Locally proper and faithful homomorphisms} These notions are useful in order to transfer properties from the image of a homomorphism to its source.
\begin{defn}\label{defn:lpropre} Let $c : \cG \to {\mathcal H}$ be a   continuous homomorphism (also called cocycle) between locally compact groupoids. We denote by $\psi$ the map $\gamma \mapsto (r(\gamma), c(\gamma), s(\gamma))$ from $\cG$ into $\cG^{(0)}\times\cH\times \cG^{(0)}$. We say that $c$ is {\it locally proper}\index{locally proper homomorphism} if $\psi$ is proper or, equivalently, if for every compact subset $K$ of $\cG^{(0)}$, the restriction of $c$ to the reduction $\cG(K)$ of $\cG$ by $K$ is proper. Following \cite{KS02,RS}, we say that $c$ is {\it faithful} if $\psi$ is injective. 
  \end{defn}

A homomorphism $c: \cG\to \cH$ induces, for every $x\in \cG^{(0)}$, a homomorphism from the group $\cG(x)$ into the group $\cH(c(x))$. Observe that $c$ is faithful if and only if, for every $x\in \cG^{(0)}$ the restriction of $c$ to $\cG(x)$ is an injective group homomorphism.

Whenever $c$  is faithful, note that $c$ is locally proper if and only if  the map $\psi$ is closed\footnote{In \cite{KS02,RS}, $c$ is said to be closed if $\psi$ is closed}, equivalently if and only if $\psi(G)$ is closed and $\psi$ is a homeomorphism from $G$ onto $\psi(G)$.

\begin{exs}\label{exs:loc_proper} (a) If $\cH$ is a closed subgroupoid of $\cG$, then the inclusion map is faithful and locally proper. 

(b) Let $E$ be a locally compact subset of $\cG^{(0)}$ such that the reduction $\cG(E)$ is  a locally compact subgroupoid of $\cG$ ({\it i.e.,} the range map of $\cG(E)$ is open). Then the inclusion map from $\cG(E)$ into $\cG$ faithful and locally proper. Indeed, it suffices to observe that if $K$ is a compact subset of $E$, then $\cG(E)(K) = \cG(K)$ is closed into $\cG$.

(c) If $Y$ is a $\cG$-space, the obvious homomorphism from $Y\rtimes \cG \to \cG$   is faithful and locally proper. 

(d) Let $\boldsymbol{\beta} : \cG\actson Y$ a partial action, as introduced in Definition \ref{def:part top}. The homomorphism $(y,\gamma) \mapsto \gamma$ from $Y\rtimes \cG$ to $\cG$ is faithful. It is locally proper if and only if the map $\psi:(y,\gamma) \to (y,\gamma,\gamma^{-1}y)$ is  closed.  Since $\psi^{-1}$ is continuous, we see that $\psi$ is locally proper if and only if its range is a closed subset of $Y\times \cG\times Y$ (see \cite[Chapter I, \S 10]{Bou}). In particular, for every $\gamma\in \cG$, the graph $\set{(y,\gamma^{-1}  y) : y\in Y_\gamma}$ must be closed into $Y\times Y$. This last condition is sufficient when $\cG$ is a discrete group.

(e) Let $\cG\actson Z$ and $\cG \actson Y$ be two actions,  and let $q:Z\to Y$ be an equivariant map. The homomorphism $(z,\gamma)\mapsto (q(z),\gamma)$ from $Z\rtimes \cG$ into $Y\rtimes \cG$ is faithful and locally proper.
\end{exs}

\begin{ex}\label{ex:ampliation} Let $\cG$ be a locally compact groupoid, let $T$ be a locally compact space and let $\varphi : T\to \cG^{(0)}$ be a continuous, open surjective map. We set 
$$\cG^\varphi = \set{(x,\gamma,y)\in T\times\cG\times T: \varphi(x) = r(\gamma), \varphi(y) = s(\gamma)}.$$
Endowed with the topology induced by that of $T\times\cG\times T$ it is a locally compact space, and it has a natural structure of groupoid with respect to the following operations: $(x,\gamma,y)(y,\eta,z) = (x,\gamma\eta,z)$, $(x,\gamma,y)^{-1} = (y,\gamma^{-1},x)$. Its unit space is $\set{(x,\varphi(x), x): x\in T}$,  canonically identified with $T$.

This groupoid $\cG^\varphi$\index{$\cG^\varphi$} is called the {\it ampliation}\index{ampliation}, or {\it blow-up} of $\cG$ with respect to $\varphi$.

 Obviously, the homomorphism $(x,\gamma,y)\mapsto \gamma$ from $\cG^\varphi$ to $\cG$ is  surjective, faithful and locally proper.
 \end{ex}

\begin{prop}\label{prop:locprop} Let $c: \cL\to \cG$ be a continuous strongly surjective homomorphism from a locally compact  groupoid $\cL$ to a locally compact groupoid $\cG$. We assume that the groupoid $\Ker(c)$ is proper. Then $c$ is locally proper.
\end{prop}

\begin{proof}  Let  $K$ be  a compact subset of $\cL^{(0)}$ and let $K_1$ be a compact subset of $\cG$. We have to show that $\cL(K)\cap c^{-1}(K_1)$ is compact. Recall that the map $\phi: \gamma\in \cL \mapsto (r(\gamma), c(\gamma))$ from $\cL$ to $\cL^{(0)} \!_{c^{(0)}}\!\!*_r \cG$ is continuous, open and surjective. Let  $C$ be a compact subset of $\cL$ such that $\phi(C) = K \!_{c^{(0)}}\!\!*_r K_1$. Then we have $\cL(K)\cap c^{-1}(K_1) \subset C\cH$. For $\gamma = \gamma' h\in \cL(K)\cap c^{-1}(K_1)$ with $\gamma'\in C$ and $h\in \cH$, we have $s(h) = s(\gamma)\in K$ and $r(h) = s(\gamma')\in s(C)$. Since $\cH$ is proper, $\cH\big(K \cup s(C)\big)$ is compact and so is $\cL(K)\cap c^{-1}(K_1))$ which is contained into $C\cH\big(K \cup s(C)\big)$.
\end{proof}

\begin{rem}\label{rem:locprop} The previous proposition is not true is $c$ is only assumed to be surjective.  Indeed,  let $\boldsymbol{\beta} : \cG\actson Y$ a partial action, as introduced in Definition \ref{def:part top}. The homomorphism $c: (y,\gamma) \mapsto \gamma$ from $Y\rtimes \cG$ to $\cG$ is faithful, and $\Ker(c) = Y$ is a proper groupoid. Recall that $c$ is  locally proper if and only if the map $(y,\gamma) \to (y,\gamma,\gamma^{-1}y)$ is  closed.  In particular, for every $\gamma\in \cG$, the graph $\set{(y,\gamma^{-1}  y) : y\in Y_\gamma}$ must be closed into $Y\times Y$. This is not always the case (see the example  $\Z_2\actson ]-1,+\infty[$ given in \ref{ex:reducpartial}).\end{rem}

\subsubsection{Generalized morphisms} The assertion (i) of the  proposition below is the property of homomorphisms that is used for instance in \cite{Tu04} in order to define a generalized morphism.

\begin{prop}\label{prop:homo} Let $\cG$ and $\cH$ be two locally compact groupoids and let $c : \cG\to \cH$ be a continuous homomorphism. We set $$Z_c = \set{(x,\eta)\in \cG^{(0)}\times \cH: c(x) = r(\eta)} = \cG^{(0)} \!_{c^{(0)}}\!\!*_r \cH.$$
 We endow $Z_c$  with the left $\cG$-action $\gamma(s(\gamma),\eta) = (r(\gamma), c(\gamma)\eta)$ $($with moment map $\rho: (x,\eta)\mapsto x$$)$, and with the right $\cH$-action $(x,\eta)\eta' = (x\,\eta\eta')$ $($with moment map  $\sigma: (x,\eta) \mapsto s(\eta)$$)$. The moment map $\rho$ has a global continuous section, namely $x \mapsto (x,c(x))$.
\begin{itemize}
\item[(i)] The right action of $\cH$ is free and proper, $\rho$ is open and $z\mapsto \rho(z)$ induces a homeomorphism from $Z_c/\cH$ onto $\cG^{(0)}$.
\item[(ii)] The left action of $\cG$ is proper if and only if $c$ is locally proper. It is free if and only if $c$ is faithful.
\end{itemize}
\end{prop}

\begin{proof} (i) To see that $\rho : (x,\eta) \mapsto x$ is open we apply Lemma \ref{lem:easy}.  Passing to quotient we get the homeomorphism from $Z_c/\cH$ onto $\cG^{(0)}$. The other assertions are immediate.
  
  (ii) Again, the verifications are straightforward.
\end{proof}

Note the the moment map $\sigma: (x,\eta) \mapsto s(\eta)$ is neither surjective nor open in general.   
  \begin{defn}(\cite[Definition 7.1]{Tu04})  \label{def:mg} Let $\cG$ and $\cH$ be two locally compact groupoids. We say that a locally compact space $Z$ (or more precisely $(Z,\rho,\sigma)$) is a {\it generalized  morphism} \index{generalized morphism} from $\cG$ into $\cH$ if the following conditions hold:
\begin{itemize}
\item[(i)] $Z$ is a left $\cG$-space with moment map $\rho$ and is a free and proper right $\cH$-space with moment map $\sigma$, with commuting actions;
\item[(ii)]  the moment map $\rho: Z\to \cG^{(0)}$  is open and induces a homeomorphism from $Z/\cH$ onto $\cG^{(0)}$.
\end{itemize}
\end{defn}

\begin{rem}\label{rem:ampl}  Ordinary homomorphisms from $\cG$ into $\cH$ are generalized morphisms such that $\rho$ has a global continuous section. Indeed, let $Z$ be a generalized morphism with a global section $p$ for $\rho$. Then for $\gamma\in \cG$ let $\eta$ be the unique element of $\cH$ such that $\gamma p(s(\gamma)) = p(r(\gamma))\eta$. We set $c(\gamma) = \eta$. Then  $c$ is a  continuous homomorphism from $\cG$ into $\cH$.  For every $r\in Z$, there exists a unique element in $\cH$, that we denote by $\psi(z)$ such that $p\circ \rho(z) = z\psi(z)$. We leave it as an exercise to show that the map $z \mapsto (\rho(z), \psi(z)^{-1})$ is an isomorphism from the ($\cG$-$\cH$)-space $Z$ onto the ($\cG$-$\cH$)-space $Z_c$. The section $p$ corresponds to the canonical section $\cG^{(0)} \to Z_c$.
\end{rem}

This remark and Proposition \ref{prop:homo} motivate the following definition.

 \begin{defn}(\cite[Definition 7.3]{Tu04}) \label{defn:genmorlp} We say that a generalized morphism $(Z,\rho,\sigma)$ from $\cG$ into $\cH$ is {\it locally proper} \index{locally proper generalized morphism} if the  action of $\cG$ on $Z$ is proper. We say that it is {\it faithful} if this action is free. 
 \end{defn}
 
If one only requires the existence of local continuous sections for the moment map $\rho$, we  have an interesting notion of generalized morphism, intermediate between  homomorphisms and generalized morphisms.
  
  \begin{defn}\label{defn:genmorloc} A {\it generalized morphism with local sections} from $\cG$ into $\cH$ is a generalized morphism $Z$ such that there exists an open covering $(U_i)$ of $\cG^{(0)}$ and local continuous sections $p_i: U_i\to Z$ of $\rho$.
 \end{defn}
 
 In this case there is an ampliation $\cG'$ of $\cG$, often more handy than $\cG^\rho$ (called, the {\it localization of $\cG$ relative to the covering $(U_i)$}) \index{localization of a groupoid} and a canonical continuous homomorphism $c: \cG'\to \cH$.  In order to define $\cG'$ we first set $T = \bigsqcup_i U_i$ and we denote by $f$ the canonical open continuous surjective map from $T$ onto $\cG^{(0)}$. Then $\cG'$ is the corresponding ampliation $\cG^f$. We have
\begin{align*}
\cG' &= \set{(t,\gamma,t')\in T\times\cG\times T: f(t) = r(\gamma), f(t') = s(\gamma)} \\
&= \bigsqcup_{i,j} \set{(x,\gamma,x')\in U_i\times\cG\times U_j: x=r(\gamma), x'=s(\gamma)}.
\end{align*}

  An element $(x,\gamma,x')$ with $x=r(\gamma)\in U_i$ and $x'=s(\gamma)\in U_j$  will be simply written $(i,\gamma,j)$. The set $\set{(i,\gamma,j): \gamma \in \cG_{U_j}^{U_i}}$ is open and closed in $\cG'$ (where  $\cG_{U_j}^{U_i} = \cG_{U_j}\cap \cG^{U_i}$) and canonically homeomorphic to $\cG_{U_j}^{U_i}$. 
  
  We have 
  $$r(i,\gamma,j) = (i,r(\gamma),i), \,s(i,\gamma,j) = (j,s(\gamma),j),$$
  $$(i,\gamma,j)(j,\gamma',k) = (i,\gamma\gamma',k), \,(i,\gamma,j)^{-1} = (j,\gamma^{-1}, i).$$
  Let $\gamma\in \cG_{U_j}^{U_i}$. Since $\rho(\gamma p_j(s(\gamma)) ) = \rho(p_i(r(\gamma)))$, there is a unique element in $\cH$, that we denote by $g_{i,j}(\gamma)$ such that $\gamma p_j(s(\gamma)) = p_i(r(\gamma)) g_{i,j}(\gamma)$. Then $(g_{i,j})$ is a cocycle from $\cG$ to $\cH$ as defined in \cite{HS}. This means that for $\gamma \in \cG_{U_j}^{U_i}$ we have $g_{j,i}(\gamma^{-1}) = g_{i,j}(\gamma)^{-1}$ and that for every $\gamma\in \cG_{U_j}^{U_i}$, $\gamma'\in \cG_{U_k}^{U_j}$ with $s(\gamma) = r(\gamma')$ we have $g_{i,k}(\gamma\gamma') = g_{i,j}(\gamma)g_{j,k}(\gamma')$. Then the homomorphism $c: \cG'\to \cH$ is defined by $c(i,\gamma,j) = g_{i,j}(\gamma)$.
 
 We observe that when $\rho$ has a global section $p$, then $\cG'= \cG$  and $c$ is the homomorphism defined in Remark \ref{rem:ampl}.

\section{\textbf{\textsc{About equivalence  of groupoids}}}\label{sec:EquGr}

There are various notions of equivalence of groupoids in the literature. In this section we discuss some of them.

\subsection{Similarity}\label{subsec:sim}
\begin{defn}\label{def:sim} Let $\cG$ and $\cH$ be two locally compact groupoids.
\begin{itemize}
\item[(i)] We say that two continuous homomorphisms $f,g : \cG\to\cH$ are {\it similar}\index{similar homomorphisms} (and we write $f\sim g$)\index{$f\sim g$} if there exists a continuous map $\theta: \cG^{(0)} \to \cH$ such that 
$$\theta\circ r(\gamma)f(\gamma) =g(\gamma)\theta\circ s(\gamma)$$
 for all $\gamma\in \cG$.
\item[(ii)] We say that $\cG$ and $\cH$ are {\it similar} \index{similar groupoids} (and we write $\cG\sim\cH$)\index{$\cG\sim\cH$} if there exist continuous homomorphisms $f:\cG\to\cH$ and $g:\cH\to\cG$ such that $f\circ g \sim\Id_\cH$ and $g\circ f \sim \Id_\cG$.
\end{itemize}
\end{defn}

From the pure algebraic point of view, forgetting the topology, a groupoid is a category in which all arrows are invertible. The set of arrows is the groupoid, the objects are the units. The homomorphisms are the functors, and  similar homomorphisms are naturally equivalent functors in the sense of category theory \cite[\S I.4]{MacL}. Similarity of locally compact groupoids corresponds to equivalence as categories.

\begin{rem}\label{rem:simeq} Let $\cG,\cH, \cK$ be locally compact groupoids, and let $f, g: \cG\to \cH$, $f_1:\cH\to \cK$ be homomorphisms. It is straightforward to check that if $f\sim g$ then $f_1\circ f\sim f_1\circ g$. From this observation, we easily deduce that similarity of groupoids is an equivalence relation.
\end{rem}

\subsection{Equivalence of groupoids}\label{subsec:eqgroupoid} Let $\cG$ and $\cH$ be two locally compact groupoids and let $(Z,\rho, \sigma)$ be a generalized morphism from $\cG$ into $\cH$. We can define a left $\cH$-action and a right $\cG$-action on $Z$ by setting $\eta z = z \eta^{-1}$ and $z\gamma = \gamma^{-1} z$ for $z\in Z$, $\eta\in \cH$, $\gamma\in \cG$  with $s(\eta) = \sigma(z)$ and $r(\gamma) = \rho(z)$. If $(Z,\sigma, \rho)$ is, similarly to   $(Z,\rho, \sigma)$, a generalized morphism from $\cH$ into $\cG$ this time, we say that   $(Z,\rho,\sigma)$ is a ($\cG$-$\cH$)-equivalence.  This property is summarized in the definition below.

\begin{defn}\label{def:equiv-groupoids}$($\cite{Ren82}, \cite{MRW}$)$ Let $\cG$ and $\cH$ be two locally compact groupoids. We say that   $(Z,\rho,\sigma)$ (or $Z$ for short) is a ($\cG$-$\cH$)-{\it equivalence} \index{equivalence of groupoids} if the following conditions hold:
\begin{itemize}
\item[(i)] $Z$ is a free and proper left $\cG$-space with moment map $\rho$ and is a free and proper right $\cH$-space with moment map $\sigma$, with commuting actions;
\item[(ii)]  the moment map $\rho: Z\to \cG^{(0)}$ is {\it open} and induces a homeomorphism from $Z/\cH$ onto $\cG^{(0)}$;
\item[(iii)]  the moment map $\sigma: Z\to \cH^{(0)}$ is {\it open} and induces a homeomorphism from $\cG\setminus Z$ onto $\cH^{(0)}$.
\end{itemize}
If such a $Z$ exists, we say that {\it $\cG$ and $\cH$ are equivalent}.
\end{defn}

This notion of equivalence is indeed an equivalence relation on the set of groupoids. It does not imply similarity in general (see \cite[Example 3.13]{FKPS}). It encompasses many interesting exemples. Let us give three of them.

\subsubsection{Equivalence induced by a locally proper faithful generalized morphism}\label{subsec:lpgm}  It is well-known \cite{Rie82} that if $G$ is a closed subgroup of a locally compact group $H$, then $H$ defines an equivalence between the group $G$ and the semi-direct product groupoid  $(G\setminus H) \rtimes H$, where the moment map for the right action of  $(G\setminus H) \rtimes H$ on $G$ is the quotient map $H\to G\setminus H = \set{Gh: h\in H}$. Denoting by $c:G\to H$  the inclusion homomorphism (obviously locally proper and faithful), we observe that $H$ is the space $Z_c$ described in Proposition \ref{prop:homo}. More generally, any  locally proper and faithful generalized morphism $(Z,\rho,\sigma)$ from a locally compact groupoid $\cG$ to a locally compact groupoid $\cH$ defines an equivalence, as shown in the following proposition. This useful result  dates back to \cite[\S 1.2]{Abad03} for partial actions and was continued in \cite{KS02, RS}.

 \begin{prop}\label{prop:homo1} Let $(Z,\rho,\sigma)$ be a locally proper and faithful generalized morphism from a locally compact groupoid $\cG$ to a locally compact groupoid $\cH$, where $\rho$ is the moment map of $\cG\actson Z$ and $\sigma$ is the moment map of $\cH\actson Z$.  Then $\cH$ acts to the right in a canonical way on  the quotient space $Y=\cG\setminus Z$ which, endowed with the quotient topology, is locally compact. Moreover there is a canonical right action of the groupoid $Y\rtimes \cH$ on $Z$, in such a way that  $Z$ is a $($$\cG$-$(Y\rtimes \cH)$$)$-equivalence of groupoids
 \end{prop}
  
  \begin{proof}   The proof is an immediate generalization of that of \cite[Theorem 6.2 (1)]{RS}.  Since the left action of $\cG$ is proper, $Y$  is locally compact and the quotient map $\sigma_Z$ from $Z$ onto $Y$ is open \cite[Lemma 2.1.11, Proposition 2.1.12]{AD-R}.   We denote by $[z]$ the class of $z\in Z$ and by $\sigma_Y$ the map $[z]\mapsto \sigma(z)$ from $Y$ into $\cH^{(0)}$. We have $\sigma = \sigma_Y\circ\sigma_Z$. Then $\cH$ acts to the right on $Y$ by $\alpha_Y: ([z],\eta) \mapsto [z\eta]$ from $Y\,_{\sigma _Y} \!*_r\cH$ into $Y$. Using the facts that $\sigma_Z$ is continuous, open and surjective, and the continuity of the right action $\cH\actson Z$, we see that the action $\alpha_Y$ is continuous.
  
  Next we define a right action of $Y\rtimes \cH$ on $Z$ with moment map $\sigma_Z$ in the following way, for $z,y,\eta \in Z\times Y \times \cH$, with $[z] = y$ and $\sigma_Y(y) = r(\eta)$:
  $$z(y,\eta) = z\eta.$$
  One checks that this action is continuous, free, proper, and commutes with the left $\cG$-action on $Z$. Moreover we have $\cG\setminus Z =Y = (Y\rtimes \cH)^{(0)}$ and $Z/(Y\rtimes \cH) = Z/\cH = \cG^{(0)}$.
  
  Note that Z is a ($\cG$-$\cH$)-equivalence precisely when $Y = \cH^{(0)}$.
 \end{proof}
   
  \begin{ex}\label{ex:Abd}(\cite{Abad03, KS02, RS}) Let $\boldsymbol{\beta } = (\set{X_t}_{t\in G}, \set{\beta_t}_{t\in G})$ be a partial action of a locally compact group $G$ on a locally compact space $X$. The previous proposition  provides a necessary and sufficient condition for $\boldsymbol{\beta}$ to be a reduction of a global action of $G$, namely that the homomorphism $c: (x,t) \mapsto t$ from $\cG =X\rtimes_\beta  G$ into $G$ is locally proper (or, equivalently that $\set{(x, t, \beta_{t^{-1}}(x): x\in X_t, t\in G}$ is a closed subset of $X\times G\times X$.  Let us give some details. We keep the notation of the section \ref{subsec:PartialAction}, adapted to our context. 
    
  First, let us assume that $\boldsymbol{\beta}$ is the reduction of an action $\alpha: G\actson Y$ to an open subset $X$ of $Y$. Then the morphism $c: (x,t) \mapsto t$ from $\cG$ into $G$ is locally proper. Indeed, let $K$ be a compact subset of $X$. We have 
  $$\cG(K) = \set{(x,t)\in (K\cap X_t)\times G: \beta_{t^{-1}}(x)\in K}= \set{(x,t): x\in K\cap \alpha_{t}(K)}$$
 where  $ K\cap \alpha_{t}(K)$ is a compact subset of $X_t$. It follows that the map $\psi:(x,t)\mapsto (x,t,\beta_{t^{-1}}(x))$ is proper from $\cG(K)$ into $X\times G\times X$.

Conversely, let us assume that the morphism $c: (x,t) \mapsto  t$ is locally proper. As in Proposition \ref{prop:homo} we introduce the space $Z_c$, which is $X\times G$ in case the target of $c$ is a group. The left $\cG$-action is defined by $(x,t)(\beta_{t^{-1}}(x),s) = (x,ts)$. The right $G$-action is defined by $(x,t)s = (x,ts)$. The moment map $\rho$ is $(x,t)\mapsto x$. We are in the situation of the previous proposition since $c$ is faithful and locally proper (see Proposition \ref{prop:homo}). The locally space $Y$ is the quotient of $X\times G$ with respect to the following equivalence relation: $(x,s)\sim (x_1,s_1)$ if $x\in X_{ss_1^{-1}}$ and $\beta_{s_1s^{-1}}(x) = x_1$. The right action  $\alpha$ of $G$ on $Y$ is given by $[x,t]s = [x,ts]$. We embed continuously $X$ into $Y$ by $\iota(x) = [x,e]$. Let us show that $\iota$ is open. If $U$ is an open subset of $X$ we have to check that $q^{-1}(\iota(U))$ is an open subset of $X\times G$ where we denote here by $q$ the quotient map from $X\times G$ onto $Y$. We have
$$q^{-1}(\iota(U)) = \set{(x,t)\in X\times G: x\in X_t, \beta_{t^{-1}}(x) \in U}.$$
It is open, since by definition of partial actions, the set $C=\set{(x,t), x\in X_t}$ is open in $X\times G$ and $(x,t) \mapsto \beta_{t^{-1}}(x)$ is continuous from $C$ into $X$.
Finally, we have $\iota(\beta_t(x)) = [\beta_t(x), e] = [x,t^{-1}] = \iota(x)t^{-1}$. It follows that $\boldsymbol{\beta}$ is identified to the restriction to $X=\iota(X)$ of the global left action $\tilde\alpha : (t,y) \mapsto yt^{-1}$ from $G$ onto $Y$.

Note that $X\rtimes_\beta G$ is equivalent to $Y\rtimes_{\tilde\alpha} G$, which is the same as the semi-direct product $Y\rtimes_\alpha G$ relative to the right $G$-action $ \alpha$.
 
    \end{ex}

\subsubsection{Transversals}

Let $\cG$ be a locally compact groupoid and let $A$ be a locally closed  subset of $\cG^{(0)}$.   What is generally missing for $\cG(A)$ to be a locally compact groupoid is that its range and source maps are not open, although $r:\cG^A\to A$ and $s:\cG_A\to A$ are open. Let mention however the following result.

\begin{lem}\label{lem:reducproperty} We keep the above notation and we assume that $s : \cG^A \to \cG^{(0)}$ is open.
\begin{itemize}
\item[(i)] The range and source maps of $\cG(A)$ are open.
\item[(ii)] If $s: \cG^A\to \cG^{(0)}$  is a local homeomorphism then the groupoid  $\cG(A)$ is \'etale.
\end{itemize}
\end{lem}

\begin{proof} In order to prove (i) we first observe that the second projection $\pi_2$ from $\cG^A  \,_s\!*_r \cG_A$ into $ \cG_A$ is open. Indeed, given two open subsets $U,V$ of $\cG$, we have
$$\pi_2\big((U\times V)\cap  ( \cG^A  \,_s\!*_r \cG_A)\big) = \pi_2(U^A \,_s\! *_r V_A) = V_A^{s(U^A)}$$
which is open since the restriction of $s$ to $\cG^A$ is open. The following diagram, where $p$ denotes the product and $s_A$ denotes the restriction of $s$ to $\cG(A)$, is commutative:
$$\xymatrix{
\cG^A  \,_s\! *_r \cG_A\ar[d]^{p} \ar[r]^{\pi_2} &  \cG_A  \ar[d]^{s}\\
\cG(A) \ar[r]^{s_A} & A}$$

If $U$ is an open subset of $\cG(A)$, then $s_A(U) = s\big(\pi_2(p^{-1}(U))\big)$ is open.

Next, let us prove (ii). Since $s_A$ is open, it suffices to show that $A$ is an open subset of $\cG(A)$ (see \cite[Lemma 1.26]{Will19}). Let $x\in A$ and let $V$ be an open neighborhood of $ x$ in $\cG^A$ on which $s$ is injective.  Since $U = V \cap A$ is an open subset of $A$ we see that $W=V\cap s_A^{-1}(U)$ is an open neighborhood of $x$ in $\cG(A)$. It is contained into $A$ because if $\gamma\in W$ we have $s(\gamma) = s(s_A(\gamma))$ with $s_A(\gamma) \in U\subset V$ and therefore $\gamma = s_A(\gamma)$. It follows that $A$ is open in $\cG(A)$.
\end{proof}

\begin{prop}\label{lem:reducproperty1} Let $\cG$ be a locally compact groupoid and let $A$ be a locally closed  subset of $\cG^{(0)}$. We assume that $s:\cG^A\to\cG^{(0)}$ is open and surjective. Then $\cG^A$, equipped with the canonical left action of $\cG(A)$ and right action of $\cG$ on $\cG^A$ is a  $($$\cG(A)$-$\cG$$)$-equivalence of groupoids.
\end{prop}

\begin{proof}  Immediate.
\end{proof}

A locally compact subset $A$ of $\cG^{(0)}$ such that $s:\cG^A\to \cG^{(0)}$ is open and surjective is called a {\it transversal}\index{transversal} (see \cite[Example 2.7]{MRW}). In particular, $A$ intersects every $\cG$-orbit $\cG\cdot x = r(\cG_x)$. \index{$\cG$-orbit}

\subsubsection{Transitive groupoids}\label{sub:tg}  We consider the particular case where $A=\set{x}$ and we assume that the locally compact groupoid $\cG$ is transitive\index{transitive groupoid}, that is, for all  $(y,z)\in  \cG^{(0)}$ there exists $\gamma\in \cG$ such that $r(\gamma) = y$ and $s(\gamma) = z$. Then all the isotropy groups of $\cG$ are isomorphic. The restriction of $s$ to $\cG^x$ is not open in general but however this is true  if $\cG$ is second countable. Then in this case $\cG^x$ is a ($\cG(x)$-$\cG$)-equivalence (see \cite[Theorem 2.2.B]{MRW}, \cite[Example 2.32]{Will19}). 

As a particular case one finds the example of the subsection \ref{subsec:lpgm} corresponding to the action of a locally compact group on the quotient space relative to a closed subgroup.

\subsubsection{A more general example}\label{subsec:FER} It generalizes the case $G\actson G/H$ and has been studied in connection  with the notions of approximate lattices and aperiodic order (see \cite{Favre, ER} and the references therein). Let $G$ be a locally compact group and let $C$ be a closed subset of $G$.   We denote by $\Omega(C)$ the closure of the set of all left translated of $C$ in the space $Closed(G)$ of all closed subsets of $G$ endowed with the Chabauty-Fell topology.  The space  $Closed(G)$ is compact and is second countable when $G$ is second countable, which will be our assumption. Its topology is generated by the subsets $\set{P\in  Closed(G): K\cap P = \emptyset}$ and $\set{P\in  Closed(G): V\cap P \not = \emptyset}$, where $K$  (resp. $V$) ranges  over all compact  (resp. open) subsets of $G$. For more on the topology of $Closed(G)$ we refer to \cite{Favre, ER}. In particular the convergence of sequences is described  as follows.
\begin{lem}\label{lem:converge}  Let $(P_n)$ be a sequence in  $Closed(G)$. Then $(P_n)$ converges to $P$ if and only if both of the
following statements hold:
\begin{itemize}
\item[(i)] whenever $x \in P$ then there exist $x_n\in P_n$ such that $\lim x_n = x$;
\item[(ii)] whenever $(n_k)_{k\in\N}$ is a subsequence of $\N$ and $x_{n_k}\in P_{n_k}$ with $\lim x_{n_k}  = x\in G$, then $x\in P$.
\end{itemize}
 \end{lem}
We note that the empty set has an open neighborhood basis consisting of the sets $\cV_K = \set{P\in Closed(G): P\cap K = \emptyset}$, where $K$ runs through all compact subsets of $G$. It follows that $\emptyset\in \Omega(C)$ if and only if for all  compact subset $K$ of $G$ there exists $g\in G$ such that  $gK\cap C=\emptyset$.

We set $\Omega(C)^\times = \Omega(C)\setminus\set{\emptyset}$ and $\Omega_0(C) = \set{P\in \Omega(C): e\in P}$, where $e$ is the unit of $G$. We observe that $\Omega(C)^\times$ is locally compact if $\emptyset \in \Omega(C)$, and compact otherwise\footnote{In this case, one says that $C$ is relatively dense in $G$}, and that $\Omega_0(C)$ is compact, since it is closed in $\Omega(C)$.

The group $G$ acts to the left on $\Omega(C)^\times$. We denote by $\cG$ the semi-direct product $G\ltimes \Omega(C)^\times$.

\begin{prop}\label{prop:transverse} The set $A= \Omega_0(C)$ is a transversal for $\cG = G\ltimes \Omega(C)^\times$: the map $s: \cG^A\to \Omega(C)^\times$ is open and surjective. As a consequence, the groupoid $\cG$ is equivalent to its reduction $\cG(\Omega_0(C)) = \set{(g,P)\in G\times \Omega_0(C), g^{-1}\in P}$. 
\end{prop}

 \begin{proof}We have $\cG^A = \set{(g,P)\in \Omega(C)^\times:  g^{-1}\in P}$ and $s: (g,P) \mapsto P$ from $\cG^A$ to $\Omega(C)^\times$. 
Obviously, $s$ is surjective. Let us check that $s$ is open. Let $V$ be an open subset of $G$ and let $W$ be an open subset of $\Omega(C)^\times$. Then 
\begin{align*} s\big((V\times W)\cap \cG^A\big) &= \set{P\in W: g^{-1} \in P\quad \hbox{for some}\quad  g\in V}\\
& = W \cap \set{P\in \Omega(C)^\times: P\cap V^{-1} \not =\emptyset},
\end{align*}
is open.
\end{proof}

\begin{rem}\label{rem:UD} Let $C$ be a closed subset of $G$. As observed in \cite{ER},  $\Omega_0(C)$ is the closure of $\set{c^{-1}C: c\in C}$. Indeed, given $P = \lim_n g_n C \in \Omega_0(C)$, for every $n$ there exists $c_n \in C$ such that $\lim_n g_n c_n = e\in P$.  Then 
$$\lim_n c_n^{-1} C = \lim_n (c_n^{-1} g_n^{-1}) g_n C = P.$$

In particular, if $H$ be a closed subgroup of $G$, then $\Omega_0(H)$ is reduced to the element $H\in Closed(G)$ and then $\Omega(H)^\times = \set{gH: g\in G}$  because if $g\in P$ with $P\in \Omega(H)^\times$, then $g^{-1}P \in \Omega_0(H)$. In this case, $\cG$ is isomorphic to $G\ltimes G/H$ and its reduction to $A=\set{H}$ is the group $H$.
\end{rem}

\begin{defn}\label{def:UD} One says that a subset $\Lambda$ of $G$ is   {\it uniformly discrete}\index{uniformly discrete subset} if there exists a neighborhood $U$ of $e$ such that $\abs{\Lambda\cap gU}\leq 1$ for every $g\in G$ (or equivalently, if $\abs{P\cap U}\leq 1$ for every $P\in \Omega(\Lambda)$). 
\end{defn}

\begin{rem}\label{rem:caracUD} It is easily checked that a uniformly discrete set $\Lambda$ is discrete and closed in $G$ and that $\Lambda$ is uniformly discrete if and only if $e$ is not an accumulation point of $\Lambda^{-1}\Lambda$ (see \cite[\S 2]{BH}). In particular, $\Lambda$ is uniformly discrete whenever $\Lambda^{-1}\Lambda$ is discrete. Every discrete subgroup of $G$ is uniformly discrete.
\end{rem}

\begin{prop}\label{prop:UD} Let $\Lambda$ be a uniformly discrete subset of a locally compact group $G$. We set $A = \Omega_0(\Lambda)$. Then the map $s: \cG^A\to \Omega(\Lambda)^\times$ is a local homeomorphism and therefore the groupoid $\cG(\Omega_0(\Lambda))$ is \'etale. \end{prop}

\begin{proof} Let us show that $s: \cG^A \to \Omega(\Lambda)^\times$ is a local homeomorphism. Since this map is open, it is enough to check that every $(g_0, P_0)\in \cG^A$ has an open neighborhood on which $s$ is injective. Let $U$ be an open symmetric neighborhood of $e$ such that $\abs{P\cap U}\leq 1$ for every $P\in \Omega(\Lambda)$ and let $W$ be an open neighborhood of $P_0$ in $\Omega(\Lambda)^\times$. If $(g,P)$ and $(h,Q)$ in $(Ug_0 \times W)\cap \cG^A$ are such that $P =s(g,P) = s(h,Q) = Q$ then $g^{-1}, h^{-1}$ are both in $P\cap g^{-1}_0 U$. It follows that $g=h$.

 We conclude that  $\cG(\Omega_0(\Lambda))$ is \'etale by using Lemma \ref{lem:reducproperty}.
\end{proof}

\begin{rem}\label{rem:strong eq} Let $\Lambda$ be a uniformly discrete subset of a locally compact group $G$.  We observe that the $(\cG(A)$-$\cG)$-equivalence defined by $\cG^A$ is such that the associated moment maps $\cG^A\to A$ and $ \cG^A \to \Omega(\Lambda)^\times$ have local sections, which proves useful.
\end{rem}

 \subsubsection{Ampliations}\label{subsec:amplibis} We keep the notation of Example \ref{ex:ampliation}.
 \begin{prop}We set $Z= \set{(x,\gamma): \varphi(x) = r(\gamma)}$. Then $Z$ defines a $($$\cG^\varphi$-$\cG$$)$-equivalence as follows. The left action of $\cG^\varphi$ is given by $(x,\gamma,y)(y,\eta) = (x,\gamma\eta)$ and the right $\cG$-action is given by $(x,\gamma)\eta = (x,\gamma\eta)$.  In particular, $\cG$ is equivalent to any of its localizations.
\end{prop}
\begin{proof} See \cite[Example 2.37]{Will19}). \end{proof}

 \begin{thm}\label{thm:ampl} Let $\cG$ and $\cH$ be two locally compact groupoids. Then they  are equivalent if and only if they have isomorphic ampliations.
 \end{thm}
 
 \begin{proof} Assume first that some ampliations $\cG^\varphi$ and $\cH^\psi$ are isomorphic. Since $\cG$ and $\cG^\varphi$ are equivalent as well as $\cH$ and $\cH^\psi$, we conclude that $\cG$ and $\cH$ are equivalent by transitivity of this property.
 
 Conversely, let $Z$ be a $($$\cG$-$\cH$$)$-equivalence with moment maps $\rho: Z\to \cG^{(0)}$ and $\sigma:Z\to \cH^{(0)}$. It is proved in \cite[Proposition 3.10]{FKPS} that the ampliations $\cG^\rho$ and $\cH^\sigma$ are isomorphic under the assumption that the groupoids $\cG$ and $\cH$ are \'etale but this assumption is not needed (see also \cite[Proposition 2.29]{Tu04}).
  \end{proof}

\begin{lem}\label{lem:ampl} Let $(Z,\rho,\sigma)$ be a generalized morphism from a locally compact grou\-poid $\cG$ to a locally compact groupoid $\cH$. Then $\cG^\rho$ is a groupoid equivalent to $\cG$ and  there is a continuous homomorphism $c: \cG^\rho \to \cH$ defined as follows. Given $(z,\gamma, z') \in \cG^\rho$,  since $\rho(z) = r(\gamma) = \rho(\gamma z')$ there exists a unique element $\eta\in \cH$ such that $z\eta = \gamma z'$ and we set $c(z,\gamma, z') = \eta$.  This homomorphism $c: \cG^{\rho} \to \cH$ is continuous, and is locally proper whenever the left action of $\cG$ on $Z$ is proper.
 \end{lem}

\begin{proof} The fact that $c$ is a homomorphism is obvious. Let us show that $c$ is continuous. Let $(z_i, \gamma_i, z_{i}')_{i\in I}$ be a net in $\cG^{\rho}$ such that $\lim_i 
(z_i, \gamma_i, z_{i}') = (z_0, \gamma_0, z_{0}')$. We have $z_i c(z_i, \gamma_i, z_{i}') = \gamma_i z_{i}'$ for all $i\in I$, $\lim_i z_i = z_0$ and $\lim_i \gamma_i z_{i}' = \gamma_0 z_{0}'$. Since the right action of $\cH$ on $Z$ is proper, there is a compact subset $K$ of $\cH$  and $i_0$ such that $c(z_i, \gamma_i, z_{i}') \in K$ for $i\geq i_0$. Therefore we can assume that $(c(z_i, \gamma_i, z_{i}'))_{i\in I}$ converges to some $\eta\in \cH$. We have $z_0\eta = \gamma_0 z_{0}'$ and therefore $\eta = c(z_0, \gamma_0, z_{0}')$.

The straightforward proof of the second assertion is left to the reader.
\end{proof}
  
  \subsubsection{The case of ample groupoids} For other facts about equivalence of groupoids see \cite[Proposition 3.10, Theorem 3.12]{FKPS} and \cite[Theorem 2.1]{CRS}. Let us mention the following particular case, for ample groupoids. Recall that an {\it ample groupoid} \index{ample groupoid} is an \'etale groupoid whose space of units is totally disconnected. Equivalently, an \'etale groupoid is ample if and only if its topology has a basis consisting of compact open  bisections.
  
  \begin{thm}\label{thmCRS} Let $\cG$ and $\cH$ be ample groupoids. Assume that $\cG^{(0)}$ and $\cH^{(0)}$ are $\sigma$-compact. We denote by $\mathcal R$ the trivial equivalence relation $\N\times \N$. The following properties are equivalent:
  \begin{itemize}
  \item[(i)] the groupoids $\cG$ and $\cH$ are similar;
  \item[(ii)] the groupoids $\cG$ and $\cH$ are equivalent ;
  \item[(iii)] the groupoid products $\cG\times\mathcal R$ and $\cH\times\mathcal R$ are isomorphic.
  \end{itemize}
   \end{thm}

\section{\textbf{\textsc{Fibrewise equivariant compactifications of $\cG$-spaces}}}\label{sec:2}

We keep the notation of the appendix where fibrewise compactifications are studied. Let  $\cG$ be a locally compact groupoid with Haar system and let $(Y,p)$ be a left $\cG$-space over $X= \cG^{(0)}$. We are interested in  $\cG$-equivariant fibrewise compactifications   of  $(Y,p)$  in the following sense.

\begin{defn}\label{def:eq_comp} A {\it $\cG$-equivariant fibrewise compactification} \index{equivariant fibrewise compactification} of the left $\cG$-space $(Y,p)$ is a fibrewise compactification $(Z,\varphi, q)$ of $(Y,p)$ such that $(Z,q)$ is a left $\cG$-space satisfying $\varphi(\gamma y) = \gamma\varphi(y)$ for every $(\gamma, y)\in \cG\,_s\!*_pY$.
\end{defn}

 For $g\in \cC_c(\cG)$ and $f\in \cC_b(Y)$ we define the convolution product  $g*f$ by
 $$(g*f)(y) = \int_{\cG} g(\gamma) f(\gamma^{-1}y) \rd \lambda^{p(y)}(\gamma).$$
This function $g*f$ is bounded but there is no reason why it should be continuous in general.

The situation is different when we are given a $\cG$-equivariant fibrewise compactification $(Z,\varphi, q)$ of $(Y,p)$  and when  $f$ is the restriction to $Y$ of $F\in \cC_0(Z)$, because $g*F \in \cC_0(Z)$ (see Lemma \ref{lem:convol})  and $g*f$ is the restriction of $g*F$ to $Y$.

Recall (Proposition \ref{prop:f_c}) that the fibrewise compactifications $(Z,\varphi, q)$ of $(Y,p)$ are in bijective correspondence with the $C^*$-algebras $A = \cC_0(Z)$ such that 
$$p^*\cC_0(X) + \cC_0(X) \subset A \subset \cC_0(Y,p),$$
 after having identified the elements of $\cC_0(Z)$ with their restriction to $Y$. This idenfication will always be done implicitly.

The above observations show that when $(Z,\varphi,q)$ is a $\cG$-equivariant   fibrewise compactification of $(Y,p)$, then  $A= \cC_0(Z)$, viewed as a subset of $\cC_b(Y)$, is stable under convolution by the elements of $\cC_c(\cG)$. Thus the first item of the following proposition \ref{prop:extaction} is satisfied. The lemma below is needed for the proof of the second item.

\begin{lem} Let $(Z,q)$ be a left $\cG$-space, $K$ a compact subset of $Z$ and $U$  a neighborhood of $K$. Then  $\Omega= \set{\gamma\in \cG: \gamma K \subset U}$  is an open neighborhood of $X$.
 \end{lem}
 
 \begin{proof} Obviously, $\Omega$ contains $X$. For $x\in X$ we set $K^x = K\cap q^{-1}(x)$. Let $\gamma_0\in \Omega$. For every $z\in K^{s(\gamma_0)}$ there exist a neighborhood $V_z$ of $\gamma_0$ and an open neighborhood $W_z$ of $z$ such that $\gamma z' \in U$ when $\gamma \in V_z$ and $z' \in W_z\cap Z^{s(\gamma)}$. Since 
$K^{s(\gamma_0)}$ is compact, there exist $z_1,\cdots, z_n$ such that $K^{s(\gamma_0)}\subset W=  \bigcup_{i= 1}^n W_{z_i}$. We set $V = \bigcap_{i= 1}^n V_{z_i}$ and $V' = V\bigcap s^{-1}\big(X\setminus q(K\setminus W)\big)$. We have $\gamma_0\in V'$ and $V'$ is a neighborhood of $\gamma_0$ since $q(K\setminus W)$ is compact, hence closed. If $\gamma\in V'$ and $z\in K^{s(\gamma)}$, see that $z\in W$ and therefore $z\in W_{z_i}$ for some $i$.  It follows that $\gamma z\in U$ since $\gamma\in V_{z_i}$. We conclude that $V'\subset \Omega$.
  \end{proof}
 
 \begin{prop} \label{prop:extaction} Let $(Z,\varphi, q)$ be a $\cG$-equivariant fibrewise compactification of $(Y,p)$ and set $A = \cC_0(Z)$. Then
 \begin{itemize}
 \item[(i)] for every $g\in \cC_c(\cG)$ and every $f \in A$ we have $g*f\in A$;
 \item[(ii)] for every $\varepsilon >0$ and every $f\in A$ there is a neighborhood $\mathcal V$ of $X$ such that $\abs{f(\gamma y) - f(y)}< \varepsilon$ for $(\gamma, y)\in \mathcal V\,_ s*_p Y$.
 \end{itemize}
 \end{prop}
 
 \begin{proof}  In order to prove (ii) it suffices to consider the case where $f\in \cC_c(Z)$. Let $K$ be the support of $f$ and choose an open set $U$ and a compact set $K_1$ such that $K\subset U \subset K_1\subset Z$.  We consider
 $$O = \set{\gamma \in \cG : \abs{f(\gamma z) - f(z)}<\varepsilon \quad \hbox{for}\quad  z \in  K_{1}^{s(\gamma)}}.$$
 Let us check that $O$ is open. Let $\gamma_0\in O$ and $z\in K_{1}^{s(\gamma_0)}$. We set 
 $$h(\gamma, z) =\abs{f(\gamma z) - f(z)}.$$
  It is a continuous function on $\cG\,_s*_q Z$. There exist open neighborhoods $V_z$ and $W_z$ of $\gamma_0$ and $z$ respectively, such that $h(\gamma, z')< \varepsilon$ for $(\gamma,z') \in (V_z\times  W_z)\cap (\cG\,_s*_q Z)$. Let $z_1,\cdots z_n\in  K_{1}^{s(\gamma_0)}$ such that $ K_{1}^{s(\gamma_0)} \subset W=\bigcup_{i= 1}^n W_{z_i}$ and set $V = \bigcap_{i= 1}^n V_{z_i}$. As in the proof of the previous lemma we see that  $V' = V\cap s^{-1}\big(X\setminus q(K_1\setminus W)\big)$ is a neighborhood of $\gamma_0$ and that $h(\gamma,z) <\varepsilon$  if $\gamma\in V'$ and $z\in K_1^{s(\gamma)}$. Therefore $O$ is an open neighborhood of $X$. 
 
 Take $\Omega$ as in the previous lemma, with respect to $K\subset U$ and set $\mathcal V = O\cap \Omega^{-1}$. Obviously $\mathcal V$ is an open set containing $X$. Let  us check that $h(\gamma,z)< \varepsilon$ for $(\gamma, z) \in \mathcal V\,_s*_q Z$. Either $z\in K_1$ and we are done or $z\notin K_1$. In this case $\gamma z\notin K$, otherwise $z\in \Omega K \subset U \subset K_1$, a contradiction; so we have $h(\gamma, z) = 0$ whenever $z\notin K_1$.
 \end{proof}

 We don't know whether the converse of the proposition \ref{prop:extaction} holds, namely, given a fibrewise compactification $(Z,\varphi, q)$ of the $\cG$-space $(Y,p)$ such that (i) and (ii) are satisfied can we extend the $\cG$-action on $Y$ to a $\cG$-action on $Z$ ? This is at least true in two important particular situations.
 
 \noindent{\bf -- Case where $\cG$ is a locally compact group.}
 
 \begin{prop}
 Let $G\actson Y$ be a left action of a locally compact group on a locallly compact space $Y$ and let $A$ be a sub-$C^*$-algebra of $\cC_b(Y)$ which contains $\C + \cC_0(Y)$.  Then the $G$-action on $Y$ extends to a $G$-action on the spectrum $Z$ of $A$ if and only if the conditions $(i)$ and $(ii)$ of the proposition \ref{prop:extaction} are fulfilled.
 \end{prop}
 
 \begin{proof} Assume that the conditions $(i)$ and $(ii)$ are satisfied. Condition (ii) means that  every $f\in A$ is left uniformly continuous: if we set $\alpha_t(f)(y) = f(t^{-1}y)$ then $\norm{\alpha_t(f) - f}_\infty$ goes to $0$ when $t$ goes to the unit element $e\in G$. Let us show that $\alpha_t(f)\in A$ for all $t\in G$. Let $g\in \cC_c(G)$ be a non-negative function such $\int_G g(s) \rd s = 1$. We have
\begin{align*}\abs{\int_G g(s) \alpha_s(f) \rd s - \alpha_t(f)} &= \abs{\int_G g(s) \big(\alpha_s(f) -\alpha_t(f)\big) \rd s} \\ &\leq \sup_{s\in K}\norm{\alpha_s(f) -\alpha_t(f)}_\infty,
\end{align*}
where $K$ is the support of $g$. Given $\varepsilon >0$ there exists a neighborhood $V$ of $e$ such that  $\sup_{s\in tV}\norm{\alpha_s(f) -\alpha_t(f)}_\infty\leq \varepsilon$. We choose $g$ such that $\supp(g) \subset tV$.
Then we get $\norm{g*f - \alpha_t(f)}_\infty \leq \varepsilon$. It follows that $\alpha_t(f)$ belongs to the closure of $A$ and therefore $\alpha_t(f)\in A$. The action of $G$ by automorphisms  of $A$ induces an action of $G$ on $Z$ which extends the action $G\actson Y$.
 \end{proof}
 
 Let $Y$ be a $G$-space. Note that $ \C + \cC_0(Y)$ satisfies the conditions (i) and (ii) and its spectrum  $Z$ is the Alexandroff compactification $Y^+$ of $Y$. Likewise, the $C^*$-algebra $A$ of the left uniformly continuous bounded functions on $Y$ satisfies both conditions. We denote by  $\beta^u Y$ its spectrum. In particular the $G$-pace $\beta^u G$ (relative to the left action of  $G$ on itself)  is the universal $G$-space for the $G$-actions on compact spaces. Note that $\beta^u Y$ is the the Stone-\v Cech compactification $\beta Y$ of $Y$ when $G$ is discrete.
 
 \noindent{\bf -- Case where $\cG$ is an \'etale groupoid.}

In the rest of this section $\cG$ {\it will always be an \'etale groupoid} and, without further mention, $X$ will denote its space of units. Let $(Y,p)$ be a left $\cG$-space, and let $(Z,\varphi, q)$ be a fibrewise compactification of $(Y,p)$. Obviously the condition (ii) of the proposition  \ref{prop:extaction} is satisfied with $\mathcal V = X$. We will show that the condition (i) suffices to extend the given $\cG$-action to $Z$.

 For $g\in \cC_c(\cG)$ and $f\in \cC_b(Y)$ the convolution product $g*f$  is defined by
\begin{equation}\label{eq:convol}
(g*f)(y) = \sum_{\gamma\in r^{-1}(p(y))} g(\gamma) f(\gamma^{-1} y).
\end{equation}
Thanks to the fact that $g$ is a finite sum of continuous functions supported in an open bisection, to study the properties of the convolution product it suffices to consider the case where $g$ has a compact support $K$ contained in an open bisection $S$. Then
\begin{align*}
(g*f)(y)& = 0 \quad\hbox{if}\quad y\notin p^{-1}(r(K))\\
 &=  g(\gamma) f(\gamma^{-1}y) \quad\hbox{if}\quad y\in p^{-1}(r(S)),
\end{align*}
where $\gamma = r_{S}^{-1}(p(y))$ and $\gamma^{-1}y = {S^{-1}}y$. 

It follows that for every $g\in \cC_c(\cG)$, the function  $y\mapsto (g*f)(y)$ is continuous and belongs to $\cC_c(Y,p)\subset \cC_b(Y)$ with $\abs{(g*f)(y)} \leq \big(\sum_{\gamma\in r^{-1}(p(y))} \abs{g(\gamma)}\big)\norm{f}_\infty$.

Let $(Z,\varphi, q)$ be a fibrewise compactification of $(Y,p)$.  Assume that $A =\cC_0(Z)$, identified to a sub-$C^*$-algebra of $\cC_b(Y)$, is stable under convolution by the elements of $\cC_c(\cG)$. For $F\in \cC_0(Z)$ we emphasize  the fact that $g*F$ will denote both the convolution product  defined on $Y$ and its unique extension as a continuous function on $Z$ (which cannot be defined, {\it a priori}, as a convolution product).  We want to show that there is a unique continuous $\cG$-action on $Z$ extending  the given $\cG$-action on $Y$. For $(\gamma_0,z_0) \in \cG\,_s\!*_q Z$, we first explain how $\gamma_0 z_0$ is defined.

\begin{lem}\label{lem:ext} We assume that $\cC_0(Z)$ is stable under convolution by the elements of $\cC_c(\cG)$. Let $(\gamma_0,z_0) \in \cG\,_s\!\!*_q Z$. Let $S$ be an
open bisection such that $\gamma_{0} \in S$ and  let $a$ be a continuous
function on $X= \cG^{(0)}$ with compact support in $s(S)$, such that $a(x)=1$ for $x$ in a neighborhood of
$s(\gamma_0)$. We set $h = (a\circ s)\,\chi_{S}$ and $g(\gamma) = h(\gamma^{-1})$ (where $\chi_S$ is the characteristic function of $S$).  
\begin{itemize}
\item[(i)] Let $F\in \cC_0(Z)$. If $z_0\in Y$, we have $(g*F)(z_0) = F(\gamma_0z_0)$.
\item[(ii)] For $F\in \cC_0(Z)$, the complex number $(g*F)(z_0)$ does not depend on the choice of
$a$ and $S$ satisfying the above properties, but only on $(\gamma_0,z_0)$.
 \item[(iii)] $F\in \cC_0(Z) \mapsto (g*F)(z_0)$
is a character of $\cC_0(Z)$ that we denote by $\gamma_0 z_0$.
\end{itemize}
\end{lem}

\begin{proof} (i)  Observe that $g\in \cC_c(\cG)$ and so $g*F\in \cC_0(Z)$.  For $y\in Y\subset Z$, we have
\begin{align}\label{al:convol}
(g*F)(y) & = a\circ p(y) F(Sy)\quad\hbox{if}\quad 
p(y) \in \hbox{Supp}(a) \subset s(S)\\
 &= 0\quad\hbox{ otherwise.}\notag
\end{align}
Obviously, if $z_0\in Y$ we have $(g*F)(z_0) = F(\gamma_0z_0)$.

(ii) Consider two open bisections $S_1$, $S_2$ with $\gamma_0
\in S_1 \cap S_2$, and let $a_1$, $a_2$
be two continuous functions with compact support in $s(S_1)$
and $s(S_2)$ respectively, with constant value $1$ in a neighborhood of
$s(\gamma_0)$. Then by \eqref{al:convol}, there is an open neighborhood $V$ of $s(\gamma_0)$ in $X$
such that $g_1 * F$ and $g_2 * F$ coincide on $p^{-1}(V)$.
By Lemma \ref{lem:up}, we see that $g_1 * F= g_2*F$ on $q^{-1}(V)$ and in particular we have $(g_1*F)(z_0) = (g_2*F)(z_0)$.

(iii) To show that $F\mapsto (g*F)(z_0)$ is a character, we first observe that
given $F$ and $F'$ in $\cC_0(Z)$, we have, for $y\in Y$ such that  $p(y)$ belongs to a neighborhood $V$ of $s(\gamma_0)$
on which $a$ takes value $1$,
$$a\circ p(y)(FF')(Sy) = a\circ p(y)F(Sy)\,
a\circ p(y)F'(Sy),$$
that is $$
(g*(FF'))(y) = (g*F)(y)\, (g*F'(y).$$
Once again we conclude thanks Lemma \ref{lem:up} that this equality still holds when $y$ is replaced by $z_0$.

It remains to check that there exists $F\in \cC_0(Z)$ such that $(g*F)(z_0) \not= 0$. We take $S$ to be relatively compact. Let $V$ be a compact neighborhood of $s(\gamma_0)$ in $X$, contained in $s(S)$ and on which $a(x) = 1$. We choose $f\in \cC_c(X)$  such that $f(x) = 1$ if $x\in r(SV)= S\cdot V$ and we set $F = f\circ q$. Since $q$ is proper, we note that $F\in \cC_c(Z)$. For $y\in p^{-1}(V)$, we have, since $p$ is equivariant,
$$(g*F)(y) = F(Sy) = f\circ p(Sy) = f(S\cdot p(y)) =1.$$
Once more, we deduce from Lemma \ref{lem:up} that $(g*F)(z_0) = 1$.
\end{proof}
Note that, by definition,
\begin{equation}\label{eq:extend}
 F(\gamma_0z_0)= (g*F)(z_0).
\end{equation}

\begin{thm}\label{thm:ext} Let $(Y,p)$ be a left $\cG$-space, where $\cG$ is an \'etale groupoid. Let $(Z,\varphi, q)$ be a fibrewise compactification of $(Y,p)$ such that $\cC_0(Z)$ is stable under convolution by the elements of $\cC_c(\cG)$. The map $(\gamma_0, z_0) \mapsto \gamma_0 z_0$ is a continuous 
${\mathcal G}$-action on $Z$, and it is the only
continuous ${\mathcal G}$-action extending the action of ${\mathcal G}$
on Y.
\end{thm}

\begin{proof} Since $\cG \,_s\!*_p Y$ is dense into $\cG \,_s\!*_q Z$, there is at most
one continuous extension of the $\cG$-action. 

For $(\gamma_0,z_0) \in \cG \,_s\!*_q Z$, let us check  that $q(\gamma_0 z_0) = r(\gamma_0)$. Let $V$ be a compact
neighborhood of $r(\gamma_0)$ in $X= \cG^{(0)}$ and let $f$ be a continuous
function supported in $V$ and taking value $1$ in a neighborhood of $ r(\gamma_0)$.
Set $F=f\circ q \in \cC_0(Z)$ and take $g$ as in the previous lemma. Here, our choices of $a$ and $S$ are such that $a=1$ on a relatively compact open neighborhood $W$ of $s(\gamma_0)$ and $f =1$ on $r(SW)$. We have
$$(g*F)(y) = a\circ p(y) F(Sy) =f\circ p(Sy) = 1$$
 for every $y\in Y$ such that $p(y)\in W$. It follows that $(g*F)(z) = 1$ for $z\in q^{-1}(W)$. In particular $F(\gamma_0 z_0) = (g *F)(z_0) =1$, and therefore $q(\gamma_0 z_0) \in V$. Since this holds for every $V$ in a basis of neighborhoods of $r(\gamma_0)$ we conclude that $q(\gamma_0 z_0) = r(\gamma_0)$.

Given $z_0\in Z$, let us check that $F(q(z_0) z_0) = F(z_0)$ for every $F\in \cC_0(Z)$. For the definition  of $F(q(z_0) z_0)$ we take $S = X$. Then in \eqref{al:convol} we get, for $y\in Y$,
$$(g*F)(y) = a\circ p(y) F(y),$$
and therefore $(g*F)(z) = a\circ q(z) F(z)$ for $z\in Z$. In particular, we have $F(q(z_0) z_0) = (g*F)(z_0) = F(z_0)$.
It follows  that $q(z_0)z_0 = z_0$. 

Thus, the first item in the definition \ref{def:Gspace} is fulfilled. 
Let us  show the continuity of $(\gamma,z)\mapsto \gamma z$. Let $(\gamma_0, z_0) \in \cG \,_s\!*_q Z$ and let $W$ be a neighborhood
of $\gamma_0 z_0$. Let us consider $F\in \cC_0(Z)$ such that $F(\gamma_0 z_0) = 1$ and $F(z) = 0$ if $z\not\in W$. We take $S$, $a$ and $g$ as  in the previous lemma. Let  $U$ be an open neighborhood of $\gamma_0$, contained into $S$ such that $a\circ s(\gamma) = 1$ if $\gamma \in U$.   For $(\gamma,z)\in U_s\!*_q Z$  we have
$F(\gamma z) = (g*F)(z)$. Since $g*F$ is continuous and $ (g*F)(z_0) = 1$, there is a neighborhood $V$
of $z_0$ in $Z$ on which $g*F > 0$. It follows that for $(\gamma, z) \in U _s\!*_qV$ we have
$\gamma z \in W$.

Finally, let us show the associativity: $(\gamma_1\gamma_2)z = \gamma_1(\gamma_2 z)$ when this expression makes sense\footnote{This would be obvious if the fibres of $Y$ were dense into those of $Z$.}.
We have to check that $F((\gamma_1\gamma_2)z) = F(\gamma_1(\gamma_2 z))$ for every $F\in \cC_0(Z)$.
For $i=1,2$, we choose an open bisection $S_i$ containing $\gamma_i$, a continuous function $a_i$ with compact support in $s(S_i)$ and taking value $1$ in a neighborhood of $s(\gamma_i)$.
 We introduce  $g_i$ (with respect now to $S_i, \gamma_i$ instead of $S,\gamma_0$) as in  the previous lemma.  A straightforward computation shows that
$$F((\gamma_1(\gamma_2 z)) = g_2*(g_1*F)(z).$$
In the groupoid algebra $\cC_c(\cG)$  the convolution product of $g_2$ by $g_1$ is defined as
\begin{align*}
(g_2*g_1)(\gamma) &= \sum_{\set{\gamma':r(\gamma') = r(\gamma)}} g_2(\gamma')g_1(\gamma'^{-1}\gamma)\\
&= a_2\circ r(\gamma)\sum_{\set{ \gamma'\in S_2: s(\gamma') = r(\gamma}} a_1\circ s(\gamma^{-1}\gamma'^{-1})\chi_{S_1}(\gamma^{-1}\gamma'^{-1}) .
\end{align*}
It follows that $(g_2*g_1)(\gamma) = 0$ except possibly whenever $\gamma'\in S_2$ together with $\gamma^{-1}\gamma'^{-1} \in S_1$ and then we get
$$(g_2*g_1)(\gamma) = a_2\circ r(\gamma)a_1\circ r(\gamma')$$
where $\gamma'$ depends on $\gamma$ as explained now.
Let $\theta_2$ denote the homeomorphism $x\mapsto S_2 \cdot x$ from $s(S_2)$ onto $r(S_2)$.
Then we have, when $r(\gamma') \in s(S_1)$, 
$$a_1\circ r(\gamma') = a_1\circ\theta_2\circ s(\gamma') = a_1\circ\theta_2\circ r(\gamma),$$ 
and therefore
$$(g_2*g_1)(\gamma) = a_2\circ r(\gamma)( a_1\circ\theta_2)\circ r(\gamma)\chi_{S_1S_2}(\gamma^{-1}).$$

We observe that $a_1\circ \theta_2$ is equal to $1$ in a neighborhood of $s(\gamma_2)$.
We can choose $S_2$ so that $r(S_2)\subset s(S_1)$ and $a_1$ with compact support in $r(S_2)$. Then $S_1S_2$ is a bisection containing $\gamma_1\gamma_2$ with $s(S_1S_2) = s(S_2)$.
 Moreover, $a_2(a_1\circ\theta_2)$ is continuous with compact support in $s(S_2)$ and is equal to $1$
in a neighborhood of $s(\gamma_2) = s(\gamma_1\gamma_2)$.

In conclusion, we get
$$F(\gamma_1(\gamma_2 z)) = g_2*(g_1*F)(z) = \big((g_2*g_1)* F\big)(z) = F((\gamma_1\gamma_2)z),$$
where the last equality follows from \eqref{eq:extend}.
Since this holds for every $F\in \cC_0(Z)$, we obtain  $\gamma_1(\gamma_2 z) = (\gamma_1\gamma_2)z$.
\end{proof}

\begin{cor}\label{cor:fin} Let $(Y,p)$ be a $\cG$-space, where $\cG$ is an \'etale groupoid. The $\cG$-equivariant fibrewise compactifications of $(Y,p)$ are in bijective correspondence with the $C^*$-subalgebras of $\cC_0(Y,p)$ which contain $p^*\cC_0(X) + \cC_0(Y)$ and are stable under convolution  by the elements of $\cC_c(\cG)$. This correspondence is the restriction to the set formed by these $C^*$-subalgebras of the bijection described in Proposition \ref{prop:f_c}.
\end{cor}

\begin{proof} The only fact that remains to show is that if $(Z,\varphi, q)$ is a  $\cG$-equivariant fibrewise compactification, then $\varphi^*\cC_0(Z)$ is stable under convolution  by the elements of $\cC_c(\cG)$.
Let $F\in \cC_0(Z)$ and $g\in \cC_c(\cG)$. We define $g\star F$ by the expression
$$(g\star F)(z) = \sum_{\gamma\in r^{-1}(q(z))}g(\gamma) F(\gamma^{-1}z).$$
This function is continuous with  support in $q^{-1}(r(K))$ where $K$ is the compact support of  $g$. Since $q$ is proper, the support of $g\star F$ is compact and so $g\star F\in \cC_0(Z)$.

To conclude, we observe that $(g\star F)\circ \varphi = g*(F\circ\varphi)$.
\end{proof}

\begin{prop}\label{prop:max_min}  Let $(Y,p)$ be a left $\cG$-space, where $\cG$ is an \'etale groupoid. The structure of $\cG$-space of $(Y,p)$ extends in a unique way to the Alexandroff fibrewise compactification $(Y^+, p^+)$ and to the Stone-\v Cech fibrewise compactification $(\beta_p Y, p_\beta)$. This makes  them $\cG$-equivariant fibrewise compactifications.
\end{prop}

 \begin{proof} We have to show that the $C^*$-algebras   $\cC_0(Y,p)$ and
$p^*{\mathcal C}_0(X) + {\mathcal C}_0(Y)$ are stable under convolution by the elements of $\cC_c(\cG)$. The verifications are immediate.
\end{proof}

We point out that $(\beta_p Y, p_\beta)$ is the solution of an universal problem.

\begin{prop}\label{prop:universal1} Let $\cG$ be an \'etale groupoid and $(Y_1,p_1)$, $(Y,p)$ be two left $\cG$-spaces. We assume that $(Y_1,p_1)$ is fibrewise compact.  Let $\varphi_1: (Y, p) \to (Y_1,p_1)$ be a $\cG$-equivariant morphism. The  unique continuous map  $\Phi_1 : \beta_p Y \to Y_1$ that extends $\varphi_1$ is $\cG$-equivariant and proper\footnote{See Proposition \ref{prop:universal} for the existence of $\Phi_1$}. 
\end{prop}

\begin{proof} The two continuous maps $(\gamma, z) \mapsto \Phi_1(\gamma z)$ and $(\gamma, z) \mapsto \gamma\Phi_1(z)$, which are defined on $\cG\,_s\!*_{p_\beta} \beta_p Y$ coincide on $\cG\,_s\!*_{p} Y$ which is dense into $\cG\,_s\!*_{p_\beta} \beta_p Y$. 
\end{proof}

\begin{rem}\label{rem:compG} Let $\cG$ be an \'etale groupoid. The above results apply to the fibre space $(\cG,r)$.  In particular  $(\cG,r)$ has two important  $\cG$-equivariant  fibrewise compactifications, its {\it Alexandroff fibrewise compactification} $(\cG_{r}^+,r^+)$ and its {\it Stone-\v Cech fibrewise compactification} $(\beta_r \cG, r_\beta)$. \index{$(\beta_r \cG, r_\beta)$} Note that since $r$ is surjective, the moment maps $r^+$ and $r_\beta$ are surjective.

 Let us consider the example described in \ref{rem:notopen}. Observe that $Y$ is in a natural way an \'etale groupoid, namely a bundle of groups over $X=  [0,1]\times\set{0}\equiv [0,1]$, the fibres over $[0,1/2]$ being the trivial group, and those over $]1/2,1]$ being the group with two elements. This example shows that the maps $r^+$ and $r_\beta$ are not always open.

\end{rem}

 With techniques similar to those used above,we  can construct useful smaller fibrewise compact $\cG$-spaces from a given one $(Z,q)$, when $\cG$ is an \'etale groupoid. This fact of independent interest will be needed in Proposition \ref{prop:nonsep_nuc}. 
 
\begin{prop}\label{prop:smaller} Let $(Z,q)$ be a left $\cG$-space where $q:Z\to X=\cG^{(0)}$ is proper. Let $A$ be a $C^*$-subalgebra of $\cC_0(Z)$ which contains $q^*\cC_0(X)$ and is stable under  convolution  by the elements of $\cC_c(\cG)$. Denote by $Y$ the Gelfand spectrum of $A$ and by $p: Y\to X$
(resp. $q_Y : Z\to Y$) the continuous surjective map corresponding to the embedding $q^*\cC_0(X)\subset \cC_0(Y)$ (resp. $\cC_0(Y)\subset \cC_0(Z))$. Then  $(Y,p)$ has a unique structure of left $\cG$-space which make $q_Y$ equivariant. Moreover, $p\circ q_Y = q$, $q_Y$ is surjective  and $p$ and $q_Y$ are proper.
\end{prop}

\begin{proof} Since $q^*\cC_0(X)$ contains an approximate unit for $\cC_0(Z)$, the maps $q_Y$ and $p$ are  well defined by  the canonical  homomorphisms from $\cC_0(Y)$ into $\cC_0(Z)$ and $\cC_0(X)$ into $\cC_0(Y)$ respectively. Since the composition of these two maps is $q^*$ we see that  that $p\circ q_Y = q$.

Let us show that $q_Y$ is proper: let $K$ be a compact subset of $Y$, then 
$$q_Y^{-1}(K) \subset q_Y^{-1}\big(p^{-1}(p(K)\big) = q^{-1}(p(K))$$ is compact. It follows that $q_Y(Z)$ is closed and so  $q_Y$ is surjective since $q_Y(Z)$ is dense into $Y$.  It is now easy to check that $p$ is proper: for every compact subset $K$ of $X$, we have $p^{-1}(K) = q_Y\big(q_Y^{-1}(p^{-1}(K))\big)$.

The uniqueness assertion for the structure of $\cG$-space of $Y$ is obvious. Given $(\gamma_0,q_Y(z_0)) \in {\mathcal G} _s\!*_p Y$, let us explain how $\gamma_0 q_Y(z_0)$ is defined. As in the proof of Lemma \ref{lem:ext}, we consider an open bisection $S$ such that $\gamma_{0} \in S$ and     a continuous
function $a$ on $X$ with compact support in $s(S)$, such that $a(x)=1$ for $x$ in a neighborhood of
$s(\gamma_0)$. We set $h = (a\circ s)\,\chi_{S}$ and $g(\gamma) = h(\gamma^{-1})$.  Let $f\in \cC_0(Y)$. Then we have $(g*(f\circ q_Y))(z_0) = f\circ q_Y (\gamma_0 z_0)$. It follows that $f\mapsto (g*(f\circ q_Y))(z_0)$ is a character of $A$
and since $g*(f\circ q_Y)$ belongs to $q_{Y}^* \cC_0(Y)$, this character only depends on $q_Y(z_0)$ and (as before) $\gamma_0$.
We denote it by $\gamma_0 q_Y(z_0)$. So, we have $f\circ q_Y (\gamma_0 z_0) = f(\gamma_0q_Y(z_0))$ for every $f\in \cC_0(Y)$ and therefore $q_Y(\gamma_0 z_0) = \gamma_0 q_Y(z_0)$. The fact that we define in this way a continuous action is proved with  arguments similar to those used in the the proof of Theorem \ref{thm:ext}.
\end{proof}

\section{\textbf{\textsc{Amenable groupoids and amenable actions}}}\label{sec:Amen}
The  reference for this section is \cite{AD-R}.  We recall the main definitions and some permanence properties of amenability, by emphasizing the subtle case of extensions. Here $\cG$ will be a locally compact groupoid   and we set $X=\cG^{(0)}$.

\subsection{Some equivalent definitions of amenability} \label{subsec:amenable} The notion of amenable locally compact groupoid has many equivalent definitions. We will recall three of them. Before,  let us recall a notation: given a locally compact groupoid $\cG$, $\gamma\in \cG$ and $\mu$ a measure on $\cG^{s(\gamma)}$, then $\gamma\mu$ is the measure on $\cG^{r(\gamma)}$ defined by 
$\int \! f \rd \gamma\mu = \int\! f(\gamma\gamma_1) \rd \mu(\gamma_1)$.

\begin{defn}\label{def:amen1}(\cite[Definitions 2.2.2, 2.2.8]{AD-R}) 
We say that $\cG$ is {\it amenable}\index{amenable locally compact groupoid} if there exists a net $(m_i)$, where $m_i = (m_i^{x})_{x\in X}$ is a family of probability measures $m_i^{x}$ on $\cG^x$, such that
\begin{itemize}
\item[(i)] each $m_i$ is continuous in the sense that for all $f\in \cC_c(\cG)$, the function $x\mapsto \int f\rd m_i^{x}$ is continuous;
\item[(ii)] $\lim_{i} \norm{\gamma m_i^{s(\gamma)} - m_i^{r(\gamma)}}_1 = 0$ uniformly on the compact subsets of $\cG$.
\end{itemize}
\end{defn}

We say that $(m_i)_i$ is an {\it approximate invariant continuous mean} \index{approximate invariant continuous mean}on $\cG$. Each $m_i$ is called a system of probability measures for $r$.\index{system of probability measures for $r$}

We say that a function $h$ on a groupoid $\cG$ is {\it positive definite}\index{positive definite function} if for every $x\in \cG^{(0)}$, $n\in \N$, and $\gamma_1,\dots,\gamma_n \in \cG^x$, the $n\times n$ matrix $[h(\gamma_i^{-1}\gamma_j)]$ is nonnegative, that is,
$$\sum_{i,j=1}^n \overline{\alpha_i}\alpha_j h(\gamma_i^{-1}\gamma_j) \geq 0$$
for $\alpha_1,\dots,\alpha_n \in \C$.

Note that for $x\in X$ we have $h(x)\geq 0$, and for all $\gamma\in \cG$
\begin{equation}\label{eq:posdef0}
h(\gamma^{-1}) = \overline{h(\gamma)},\quad\quad \abs{h(\gamma)}^2 \leq h(r(\gamma))h(s(\gamma)).
\end{equation}

A second useful characterization of amenability is  as follows.

\begin{prop}\label{prop:amen2}{\rm(\cite[Proposition 2.2.13 (iv)]{AD-R})} Let $\cG$ be a locally compact groupoid with Haar system. Then $\cG$ is amenable if and only if there exists a net $(h_i)$ of continuous positive definite functions with compact support in $\cG$ such that 
\begin{itemize}
\item[(i)] for every $i$, the restriction of $h_i$ to $X$ is uniformly bounded by $1$;
\item[(ii)] $\lim_n h_i = 1$ uniformly on the compact subsets of $\cG$.
\end{itemize}
\end{prop}

Finally, we  will also need the following  characterization of amenability.

\begin{prop}\label{prop:amen3}{\rm (\cite[Proposition 2.2.13 (i)]{AD-R})} Let $(\cG,\lambda)$ be a locally compact groupoid with Haar system $\lambda$. Then $\cG$ is amenable if and only if there exists  a net $(g_i)$ of nonnegative functions in $\cC_c(\cG)$ such that 
\begin{itemize}
\item[(a)] $\int g_i \rd\lambda^{x} \leq 1$ for every $x\in \cG^{(0)}$;
\item[(b)] $\lim_i \int g_i \rd\lambda^{x} = 1$ uniformly on the compact subsets of $\cG^{(0)}$;
\item[(c)] $\lim_i \int \abs{g_i(\gamma^{-1}\gamma_1) -  g_i(\gamma_1)}\rd\lambda^{r(\gamma)}(\gamma_1) = 0$ uniformly on the compact subsets of $ \cG$.
\end{itemize}
\end{prop}

\begin{rems}\label{rem:nonsep}  (1) We will have to work occasionally  with $\sigma$-compact groupoids, not necessarily second countable, with second countable space of units.  One checks that the results of \cite[\S 2.2.a, \S 2.2.b]{AD-R} remain true in this setting. For them, the above characterizations of amenability are still valid. Moreover, as with second countable locally compact groupoids, in the above statements, we can replace nets by sequences.

(2) As observed in \cite[Proposition 2.2.13]{AD-R}, \cite[Remark 9.4]{Will19}, the condition (a) in the previous proposition is not necessary. It follows that in the proposition \ref{prop:amen2} we do not need to ask for condition (i).
\end{rems}

We say that an action $\cG\actson Y$ of a locally compact groupoid $\cG$ on a locally compact space $Y$ is amenable \index{amenable action}, or that the {\it $\cG$-space $Y$ is amenable}, if the semi-direct product groupoid $Y\rtimes \cG$ is amenable. 

If $\cG$ is an amenable groupoid, then every left $\cG$-space $(Y, p)$ is amenable. Indeed, if $(m_i) $ is an approximate invariant continuous mean for $\cG$, then $(n_i)$, where $n_i^{y}(f) = \int f(y,\gamma)\rd m_i^{p(y)}(\gamma)$ for $f\in \cC_c(Y\rtimes \cG)$, is an approximate invariant continuous mean for $Y\rtimes \cG$.

\begin{rem}\label{rem:amenable} Every locally compact group $G$ (or groupoid) has an amenable action on a locally compact space, namely its left action on itself, but the existence of an amenable action on a compact space is not always realized. Groups for which such an action exists are called amenable at infinity (see Section \ref{sec:ameninf}). 

We note that if there exists a $G$-invariant probability measure $\mu$ on a locally compact space $Y$ on which $G$ acts amenably, then $G$ is an amenable group. Indeed, if $(h_i)$ is a net of continuous positive definite functions in $\cC_c(Y\times G)$ as in Proposition \ref{prop:amen2}, it suffices to consider the net $(\widetilde{h_i})$ of positive definite functions on $G$, where $\widetilde{h_i}(t) = \int_Y h_i(x,t)\rd\mu(x)$. It converges to $1$ uniformly on compact subsets of $G$.

\end{rem}

\subsection{Some stability properties of amenability}\label{subsec:stable}  We refer to \cite[\S 2.2 b, \S 5.1]{AD-R} for this subject. In particular amenability is preserved under equivalence of groupoids \cite[Theorem 2.2.17]{AD-R}. 

\begin{prop}\label{cor:locpropre} Let $Y$ and $Z$ be two locally compact left $\cG$-spaces, where $\cG$ is a locally compact groupoid. We assume that there exists a continuous $\cG$-equivariant map $p:Y\to Z$ and that $Z\rtimes \cG$ is amenable. Then so is $Y\rtimes \cG$. 
\end{prop}

\begin{proof} 

That $Y\rtimes \cG$ is amenable follows  from the fact that this groupoid is isomorphic to $Y\rtimes (Z\rtimes \cG)$  (see Lemma \ref{lem:idengr} (i)), where $Z\rtimes \cG$ is amenable\end{proof}

\begin{prop}\label{prop:locpropre} Let $\cG$ and $\cH$ be two locally compact groupoids with Haar systems and let $c:\cG\to \cH$ be a locally proper continuous homomorphism. If $\cH$ is amenable, then so is $\cG$.
\end{prop}

\begin{proof} Given a compact subset $K$ of $\cG$  and $\varepsilon >0$, we want to construct a continuous  positive definite function $h\in \cC_c(\cG)$ such that $\abs{h(\gamma) -1}\leq \varepsilon$ for $\gamma\in K$. We choose a continuous  positive definite function $k\in \cC_c(\cH)$ such that $\abs{k(\eta) -1}\leq \varepsilon$ for $\eta\in c(K)$ and a continuous function $\varphi:\cG^{(0)} \to [0,1]$, with compact support such that $\varphi(x) = 1$ for $x\in r(K)\cup s(K)$. The function 
$$h: \gamma \mapsto \varphi\circ r(\gamma)\varphi\circ s(\gamma) k\circ c(\gamma)$$
 satisfies the required conditions.
\end{proof}

For instance every proper locally compact groupoid $\cG$ with Haar system is amenable. Indeed, the map $r\times s: \cG \to \cG^{(0)}\times \cG^{(0)}$ is a proper homomorphism into the pair groupoid $\cG^{(0)}\times \cG^{(0)}$ which is amenable.

\begin{prop}\label{prop:stablegenmor}  Let $\cG$ and $\cH$ be two locally compact groupoids. We assume that $\cH$ is amenable.  Let $Z$ be a faithful locally proper  generalized morphism from $\cG$ to $\cH$. Then $\cG$ is amenable.
\end{prop}

\begin{proof}  We keep the notation of Proposition \ref{prop:homo1} which shows that the groupoids $\cG$ and $Y\rtimes \cH$ are equivalent. The groupoid $\cG$ is amenable since $Y\rtimes \cH$ is amenable.
\end{proof}

\subsection{Amenability of extensions}\label{subsec:amenext} It is a natural question to ask whether a locally compact groupoid $\cL$ is amenable if there is a surjective continuous homomorphism $c: \cL \to \cG$ where $\cG$ is a locally compact groupoid  which is amenable as well as $\Ker(c)$. The answer is no in general as shown  by the following simple example given in \cite[Remark 5.3.15]{AD-R}: $\cL$ is the group bundle groupoid $\cL = G\bigsqcup H$ where $H$ is non-amenable locally compact group equiped with a continuous injective homomorphism $\psi$ from $H$ into an amenable locally compact group $G= \cG$.  We consider the homomorphism  $c: \cL\to \cG$ that is the identity on $G$ and is equal to $\psi$ on $H$. Then $\cL$ is not amenable, although $\cG$  and $\Ker(c)$ are amenable. If we remove one of the following two obstacles, namely that $G$ is not discrete, or that $c$ is not strongly surjective, we will see that the answer is positive.

\subsubsection{The strongly surjective case} It is solved in \cite[Theorem 5.3.14]{AD-R} in the setting of measurewise amenability and in \ \cite[Proposition 5.1.2]{AD-R}, \cite[Theorem 9.80]{Will19}, for unit fixing extensions. For the general case below, we simply adapt the proof of this last result.

\begin{thm}\label{thm:amen-ext} Let $c:\cL \to \cG$ be a strongly surjective continuous homomorphism from a locally compact  groupoid $\cL$ to a locally compact groupoid $\cG$. We assume that $\cG$ and $\cH=\Ker(c)$ are amenable. Then, so is $\cL$.
\end{thm}

\begin{proof}  Given a compact subset $K_1$ of $\cL$ and $\varepsilon >0$, we will construct a continuous system of probability measures $\nu$ tor $r:\cL\to \cL^{(0)}$ such that 
\begin{equation}\norm{\gamma_1\nu^{s(\gamma_1)} - \nu^{r(\gamma_1)}}_1 \leq \varepsilon \quad \hbox{for} \quad\gamma_1\in K_1.
\end{equation}
We keep the notation of the definition \ref{def:strong-surj}. Since $\phi : \cL \to \cL^{(0)} \!_{c^{(0)}}\!\!*_r \cG$ is surjective and open there is a system $m=(m^{(y,g)})_{(y,g)\in  \cL^{(0)} \!_{c^{(0)}}\!\!*_r \cG}$ of probability measures for $\phi$ with $m^{(y,g)} \not= 0$ for all $(y,g)$ (see for instance \cite[Lemma 3.20]{Will19}). Moreover, using \cite[Lemma 1.1.2]{AD-R}, we can choose $m$ so that for every compact subset $C$ of $ \cL^{(0)} \!_{c^{(0)}}\!\!*_r \cG$ there exists a compact subset $K$ of $\cL$ such that $m^{(y,g)}(K) = 1$ when $(y,g)\in C$. We set $\varepsilon' = \varepsilon/4$. Since $\cG$ is amenable there is a continuous system $\mu = (\mu^x)_{x\in \cG^{(0)}}$ for $r: \cG\to \cG^{(0)}$ such that 
\begin{equation}\label{eq:G}
\norm{g\mu^{s(g)} - \mu^{r(g)}}_1\leq \varepsilon' \quad \hbox{for} \quad g\in c(K_1).
\end{equation}
There exists a compact subset $L$ of $\cG$ such that $\mu^x(\cG\setminus L)\leq \varepsilon'$ for $x\in c^{(0)}\circ s(K_1)$. Then we choose a compact subset $K$ of $\cL$ such that 
$$m^{(y,g)}(K) = 1 \,\hbox{ for }\, (y,g)\in \big(\cL^{(0)} \!_{c^{(0)}}\!\!*_r \cG\big) \cap \big(s(K_1)\cup r(K_1))\times (L \cup c(K_1)L)\big).$$
Now, let $M$ be a compact subset of $\cL$ containing $K^{-1}K$ and $K^{-1}K_1K$. Since $\cH = \Ker(c)$ is amenable we can find a continuous system $\lambda = (\lambda^y)_{y\in \cH^{(0)}}$ such that
\begin{equation}
\norm{h\lambda^{s(h)} - \lambda^{r(h)}}_1\leq \varepsilon' /2\quad \hbox{if} \quad h\in M\cap\cH.
\end{equation}
Then, for $f\in \cC_c(\cL)$ and $y\in \cL^{(0)}$ we set
$$\nu^y(f) = \int_\cG\big[\int_\cL\big(\int_\cH f(\gamma h)\rd \lambda^{s(\gamma)}(h)\big)\rd m^{(y,g)}(\gamma)\big]\rd \mu^{c^{(0)}(y)}(g).$$
Then, $\nu^y$ is a probability measure supported on $\cL^y$ and $y\mapsto \nu^y(f)$ is continuous.

Let $f\in \cC_c(\cL)$ with $\norm{f}_\infty\leq 1$. To end the proof of the theorem it suffices  to show that
\begin{equation}\label{eq:toshow}
\abs{\int_\cL f(\gamma_1\gamma)\rd\nu^{s(\gamma_1)}(\gamma) -\int_\cL f(\gamma)\rd\nu^{r(\gamma_1)}(\gamma)}\leq \varepsilon \quad\hbox{if} \quad \gamma_1\in K_1.
\end{equation}
We set $\underline{\lambda}(f)(\gamma) =\int_\cH f(\gamma h)\rd^{s(\gamma)}(h)$. 
 If $\gamma,\gamma'$ are such that $\phi(\gamma) = \phi(\gamma')$ with $\gamma^{-1}\gamma'\in M\cap \cH$, we have
\begin{equation}\label{eq:underline}
\abs{\underline{\lambda}(f)(\gamma) - \underline{\lambda}(f)(\gamma')}\leq \varepsilon'/2,
\end{equation}
since 
$$\abs{\underline{\lambda}(f)(\gamma) - \underline{\lambda}(f)(\gamma')} = \abs{\int f(\gamma hh_1\rd\lambda^{s(h)}(h_1) = \int f(\gamma h_1\rd\lambda^{r(h)}(h_1)}$$
with $h = \gamma^{-1}\gamma'$.

Let $\gamma_1\in K_1$ and $g\in L$ with $r(g) = c^{(0)}(s(\gamma_1))$ and let $\gamma'\in K$ such that $r(\gamma') = s(\gamma_1)$ and $c(\gamma') = g$.  We have
\begin{multline}\label{eq:1}
\abs{\int_\cL\underline{\lambda}(f)(\gamma_1\gamma) \rd m^{(s(\gamma_1),g)}(\gamma) - \underline{\lambda}(f)(\gamma_1\gamma')}\\
= \abs{\int_K \big(\underline{\lambda}(f)(\gamma_1\gamma) -\underline{\lambda}(f)(\gamma_1\gamma')\big)\rd m^{(s(\gamma_1),g)}(\gamma)}\leq \varepsilon'/2
\end{multline}
due to \eqref{eq:underline}. Similarly we have
\begin{equation}\label{eq:2}
\abs{\int_\cL\underline{\lambda}(f)(\gamma) \rd m^{(r(\gamma_1), c(\gamma_1)g)}(\gamma) - \underline{\lambda}(f)(\gamma_1\gamma')}\leq \varepsilon'/2.
\end{equation}
From these \eqref{eq:1} and \eqref{eq:2} we deduce the inequality
\begin{equation}\label{eq:3}
\abs{\int_\cL\underline{\lambda}(f)(\gamma_1\gamma) \rd m^{(s(\gamma_1),g)}(\gamma) - \int_\cL\underline{\lambda}(f)(\gamma) \rd m^{(r(\gamma_1), c(\gamma_1)g)}(\gamma)}\leq \varepsilon'.
\end{equation}
Now, given $\gamma_1\in K_1$, we turn to the proof of the inequality \eqref{eq:toshow}. One one hand we have
$${\int_\cL f(\gamma_1\gamma)\rd\nu^{s(\gamma_1)}(\gamma) = \int_\cG\big(\int_\cL \underline{\lambda}(f)(\gamma_1\gamma)\rd m^{(s(\gamma_1),g)}(\gamma)\big)\rd\mu^{c^{(0)}(s(\gamma_1))}}(g),$$
while
\begin{align*}\int_\cL f(\gamma)\rd\nu^{r(\gamma_1)}(\gamma) &= \int_\cG\big(\int_\cL \underline{\lambda}(f)(\gamma)\rd m^{(r(\gamma_1),g)}(\gamma)\big)\rd\mu^{c^{(0)}(r(\gamma_1))}(g)\\
& =  \int_\cG \psi(g) \rd\mu^{c^{(0)}(r(\gamma_1))}(g),
\end{align*}
where $\psi(g) = \int_\cL \underline{\lambda}(f)(\gamma)\rd m^{(r(\gamma_1),g)}(\gamma)$.
The quantity 
$$\int_\cL f(\gamma_1\gamma)\rd\nu^{s(\gamma_1)}(\gamma) -\int_\cL f(\gamma)\rd\nu^{r(\gamma_1)}(\gamma)$$
is the sum $A+B$ where 
\begin{multline*} 
A = \int_\cG\Big(\int_\cL \underline{\lambda}(f)(\gamma_1\gamma)\rd m^{(s(\gamma_1),g)}(\gamma)\Big)\rd\mu^{c^{(0)}(s(\gamma_1))}(g)\\
-\int_\cG\Big(\int_\cL \underline{\lambda}(f)(\gamma)\rd m^{(r(\gamma_1),c(\gamma_1)g)}(\gamma)\Big)\rd\mu^{c^{(0)}(s(\gamma_1))}(g),
\end{multline*}
and
$$B = \int_\cG \psi(c(\gamma_1)g) \rd \mu^{c^{(0)}(s(\gamma_1))}(g) - \int_\cG \psi(g) \rd \mu^{c^{(0)}(r(\gamma_1))}(g).$$
Since $\norm{\psi}_\infty\leq 1$, using \eqref{eq:G} we get $\abs{B}\leq \varepsilon'$. 
On the other hand, we have $A = A_1 + A_2$ where
\begin{multline*} 
A _1= \int_L\Big(\int_\cL \underline{\lambda}(f)(\gamma_1\gamma)\rd m^{(s(\gamma_1),g)}(\gamma) \\
- \int_\cL \underline{\lambda}(f)(\gamma)\rd m^{(r(\gamma_1),c(\gamma_1)g)}(\gamma)\Big)\rd\mu^{c^{(0)}(s(\gamma_1))}(g),
\end{multline*}
and
\begin{multline*} 
A_2 = \int_{\cG\setminus L}\Big(\int_\cL \underline{\lambda}(f)(\gamma_1\gamma)\rd m^{(s(\gamma_1),g)}(\gamma)\\
-\int_\cL \underline{\lambda}(f)(\gamma)\rd m^{(r(\gamma_1),c(\gamma_1)g)}(\gamma)\Big)\rd\mu^{c^{(0)}(s(\gamma_1))}(g).
\end{multline*}
Then \eqref{eq:3} applies to give $\abs{A_1}\leq\varepsilon'$. Moreover, we recall that we have chosen $L$ such that $\mu^x(\cG\setminus L)\leq \varepsilon'$ if $x\in c^{(0)}(s(K_1))$. It follows that $\abs{A_2}\leq 2\varepsilon'$. Combining all these facts we establish the inequality\eqref{eq:toshow}.
\end{proof}

\subsubsection{Case of homomorphisms $c:\cL\to \cG$ with some discreteness assumption on the groupoid $\cG$}  Let us first recall the following result proved in  \cite{RW17}.

\begin{thm}$($\cite[Corollary 4.5]{RW17}$)$\label{thm:RW17} Let $\cL$ be a locally compact  groupoid with a Haar system and let $c :\cL \to \Gamma$ be a continuous homomorphism into a discrete group $\Gamma$. We assume   that $\Gamma$ is amenable and that $c^{-1}(e)$
 is amenable. Then the groupoid $\cL$ is amenable.
\end{thm}

In order to study amenability of partial actions, we need the more general version of this result, where $\Gamma$ is replaced by a semi-direct product $X\ltimes \Gamma$.

\begin{thm}\label{thm:RW17Bis} Let $\cL$ be a locally compact  groupoid with a Haar system and let $\cG = \Gamma\ltimes X$, where $\Gamma\actson  X$ is an amenable left action of a discrete group $\Gamma$ on a locally compact space $X$. Let  $c :\cL \to \cG$ be a continuous homomorphism. We assume  that $Ker(c)$
 is amenable. Then the groupoid $\cL$ is amenable.
\end{thm}

The proof is obtained by combining the propositions \ref{prop:RW17Thm4.2}  and \ref{prop:RW17-4.1} below\footnote{We do not know whether in this theorem  we can replace $\Gamma\ltimes X$ by any \'etale groupoid.}.

We will follow the ideas of the proof of the theorem 4.2 in \cite{RW17} which uses in a crucial way the notion of {\it skew-product}\index{skew-product}, that we first recall.

 Let $\cG$ and $\cL$ be two locally compact groupoids and let $c: \cL\to \cG$ be a continuous  homomorphism. As a set, the skew-product groupoid $\cL(c)$ associated with $c$  is defined as follows:
$$\cL(c) = \set{(\eta,\gamma, c(\gamma^{-1})\eta) : c^{(0)}(r(\gamma)) = r(\eta)}\subset \cG\times \cL\times\cG.$$
It is a closed subset of $\cG\times \cL\times\cG$, and we endow it with the topology induced by the product topology. The range of $(\eta,\gamma, c(\gamma^{-1})\eta)$ is $(\eta, r(\gamma), \eta)$ and its source is $(c(\gamma^{-1})\eta,s(\gamma),c(\gamma^{-1})\eta)$, which are identified to $(r(\gamma),\eta)$ and $(s(\gamma),c(\gamma^{-1})\eta)$ res\-pectively. The product in $\cL(c)$ is defined as follows if $(s(\gamma),c(\gamma^{-1})\eta) = (r(\gamma_1),\eta_1)$:
$$(\eta,\gamma, c(\gamma^{-1})\eta)(\eta_1,\gamma_1, c(\gamma_{1}^{-1})\eta_1) = (\eta,\gamma\gamma_1, c(\gamma_{1}^{-1})\eta_1).$$
The inverse is given by
 $$(\eta,\gamma, c(\gamma^{-1})\eta)^{-1} = (c(\gamma^{-1})\eta, \gamma^{-1}, \eta).$$
   We have $\cL(c)^{(0)} = \cL^{(0)} \!_{c^{(0)}} \!*_r \cG$, that we denote by $Z$.

Let us observe that $(\eta,\gamma, c(\gamma^{-1})\eta) \mapsto (\gamma, (s(\gamma), c(\gamma^{-1})\eta))$ is an isomorphim from $\cL(c)$ onto $\cL\ltimes Z$,
where $\gamma(s(\gamma), \eta) = (r(\gamma), c(\gamma)\eta)$.

For $(x,\eta) \in Z=\cL(c)^{(0)}$ we have $\cL(c)^{(x,\eta)} = \set{(\eta,\gamma, c(\gamma^{-1})\eta)\in \cL(c) : r(\gamma) = x}$. If $\lambda = (\lambda^x)_{x\in \cL^{(0)}}$ is a Haar system for $\cL$, then $(\lambda^{(x,\eta)})_{(x,\eta)\in Z}$ is a Haar system for $\cL(c)$, where
$$\int_{\cL(c)} f\rd\lambda^{(x,\eta)} = \int_\cL f(\eta,\gamma, c(\gamma^{-1})\eta)\rd\lambda^x(\gamma).$$

For $\gamma\in \cL$, we set $\phi(\gamma) = (r(\gamma), c(\gamma))\in \cL^{(0)} \!_{c^{(0)}} \!*_r \cG = \cL(c)^{(0)}$, and we set $W= \phi(\cL)$. This set is $\cL(c)$-invariant. Indeed, assume that $s(\eta,\gamma, c(\gamma^{-1})\eta) = \phi(\gamma_1)$, that is $s(\gamma) = r(\gamma_1)$ and $c(\gamma^{-1})\eta = c(\gamma_1)$,  then $\phi(\gamma\gamma_1) = (r(\gamma),\eta) = r(\eta,\gamma, c(\gamma^{-1})\eta)$. Moreover, since $\cL$ is second countable\footnote{We remind the reader that we always assume the locally compact groupoids to be second countable.}, hence $\sigma$-compact, and since $\phi$ is continuous, $W$ is $\sigma$-compact and therefore is a Borel subset of $\cL(c)^{(0)}$.

Note that $\cL$ is a right $\Ker(c)$-space with respect to the groupoid multiplication. It is a left $\cL(c)(W)$-space where $\phi$ is the moment map $\cL\to W$ and the left action is defined by $(\eta,\gamma, c(\gamma^{-1})\eta)\gamma_1 = \gamma\gamma_1$. Exactly as in \cite[Proposition 4.3]{RW17} one shows the following proposition. For the notion of Borel equivalence see for instance \cite[Definition 3.1]{RW17}.

\begin{prop}\label{prop:Boreleq} With the above notation, $\cL$ is a Borel equivalence between the Borel groupoids $\cL(c)(W)$ and $\Ker(c)$.
\end{prop}

\begin{prop}\label{prop:RW17Thm4.2} Let $\cL$ be a locally compact groupoid with Haar system, let $\cG = \Gamma\ltimes X$ where $\Gamma \actson X$ be a left action of a  discrete group $\Gamma$ on a locally compact space $X$ and let $c: \cL \to \cG$ be a continuous homomorphism. Assume that the groupoid $\Ker(c)$ is amenable. Then the skew-product $\cL(c)$ is amenable.
\end{prop}

\begin{proof} For the notion of Borel amenability of locally compact groupoids, we refer to \cite{Ren_13, RW17}. Recall that (topological) amenability implies Borel amenability, and that
locally compact groupoids with Haar system that are Borel amenable are amenable by \cite[Corollary 2.15]{Ren_13}.  By \cite[Theorem 3.2]{RW17}, Borel amenability is preserved under Borel equivalence. It follows from the previous proposition that the groupoid $\cL(c)(W)$ is Borel amenable.

We set $\Gamma =\set{g_n: n\in \N}$ and $S_n = \set{g_n}\times X$ for $n\in \N$. We let $S_n$ act on the right of $\cL(c)$ in the following way:
$$(\eta,\gamma, c(\gamma^{-1})\eta)\cdot S_n = (\eta\eta_1,\gamma, c(\gamma^{-1})\eta\eta_1),$$
where $\eta_1$ is the unique element of $S_n$ such that $r(\eta_1) =s(\eta)$. Let us check that this right action of $S_n$ is a groupoid automorphism.
Assume that 
$$(\eta',\gamma', c(\gamma'^{-1})\eta')\cdot S_n = (\eta,\gamma, c(\gamma^{-1})\eta)\cdot S_n.$$
Then $\gamma = \gamma'$ and $\eta'\eta'_{1} = \eta\eta_1$ where $\eta'_{1} = (g_n,x'_{1})$ and $\eta_1= (g_n, x_1)$ are the unique elements of $S_n$ such that $g_n x'_{1} = s(\eta')$ and $g_n x_1 = s(\eta)$. If $\eta' = (g',x')$ and $\eta =(g,x)$, we have $g'x' = c^{(0)}(r(\gamma)) = gx$ and $(g'g_n, x'_1) = (gg_n, x_1)$. It follows that $g'= g$ and $x' = x$. Therefore the right action of $S_n$ is injective. Its surjectivity and the fact that it is a groupoid automorphism are also tedious but straightforward verifications.

It follows that $W\cdot S_n$ is $\cL(c)$-invariant and that $\cL(c)(W)\cdot S_n$ is Borel amenable since it is isomorphic to $\cL(c)(W)$. Moreover, we have $Z = \cup_n W\cdot S_n$. Indeed, let $(\eta,x,\eta)\in Z$ and let  $\gamma\in \cL^x$. We have $r(\eta) = c^{(0)}(x)= r(c(\gamma))$. If $k$ is such that $\eta_1 = c(\gamma)^{-1} \eta \in S_k$, then $(\eta,x,\eta)= (c(\gamma), x, c(\gamma))\cdot S_k$.

Now, using Lemma 3.3 in \cite{RW17} we see that $\cL(c)$ is Borel amenable and therefore amenable.
\end{proof}

Finally, in order to prove Theorem \ref{thm:RW17Bis}, we need the following generalization of \cite[Proposition 4.1]{RW17}.

\begin{prop}\label{prop:RW17-4.1} Let $\cG$ and $\cL$ be locally compact groupoids with Haar systems and let $c:\cL\to\cG$ be a continuous homomorphism. If $\cL(c)$ and $\cG$ are both amenable, then so is $\cL$.
\end{prop}

\begin{proof} We denote by the same letter $\lambda$ the Haar systems of $\cG$ and $\cL$. We fix two compact sets $K\subset\cL$ and $L\subset \cL^{(0)}$, and $\varepsilon >0$. We need to construct a nonnegative function $f\in \cC_c(\cL)$ such that 
\begin{itemize}
\item [(1)]$\int_{\cL} f(\gamma) \rd \lambda^x \leq 1$ for every $x\in \cL^{(0)}$;
\item [(2)]$\int_{\cL} f(\gamma) \rd \lambda^x \geq 1-\varepsilon$ for every $x\in L$;
\item [(3)] $\int_\cL \abs{f(\gamma^{-1}\gamma') - f(\gamma')} \rd \lambda^{r(\gamma)}(\gamma') \leq \varepsilon$ for every $\gamma \in K$.
\end{itemize}
Since $\cG$ is amenable, we find a nonnegative function $k\in\cC_c(\cG)$ satisfying the following conditions:
\begin{itemize}
\item $\int_\cG k(\eta)\rd \lambda^y(\eta) \leq 1$ for every $y\in \cG^{(0)}$;
\item$\int_\cG k(\eta)\rd \lambda^y(\eta) \geq 1- \varepsilon/2$ for every $y\in c^{(0)}(L)$;
\item $\int_\cG \abs{k(\eta^{-1}\eta') - k(\eta')}\rd\lambda^{r(\eta)}(\eta') \leq \varepsilon/2$ for every $\eta\in c(K)$.
\end{itemize}
Similarly, since $\cL(c)$ is amenable, there exists a nonnegative function $g\in \cC_c(\cL(c))$ such that:
\begin{itemize}
\item $\int_{\cL} g(\eta,\gamma, c(\gamma^{-1})\eta) \rd\lambda^{x}(\gamma) \leq 1$ for every $(x,\eta)\in \cL(c)^{(0)}$;
\item $\int_{\cL}g(\eta,\gamma, c(\gamma^{-1})\eta) \rd\lambda^{x}(\gamma)\geq 1-\varepsilon/2$ for every \\$(x,\eta)\in (L\times\supp(k))\cap \cL(c)^{(0)}$;
\item $\int_\cL \abs{g(c(\gamma^{-1})\eta, \gamma^{-1}\gamma', c(\gamma'^{-1})\eta) - g(\eta,\gamma', c(\gamma'^{-1})\eta)}\rd^{r(\gamma)}(\gamma') \leq \varepsilon/2$ for eve\-ry $(\eta,\gamma)\in (c(K)\supp(k) \times K) \cap\cG\, _r\!\!*_{c^{(0)}\circ r} \cL$.
\end{itemize}
We define a nonnegative function $f\in \cC_c(\cL)$ by the formula
$$f(\gamma) = \int_\cG k(\eta)g(\eta,\gamma,c(\gamma^{-1})\eta)\rd\lambda^{c^{(0)}(r(\gamma))}(\eta).$$
We have
$$\int_\cL f(\gamma)\rd\lambda^x(\gamma) = \int_\cL \int_\cG k(\eta)g(\eta,\gamma,c(\gamma^{-1})\eta)\rd\lambda^{c^{(0)}(x)}(\eta)\rd\lambda^x(\gamma).$$
Condition (1) above is obviously satisfied and if $x\in L$ we have 
$$\int_\cL f(\gamma)\rd\lambda^x(\gamma) \geq (1-\varepsilon/2)^2 \geq 1-\varepsilon.$$

For the proof of Condition (3) we note that $\int_\cL  \abs{f(\gamma^{-1}\gamma') - f(\gamma')} \rd \lambda^{r(\gamma)}(\gamma')$
is the integral with respect to $\lambda^{r(\gamma)}$ of  the absolute value of the function of $\gamma'$ below:
\begin{align*}
\int_\cG k(\eta)g(\eta,\gamma^{-1}\gamma',&c(\gamma'^{-1})c(\gamma)\eta)\rd\lambda^{c^{(0)}(s(\gamma))}(\eta)\\ 
&- \int_\cG  k(\eta)g(\eta,\gamma',c(\gamma'^{-1})\eta)\rd\lambda^{c^{(0)}(r(\gamma'))}(\eta).
\end{align*}By change of variable, the first integral of this expression is equal to
$$\int_\cG k(c(\gamma^{-1})\eta)g(c(\gamma^{-1})\eta, \gamma^{-1}\gamma',c(\gamma'^{-1})\eta)\rd\lambda^{c^{(0)}(r(\gamma))}(\eta) .$$
Then we have, for $\gamma\in K$,
\begin{align*}
&\int_\cL  \abs{f(\gamma^{-1}\gamma') - f(\gamma')} \rd \lambda^{r(\gamma)}(\gamma')\leq
\Big(\int_\cL \int_\cG k(c(\gamma^{-1})\eta)\times\\
&\abs{g(c(\gamma^{-1})\eta, \gamma^{-1}\gamma',c(\gamma'^{-1})\eta) - g(\eta,\gamma',c(\gamma'^{-1})\eta)}\rd\lambda^{c^{(0)}(r(\gamma)}(\eta)\rd\lambda^{r(\gamma)}(\gamma')\Big)\\
&+ \int_\cL\int_\cG\abs{ k(c(\gamma^{-1})\eta) - k(\eta)}g(\eta,\gamma',c(\gamma'^{-1})\eta)\rd^{c^{(0)}(r(\gamma)}(\eta)\rd\lambda^{r(\gamma)}(\gamma')\\
&\leq\varepsilon/2\int_\cG k(c(\gamma^{-1})\eta)\rd^{c^{(0)}(r(\gamma)}(\eta) + \varepsilon/2\int_\cL g(\eta,\gamma',c(\gamma'^{-1})\eta)\rd\lambda^{r(\gamma)}(\gamma') \leq \varepsilon.
\end{align*}
\end{proof}

\subsubsection{Application to amenable partial actions}\label{subsec:amenpa}  Let $\boldsymbol{\beta} = ((\beta_\eta)_{\eta\in \cG}, (Y_\eta)_{\eta\in \cG})$ be a partial action of a locally compact groupoid $\cG$ on a locally compact space $(Y,p)$ fibered over $X = \cG^{(0)}$.

\begin{defn}\label{def:amenac} We say that the partial action $\boldsymbol{\beta}$  is {\it amenable} \index{amenable groupoid partial action} if the  groupoid $\cL =Y\rtimes \cG$ is amenable.
\end{defn}

Recall that $Y$ is the set of units of the semi-direct product $\cL=Y\rtimes \cG$ and that for $y\in Y$ we have $\cL^y = \set{y}\times \set{\eta\in \cG:y\in Y_\eta}$ which can be
 canonically identified to an open subset of $\cG^{p(y)}$. If $\cG$ has a Haar system $\lambda$, then $Y\rtimes \cG$ has a  Haar system defined in Proposition \ref{prop:Abad2.2}.  
A positive definite function  on $Y\rtimes \cG$ is a function $h$ such that, for every $y\in Y$, $n\in \N$, and $\eta_1,\dots,\eta_n \in \set{\eta\in \cG:y\in Y_\eta}$, the $n\times n$ matrix $[h(\eta_i^{-1} y,\eta_i^{-1}\eta_j)]$ is nonnegative. Then, the  characterizations given in Propositions \ref{prop:amen2} and \ref{prop:amen3} are easily spelled out for the groupoid $Y\rtimes \cG$. Note that when $\boldsymbol{\beta}$ is a global action, we have $\set{\eta\in \cG:y\in Y_\eta} = \cG^{p(y)}$.

 We have already observed that a global action of an amenable locally compact groupoid  is amenable. 
 
 For a partial action $\cG\actson Y$ the situation is more complicated because the cano\-nical map $c:Y\rtimes \cG \to \cG$ is not always locally proper and also because
 the groupoid $Y\rtimes \cG$ is not an extension of $\cG$ by the space $Y$ (viewed as a groupoid), in the sense  of  Definition \ref{def:extension}. Indeed  $c$ is not always strongly surjective  even if we assume that the moment map is surjective. 
 
  The following result was obtained in \cite[Proposition 7.1]{AD20} when $\cG$ is, more generally, an \'etale groupoid, but with a proof  involving more sophisticated tools that we do not consider necessary for our objective here.
 
  \begin{prop}\label{cor:amen-pa} Let $\cG =\Gamma \ltimes X$ where $\Gamma\actson X$ is an amenable left action of a discrete group $\Gamma$ on a locally compact space $X$. Let $\boldsymbol{\beta}$ be a partial action of $\cG$ on a locally compact space $Y$. Then the semi-direct product groupoid $\cL =Y\rtimes \cG$ is amenable.
\end{prop}

\begin{proof} We have $\cL = \set{(y,\eta)\in Y\times \cG:  y\in Y_{\eta}}$. We consider the homomorphism $c : (y,\eta) \mapsto \eta$ from $\cL$ to $\cG$. Then $\Ker(c) = Y$ is amenable and we use the theorem \ref{thm:RW17Bis}. 
\end{proof}

\section{\textbf{\textsc{Inner amenable groupoids}}}\label{section:IA}

We will use in Section \ref{sect:eiai} another kind of amenability\footnote{This notion was considered in \cite{AD02}  for transformation groupoids, under the name of Property (W).}, that we describe now.

 \subsection{Definition, examples and some stability properties} \label{subsection:IA}Let us first consider the case of a locally compact group $G$. Let $\lambda$ be its left regular representation and $\rho$ its right regular representation. These representations act on the Hilbert space $L^2(G)$ of square integrable functions. For $\xi\in L^2(G)$ and $s,y\in G$, we have
$$(\lambda_s \xi)(y) = \xi(s^{-1}y),\quad (\rho_s \xi)(y) = \Delta(s)^{1/2}\xi(ys)$$
 where $\Delta$ is the modular function of $G$.

\begin{defn}\label{def:IA}Following \cite[page 84]{Pat88}, we say that $G$ is {\it inner amenable} if there exists an inner invariant mean  on $L^\infty(G)$, that is, a state $m$ such that $m(sfs^{-1}) = m(f)$ for every $f\in L^\infty(G)$ and $s\in G$, where $(sfs^{-1})(y) = f(s^{-1}ys)$. This is equivalent to the existence of a net $(\xi_i)$ in $\cC_c(G)$ such that $\| \xi_i\|_2 = 1$ for all $i$ and
$\langle \xi_i, \lambda_s\rho_s \xi_i \rangle$ goes to one uniformly on compact subsets of $G$ (see \cite{LR87}). 
\end{defn}

 \begin{exs} Amenable groups are inner amenable. Every discrete group $G$ is inner amenable in this sense, since the Dirac measure $\delta_e$ is an inner invariant mean\footnote{Effros \cite{Effros} excludes this trivial inner invariant mean in his definition of inner amenability.}. On the other hand,  
a locally compact group which is either almost connected or type $I$  is  inner amenable if and only if it is amenable (see  \cite{LR87}, \cite[Remark 5.10]{AD02}). For instance, the type I, totally disconnected group $G(\Q_p)$, where $G$ is a linear algebraic group over $\Q_p$, $p$ a prime number, is not inner amenable
 
 There exist non-discrete, totally disconnected, non-amenable  groups that are inner amenable. Consider for example any product of an infinite compact totally disconnected group by a discrete non-amenable group. 
 
 Finally, let us mention that inner amenability of locally compact groups is preserved under passing to quotient \cite[Proposition 6.2]{LP}, closed subgroups \cite[Corollary 3.3]{CT}, but not under extension \cite[Proposition 4.14]{Man}.
 \end{exs}
 
 Let  $(\xi_i)$ be a net of elements of $\cC_c(G)$ satisfying the  condition stated in the definition \ref{def:IA}. Let us define $f_i$ on $G\times G$ by
$$f_i(s,t) = \langle \xi_i, \lambda_s \rho_t \xi_i \rangle.$$
 Then $f_i$ is a positive definite function on the product group $G\times G$. For each compact subset $K$ of $G$, the intersections with $K\times G$ and with $G\times K$ of the support of $f_i$ are compact. Moreover, $\lim_i f_i = 1$ uniformly on compact subsets of the diagonal.

 \begin{thm}\cite{CT} Let $G$ be a locally compact group. The following conditions are equivalent:
 \begin{itemize}
 \item[(i)] $G$ is inner amenable;
 \item[(ii)]  there exists a net $(f_i)$ of positive definite functions on $G\times G$ satisfying the above mentioned properties.
 \end{itemize}
 \end{thm}

\subsubsection{Definition of inner amenability for groupoids} This theorem motivates our definition.  We first  need to introduce the following notions.

\begin{defn} Let $\cG$ be a locally compact groupoid. Following \cite[Definition 2.1]{Roe}, we say that a closed subset $A$ of $\cG \times \cG$ is {\it proper}  if for every compact subset $C$ of $\cG$, the sets $(C\times \cG) \cap A$ and $(\cG \times C) \cap A$ are compact. We say that a function $f: \cG\times \cG \to \C$ is {\it properly supported} \index{properly supported function} if its support is  proper.
\end{defn}

Given a groupoid $\cG$, let us recall that a function  $f:\cG\times \cG\to \C$ on  the product groupoid $\cG\times \cG$ is positive definite if and only if, for every integer $n$, $(x,y)\in \cG^{(0)}\times \cG^{(0)}$ and $\gamma_1,\dots, \gamma_n\in \cG^x$, $\eta_1,\dots,\eta_n\in \cG^y$, the matrix $[f(\gamma_i^{-1}\gamma_j, \eta_i^{-1}\eta_j)]_{i,j}$ is nonnegative.

\begin{defn}\label{def:wia} We say that a locally compact groupoid $\cG$ is {\it  inner amenable}\index{inner amenable l. c. groupoid}  if for every compact subset $K$ of $\cG$ and for every $\varepsilon >0$
there exists a continuous   positive definite function $f$ on the product groupoid $\cG\times \cG$, properly supported,  such that  $f(x,y) \leq 1$ for $x,y\in \cG^{(0)}$ and $|f(\gamma,\gamma) - 1| < \varepsilon$ or all $\gamma \in K$.
\end{defn}

Note that $f$ is bounded by $1$ (see Equation \ref{eq:posdef0}). Its restriction to the diagonal  is a continuous positive definite function on $\cG$. Therefore, if  $f$ was required to be compactly supported, then $\cG$ would be amenable.

  Every amenable locally compact groupoid $\cG$ with Haar system is  inner amenable since the groupoid $\cG\times \cG$ is amenable and therefore Proposition \ref{prop:amen2} applies to this groupoid.

\subsubsection{Some permanence properties} We will use several times  the following observation: let $f$ be a positive definite function on a product groupoid $\cG\times \cG$, and let $\varphi$ be a function on $\cG^{(0)}$; then the function 
\begin{equation}\label{eq:posdef}
f_\varphi : (\gamma_1,\gamma_2) \mapsto \varphi\circ r(\gamma_1)\varphi\circ s(\gamma_1)\varphi\circ r(\gamma_2)\varphi\circ s(\gamma_2)
f(\gamma_1, \gamma_2)
\end{equation}
is positive definite.

\begin{prop}\label{prop:wai} Let $c : \cL \to {\mathcal G}$ be a locally proper continuous homomorphism
between locally compact groupoids. Assume that ${\mathcal G}$ is  inner amenable. Then the groupoid $\cL$ is also  inner amenable.
\end{prop}

\begin{proof} Let $K$ be a compact subset of $\cG$ and let $\varepsilon >0$ be given.  We choose a continuous function $\varphi : \cL^{(0)} \to [0,1]$ with compact support $K'$, such that $\varphi(x) = 1$ if $x\in r(K)\cup s(K)$. Let $f:\cG\times \cG \to \C$ be a continuous  positive definite function 
on the product groupoid $\cG\times\cG$, properly supported, such that $f(x,y) \leq 1$  for $x,y\in \cG^{(0)}$ and  $|f(\gamma,\gamma) - 1| < \varepsilon$
for all $\gamma \in c(K)$. Then $h= f\circ(c\times c)$ is positive definite on $\cL\times\cL$ as well as $F= h_\varphi$.
Obviously we  have   $\abs{F(\gamma,\gamma) -1} < \varepsilon$ if $\gamma\in K$.
 Let us check that $F$ is properly supported. We denote by $S_F$ and $S_f$ the supports of $F$ and $f$ respectively. We fix a compact subset $K_1$ of $\cL$. Let $C$ be a compact subset of $\cG$ such that 
$$( c(K_1)\times \cG)\cap S_f \subset C\times C.$$ 
Then $\big(K_1\times \cL\big)\cap S_F$ is contained into $\big(c^{-1}(C)\cap \cL(K')\big)\times \big(c^{-1}(C)\cap \cL(K')\big)$ and therefore is compact. The case of $\big(\cL  \times K_1 )\cap S_F $ is similar.
  \end{proof}

Applying Proposition \ref{prop:wai}  to the examples given in \ref{exs:loc_proper} we obtain the following corollaries.

\begin{cor}\label{cor:induc} Let $\cG$ be an  inner amenable locally compact groupoid.
\begin{itemize}
\item[(i)] Every closed subgroupoid of $\cG$ is  inner amenable.
\item[(ii)] Let $E$ be a locally compact subset of $\cG^{(0)}$ whose reduction $\cG(E)$ is a subgroupoid. Then $\cG(E)$ is   inner amenable.
\end{itemize}
\end{cor}

\begin{rem} 
In particular, we see that if $\cG$ is  inner amenable, all its isotropy subgroups must be amenable whenever they are connected. 
\end{rem}

\begin{cor}\label{cor:transf} Let $\cG$ be an  inner amenable locally compact groupoid and $(Y,p)$ a $\cG$-space.
Then the semi-direct product groupoid $Y\rtimes \cG$ is  inner amenable.
\end{cor}

The following more general result  is a consequence of the propositions \ref{prop:wai} and \ref{prop:locprop}.

\begin{cor}\label{cor:transf1} Let $c:\cL\to \cG$ be a continuous strongly surjective homomorphism from a locally compact groupoid $\cL$ to an inner amenable locally compact groupoid $\cG$. If  $\Ker(c)$ is a proper groupoid, then $\cL$ is inner amenable.
\end{cor}

We note that this result is not true if we only assume that $\Ker(c)$ is amenable: there is in \cite[Proposition 4.14]{Man} an example of an extension of the discrete group $\F_6$ by $\R^2$ that is not inner amenable.

\begin{cor}\label{cor:amplinamen} Let $\cG$ be a locally compact groupoid,  and let $\varphi : T\to \cG^{(0)}$ be a continuous, open, surjective map. Assume that $\cG$ is inner amenable. Then the ampliation $\cG^\varphi$  is inner amenable.
\end{cor}

\begin{proof} Recall that $$\cG^\varphi = \set{(x,\gamma,y)\in T\times\cG\times T: \varphi(x) = r(\gamma), \varphi(y) = s(\gamma)}.$$
The inner amenability of $\cG^\varphi$ follows from the fact that the map $(x,\gamma,y) \mapsto \gamma$ from $\cG^\varphi $ into $\cG$ is locally proper.
\end{proof}

\begin{rem}\label{rem:??}  We do not know whether the inner amenability of $\cG^\varphi$ implies the same property for $\cG$. We point out that these two groupoids are  equivalent, but  whether inner amenability is preserved in general under equivalence of groupoids is still an open problem.
\end{rem}

\begin{prop}\label{prop:IAproduct} Let $\cG$ and $\cH$ be two locally compact groupoids. Then the groupoid product $\cG\times \cH$ is inner amenable if and  only if so are $\cG$ and $\cH$.
\end{prop}

\begin{proof} In one direction we use the fact that $\cG$ and $\cH$ can be identified to closed subgroupoids of $\cG\times\cH$. For the converse, assume that $\cG$ and $\cH$ are inner amenable. For $f: \cG\times\cG \to \C$ and $h:\cH\times\cH \to \C$ we set $F: \big((\gamma,\eta),(\gamma',\eta')\big)\mapsto f(\gamma,\gamma')h(\eta,\eta')$ from $(\cG\times\cH)\times(\cG\times\cH)$ to $\C$. It is easy to construct $F$ satisfying the conditions stated in Definition \ref{def:wia} for $\cG\times\cH$ from the analogous conditions for $f$ and $h$.
\end{proof}

\begin{prop}\label{prop:inductivelim}  Let $(\cG_i)_{i\in I}$ be a directed set of open and closed inner amenable subgroupoids of a locally compact groupoid $\cG$ such that $\bigcup_{i\in I} \cG_i = \cG$. Then $\cG$ is inner amenable. \end{prop}

\begin{proof} Let $K$    be a compact subset of $\cG$ and $\varepsilon >0$. Then we must find a conti\-nuous function $f: \cG\times \cG \to \C$, positive  definite and properly supported such that $\sup_{\gamma\in K}\abs{f(\gamma,\gamma) -1}\leq \varepsilon$ and $f(x,y) \leq 1$ for $x,y\in \cG^{(0)}$.

Let $i_0\in I$ such that $K\subset \cG_{i_0}$.  There exists a continuous positive definite function $f_0$ on $\cG_{i_0}\times \cG_{i_0}$, bounded by $1$ on the units,  properly supported and  such that $\sup_{\gamma\in K}\abs{f_0(\gamma,\gamma) -1}\leq \varepsilon$. We extend $f_0$ to a function $f$ on $\cG\times\cG$ by giving it the value $0$ outside $\cG_{i_0}\times \cG_{i_0}$. Then $f$ is continuous and positive definite (see the next lemma). It remains to show that $f$ is properly supported. Let $C$ be a compact subset of $\cG$. Then we have
\begin{align*}
\big(C\times \cG\big) \cap \supp(f) & = \big(C\times \cG\big) \cap \big(\cG_{i_0}\times \cG_{i_0}\big) \cap \supp(f_0)\\
&=\Big( \big(C\cap\cG_{i_0}\big)\times \cG_{i_0}\Big)\cap \supp(f_0).
\end{align*}
It follows that $\big(C\times \cG\big) \cap \supp(f)$ is compact. The case of  $\big(\cG\times C\big) \cap \supp(f) $ is similar.
\end{proof}

\subsubsection{Other examples} Unfortunately, we lack tools in order to construct  positive definite functions on groupoids. The following well-known lemma is sometimes useful.

\begin{lem}\label{lem:subgroupoid} Let $\cG$ be a groupoid and let $f$ be a positive definite function on a subgroupoid $\cH$ of $\cG$. Let $h:\cG\to \C$  be such that $h(\gamma) = f(\gamma)$ if $\gamma \in \cH$ and $h(\gamma) = 0$ otherwise. Then $h$ is  positive definite. In particular, the characteristic function of $\cH$  is 
 positive definite.
\end{lem}
 
\begin{proof} Let $x\in \cG^{(0)}$, $\gamma_1,\cdots,\gamma_n \in \cG^x$ and $\lambda_1,\cdots,\lambda_n\in \C$. By rearranging the elements of $I= [1,\cdots,n]$ we can find a partition $I_1,\cdots, I_k$ of $I$, where $I_1 = [1,\cdots,j_1]$ is the set of indices $j$ such that $\gamma_1^{-1}\gamma_j\in \cH$, $I_2= [j_1+1,\cdots j_2]$ is the set of indices $j$ such that $\gamma_{j_1+1}^{-1}\gamma_j \in \cH$ and so on. Then we have
$$\sum_{i,j\in I} \overline{\lambda_i}\lambda_j h(\gamma_i^{-1}\gamma_j) = \sum_{i,j\in I_1}\overline{\lambda_i}\lambda_j f(\gamma_i^{-1}\gamma_j)+\cdots +   \sum_{i,j\in I_k}\overline{\lambda_i}\lambda_j f(\gamma_i^{-1}\gamma_j)\geq 0.$$
\end{proof}

\begin{lem}\label{lem:opclos} Let $\cG$ be a locally compact groupoid. We assume the existence of an open and closed subgroupoid $\cH$ of $\cG\times\cG$ that contains the diagonal $\Delta_\cG$ of $\cG\times\cG$ and which has in addition the following property: for every compact subset $K$ of $\cG^{(0)}$ and every compact subset $C$ of $\cG(K)$, the sets $(C\times\cG(K))\cap \cH$ and $(\cG(K)\times C)\cap \cH$ are compact. Then $\cG$ is inner amenable.
\end{lem}

\begin{proof} Let $\varphi: \cG^{(0)} \to [0,1]$ be a continuous function with compact support $K$ and let $f_\varphi$ be the continuous positive definite function on $\cG\times\cG$ defined in the formula \eqref{eq:posdef}, where $f$ is the characteristic function of $\cH$.  It is a positive definite continuous function since the characteristic function of $\cH$ is so. Moreover  $f_\varphi$ is properly supported since for every compact subset $C$ of $\cG$ we have
$$(C\times \cG)\cap \supp(f_\varphi) \subset \Big(\big(C\cap\cG(K)\big)\times\cG(K)\Big)\cap \cH.$$
The fact that $(\cG\times C)\cap \supp(f_\varphi)$ is compact is proved similarly.

Morever, given a compact subset $C$ of $\cG$, we can choose $\varphi$ such that $f_\varphi(\gamma,\gamma)=1$ for $\gamma\in C$.
\end{proof}

The simplest example is for $\cG$ a discrete group $\Gamma$, where we take for $\cH$ the diagonal of $\Gamma\times\Gamma$. This is a particular case of the following situation.

\begin{prop}\label{prop:newin} Let $c : \cG \to {\mathcal \cG_1}$ be a locally proper continuous homomorphism
between locally compact groupoids.  We assume that 
$$\cH = \set{(\gamma,\eta)\in \cG\times \cG: c(\gamma) = c(\eta)}$$
 is an open subset of $\cG\times\cG$. Then the groupoid $\cG$ is inner amenable.
\end{prop}

\begin{proof} Since $\cH$ is an open and closed subgroupoid of $\cG\times\cG$ which contains the diagonal $\Delta_\cG$, it suffices to show that for every compact subset $K$ of $\cG^{(0)}$ and every compact subset $C$ of $\cG(K)$  the sets $(C\times\cG(K))\cap \cH$ and $(\cG(K)\times C)\cap \cH$ are compact. This is immediate, since for instance
$$(C\times\cG(K))\cap \cH \subset C\times \big(\cG(K)\cap c^{-1}(c(C))\big)$$
with $c(C)$ compact.
\end{proof}

Compared with Proposition \ref{prop:wai}, the important point here is that the target groupoid $\cG_1$ is not assumed to be inner amenable.

\begin{cor}\label{prop:moreInAmen} Let $\cG = G\ltimes Y$ be a semi-direct product groupoid where $G$ is a locally compact group  acting to the left on the  locally compact space $Y$. Let $A$ be a  locally compact  subset of $Y$ such that $\cG(A)$ is a subgroupoid of $\cG$. We denote by $c$ the restriction to $\cG(A)$ of the first projection $G\times Y\to G$. We assume  that $c(\cG(A))$ is a discrete subset of $G$. Then the groupoid $\cG(A)$ is inner amenable.
\end{cor}

\begin{proof}  We first observe that $c$ is locally proper. We set $E = c(\cG(A))$. The map $\big((t,x), (s,y)\big) \mapsto (t,s)$ from $\cG(A)\times\cG(A)$ into $G\times G$ is continuous and its range is contained into the discrete set $E\times E$. It follows that the subgroupoid $\cH = \set{\big((t,x),(t,y)\big): (t,x), (t,y)\in  \cG(A)}$ of $\cG(A)\times \cG(A)$ is open and we apply the proposition \ref{prop:newin}.
\end{proof}

The fact that the range and source of $\cG(A)$ are open is not important in this proof. This is ensured when $A$ is open, or invariant, or when the restriction of the source of $\cG$ to $\cG^A$ is open (see Lemma \ref{lem:reducproperty}).

 \begin{rem}\label{rem:partial??}   Let us recall  that if $A$ is an open subset of $Y$, the reduction $\cG(A)$ of the groupoid $\cG= G \ltimes_\alpha Y $ is the semi-direct product groupoid $G\ltimes_\beta A$ associated to the partial action $\boldsymbol{\beta } = (\set{A_g}_{g\in G}, \set{\beta_g}_{g\in G})$ where, for $g\in G$, we set $A_g= A\cap \alpha_g(A)$ and where we denote by $\beta_g$ the restriction of $\alpha_g$ to $A_{g^{-1}}$. Applying the  corollaries \ref{cor:induc} and \ref{cor:transf}, we see that the groupoid $\cG(A) = G\ltimes_\beta A$   is inner amenable if $G$ is inner amenable. 

We emphasize that the interesting point in the corollary \ref{prop:moreInAmen} is that it holds without assuming that $G$ is inner amenable. Its origin comes  from \cite{Favre} where the following example is considered.
\end{rem}

\begin{ex}\label{ex:Favre}  Let $G$ be a locally compact space. Let
 $\Lambda$ a be subset of $G$ such that   $\Lambda^{-1} \Lambda$ is uniformly discrete\footnote{This extends the case of a discrete subgroup $\Lambda$.}. This implies that $\Lambda$ is also uniformly discrete. Such a set $\Lambda$ is said to have finite local complexity. 
 
 We keep the notation of the subsection \ref{subsec:FER}, with $C$ replaced by $\Lambda$. We still set $A= \Omega_0(\Lambda)$. The proposition \ref{prop:UD} shows that the groupoid $\cG(A)$ is \'etale.  As observed in the remark \ref{rem:UD}, $A$ is the closure of the set $\set{\lambda^{-1}\Lambda : \lambda\in \Lambda}$ in the Chabauty-Fell topology. Since the set of closed subsets of $G$ contained in $\Lambda^{-1}\Lambda$ is closed, we see that  $\bigcup_{P\in A}P= \Lambda^{-1}\Lambda$. 
 
 Next, since $c(\cG(A))= \set{g\in G: \exists P\in A, g^{-1} \in P}$, we see that 
 $$c(\cG(A))\subset \bigcup_{P\in A} P^{-1} = \Lambda^{-1}\Lambda.$$
 Therefore $c(\cG(A))$ is discrete and the \'etale groupoid $\cG(A)$ is inner amenable\footnote{Recall that when $\Lambda$ is a discrete subgroup, we have $\cG(A) = \Lambda$.}.

Interesting examples of such subsets $\Lambda$, called model sets, in non inner amenable groups $G$ are given in \cite{BHP}. They provide a natural generalization of model sets in locally compact abelian groups as defined by Y. Meyer and used as mathema\-tical models of quasicrystals.  For instance it is explained how  such sets exist in any  semisimple Lie group $G$ without compact factors. The  starting point of the construction is the data  of a  lattice $\Gamma$ in a product $G\times H$, where $H$ is another locally compact group. We refer to \cite{BHP} for the details of such constructions of $\Lambda$. It is shown in \cite[Corollary 2.11]{BHP} that $\emptyset \in \Omega(\Lambda)$ if and only if the lattice $\Gamma$ is non-uniform and that in any case there is a unique $G$-invariant probability measure on $\Omega(\Lambda)^\times$ (see \cite[Theorem 3.4]{BHP}). Nice examples are given from the embedding of $\Gamma = SL_n(\Z[\sqrt{2}])$ into $SL_n(\R)\times SL_n(\R)$ (starting from the embedding $a+b\sqrt{2} \mapsto (a+b\sqrt{2}, a-b\sqrt{2})$ from $\Z[\sqrt{2}]$ into $\R\times \R$) or the embedding from $SL_n(\Z[1/p])$ into  $SL_n(\Q_p)\times SL_n(\R)$.

We note that  the groupoid $\cG = G\ltimes \Omega(\Lambda)^\times$ is inner amenable  in such examples (by Remark \ref{rem:UD2}) although $G$ is not inner amenable.

Let us observe that if we take $H$ to be trivial in this construction, then $\Lambda = \Gamma$ and $\Omega(\Lambda)^\times$ is the $G$-space $G/\Gamma$.  Note that we have $\Omega(\Gamma)^\times = \Omega_0(\Gamma)$ if and only if the lattice $\Gamma$ is uniform.
\end{ex}

 Let us come back to the general case studied in the corollary \ref{prop:moreInAmen}. For $t\in G$ we set $S_t = (\set{t}\times Y)\cap \cG(A)$. Note that $S_t \not =\emptyset$ if and only if $t\in E$. The family $\cS = \set{S_t:t\in G}$ has the properties described in the following proposition, which suffice to imply inner amenability.
 
 \begin{prop}\label{prop:moremoreInAmen} Let $\cG$ be a locally compact groupoid that is covered by a locally finite familiy $\cS$ of closed and open subsets such that 
 \begin{itemize}
 \item[(i)] $S\in \cS\Rightarrow S^{-1}\in \cS$ ;
 \item[(ii)] given $S, T\in \cS$ there exists $R\in \cS$ such that $ST\subset R$;
 \item[(iii)] for every compact subset $K$ of $\cG^{(0)}$ and every $S\in \cS$, the intersection $\cG(K)\cap S$ is compact.
 \end{itemize}
 \end{prop}
 
 \begin{proof} We set $\cH = \bigcup_{S\in \cS} S\times S$. It is an open subgroupoid of $\cG\times\cG$ which contains $\Delta_\cG$. Let us show that $\cH$ is closed. Let $(\gamma,\eta)\not\in \cH$. Let $S_1,\cdots, S_k$ be the elements of $\cS$ which contain $\gamma$ and let $T_1,\cdots, T_l$ be the elements of $\cS$ which contain $\eta$. The two sets $\set{S_1,\cdots, S_k}$ and $\set{T_1,\cdots, T_l}$ are disjoint. Since the elements of $\cS$ are closed, and since $\cS$ is locally finite there exists a neighborhood $V$ of $\gamma$ that  intersects no other element of $\cS$ than $S_i$, $i= 1,\cdots, k$. We choose similarly a neighborhood $W$ of $\eta$ that only intersects $T_1,\cdots, T_l$. Then we have $(V\times W)\cap \cH = \emptyset$.
 
 Now, let $C$ be compact subset of $\cG$ and let $K$ be a compact subset of $\cG^{(0)}$. Then
 $$(C\times \cG(K)) \cap \cH = \bigcup_{S\in \cS} (C\cap S) \times (\cG(K)\cap S).$$
 Since $C$ only intersects a finite number of elements of $\cS$, it follows from (iii) that $(C\times \cG(K)) \cap \cH$ is compact, and the same conclusion holds for $(\cG(K)\times C) \cap \cH$.
 \end{proof}
 
 \begin{cor}\label{cor:moremoreInAmen}  Every locally compact groupoid that is covered by a locally finite inverse semigroup of closed and open bisections is inner amenable.
 \end{cor}

\subsection{About invariance under equivalence}\label{subsec:aboutIA} Let us consider first the case of similarities.

\begin{prop}\label{prop:midsim} Let $\cG$ and $\cH$ be two locally compact groupoids, and let $f: \cG\to \cH$, $g: \cH\to \cG$ be two continuous homomorphisms. We assume that there exists a continuous map  $\theta: \cG^{(0)} \to \cG$ such that $\theta\circ r(\gamma)g\circ f(\gamma) = \gamma\theta\circ s(\gamma)$ for all $\gamma\in \cG$. Then, $\cG$ is  inner amenable whenever $\cH$ is inner amenable.
\end{prop}

\begin{proof} It suffices to check that $f$ is locally proper. Let $K$ be a compact subset of $\cG$  and let $C$ be a compact subset of $\cH$. Then we have
$$\cG(K)\cap f^{-1}(C) \subset \theta(K)g(C)\theta(K)^{-1}$$
where $\theta(K)g(C)\theta(K)^{-1}$ is compact. 
\end{proof}

\begin{cor}\label{cor:similar} Inner amenability is preserved under similarity.
\end{cor}

Using Theorem \ref{thmCRS} and the proposition \ref{prop:IAproduct}, we obtain the result below.

\begin{prop}\label{cor:ample} Let $\cG$ and $\cH$ be two ample equivalent groupoids with $\sigma$-compact sets of units. Then $\cG$ is inner amenable if and only if so is $\cH$.
\end{prop}

In particular, every ample transitive groupoid is inner amenable, since it is equi\-valent to a discrete group.

We end this section with the study of a particular case where the inner amena\-bility of an ampliation of $\cG$ is im\-plied by the inner amenability of $\cG$.  In the following lemma, we keep the notation of the end of Subsection  \ref{subsec:gen-mor}.

\begin{lem}\label{lem:inamenampl} Let $\cG$ be a locally compact groupoid, $(U_i)_{i\in I}$ a locally finite covering of $\cG^{(0)}$ by open subsets, and let $\cG'$ be the ampliation of $\cG$ defined by the natural map $f$ from $T=\sqcup_i U_i$ onto $\cG^{(0)}$. Assume that $\cG'$ is inner amenable. Then $\cG$ is also inner amenable.
\end{lem}

\begin{proof} Recall that $\cG' = \bigsqcup_{i,j} \set{(i,\gamma,j): \gamma \in \cG_{U_j}^{U_i} }$. 
 A positive definite kernel $h$ on $\cG'\times\cG'$ is such that if we fix $x,x'\in T$, and elements $\widetilde{\gamma_k} = (x,\gamma_k,y_k)$, $\widetilde{\gamma_{k}' }= (x',\gamma_{k}', y_{k}')$ of $\cG'$ , with $k= 1,\cdots, n$, then for every $\lambda_1,\cdots,\lambda_n\in \C$ we have

$$\sum_{k,l =1}^{n} \overline{\lambda_k}\lambda_l h\big((y_k,\gamma_{k}^{-1}\gamma_l),  y_{l}), (y_{k}', \gamma_{k}'^{-1}\gamma_{l}', y_{l}')\big)
=\sum_{k,l =1}^{n} \overline{\lambda_k}\lambda_l h\big(\widetilde{\gamma_{k}}^{-1}\widetilde{\gamma_l}, \widetilde{\gamma_{k}'}^{-1}\widetilde{\gamma_{l}'}\big) \geq 0.$$

 Let $(\varphi_i)_i$ be a partition of unit subordinated to the covering $(U_i)$ and let $h$ be a continuous properly supported positive definite kernel on $\cG'\times\cG'$. For $(\gamma_1,\gamma_2) \in \cG\times \cG$ we set
 $$h'(\gamma_1,\gamma_2) = \sum_{i,j} \varphi_i(r(\gamma_1))^{1/2}\varphi_j(s(\gamma_1))^{1/2}\varphi_i(r(\gamma_2))^{1/2}\varphi_j(s(\gamma_2))^{1/2} h\big((i,\gamma_1,j), (i,\gamma_2,j)\big).$$
 Let us check that this continuous function is positive definite on $\cG\times\cG$. We fix $x,y\in \cG^{(0)}$, an integer $n$ and $\gamma_1,\cdots,\gamma_n\in \cG^x$, $\gamma_{1}',\cdots,\gamma_{n}' \in\cG^y$. Given $\lambda_1,\cdots,\lambda_n\in \C$, we have to show that
 $$\sum_{k,l}\overline{\lambda_k}\lambda_l h'(\gamma_{k}^{-1}\gamma_l, \gamma_{k}'^{-1}\gamma_{l}') \geq 0,$$
 where $ h'(\gamma_{k}^{-1}\gamma_l, \gamma_{k}'^{-1}\gamma_{l}')$ is equal to
 $$\sum_{i,j} \varphi_i(s(\gamma_k))^{1/2}\varphi_j(s(\gamma_l))^{1/2}\varphi_i(s(\gamma_{k}'))^{1/2}\varphi_j(s(\gamma_{l}'))^{1/2} h\big((i,\gamma_{k}^{-1}\gamma_l,j), (i,\gamma_{k}'^{-1}\gamma_{l}',j)\big).$$
We note that that the sources of $\gamma_1,\cdots,\gamma_n, \gamma_{1}',\cdots,\gamma_{n}'$ belong to a finite set of the covering $(U_i)_i$ and therefore  the above sum is limited to $i,j$ in a finite subset $F$ of $I$. For $i\in F$ and $k = 1\cdots n$ we set $\lambda_{i,k} =  \varphi_i(s(\gamma_k))^{1/2}\varphi_i(s(\gamma_{k}'))^{1/2}$. We choose $i_0$ such that $x\in U_{i_0}$ and $i_{0}'$ such that $y\in U_{i_{0}'}$ and we set $\widetilde\gamma_{i,k} = (i_{0},\gamma_{k},i)$, $\widetilde\gamma_{i,k}' = (i_{0}',\gamma_{k}',i)$. Then we have
$$h'(\gamma_{k}^{-1}\gamma_l, \gamma_{k}'^{-1}\gamma_{l}')= \sum_{i,j\in F} \lambda_{i,k}\lambda_{j,l} h(\widetilde\gamma_{i,k}^{-1}\widetilde\gamma_{j,l}, \widetilde\gamma_{i,k}'^{-1}\widetilde\gamma_{j,l}') .$$
Therefore
$$\sum_{k,l}\overline{\lambda_k}\lambda_l h'(\gamma_{k}^{-1}\gamma_l, \gamma_{k}'^{-1}\gamma_{l}') = \sum_{i,j\in F, 1\leq k,l\leq n} \overline{\lambda_k} \lambda_{i,k} \lambda_l\lambda_{j,l} h(\widetilde\gamma_{i,k}^{-1}\widetilde\gamma_{j,l}, \widetilde\gamma_{i,k}'^{-1}\widetilde\gamma_{j,l}')\geq 0.$$

 Let us show now that $h'$ is properly supported.  For $(\gamma_1,\gamma_2)\in \cG\times \cG$, we set 
 $$\widetilde{h_{i,j}}(\gamma_1,\gamma_2) =  \varphi_i(r(\gamma_1))^{1/2}\varphi_j(s(\gamma_1))^{1/2}\varphi_i(r(\gamma_2))^{1/2}\varphi_j(s(\gamma_2))^{1/2} h\big((i,\gamma_1,j), (i,\gamma_2,j)\big) .$$
 Let $K$ be a compact subset of $\cG$. We have to show that 
 $$(K\times \cG) \bigcap \hbox{Supp}\big(\sum_{i,j} \widetilde{h_{i,j}}\big)$$
 is compact.      
  Since $r(K)$ and $s(K)$ are compact the sum in this expression is finite, say limited to $i,j\in F$, where $F$ is a finite subset of $I$. Therefore it suffices to show that for $i,j\in F$ the space $(K\times \cG) \bigcap \hbox{Supp}\big(\widetilde{h_{i,j}}\big)$ is compact. Since $h$ is properly supported, the set of $(i,\gamma_2,j)\in \cG'$ such that there exists $\gamma_1\in K$ with $h\big((i,\gamma_1,j), (i,\gamma_2,j)\big)\not=0$ is relatively compact in $\cG'$, and thus the set $E$ of such $\gamma_2$ is relatively compact in $\cG_{U_j}^{U_i}$. The conclusion follows because $(K\times \cG) \bigcap \hbox{Supp}\big(\widetilde{h_{i,j}}\big)\subset K\times E$.
 
 Finally, let $K$ be a compact subset of $\cG$ and $\varepsilon >0$. We show that we can choose $h$ such that $\sup_{\gamma\in K}\abs{h'(\gamma,\gamma) - 1}\leq \varepsilon$. Since the covering $(U_i)$ is locally finite, there exists a finite subset $F$ of $I$ such that $r(K)\cap \supp(\varphi_i) = \emptyset=s(K)\cap \supp(\varphi_i)$ if $i\notin F$. Therefore $r(K)= \bigcup_{i\in F} r(K)\cap \supp(\varphi_i)$ and $s(K)= \bigcup_{i\in F} s(K)\cap \supp(\varphi_i)$. It follows that, if we set $K_j^{i} = K\cap r^{-1}(\supp(\varphi_i)) \cap s^{-1}(\supp(\varphi_j))$, we have $K = \bigcup_{i,j\in F} K_j^{i}$ where $K_j^{i}$ is a compact subset  of $\cG_{U_j}^{U_i}$.
 We choose  $h$ such that $\abs{h\big((i,\gamma,j), (i,\gamma,j)\big) -1}\leq \varepsilon$ for $i,j\in F$ and $\gamma\in K_j^{i}$.  We have, for $\gamma\in K$,
 $$\abs{h'(\gamma,\gamma)-1} =\abs{ \sum_{i,j\in F} \varphi_i(r(\gamma))\varphi_j(s(\gamma)) \Big(h\big((i,\gamma,j), (i,\gamma,j)\big)-1\Big)}\leq \varepsilon.$$
\end{proof}

In the following theorem, in the case where $\rho$ has a global continuous section, we recover the proposition \ref{prop:wai}.

\begin{thm}\label{thm:inamen} Let $(Z,\rho,\sigma)$ be a  generalized locally proper morphism from $\cG$ into $\cH$. We assume that $\rho$ has local sections. If $\cH$ is inner amenable, then $\cG$ is inner amenable.
\end{thm}

\begin{proof}  We consider a locally finite covering $(U_i)_{i\in I}$ of $\cG^{(0)}$ by open subsets such for each $i$ there is a local section $p_i: U_i \to \cG$ of $\rho$. We keep the notation following the definition \ref{defn:genmorloc}. It suffices to show that the  homomorphism $c:\cG'\to \cH$ is locally proper and then to use Proposition \ref{prop:wai} and Lemma \ref{lem:inamenampl}. Let $K$ be a compact subset of $T= \sqcup_i U_i$ and let $C$ be a compact subset of $\cH$. We write $K$ as $K = \sqcup_{i\in F} K_i$ where  $K_i$ is a compact subset of $U_i$  and $F$ is a finite subset of $I$. We have to show that  the set $S$ of $(i,\gamma,j)$, with $i,j\in F$ and $r(\gamma)\in K_i$, $s(\gamma)\in K_j$, such that $g_{i,j}(\gamma)\in C$, is compact. Recall that $\gamma p_i(s(\gamma)) = p_j(r(\gamma))g_{i,j}(\gamma)$. For $(i,\gamma,j)\in S$, we observe that
$\gamma$ is contained in the set $\set{\eta: \eta p_i(s(K))\cap p_j(r(K))C\not=\emptyset}$.
The generalized morphism being locally proper, the left action of $\cG$ on $Z$ is proper and therefore this set is compact. It follows that $c$ is locally proper.
\end{proof}

\begin{cor}\label{cor:inamen} Let $(Z,\rho,\sigma)$ be a $(\cG$-$\cH)$-equivalence. We assume that $\rho$ and $\sigma$ have local sections. Then $\cG$ is inner amenable if and only if $\cH$ is inner amenable.
\end{cor}

\begin{cor}\label{cor:subgroup} We keep the notation of the proposition \ref{lem:reducproperty1} and we assume that $s: \cG^A\to \cG^{(0)}$ is a surjective local homeomorphism.
Then $\cG(A)$ is inner amenable if and only if so is $\cG$.
\end{cor}

\begin{proof} Recall that $(\cG^A, r, s)$ is an equivalence between the groupoids $\cG(A)$ and the groupoid $\cG$, where $r: \cG^A \to A$ and $s : \cG^A\to \cG^{(0)}$ are the restrictions of the range and source maps of $\cG$ to $\cG^A$. We apply the previous corollary.
 \end{proof}
 
 Now we keep the notation of Proposition \ref{prop:UD}.
 
 \begin{cor}\label{cor:UD2} Let $\Lambda$ be a uniformly discrete subset of a locally compact group $G$. Then the groupoid $\cG(\Omega_0(\Lambda))$ is  inner amenable if and only if so is the groupoid $G\ltimes \Omega(\Lambda)^\times$.
 \end{cor}
\begin{proof} This follows from the previous corollary, since we have shown in Proposition \ref{prop:UD} that $s: \cG^A \to  \Omega(\Lambda)^\times$ is a local homeomorphism, where $A = \Omega_0(\Lambda)$.
\end{proof}

\begin{rem}\label{rem:UD2} Let us assume that $\Lambda^{-1}\Lambda$ is uniformly discrete (which implies that $\Lambda$ is also uniformly discrete). Then the groupoid $\cG(\Omega_0(\Lambda))$ is inner amenable (see Example  \ref{ex:Favre}) and therefore $G\ltimes \Omega(\Lambda)^\times$ is inner amenable. In particular, when $\Lambda$ is a discrete subgroup of $G$ the transformation groupoid $G\ltimes G/\Lambda$ is inner amenable.
\end{rem}

\section{\textbf{\textsc{Amenability at infinity}}}\label{sec:ameninf}
\subsection{Definitions and first properties}\label{def:ameninf}
Let $\cG$ be a  locally compact groupoid and set $X=\cG^{(0)}$.  Recall that a  {\it fibrewise compact $\cG$-space}\index{fibrewise compact $\cG$-space} is a $\cG$-space $(Y,p)$ which is fibrewise compact, that is $p:Y\to X$ is a proper map.

\begin{defn}\label{def:sai}    We say that a locally compact groupoid $\cG$ is {\it amenable at infinity}\index{amenability at infinity} if there is an amenable fibrewise compact $\cG$-space $(Y,p)$, where $p$ is {\it surjective}.  Whenever $(Y,p)$ can be chosen such that, in addition, there exists a continuous section $\sigma : X\to Y$ of $p$, then we say that $\cG$ is {\it strongly amenable at infinity}.
\index{strong amenability at infinity}
\end{defn}

\begin{rems} The interest of this last notion will become clear later (see Proposition \ref{prop:SC} and Theorem \ref{prop:amen_inf})  when we will deal with the action $\cG\actson \beta_r \cG$ for an \'etale groupoid $\cG$. It has a continuous section, namely the inclusion of $\cG^{(0)}$ into $\cG$. On the other hand, as said in the remark \ref{rem:compG}, the moment map $p_\beta : \beta_r \cG \to \cG^{(0)}$ is not always open.

 In \cite[D\'efinition 3.2]{Las}  an \'etale locally compact groupoid $\cG$ is said to be amenable at infinity if there is an amenable fibrewise compact $\cG$-space $(Y,p)$ where in addition $p$ is assumed to be open. In this situation, it is shown in \cite{Las} that there exists an amenable fibrewise compact $\cG$-space $(M,\tilde p)$ such that $\tilde p$ admits a continuous section but   the proof  has some inaccuracies.
  
 We don't have any example of a locally compact groupoid where the two notions  of amenability at infinity differ.

 Note that a locally compact group $G$ is amenable at infinity if and only if it has an amenable action on a compact space. In this case there is no difference between amenability at infinity and strong amenability at infinity. 

Of course, every amenable locally compact groupoid is strongly amenable at infinity since the canonical left action of $\cG$ on $\cG^{(0)}$ is amenable.
\end{rems}

We end this section with a first study of some permanence properties of amenabi\-lity at infinity.  Analogous  results for KW-exactness have been obtained in \cite{Lal14, Lal17}. The first invariance that we consider is by passing to subgroupoids.

\begin{prop}\label{prop:subgroupoid} Let $\cG$ be an amenable (resp. strongly amenable) at infinity locally compact groupoid. Let $\cH$ be a locally compact subgroupoid  of $\cG$. Then $\cH$ is amenable (resp. strongly amenable) at infinity.
\end{prop}

\begin{proof} Let $(Y,p)$ be a fibrewise compact amenable $\cG$-space with $p$ surjective and set $Y_E = p^{-1}(E)$ where $E = \cH^{(0)}$. Note that the restriction of $p$ to $Y_E$ is proper and surjective, and that it has a continuous section whenever $p$ has a continuous section. Moreover, it is easily seen that $Y_E\rtimes \cH$ is a subgroupoid of $Y\rtimes \cG$. Therefore, by \cite[Proposition 5.1.1]{AD-R}, the groupoid $Y_E\rtimes \cH$ is amenable. \end{proof}

It follows in particular that if $\cG$ is amenable at infinity, its isotropy groups are amenable at infinity.
\begin{prop}\label{prop:inv_sd} Let $\cG$ be a locally compact groupoid.
\begin{itemize}
\item[(i)] Assume that $\cG$ is amenable (resp. strongly amenable) at infinity. Then for every $\cG$-space $(Z,q)$, the semi-direct product groupoid $Z\rtimes \cG$ is amenable (resp. strongly amenable) at infinity.
\item[(ii)] Let $Z$ be a $\cG$-space whose moment map $q$ is proper and surjective. Assume that $Z\rtimes \cG$  is amenable at infinity. Then $\cG$ is amenable at infinity.
\item[(iii)] Let $Z$ be a $\cG$-space whose moment map $q$ is proper, surjective,  and has a continuous section. Assume that $Z\rtimes \cG$ is strongly amenable at infinity. Then $\cG$ is strongly amenable at infinity.
\end{itemize}
\end{prop}

\begin{proof} (i)  Let $(Y,p)$ be a fibrewise compact amenable $\cG$-space with $p$ surjective. We denote by $l: Z\, _q\!*_pY \to Z$ the first projection. We observe that $l$ is proper and surjective, and that whenever $p$ has a continuous section $\sigma$, then $z\mapsto (z,\sigma(q(z)))$ is a continuous section for $l$. We let $Z\rtimes \cG$ act to the left on $Z\,_q\!*_pY \to Z$ by
$$(z,\gamma)(\gamma^{-1}z,y) = (z,\gamma y).$$
To prove (i), it suffices to show that this action is amenable. The map
$$\big((z,y),(z,\gamma)\big) \mapsto ((z,y), \gamma)$$
is an isomorphism of groupoids from $(Z\,_q\!*_pY)\rtimes (Z\rtimes \cG)$ onto $(Z\,_q\!*_pY)\rtimes \cG$, where $\cG$ acts diagonally on $Z\,_q\!*_pY$. Similarly, the groupoids $(Z\,_q\!*_pY)\rtimes \cG$ and $(Z\,_q\!*_pY)\rtimes (Y\rtimes \cG)$ are isomorphic (see Lemma \ref{lem:idengr}) and they are amenable since $Y\rtimes \cG$ is amenable. 

(ii) Let $(Y,p)$ be a fibrewise compact amenable $Z\rtimes\cG$-space. We define a continuous $\cG$-action on $Y$, whose moment map is $q\circ p$ by
$$\gamma y = (\gamma p(y), \gamma) y.$$
The map $q\circ p$ is proper and surjective. Moreover, the map $(y,\gamma) \mapsto (y,(p(y),\gamma))$ is an isomorphism of groupoids from $Y\rtimes \cG$ onto $Y\rtimes \big(Z\rtimes\cG\big)$ (see again Lemma \ref{lem:idengr}) and therefore  $Y\rtimes \cG$ is amenable.

(iii) is proved in the same way.
\end{proof}

 Of course, (ii) and (iii) do not extend to the case where $q$ is not proper: for every group $G$ acting by left translations onto itself, the transformation groupoid $G\rtimes G$ is amenable (even proper), although there exist discrete groups that are not amenable at infinity, for instance the Gromov monster groups \cite{Gro} or the groups introduced in \cite{Osa}.

\begin{prop}\label{prop:stab-homo1} Let $(Z,\rho,\sigma)$ be a locally proper and faithful generalized morphism from a locally compact groupoid $\cG$ into a locally compact groupoid $\cH$. If $\cH$ is amenable at infinity, then so is $\cG$.
\end{prop}
\begin{proof} By Proposition \ref{prop:homo1}  there is a right $\cH$-space $Y$ such that $Z$ is a  ($\cG$-$Y\rtimes \cH$)-equivalence. Since the groupoid $Y\rtimes \cH$ is amenable at infinity (Proposition \ref{prop:inv_sd}), the conclusion follows from the proposition \ref{prop: stab-amen-infi}  below.
\end{proof}

We will see in Proposition \ref{prop:locpropinv} that for genuine homomorphisms, the faithfulness  assumption is not necessary.

\begin{prop}\label{prop: stab-amen-infi}  Let $\cG$ and $\cH$ be two equivalent locally compact groupoids. Then $\cG$ is amenable at infinity if and only if so is $\cH$.
\end{prop}

\begin{proof} Let $(Z,r_Z,s_Z)$ be a ($\cG$-$\cH$)-equivalence and let  $(Y,p)$ be a fibrewise compact amenable left $\cH$-space with $p$ surjective. We set $T = Z\,_{s_Z}\!\!*_p Y$. We define a right action of $\cH$ on $T$ by setting $(z,y)\eta = (z\eta, \eta^{-1}y)$ when $r(\eta) = s_Z(z) = p(y)$. This action is free and proper since so is $\cH\actson Z$.  It follows from \cite[Lemma 2.1.11]{AD-R} that  $r_T:T\to T/\cH$ is open and therefore from \cite[Proposition 2.1.12]{AD-R} that the topological quotient space $S= T/\cH$ is locally compact. Given $(z,y) \in T$ we denote by $[z,y]$ its equivalence class. Observe that $\cG$ acts on $S$ by $\gamma [z,y] = [\gamma z,y]$ if $s(\gamma) = r_Z(z)$.

Furthermore, we denote by $s_T$ the map $(z,y) \mapsto y$ from $T$ onto $Y$. This map is continuous and open. Indeed, let $U,V$ be open subsets of $Z$ and $Y$ respectively. Then we have $s_T(U\,_{s_Z}\!\!*_p V) = V \cap p^{-1}(s_Z(U))$ which is an open subset of $Y$ since $s_Z$ is open.  The groupoid $Y\rtimes \cH$ acts to the right on $T$, with moment map $s_T$, by
$$(z,y)(y,\eta) = (z\eta, \eta^{-1} y), \quad \forall (z,y,\eta)\in Z\times Y\times \cH, \,\hbox{with}\,\, s_Z(z) = p(y) = r(\eta),$$
and the groupoid $\cG\ltimes S$ acts on $T$, with moment map $r_T$, by
$$(\gamma, w)(z,y) = (\gamma z,y)$$
where $(\gamma, w, z,y)\in \cG\times S\times Z\times Y$ with $s(\gamma)=  r_Z(z) = p(y)$ and $r_T(z,y) = w$.

Let us show that $T$ is a ($\cG\ltimes S$-$Y\rtimes \cH$)-equivalence of groupoids. One immediately sees that the left and right actions commute and are free and proper. Obviously $r_T$ induces a homeomorphism from  $T/(Y\rtimes \cH)$ onto $S$, and since $\cG$ acts transitively on the fibres $s_Z^{-1}(x)$ with $x\in \cH^{(0)}$, we see that $s_T$ induces a homeomorphism from  $(\cG\ltimes S)\setminus T$ onto $Y$. 

We know that amenability is preserved under equivalence of groupoids (see \cite[Theorem 2.2.17]{AD-R}), and therefore the action of $\cG$ on $S$ is amenable. To conclude that $\cG$ is amenable at infinity it remains to show that the map $q: S\to \cG^{(0)}$ such that $q([z,y)] = r_Z(z)$ is proper and surjective. Surjectivity follows from the fact that for every $x\in \cG^{(0)}$ there exists $(z,y)\in Z\times Y$ with $r_Z(z) = x, s_Z(z) = p(y)$.  Let us show that $q$ is proper. Let $K$ be a compact subset of $\cG^{(0)}$. Since the map $r_Z$ is open and surjective, there exists a compact subset $K_1$ of $Z$ such $r_Z(K_1) = K$.  Using the fact that $p$ is proper, we see that $C_1= \set{(z,y)\in K_1\times Y: r_Z(z) = p(y)}$ is compact and therefore that $C= r_T(C_1) = \set{[z,y]: (z,y)\in C_1}$ is a compact subset of $S$. Let us show that $q^{-1}(K) = C$. Obviously we have $q(C) = K$. Moreover let $[z,y]\in S$ such that $q[z,y] = r_Z(z)\in K$. There exists $z'\in K_1$ with $r_Z(z') = r_Z(z)$, and so  $z= z'\eta$ with $\eta\in \cH$. It follows that $[z,y] = [z'\eta, y] = [z',\eta y]\in C$.
\end{proof}

\begin{prop}\label{prop: stab-amen-infi1}  Let $(Z,r_Z,s_Z)$ be an equivalence between two locally compact groupoids $\cG$ and $\cH$. We assume that $\cH$ is strongly amenable at infinity and that $r_Z$ has a global continuous section $\sigma$. Then $\cG$ is strongly amenable at infinity.
\end{prop}

\begin{proof} We keep the notation of the proof of the previous proposition. We have to show that the map $q: [z,y] \in S \mapsto r_Z(z)$ have a global continuous section.
Let $\sigma'$ be a continuous section of $p:Y\to \cH^{(0)}$. Then $x\mapsto [\sigma(x), \sigma'\big(s_Z(\sigma(x))\big)]$ is such a section.
\end{proof}

\subsection{About invariance under extensions}\label{subsec:iue}  We have already considered in Proposition \ref{prop:inv_sd} the case of transformation groupoids.  We now study other situations.
In the following proposition and its corollaries we assume that the groupoids have a Haar system.

\begin{prop}\label{prop:locpropinv} Let $c$ be a locally proper continuous homomorphism from a locally compact groupoid $\cL$ into a locally compact groupoid $\cG$ which is amenable (resp. strongly amenable) at infinity. Then so is $\cL$.
\end{prop}

\begin{proof}  Let $(Y,p)$ be a fibrewise compact  amenable $\cG$-space with $p$ surjective. We set $Z= \cL^{(0)}$ and $X= Z\, _{c^{(0)}}\!*_pY$, and we denote by $l$ the first projection $Z\, _{c^{(0)}}\!*_pY \to Z$. We observe that $l$ is proper and surjective, and that whenever $p$ has a continuous section $\sigma$, then $z\mapsto (z,\sigma(c^{(0)}(z)))$ is a continuous section for $l$. We let $\cL$ act to the left on $Z\, _{c^{(0)}}\!*_pY$ by
$$\eta(s(\eta),y) = (r(\eta),c(\eta)y).$$
To conclude, it suffices to show that the groupoid $X\rtimes \cL$ is amenable. Note that the range of $\big((z,y),\eta\big)\in X\rtimes \cL$ is $(r(\eta),y)$ and that its source is $(s(\eta),c(\eta)^{-1} y)$. Let $\Phi: X\rtimes \cL\to Y\rtimes \cG$ be defined by
$$\Phi\big(((z,y), \eta\big) = (y,c(\eta)),$$
where $z = r(\eta)$ and $r(c(\eta))= c^{(0)}(r(\eta)) = p(y)$. It is easily checked that $\Phi$ is a continuous homomorphism. Let us show that $\Phi$ is locally proper.  Let $C$ and $K$ be compact subspaces of $X$ and $Y\rtimes \cG$ respectively . We  have to check that $\set{(x,\eta)\in (X\rtimes\cL)(C): \Phi(x,\eta)\in K}$ is compact.  We can assume that $C= C_1\,_{c^{(0)}}\!*_p C_2$ and $K = K_1\,_p\!*_r K_2$, where $C_1\subset Z$, $C_2\subset Y$, $K_1\subset Y$, $K_2\subset \cG$ are compact spaces.

We have $(X\rtimes\cL)(C) \subset C\,_l\!*_r\cL(C_1)$ and 
$$(X\rtimes\cL)(C) \cap\Phi^{-1}(K) \subset C\,_l\!*_r\Big(\cL(C_1)\cap c^{-1}(K_2)\Big).$$
The right-hand side is compact since $c$ is locally proper, and therefore $\Phi$ is locally proper. Since $Y\rtimes \cG$ is amenable, we see that $X\rtimes \cL$ is amenable (Proposition \ref{prop:locpropre}).
\end{proof}

We observe that in this result we do not assume that $c$ is surjective.

\begin{cor}\label{prop:equivproper} Let $\cG\actson Z$ and $\cG \actson Y$ be two left $\cG$-actions and let $q:Z\to Y$ be an equivariant proper map. Assume that $Y\rtimes \cG$ is amenable (resp. strongly amenable) at infinity. Then $Z\rtimes \cG$ is amenable (resp. strongly amenable) at infinity.
\end{cor}

\begin{proof} This follows from the fact that $\Phi: Z\rtimes \cG \to Y\rtimes \cG$ defined by $\Phi(z,\gamma) = ((q(z),\gamma)$ is locally proper.
\end{proof}

\begin{cor}\label{cor:kerproper} Let  $c$ be a  continuous strongly surjective homomorphism from a locally compact groupoid $\cL$ into a locally compact groupoid $\cG$. We assume that $\cG$ is amenable (resp. strongly amenable) at infinity and that $\Ker(c)$ is a proper groupoid. Then $\cL$ is amenable  (resp. strongly amenable) at infinity.
\end{cor}
\begin{proof} This follows from Proposition  \ref{prop:locpropinv}, since $c$ is locally proper by Proposition \ref{prop:locprop}.
\end{proof}

It is likely that the above result remains true when $Ker(c)$ is assumed to be (strongly) amenable at infinity, instead of being proper.  At least in the case of locally compact groups the answer is positive. Indeed,  KW-exactness is preserved under extension (\cite[Theorem 5.1]{KW99bis}), and KW-exactness is equivalent to amenabi\-lity at infinity \cite{BCL, OS20}. A direct proof in case of discrete groups is given in \cite[Proposition 5.2.6]{AD-R}. To adapt it to the case of locally compact groupoids seems quite technical. 
At least, we can easily replace the fact that $\Ker(c)$ is proper by amenability.

\begin{prop}\label{cor:keramen} Let  $c$ be a  continuous strongly surjective homomorphism from a locally compact groupoid $\cL$ into a locally compact groupoid $\cG$. We assume that $\cG$ is amenable (resp. strongly amenable) at infinity and that $\Ker(c)$ is an amenable groupoid. Then $\cL$ is amenable  (resp. strongly amenable) at infinity.
\end{prop}

\begin{proof} We keep the notation of the proof of the proposition \ref{prop:locpropinv}. First we observe that $\Phi: X\rtimes \cL\to Y\rtimes \cG$ is strongly surjective: indeed, given a unit $(z,y)$ of $X\rtimes \cL$, we have $\big(X\rtimes\cL)^{(z,y)} = \set{(z,y)}\times \cL^z$ and
$$\Phi\Big(\set{(z,y)}\times \cL^z\Big) = \set{y}\times \cG^{p(y)} = \Big(Y\rtimes\cG\Big)^{\Phi^{(0)}(z,y)}.$$
We have 
$$\Ker(\Phi) = \set{((r(\eta),y),\eta): c^{(0}(r(\eta)) = p(y), \eta\in \Ker(c)}= X\rtimes \Ker(c),$$
and therefore $\Ker(\Phi)$ is an amenable groupoid, since $\Ker(c)$ is amenable. It follows that $X\times \cL$ is an extension of the amenable groupoid $Y\rtimes \cG$ by the amenable groupoid $\Ker(\Phi)$ and therefore is amenable by the theorem \ref{thm:amen-ext}.
\end{proof}

  Let $\boldsymbol{\beta} = (\set{X_g}_{g\in \Gamma}, \set{\beta_g}_{g\in \Gamma})$ be a partial action of a discrete group $\Gamma$. The cocycle $c: (x,g,y) \mapsto g$ has a proper kernel. It is locally proper if and only if for every $g\in \Gamma$ the graph $\set{(\beta_g(y),y): y\in X_{g^{-1}}}$ is closed in $X\times X$ (see \ref{exs:loc_proper}  (d)) and so is not locally proper in general. However, we can take advantage of the fact that $\Gamma$ is discrete to obtain the following positive result (where exact is the same at amenable at infinity in this case).

\begin{prop}\label{prop:partial} Let $\boldsymbol{\beta} = (\set{X_g}_{g\in \Gamma}, \set{\beta_g}_{g\in \Gamma})$ be a partial action of an exact discrete group $\Gamma$ on a locally compact space $X$. Then the partial transformation groupoid $\Gamma\ltimes_\beta X$  is strongly amenable at infinity. 
\end{prop}

\begin{proof}  The proof is similar to that of Proposition \ref{prop:inv_sd} (i) and this time uses  the proposition \ref{thm:RW17Bis}. We write $gx$ instead of $\beta_g(x)$. Let $Y$ be a compact space on which $\Gamma$ acts amenably.  We are going to define a left action of $\Gamma\ltimes_\beta X  $ on $X\times Y$.  Let $l : X\times Y \to X$ be the first projection.
Then $$(\Gamma\ltimes_\beta X )\,_s*_l (X\times Y) = \set{\big((g,x),(x,y)\big) : g\in \Gamma, x\in X_{g^{-1}}, y\in Y}$$
and we set $(g,x)(x,y) = (gx,gy)$. 

Since the momentum  map $l: X\times Y\to X$ is proper and surjective and has obviously continuous sections, its remains to show that the groupoid we have just defined is amenable.

We let $\cG = \Gamma\ltimes Y$ act partially on $X\times Y$ as follows. The domain of the action of $(g,y)$ is $X_{g^{-1}} \times \set{y}$. This time the moment map is the second projection $q$ and
$$\cG\,_s*_q (X\times Y) = \set{\big((g,y),(x,y)\big) : g\in \Gamma, x\in X_{g^{-1}}, y\in Y}.$$ The partial action of $\cG$ on $X\times Y$ is defined by $(g,y)(x,y) = (gx ,gy)$, where $x\in X_{g^{-1}}$. 

 One immediately check that the map $\big((g,y),(x,y)\big) \mapsto \big((g,x),(x,y)\big)$ is an isomorphism of groupoids from $\cG\ltimes (X\times Y)$ onto the groupoid $(\Gamma\ltimes_\beta X)\ltimes (X\times Y)$. 
 
 To end the proof,  one applies the proposition \ref{thm:RW17Bis} to the amenable groupoid $\cG$ acting partially on $X\times Y$. The corresponding semi-direct groupoid is amenable and so is  $(\Gamma\ltimes_\beta X)\ltimes (X\times Y)$. 
\end{proof}

\subsection{Amenability and amenability at infinity for \'etale groupoids}\label{subsec:aietale}

A locally compact group is amenable if and only if its action on its Alexandroff compactification is amenable. We have a similar characterization of amenability for an \'etale groupoid, for which we need the observation below. 

\begin{prop}\label{prop:sat} Let $\cG$ be an \'etale groupoid and $(Y,p)$ a left $\cG$-space. Let $U$ be the greatest open subset of $X=\cG^{(0)}$ such that the restriction of $p$ to $p^{-1}(U)$ is proper. Then $U$ is invariant under the $\cG$-action, that is, $r(\gamma)\in U$ if and only if $s(\gamma)\in U$.
\end{prop}

\begin{proof} We refer to Subsection \ref{subsec:pfc}  for the properties of $U$ used in this proof. Let $\gamma$ be such that $s(\gamma)\in U$ and let $S$ be a compact bisection which is a neighborhood of $\gamma$. There exists a compact neighborhood $K$ of $s(\gamma)$ contained into $s(S)$ such that $p^{-1}(K)$ is compact. Then $S\cdot K$ is a compact neighborhood of $r(\gamma)$ and $p^{-1}(S\cdot K) = Sp^{-1}(K)$ is compact. Therefore we have $r(\gamma)\in U$.
\end{proof}

\begin{prop}\label{prop:carac_amen} An \'etale groupoid $\cG$ is amenable if and only if the $\cG$-space $\cG_r^{+}$ is amenable.
\end{prop}
\begin{proof}  Let $U $ be the greatest open subset of $X$ such that the restriction of $r$ to $r^{-1}(U)$ is proper.
The structure of $\cG_{r}^+$ is described in Proposition \ref{prop:alex}. One of its main features is that it contains a $\cG$-invariant closed subset $F$ such that the restriction of $p=r^+$ to $F$ is a $\cG$-equivariant homeomorphism onto   $F'=X\setminus U$ (see Propositions \ref{prop:alex} and \ref{prop:sat})\footnote{When $\cG$ is an infinite discrete group $G$, then $F = G^+\setminus G$ is reduced to the point at infinity, $F' = \set{e}$, $U$ is empty and $\cG(F') = G$.}. Assume that the $\cG$-space $\cG_{r}^+$ is amenable. Let $(h_i)$ be a net of continuous positive definite functions on $\cG_{r}^+\rtimes \cG$, with compact support, that satisfies the conditions of Proposition \ref{prop:amen2}. Let us observe that 
$$F\,_p\!*_r\cG = \set{(y,\gamma) \in F\times \cG: p(y) = r(\gamma)}$$
is canonically identified to the groupoid $\cG(F')$ thanks to the map $(y,\gamma) \mapsto \gamma$.
We denote by $k_i$ the restriction of $h_i$ to $F\,_p\!*_r\cG$, and we view it as a function on $\cG(F')$. It is a net of positive definite functions on $\cG(F')$ that satisfies the conditions of Proposition \ref{prop:amen2}. So the groupoid $\cG(F')$ is amenable. On the other hand $\cG(U)= r^{-1}(U)$ is amenable since it is proper. It follows that $\cG$ is amenable (see \cite[Proposition 9.83]{Will19}).
\end{proof}

Let us now turn to amenability at infinity. Our standing assumption is that $\cG$ is second countable, and \'etale here, but its Stone-\v Cech fibrewise compactification $\beta_r \cG$ is only $\sigma$-compact in general. The groupoid $\beta_r \cG \rtimes \cG$ is $\sigma$-compact with countable $r$-fibres.  The following  technical result (a generalization of \cite[Lemma 3.5]{H00}) shows how to build a {\it second countable} amenable fibrewise compact $\cG$-space out of any amenable fibrewise compact $\cG$-space.

\begin{lem}\label{lem:scf} Let $\cG$ be a second countable \'etale groupoid that acts amenably on a  fibrewise compact fibre space $(Z,q)$ with $q$ surjective. Then   there is a second countable fibrewise compact $\cG$-space  $(Y,p)$ and a $\cG$-equivariant proper surjective continuous morphism $q_Y:Z\to Y$ such that $\cG$  also acts amenably on $(Y,p)$ and $q= p\circ q_Y$. Moreover, if there is a continuous section $\sigma$ for $q$, then there is a continuous section for $p$, namely $q_Y\circ \sigma$. 
\end{lem}

\begin{proof} Here, $Z$ is not assumed to be second countable, but it is $\sigma$-compact since $q$ is proper and $\cG^{(0)}$ is second countable. The semi-direct product groupoid $Z\rtimes \cG$ is $\sigma$-compact. 
We will use the characterization of an amenable action on a $\sigma$-compact fibre space recalled in Proposition \ref{prop:amen3} and Remark \ref{rem:nonsep}: there exists a sequence $(g_n)$ of nonnegative functions in $\cC_c(Z\rtimes\cG)$ such that 
\begin{itemize}
\item[(a)] $\int g_n(z,\gamma) \rd\lambda^{q(z)}(\gamma) \leq 1$ for every $z\in Z$;
\item[(b)] $\lim_n \int g_n(z,\gamma) \rd\lambda^{q(z)}(\gamma) = 1$ uniformly on the compact subsets of $Z$;
\item[(c)] $\lim_n \int\abs{ g_n(\gamma^{-1}z,\gamma^{-1}\gamma_1) - g_n(z,\gamma_1)} \rd\lambda^{q(z)}(\gamma_1) = 0$ uniformy on the compact subsets of $Z\rtimes \cG$.
\end{itemize}

{\it First observation.} Let $g: Z\rtimes \cG\to \C$. Then $g$ is continuous if and only if for every open bisection $S$ of $\cG$, the function $z \mapsto g(z,r_S^{-1}(q(z)))$ is continuous on $q^{-1}(r(S))$. This is immediate since $z\mapsto (z,r_S^{-1}(q(z)))$ is a homeomorphism from the open subset $q^{-1}(r(S))$ of $Z$ onto the open subset $q^{-1}(r(S))\,_q\!*_r S$ of $Z\rtimes \cG$.

{\it Construction of $Y$.} Let $\cF$ be a countable family of open bisections of $\cG$ which covers $\cG$, is stable under product and inverse and  contains a basis $\cB$ of open sets relative to the topology of $X= \cG^{(0)}$. For $S,S'\in \cF$ we denote by $g_{n,S,S'}$ the function $z \mapsto g_n(Sz, r^{-1}_{S'}(q(Sz)))$. The domain of definition $\Omega_{S,S'}$ of this function is the set of $z\in Z$ such that $q(z)\in s(S)$ and $S\cdot q(z)\in r(S')$, that is, $q^{-1}\big(S^{-1}\cdot \big(r(S)\cap r(S')\big)\big)$.  We observe that if $S$ is an open subset of $X$ then $\Omega_{S,S} = q^{-1}(S)$, and for $z\in q^{-1}(S)$ we have $g_{n,S,S}(z) = g_n(z,q(z))$.

We endow $Z$ with the weakest topology $\cT$ that contains the $\Omega_{S,S'}$ as open subsets and makes continuous the functions $g_{n,S,S'}$ on $\Omega_{S,S'}$. It is second countable but not necessarily Hausdorff. Now, on $Z$ we define the following equivalence relation:
$$z_1\sim z_2 \quad\hbox{if}\quad q(z_1) = q(z_2)\quad\hbox{and}\quad g_{n,S,S'}(z_1) = g_{n,S,S'}(z_2)$$
for all $n$ and all $S,S'$ with $z_1,z_2\in \Omega_{S,S'}$. We denote by $Y$ the quotient space endowed with the quotient topology and by $q_Y:Z\to Y$ the quotient map. Let us observe that $\cT$ is generated by subsets that are saturated with respect to this equivalence relation and therefore $q_Y$ is an open map from $(Z,\cT)$ onto $Y$.
 We denote by $\widetilde{g}_{n,S,S'}$ the continuous function on $q_Y(\Omega_{S,S'})$, deduced from $g_{n,S,S'}$ by passing to the quotient. We also introduce the  map $p$  from $Y$ onto $X= \cG^{(0)}$ such that $p(q_Y(z)) = q(z)$. We remark that $q$ is still continuous when $Z$ is equipped with the topology $\cT$, and therefore $p$ is continuous.

{\it The space $Y$ is Hausdorff, second countable and locally compact.} Let $z_1,z_2\in Z$ be such that $q_Y(z_1) \not= q_Y(z_2)$ in $Y$.
If $q(z_1)\not= q(z_2)$, then $q_Y(z_1)$ and $q_Y(z_2)$ belong to disjoint open subsets of the form $q_Y(\Omega_{S,S})$ since $\cF$ contains a basis of the topology of $X$. Now, if $q(z_1)=q(z_2)$, there exists $S,S'\in \cF$ and $n$ such that  $z_1,z_2\in \Omega_{S,S'}$ and $g_{n,S,S'}(z_1) \not= g_{n,S,S'}(z_2)$. It follows that $\widetilde{g}_{n,S,S'}(q_Y(z_1))\not=\widetilde{g}_{n,S,S'}(q_Y(z_2))$ and therefore $q_Y(z_1)$ and $q_Y(z_2)$ have disjoint neighborhoods. Therefore, $Y$ is Hausdorff. Since $(Z,\cT)$ is second countable and $q_Y$ is open, we see that $Y$ is second countable. 

{\it The map $p$ is proper and $Y$ is locally compact.} Let $K$ be a compact subset of $X$. Then $p^{-1}(K) =q_Y(q^{-1}(K))$. Since $q$ is proper, $q^{-1}(K)$ is compact for the initial topology and therefore $q_Y(q^{-1}(K))$ is compact. It follows that $Y$ is locally compact, because $p$ is a continuous proper map onto  the locally compact space $X$.

We also see that $q_Y$ is proper when $Z$ is endowed with its initial topology, since   we have
$$q_{Y}^{-1}(K)\subset q_{Y}^{-1}(p^{-1}(p(K))) = q^{-1}(p(K)),$$
for every compact subset $K$ of $Y$.
 
{\it The groupoid $\cG$ acts continuously on $(Y,p)$.}  First, if $z_1\sim z_2$, we check that for $\gamma$ such that $s(\gamma) =  q(z_1)$ then $\gamma z_1 \sim \gamma z_2$. Of course, we have $q(\gamma z_1) = r(\gamma) = q(\gamma z_2)$. If $S,S'\in \cF$ are such that $r(\gamma)\in S^{-1}\cdot (r(S)\cap r(S'))$, we have to verify that $g_{n,S,S')}(\gamma z_1) = g_{n,S,S'}(\gamma z_2)$, that is,
$$g_n(S\gamma z_1, r^{-1}_{S'}(q(S\gamma z_1))) = g_n(S\gamma z_2, r^{-1}_{S'}(q(S\gamma z_2))).$$
Let $T\in \cF$ such that $\gamma\in T$. Then $g_{n,S,S')}(\gamma z_1) = g_{n,ST,S'}(z_1)$ and similarly for $z_2$.
Since $ST\in \cF$ we get the desired equality.

Now, we set $\gamma q_Y(z) = q_Y(\gamma z)$ if $s(\gamma) = q(z) = p(q_Y(z))$. We  have to show that the map $(z,\gamma)\mapsto q_Y(\gamma z)$ is continuous from $Z_q\!*_s\cG$ into $Y$, when $Z$ is equipped with the topology $\cT$. Let $z'\in \Omega_{S,S'} $ with $S,S'\in \cF$.  For $\varepsilon >0$ we set
$$V(n,S,S', \varepsilon; z') = \set{z\in \Omega_{S,S'} : \abs{g_{n,S,S'}(z) - g_{n,S,S'} (z')}\leq \varepsilon}.$$
We note that the finite intersections of such $V(n,S,S', \varepsilon;z')$  form a basis of neighborhoods of $z'$.

We check the continuity the map $(z,\gamma)\mapsto q_Y(\gamma z)$ in $(z_0,\gamma_0)$. We consider a neighborhood of $\gamma_0 z_0$ of the form $V(n,S,S', \varepsilon; \gamma_0 z_0)$. 
Let $T\in \cF$ be an open bisection which contains $\gamma_0$. Then if $(z,\gamma)\in V(n,ST,S',\varepsilon; z_0) \,_q\!*_s T$
we have 
$$\abs{g_{n,ST,S'}(z) - g_{n,ST,S'} ( z_0)} =\abs{g_{n,S,S'}(\gamma z) - g_{n,S,S'} (\gamma_0 z_0)} \leq \varepsilon,$$
and therefore $\gamma z \in V(n,S,S', \varepsilon; \gamma_0 z_0)$.

{\it The action of $\cG$ on $Y$ is amenable.} For $(q_Y(z),\gamma) \in Y\rtimes \cG$ we set $\widetilde{g}_n(q_Y(z), \gamma) = g_n(z,\gamma)$. This is well defined. Indeed, let $z_1\sim z_2$. Let $S\in \cB$ such that $q(z_1)\in S$ and let $S'$ be an open bisection containing $\gamma$. Then
$$g_n(z_1,\gamma)= g_{n,S,S'}(z_1) = g_{n,S,S'}(z_2)= g_n(z_2,\gamma).$$
It is now straightforward to see that  $(\widetilde{g}_n)$ is a sequence in $\cC_c(Y\rtimes \cG)$ which satisfies the appropriate version of the above conditions (a), (b) and (c).
\end{proof}

\begin{prop}\label{prop:SC} An \'etale groupoid $\cG$ is strongly amenable at infinity if and only if the Stone-\v Cech fibrewise compactification $(\beta_r \cG, r_\beta)$ is an amenable $\cG$-space.
\end{prop}

\begin{proof} In one direction, let us assume that the action of $\cG$ on $\beta_r \cG$ is amenable. We note that the inclusions $\cG^{(0)}\subset \cG\subset \beta_r \cG$ provide a continuous section for $r_\beta$. The previous lemma shows that   $\cG$ is strongly amenable at infinity (in order to comply  to our standing assumption, we wish  to limit ourselves to actions on second countable locally compact spaces).   Conversely, assume that $(Y,p,\sigma)$ satisfies the conditions of the definition \ref{def:sai}. We define a continuous $\cG$-equivariant morphism $\varphi : (\cG,r)\to (Y,p)$ by
$$\varphi(\gamma) = \gamma \sigma\circ s(\gamma).$$
Then, by the proposition \ref{prop:universal1}, $\varphi$ extends in a unique way to a continuous $\cG$-equivariant morphism $\Phi$ from $(\beta_r\cG, r_\beta)$ into $(Y,p)$.
Since $\Phi$ is proper and equivariant, the map $(z,\gamma)\mapsto (\Phi(z),\gamma)$ is a proper homomorphism from $\beta_r\cG\rtimes \cG$ into $Y\rtimes\cG$. It follows from Proposition \ref{prop:locpropre} that $\beta_r\cG\rtimes \cG$ is amenable, since  $Y\rtimes\cG$ is amenable.
\end{proof}

{\subsection{A characterization of strong amenability at infinity for  \'etale  grou\-poids}\label{subsec:charai} Let $\cG$ be an \'etale groupoid. It is strongly amenable at infinity if and only if the semi-direct product groupoid $\beta_r\cG\rtimes \cG$ is amenable. This can be defined in terms of positive definite functions as recalled in Proposition \ref{prop:amen3}. We are going to provide a characterization where the obscure space $\beta_r \cG$ does not appear. To that end we introduce the dense open subspace\index{$\cG*_r\cG$}
$$\cG*_r\cG = \set{(\gamma,\gamma_1)\in \cG\times \cG : r(\gamma) = r(\gamma_1)}$$\index{$\cG*_r\cG$}
of $\beta_r\cG\rtimes \cG$. For simplicity, we will use the notation $q$ instead of $r_\beta : \beta_r\cG \to X=\cG^{(0)}$. We recall that the restriction of $q$ to $\cG$ is the range map $r$.

Given  $f\in \cC_c(\beta_r\cG\rtimes \cG)$, let us observe that the support of $f$ is contained in $q^{-1}(r(K))*_r K$ \footnote{If $A\subset \beta_r\cG$ and $B\subset \cG$, we write $A *_r B$ instead of $A \,_q\!*_r B$.} for some compact subset $K$ of $\cG$. 

\begin{lem}\label{lem:prepar} Let $\cG$ be an \'etale groupoid. Let $h$
be a bounded continuous function on $\cG *_r \cG$ with support in  $r^{-1}(r(K))*_r K$ for some compact subset $K$ of $\cG$. Then $h$ extends continuously to an  element $\tilde h$ of $\cC_c(\beta_r \cG\rtimes\, \cG)$.
\end{lem}

\begin{proof} Let $\displaystyle K \subset \cup_{i=1}^{n} S_i$ be a covering of $K$
by a finite number of open bisections, and let $\psi_i, i = 1,\dots, n,$ be  continuous nonnegative
functions on $\cG$, the support of each $\psi_i$ being compact and contained in $S_i$, such that $\displaystyle \sum_{i=1,\dots,n} \psi_i(\gamma)
= 1$ if $\gamma\in K$.  We will denote by $S_{i,r}(\gamma)$  the unique element $\gamma'$ of $S_i$ such that $r(\gamma') = r(\gamma)$ whenever $r(\gamma)\in r(S_i)$. Then $S_{i,r}$ is a continuous map from $r^{-1}\big(r(S_i)\big)$ onto $S_i$.

For $ i = 1,\dots, n$, we define a continuous function $k_i$ on $\cG$ as follows:
\begin{align*}
k_i(\gamma) &= h(\gamma, S_{i,r}(\gamma)) \sqrt{\psi_i(S_{i,r}(\gamma))}\quad\hbox{if} \quad  r(\gamma)\in r(S_i)\\
&=0 \quad\hbox{otherwise}.
\end{align*}
 
 Note that $k_i$ is a continuous bounded function on $\cG$. We have $k_i(\gamma) = 0$ when $\gamma\notin r^{-1}\big(r(\supp(\psi_i)\cap K)\big)$
and therefore $k_i$ belongs to $\cC_c(\cG, r)$ (see the notation in Section \ref{subsec:fc}). Thus $k_i$ extends to an element of $\cC_0(\beta_r \cG)$, still denoted by $k_i$,
whose support is  contained into  $q^{-1}\big(r(\supp(\psi_i)\cap K)\big)$ and therefore is compact. 

For $(z,\gamma) \in \beta_r\cG\rtimes \cG$, let us set 
$$\tilde h(z,\gamma) = \sum_{i=1}^{n} k_i(z) \sqrt{\psi_i(\gamma)}.$$
Obviously, $\tilde h$ is a continuous function on  $\beta_r \cG\rtimes\, \cG$ with compact support. Observe that $\tilde h(z,\gamma)= 0$ if $z\notin q^{-1}(r(K))$.

For $(\gamma,\gamma_1) \in \cG *_r \cG$ with $\gamma_1\in K$ we have
\begin{align*}
\tilde h(\gamma,\gamma_1) &= \sum_{i=1}^{n} k_i(\gamma)\sqrt{\psi_i(\gamma_1)}\\
& = \sum_{i=1}^{n} h(\gamma, S_{i,r}(\gamma))\sqrt{\psi_i(S_{i,r}(\gamma))}\sqrt{\psi_i(\gamma_1)}.
\end{align*}
Since $\gamma_1 = S_{i,r}(\gamma)$ for every $i$ such that $\gamma_1 \in S_i$
it follows that
$$\tilde h(\gamma,\gamma_1) = \sum_{i=1}^{n} h(\gamma,\gamma_1) \psi_i(\gamma_1) = h(\gamma,\gamma_1).$$
If $\gamma_1\notin K$ we have $\tilde h(\gamma,\gamma_1) = 0 = h(\gamma,\gamma_1)$.
\end{proof}

\begin{defn} et ${\mathcal G}$ be a locally compact groupoid. We say that a function $k: \cG *_r \cG \to \C$ is a {\it positive definite kernel}\index{positive definite kernel} if for every $x\in X$, $n \in \N$ and $\gamma_1,\dots,\gamma_n \in \cG^x$, the matrix $[k(\gamma_i,\gamma_j)]$ is nonnegative, that is
$$\sum_{i,j=1}^n \overline{\alpha_i}\alpha_j k(\gamma_i,\gamma_j) \geq 0$$
for $\alpha_1,\dots,\alpha_n \in \C$.
\end{defn}

Note that we have $k(\gamma,\gamma) \geq 0$ for all $\gamma\in \cG$ and, for $(\gamma,\gamma_1)\in  {\mathcal G}*_r {\mathcal G}$,
$$k(\gamma,\gamma_1) = \overline{k(\gamma_1,\gamma)},\quad\quad \abs{k(\gamma,\gamma_1)}^2\leq k(\gamma,\gamma)k(\gamma_1,\gamma_1).$$

\begin{defn}\label{def:tube} Let ${\mathcal G}$ be a locally compact groupoid. A {\it tube}\index{tube}
is a subset of ${\mathcal G}*_r {\mathcal G}$ whose image by the map
$(\gamma,\gamma_1) \mapsto \gamma^{-1}\gamma_1$ is relatively compact in ${\mathcal G}$. If $K$ is a compact subset of $\cG$, we denote by $T_K$ \index{$T_K$} the tube $\set{(\gamma,\gamma_1)\in  {\mathcal G}*_r {\mathcal G}:\gamma^{-1}\gamma_1\in K}$.
We denote by
$\cC_t(\cG*_r\cG)$\index{$\cC_t(\cG*_r\cG)$} the space of continuous bounded functions on $\cG*_r\cG$ with support in a tube.

We define a linear bijection from $\cC_b(\cG*_r\cG)$ onto itself by setting
$$\theta(f)(\gamma,\gamma_1) = f(\gamma^{-1},\gamma^{-1}\gamma_1).$$
Note that $\theta = \theta^{-1}$.

For $f\in \cC_c(\beta_r\cG\rtimes \cG)$, we denote by $\rho(f)$ its restriction to $\cG*_r\cG$.
\end{defn} 
\begin{thm}\label{thm:versus_tube} Let $\cG$ be an \'etale groupoid.
\begin{itemize}
\item[(i)] The map $\Theta : f\mapsto \theta(\rho(f))$ is a linear bijection from $\cC_c(\beta_r\cG\rtimes \cG)$ onto $\cC_t(\cG*_r\cG)$. 
\item[(ii)] The map $\Theta$  induces a bijection between the space of continuous positive definite functions with compact support on the groupoid $\beta_r\cG\rtimes \cG$ and the space of continuous positive definite kernels contained in $\cC_t(\cG*_r\cG)$.
\end{itemize}
\end{thm}

\begin{proof} Let us prove (i). We observe that $h\in \cC_t(\cG*_r\cG)$ if and only if there exists a compact subset $K$ of $\cG$ such that the support of $\theta(h)$ is contained in $r^{-1}(r(K))*_r K$. Let $f\in \cC_c(\beta_r\cG\rtimes \cG)$. Then its support is contained in some $q^{-1}(r(K))*_r K$
where $K$ is a compact subset $\cG$ and therefore we have $\rho(f) \in \theta(\cC_t(\cG*_r\cG))$. Moreover, $\rho$ is injective since $\cG*_r\cG$ is dense into $\beta_r\cG\rtimes \cG$. Thus $\Theta$ is injective. Its surjectivity follows from the previous lemma.

(ii) Assume that $f\in \cC_c(\beta_r\cG\rtimes \cG)$ is positive definite, that is, for every $z\in \beta_r\cG$, $n \in \N$ and $\gamma_1,\dots,\gamma_n \in \cG^{q(z)}$, the matrix $[f(\gamma_i^{-1}z,\gamma_i^{-1}\gamma_j)]$ is nonnegative. Observe that $\Theta(f)(\gamma_i,\gamma_j) = f(\gamma_i^{-1},\gamma_i^{-1}\gamma_j)$ and so, by taking $z = r(\gamma_i)\in \cG^{(0)}\subset \cG\subset \beta_r\cG$, we get the nonnegativity of $[\Theta(f)(\gamma_i,\gamma_j)]$.

Conversely, assume that the kernel $\Theta(f)$ is positive definite. Let $z\in \beta_r\cG$ and $\gamma_1,\dots,\gamma_n \in \cG^{q(z)}$ as above. Let $(z_\alpha)$ be a net in $\cG$ such that $\lim z_\alpha = z$. For each $i$ we choose an open bisection $S_i$ in $\cG$ with $\gamma_i\in S_i$ and for $\alpha$ large enough we set $\gamma_{i,\alpha} = r_{S_i}^{-1}(r(z_\alpha))$. Given $\lambda_1, \dots,\lambda_n\in \C$, we have
$$\sum \overline{\lambda}_i\lambda_j f(\gamma_{i,\alpha}^{-1} z_\alpha,  \gamma_{i,\alpha}^{-1}\gamma_{j,\alpha}) = \sum \overline{\lambda}_i\lambda_j \Theta(f)(\wt{\gamma}_{i,\alpha}, \wt{\gamma}_{j,\alpha})\geq 0,$$
 where we have set $\wt{\gamma}_{i,\alpha} = z_{\alpha}^{-1}\gamma_{i,\alpha}$. The fact that $\sum \overline{\lambda}_i\lambda_j f(\gamma_i^{-1}z,\gamma_i^{-1}\gamma_j)\geq 0$ is obtained by passing to the limit.
\end{proof}

\begin{rem} Let $f\in \cC_c(\beta_r \cG\rtimes\cG)$ be positive definite and let $h = \Theta(f)$, that is, $h(\gamma,\gamma_1) = f(\gamma^{-1},\gamma^{-1}\gamma_1)$ for $(\gamma,\gamma_1) \in \cG  *_r\cG$. Then we have $f(z,q(z)\leq 1$ for all $z\in \beta_r\cG$ if and only if $h(\gamma,\gamma)\leq 1$ for all $\gamma\in\cG$.
\end{rem}

\begin{thm}\label{prop:amen_inf} Let $\cG$ be an \'etale groupoid. The following conditions are equivalent:
\begin{itemize}
\item[(i)] $\cG$ is strongly amenable at infinity;
\item[(ii)]  there exists a net $(k_i)_{i\in I}$ of  bounded positive definite  continuous
kernels on $\cG *_r \cG$ supported in tubes such that
\begin{itemize}
\item[(a)] for every $i$, the restriction of $k_i$ to the diagonal of $\cG *_r \cG$ is uniformly bounded by $1$;
\item[(b)]  $\lim_i k_i = 1$ uniformly on tubes.
\end{itemize}
\end{itemize} 
\end{thm}

\begin{proof} By Proposition \ref{prop:amen2} (and Remark \ref{rem:nonsep} (1)), the groupoid $\beta_r\cG\rtimes \cG$ is amenable if and only if there exists a net $(h_i)_{i\in I}$ of continuous positive definite functions in $\cC_c(\beta_r\cG\rtimes \cG)$, with $h_i(z,q(z)) \leq 1$ for all $z\in \beta_r\cG$, such that 
 $\lim_i h_i = 1$ uniformly on every compact subset of $\beta_r\cG\rtimes \cG$.

Therefore, to prove the theorem we just have to check that $\lim_i h_i = 1$ uniformly on every compact subset of $\beta_r\cG\rtimes \cG$ if and only if $\lim_i k_i =1$ uniformly on tubes where we set $k_i = \Theta(h_i)$. First, given a compact subset $K$ of $\cG$, if $i$ is such that $\abs{h_i(z,\gamma) - 1} \leq \varepsilon$ on the compact set $q^{-1}(r(K))*_r K$, then we have $\abs{k_i(\gamma_1,\gamma_2) - 1}\leq \varepsilon$ whenever $\gamma_1^{-1}\gamma_2 \in K$.
Conversely,  let $K$  be a compact subset of $\cG$ and let $\Omega$ be an open relatively compact subset of $\cG$ containing $K$. Let $i$ be such that $\abs{k_i(\gamma_1,\gamma_2) - 1}\leq \varepsilon$ whenever $\gamma_1^{-1}\gamma_2 \in \Omega$. Then we have $\abs{h_i(\gamma, \gamma') -1} \leq \varepsilon$ for $(\gamma,\gamma')\in \cG*_r\Omega$. Since $\cG*_r\Omega$ is dense in $\beta_r\cG\,_q\!*_r\Omega$, we get  $\abs{h_i(z, \gamma') -1} \leq \varepsilon$  for
$(z,\gamma') \in \beta_r\cG *_r K$. To conclude, we observe that every compact subset of $\beta_r\cG\rtimes \cG$ is contained in such a subset $\beta_r\cG *_r K$.
\end{proof}

\begin{rems}\label{rem:bounded} (a) In fact, thanks to Remark \ref{rem:nonsep} (2), it is not necessary to require Condition (ii) (a). It is enough to require that each $k_i$ is bounded.

(b) For the analogous theorem when $\cG$ is a group we refer to \cite[Theorem 2]{Oza} when $\cG$ is discrete and to \cite[Theorems 3.4, 3.5]{AD02} for any locally compact group. In this case, amenability at infinity is equivalent to the fact that $\cG$ is KW-exact (see \cite[Theorem 5.6]{BCL}, \cite[Proposition 2.5]{OS20}) and even to the $C^*$-exactness of the group when it is discrete\footnote{Since there is no ambiguity in the latter case, we simply say that the discrete group is exact. These definitions of exactness are recalled in \ref{def:exact}, \ref{def:exact1} }.
\end{rems}

\begin{rem} This theorem can be used to show that strongly amenable at infinity \'etale groupoids are uniformly embeddable in continuous fields of Hilbert spaces.
\end{rem}

\newpage
\part{}
\section{\textbf{\textsc{Groupoids and $C^*$-algebras}}}\label{groupoid-alg}
\subsection{Full and reduced  $C^*$-algebras}\label{sec:full}

Let $(\cG, \lambda)$ be a locally compact groupoid with a Haar system\footnote{Recall that in the rest of this text, since we are interested in groupoid $C^*$-algebras, groupoids are  implicitely  equipped with a Haar system.} $\lambda$. We set $X = \cG^{(0)}$. The space $\cC_c(\cG)$ is an involutive algebra with respect to the following operations for $f,g\in \cC_c(\cG)$:
\begin{align*}
(f*g)(\gamma) &= \int f(\gamma_1)g(\gamma_{1}^{-1}\gamma) d\lambda^{r(\gamma)}(\gamma_1)\\
f^*(\gamma) & =\overline{f(\gamma^{-1})}. 
\end{align*}

We define a norm on $\cC_c(\cG)$ by
$$\norm{f}_I = \max \set{\sup_{x\in X} \int \abs{f(\gamma)}\rd\lambda^x(\gamma), \,\, \sup_{x\in X}\int \abs{f(\gamma^{-1})}\rd\lambda^x(\gamma)}.$$
The {\it full $C^*$-algebra $C^*(\cG)$\index{full groupoid $C^*$-algebra} of the groupoid} $(\cG,\lambda)$ is the enveloping $C^*$-algebra of the Banach $*$-algebra obtained by completion of $\cC_c(\cG)$ with respect to the norm $\norm{\cdot}_I$.

For the notion of right Hilbert   $C^*$-module $\cH$ over a $C^*$-algebra $A$ (or Hilbert $A$-module) that we use in the sequel, we refer to \cite{Lance_book}. We will denote by  $\cB_A(\cH)$ \index{$\cB_A(H)$} the $C^*$-algebra of $A$-linear adjointable maps from $\cH$ into itself.

We denote by $L^2_{\cC_0(X)}(\cG,\lambda)$\index{$L^2_{\cC_0(X)}(\cG,\lambda)$} the $C^*$-module\footnote{When $\cG$ is \'etale, we will use the notation $\ell^2_{\cC_0(X)}(\cG)$\index{$\ell^2_{\cC_0(X)}(\cG)$} rather than  $L^2_{\cC_0(X)}(\cG,\lambda)$. }  over  ${\mathcal C}_0(X)$  obtained by completion of $\cC_c(\cG)$ with respect to the $\cC_0(X)$-valued inner product
$$\scal{\xi,\eta}(x) = \int_{\cG^x} \overline{\xi(\gamma)}\eta(\gamma)\rd\lambda^x(\gamma).$$
The $\cC_0(X)$-module structure is given by
$$(\xi f)(\gamma) = \xi(\gamma)f\circ r(\gamma).$$
Let us observe that $L^2_{\cC_0(X)}(\cG,\lambda)$ is the space of continuous sections vanishing at infinity of a continuous field of Hilbert spaces with fibre $L^2(\cG^x,\lambda^x)$ at $x\in X$.

For simplicity of notation we set $\cE = L^2_{\cC_0(X)}(\cG,\lambda)$. We let $\cC_c(\cG)$ act on $\cE$ by the formula
$$(\Lambda(f)\xi)(\gamma) = \int f(\gamma^{-1}\gamma_1) \xi(\gamma_1) \rd\lambda^{r(\gamma)}(\gamma_1).$$
Then, $\Lambda$\index{$\Lambda$: regular representation of a groupoid} extends to a representation of $C^*(\cG)$ in the Hilbert $\cC_0(X)$-module $\cE$, called the {\it regular representation of} $(\cG,\lambda)$. Its range $\Lambda(C^*(\cG))\subset \cB_{\cC_0(X)}(\cE)$ is denoted by $C^*_{r}(\cG)$\index{$C^*_{r}(\cG)$} and called the {\it reduced $C^*$-algebra}\index{reduced groupoid $C^*$-algebra}\footnote{Very often, the Hilbert $\cC_0(X)$-module $L^2_{\cC_0(X)}(\cG,\lambda^{-1})$ is considered in order to define the reduced $C^*$-algebra (see for instance \cite{KS02,KS04}). We pass from this setting to ours (which we think to be more convenient for our purpose) by considering the isomorphism $U: L^2_{\cC_0(X)}(\cG,\lambda^{-1})\to L^2_{\cC_0(X)}(\cG,\lambda)$ such that $(U\xi)(\gamma) = \xi(\gamma^{-1})$.}{\it of the groupoid} $\cG$. We still denote by $\Lambda$ the canonical inclusion  of $C^*_{r}(\cG)$  in $\cB_{\cC_0(X)}\big(L^2_{\cC_0(X)}(\cG,\lambda)\big)$. Note that $\Lambda(C^*_{r}(\cG))$ acts fibrewise on the corresponding continuous field of Hilbert spaces with fibres $L^2(\cG^x,\lambda^x)$ by the formula
\begin{equation}\label{eq:rep}(\Lambda_x(f)\xi)(\gamma) = \int_{\cG^x} f(\gamma^{-1}\gamma_1) \xi(\gamma_1) \rd\lambda^x(\gamma_1)
\end{equation}
for $f\in \cC_c(\cG)$ and $\xi\in L^2(\cG^x,\lambda^x)$. Moreover, for $a\in C^*_{r}(\cG)$ we have
$\norm{\Lambda(a)}= \sup_{x\in X} \norm{\Lambda_x(a}$. The norm of $\Lambda(a)$ is denoted by $\norm{\Lambda(a)}$ (or sometimes $\norm{\Lambda(a)}_r$) \index{$\norm{\Lambda}_r$}.

\begin{ex}\label{ex:red} Let $G$ be a locally compact group with modular function $\Delta$. The usual formula for the left regular representation  $\lambda$ of $G$ is given, for $f\in \cC_c(G)$ and $\xi\in L^2(G)$,  by $(\lambda(f)\xi)(s) = \int f(st)\xi(t^{-1}) \rd(t)$. The unitary $V\in \cB(L^2(G))$ defined by $V(\xi)(t) = \xi(t^{-1})\Delta(t)^{-1/2}$ is such that $V\lambda(f)V^{-1} = \Lambda(f\Delta^{1/2})$. The bijection $\Phi: \cC_c(G) \to  \cC_c(G)$ given by $\Phi(f)(s) = f(s)\Delta(s)^{1/2}$  is such that $\Phi(f*g) = \Phi(f)*\Phi(g)$ but $\Phi(f^\star) = \Phi(f)^*$ where $f^\star(s) = \overline{f(s^{-1})}\Delta(s)^{-1}$ and $f^*(s) = \overline{f(s^{-1})}$. The closure of $\lambda(\cC_c(G))$ and $\Lambda(\cC_c(G))$ in $\cB(L^2(G))$ are the same and give the reduced $C^*$-algebra $C^*_{r}(G)$.
\end{ex}
\begin{rem}\label{rem:propRed}  We will use later on the  following properties of $C^*_{r}(\cG)$:
\begin{itemize}
\item[(a)] for every $a\in C^*_{r}(\cG)$ the function $x\mapsto \norm{\Lambda_x(a)}$ is lower semicontinuous. This follows from the fact that, $$x\mapsto \norm{\Lambda_x(a)}=\sup\abs{\scal{\xi,\Lambda(a)\eta}(x)},$$
 where $\xi,\eta$ run over the unit ball of  $L^2_{\cC_0(X)}(\cG,\lambda)$;
\item[(b)] for every $\gamma\in \cG$, the representations $\Lambda_{s(\gamma)}$ and $\Lambda_{r(\gamma)}$ are equivalent. So, for $a\in C^*_{r}(\cG)$, the function $x\mapsto \norm{\Lambda_x(a)}$ is constant on each orbit $\cG\cdot x = r(s^{-1}(x))$;
\item[(c)] if the support of $f \in \cC_c(\cG)$ is contained into a compact set $K$, then we have $\Lambda_x(f) = 0$ when $x$ is not in the saturation $[r(K)] = r\big(s^{-1}(r(K))\big)$ of $r(K)$. It follows that for every $a\in C^*_{r}(\cG)$ and every $\varepsilon >0$ there exists a compact subset $K$ of $\cG$ such that $\norm{\Lambda_x(a)} \leq \varepsilon$ if $x\notin [r(K)]$.
\end{itemize}
\end{rem}

\subsection{Uniform $C^*$-algebra}\label{subsec:Roe}
We now introduce a third $C^*$-algebra associated with a groupoid $(\cG,\lambda)$, which extends the notion of uniform Roe algebra associated with a countable discrete group \cite{Roe}. Recall that $\cC_t(\cG*_r\cG)$ denotes the space of conti\-nuous bounded functions on $\cG*_r\cG$ with support in a tube. We define on $\cC_t(\cG*_r\cG)$ the following operations:
\begin{align*}
(f*g)(\gamma,\gamma') &= \int f(\gamma,\gamma_1)g(\gamma_1,\gamma') \rd\lambda^{r(\gamma)}(\gamma_1)\\
f^*(\gamma,\gamma') &= \overline{f(\gamma',\gamma)},
\end{align*}
which make $\cC_t(\cG*_r\cG)$ a $*$-algebra.

For $f\in \cC_t(\cG*_r\cG)$ and $\xi\in \cE$, we set
$$(T(f)\xi)(\gamma) = \int f(\gamma,\gamma_1) \xi(\gamma_1) \rd\lambda^{r(\gamma)}(\gamma_1).$$
Then $T$ is a $*$-homomorphism from $\cC_t(\cG*_r\cG)$ into $\cB_{\cC_0(X)}(\cE)$.

\begin{defn} The {\it uniform} $C^*$-algebra\index{uniform $C^*$-algebra} of $\cG$ is the $C^*$-subalgebra $C_{u}^{*}(\cG)$\index{$C_{u}^{*}(\cG)$}
of $\cB_{\cC_0(X)}(\cE)$ generated by the operators $T(f)$ associated with the bounded conti\-nuous
kernels supported in tubes.
\end{defn}

\begin{rem} Our notion of Roe algebra differs from the notion introduced in \cite[Definition 3.9]{TWY} for a another purpose. Let us briefly compare the two notions. Let $\cG$ be a locally compact groupoid with Haar system and set $\cH = \cG*_r\cG$. It is a locally compact groupoid with $\cG$ as set of units. We have $r(\gamma_1,\gamma_2)= \gamma_1$, $s(\gamma_1,\gamma_2) = \gamma_2$, $(\gamma_1,\gamma_2)^{-1} = (\gamma_2,\gamma_1)$, $(\gamma_1,\gamma_2)(\gamma_2,\gamma_3) = (\gamma_1,\gamma_3)$. After having identified $\cH^\gamma$ with $\cG^{r(\gamma)}$, we see that the Haar system of $\cG$ induces a Haar system on $\cH$. We define a coarse structure on $\cH$ in the sense of \cite[definition 2.1]{TWY} by saying that an open subset $E$ of $\cH$ is an entourage if there exists an open relatively compact  subset $\Omega$ of $\cG$ such that $E$ is contained into the tube $\set{(\gamma_1,\gamma_2)\in \cH, (\gamma_1^{-1}\gamma_2) \in \Omega}$. Then $C^*_{u}(\cG)$ is the same as the Roe algebra of $\cH$ equipped with this coarse structure, as defined in \cite{TWY}.
\end{rem}

Our goal  is now to prove that the $C^*$-algebras $C_{u}^{*}(\cG)$ and  $C^{*}_{r}(\beta_r \cG\rtimes \,\cG)$ are canonically isomorphic when $\cG$ is an \'etale groupoid.  We keep the notation of Theorem \ref{thm:versus_tube}. Recall that we have set $q = r_\beta : \beta_r \cG \to X$. We will have to apply the definition of the reduced $C^*$-algebra to the groupoid $\beta_r\cG\rtimes \cG$. We will use the notation $\Lambda'$ and $\cE'$ for this groupoid to avoid any confusion with $\Lambda$ and $\cE$ that we reserve to the groupoid $\cG$. Hence  $\Lambda'$ is a representation of $\cC_c(\beta_r\cG\rtimes \cG)$ in the Hilbert $\cC_0(\beta_r\cG)$-module $\cE'$ which is the completion of $\cC_c(\beta_r\cG\rtimes \cG)$ with respect to the inner product
$$\scal{\xi,\eta}(z) = \int \overline{\xi(z,\gamma)}\eta(z,\gamma) \rd\lambda^{q(z)}(\gamma).$$
For $f \in \cC_c(\beta_r\cG\rtimes \cG)$ and $\xi\in \cC_c(\beta_r\cG\rtimes \cG)$ we have
$$\big(\Lambda'(f)\xi\big)(z,\gamma) = \int f(\gamma^{-1}z,\gamma^{-1}\gamma_1)\xi(z,\gamma_1)\rd\lambda^{r(\gamma)}(\gamma_1).$$

\begin{lem}\label{lem:versus_tube} Let $\cG$ be an \'etale groupoid.
\begin{itemize}
\item[(i)] The $*$-algebra $\cC_c(\cG)$ is canonically embedded into the $*$-algebras $\cC_c(\beta_r\cG\rtimes \cG)$ and  $\cC_t(\cG*_r\cG)$ (via $*$-homomorphisms).
\item[(ii)] These embeddings extend into embeddings of $C^*_{r}(\cG)$ into $C^{*}_{r}(\beta_r \cG\rtimes \,\cG)$ and $C^*_{u}(\cG)$.
\item[(iii)] The map $\Theta : f\mapsto \theta(\rho(f))$ (defined in \ref{def:tube}) is an isomorphism of $*$-algebras from $\cC_c(\beta_r\cG\rtimes \cG)$ onto $\cC_t(\cG*_r\cG)$ which preserves the above mentioned embeddings of $\cC_c(\cG)$.

\end{itemize}
\end{lem}

\begin{proof} 
(i) The embedding from $\cC_c(\cG)$ into $\cC_c(\beta_r\cG\rtimes \cG)$ is given by $f\mapsto f\circ \pi$ where $\pi: \beta_r\cG\rtimes \cG \to \cG$ is the second projection, which is proper.
 We embed $\cC_c(\cG)$ into  $\cC_t(\cG*_r\cG)$ by sending $f\in\cC_c(\cG)$ onto $\tilde f$ such that $\tilde f(\gamma,\gamma_1) = f(\gamma^{-1}\gamma_1)$. 

(ii) For $f\in \cC_c(\cG)$, we have $\norm{\Lambda'(f\circ \pi)} = \sup_{z\in \beta_r\cG} \norm{\Lambda_{z}'(f\circ\pi)}$ 
with
$$\big(\Lambda_{z}'(f\circ\pi)\xi\big)(z,\gamma) = \int f(\gamma^{-1}\gamma_1)\xi(z,\gamma_1)\rd\lambda^{q(z)}(\gamma_1).$$
Observe that $\xi\in \ell^2((\beta_r\cG\rtimes \cG)^z)$ can be identified to the element $\gamma\mapsto \xi(z,\gamma)$ of $\ell^2(\cG^{q(z)})$. It follows that $\norm{\Lambda_{z}'(f\circ\pi)} = \norm{\Lambda_{q(z)}(f)}$ and so $\norm{\Lambda'(f\circ \pi)} = \norm{\Lambda(f)}$. The second assertion of (ii) is immediate.

Taking into account Theorem \ref{thm:versus_tube},  straightforward computations prove (iii).
\end{proof}

\begin{thm}\label{prop:Roe} Let $\cG$ be an \'etale groupoid.
The $*$-isomorphism  $\Lambda'(f) \mapsto T(\Theta(f))$ defined on $\Lambda'(\cC_c(\beta_r \cG\rtimes \cG)$ extends to an isomorphim
from $C^{*}_{r}(\beta_r \cG\rtimes \,\cG)$ onto $C^*_{u}(\cG)$, which is the identity on $C_{r}^*(\cG)$.
\end{thm}

\begin{proof}  Let $\Phi$ be the faithful non-degenerate homomorphism from $\cC_0(\beta_r \cG)$ into
$\cB_{\cC_0(X)}(\cE)$ defined by
$$\big(\Phi(f) \xi\big)(\gamma) = f(\gamma^{-1}) \xi(\gamma).$$
The relative (or interior) tensor product
$$\cH=\cE'\otimes_{\cC_0(\beta_r\cG)} \cE,$$  is a Hilbert $\cC_0(X)$-module (see \cite{Lance_book}) whose $\cC_0(X)$-inner product is defined, for $\xi,\xi'\in \cC_c(\beta_r\cG\rtimes \cG)$ and $\eta,\eta'\in \cC_c(\cG)$, by
\begin{align*}
\scal{\xi\otimes \eta,\xi'\otimes\eta'}(x) &= \int_{\cG^x}\overline{\eta(\gamma)}\scal{\xi,\xi'}(\gamma^{-1})\eta'(\gamma)\rd\lambda^x(\gamma)\\
& =  \int_{\cG^x}\overline{\eta(\gamma)}\eta'(\gamma)\Big(\int\overline{\xi(\gamma^{-1},\gamma_1)}\xi'(\gamma^{-1},\gamma_1)\rd\lambda^{s(\gamma)}(\gamma_1)\Big)\rd\lambda^x(\gamma).
\end{align*}

We will first check that $\cH$ is isomorphic to the Hilbert $\cC_0(X)$-module 
$$\widetilde{\cH} = L^2_{\cC_0(X)}(\cG *_r \cG, \lambda \otimes \lambda)$$
 which is defined as the completion of $\cC_c(\cG*_r\cG)$ with respect to the $\cC_0(X)$-valued inner product
$$\scal{\xi,\xi'}(x) = \int \overline{\xi(\gamma,\gamma_1)} \xi'(\gamma,\gamma_1)\rd\lambda^x(\gamma)\rd\lambda^x(\gamma_1).$$
Its $\cC_0(X)$-module structure is defined by
$$(\xi f)(\gamma,\gamma_1) = \xi(\gamma,\gamma_1) f\circ r(\gamma).$$

For $\xi \in \cC_c(\beta_r\cG\rtimes \cG)$ and $\eta\in\cC_c(\cG)$ we set, for $(\gamma, \gamma_1) \in \cG *_r \cG$,
$$\big(W(\xi\otimes \eta)\big)(\gamma,\gamma_1) = \xi(\gamma_1^{-1},\gamma_1^{-1}\gamma)\eta(\gamma_1).$$
We have $W(\xi\otimes \eta)\in \cC_c(\cG*_r\cG)$. A straightforward computation shows that $W$ extends to an isomorphism of Hilbert $\cC_0(X)$-module from $\cH$ onto $\widetilde{\cH}$.

We also observe that the map from $\cC_c(\cG) \times \cC_c(\cG)$ into $\cC_c(\cG *_r \cG)$ sending $(\xi,\eta)$ to $(\gamma,\gamma_1)\in \cG *_r \cG \mapsto \xi(\gamma)\eta(\gamma_1)$ defines an isomorphism of Hilbert $\cC_0(X)$-module from  $\cE \otimes_{\cC_0(X)} \cE$ onto $\widetilde{\cH}$. We identify these two Hilbert $\cC_0(X)$-modules.

For $f\in \cC_c(\beta_r\cG\rtimes \cG)$, we claim that
$$W\circ (\Lambda'(f)\otimes \Id_\cE) = \big(T(\Theta(f))\otimes \Id_{\cE}\big)\circ W.$$
This will imply that $\Lambda'(f) \mapsto T(\Theta(f))$ is isometric and thus extends to an isomorphism of the completions. This isomorphism will be the identity on $C^*_{r}(\cG)$ by Lemma \ref{lem:versus_tube} (iii).
Let us prove our claim. Given $\xi \in \cC_c(\beta_r\cG\rtimes \cG)$ and $\eta\in\cC_c(\cG)$, we have
\begin{align*}
W\circ (\Lambda'(f)\otimes \Id_\cE)&(\xi\otimes \eta)(\gamma_1,\gamma_2) =\big(\Lambda'(f)\xi\big)(\gamma_2^{-1},\gamma_2^{-1}\gamma_1) \eta(\gamma_2)\\
&= \Big(\int f(\gamma_1^{-1},\gamma_1^{-1}\gamma_2\gamma)\xi(\gamma_2^{-1},\gamma)\rd\lambda^{s(\gamma_2)}(\gamma)\Big)\eta(\gamma_2)\\
&= \Big(\int f(\gamma_1^{-1}, \gamma_1^{-1}\gamma)\xi(\gamma_2^{-1},\gamma_2^{-1}\gamma)\rd\lambda^{r(\gamma_2)}(\gamma)\Big)\eta(\gamma_2).
\end{align*}
On the other hand, we have
\begin{align*}
\Big(\big(T(\Theta(f))\otimes \Id_\cE\big)\circ W\Big)(\xi\otimes \eta)(\gamma_1,\gamma_2)&= \int \Theta(f)(\gamma_1,\gamma)\big(W(\xi\otimes\eta)\big)(\gamma,\gamma_2)\rd\lambda^{r(\gamma_2)}(\gamma)\\
&= \int f(\gamma_1^{-1},\gamma_1^{-1}\gamma)\xi(\gamma_2^{-1},\gamma_2^{-1}\gamma)\eta(\gamma_2)\rd\lambda^{r(\gamma_2)}(\gamma),
\end{align*}
and so our claim is proved.
\end{proof}

 \subsection{Groupoid actions on $C^*$-algebras and crossed products}\label{subsect:Gaction}  In this section we extend  the notions of full and reduced groupoid $C^*$-algebras to the case of groupoid actions on any $C^*$-algebra.
 
 \subsubsection{$\cC_0(X)$-algebras and groupoid  actions on $C^*$-algebras} We begin by recalling some facts and definitions that are mostly borrowed from \cite{Blan,LeG, KS04, Rie}.
\begin{defn}\label{def:C-alg} Let $X$ be a locally compact space. A $\cC_0(X)$-{\it algebra} \index{$\cC_0(X)$-algebra} is a $C^*$-algebra $A$ equipped with a homomorphism $\rho$ from $\cC_0(X)$ into the center of the multiplier algebra of $A$, which is non-degenerate  in the sense that there exists an approximate unit $(u_\lambda)$ of $\cC_0(X)$ such that $\lim_\lambda \rho(u_\lambda) a = a$ for every $a\in A$, that is, $\lim_\lambda \rho(u_\lambda) = 1$ in the strict topology.
\end{defn}

Given $f\in \cC_0(X)$ and $a\in A$, for simplicity we will often write $f a$ instead of $\rho(f) a$. Let $U$ be an open subset of $X$ and $F= X\setminus U$. We view $\cC_0(U)$ as an ideal of $\cC_0(X)$ and we denote by $\cC_0(U) A$ the closed linear span of $\set{fa: f\in \cC_0(U), a\in A}$. It is a closed ideal of $A$ and in fact, we have $\cC_0(U) A = \set{fa: f\in \cC_0(U), a\in A}$ (see \cite[Proposition 1.8]{Blan}). We set $A_F = A/\cC_0(U) A$ and whenever $F = \set{x}$ we write $\cC_x(X)$\index{$\cC_x(X)$} instead of $\cC_0(X\setminus\set{x})$ and $A_x$ \index{$A_x$} instead of $A_{\set{x}}$. We denote by $e_x : A \to A_x$ the quotient map and for $a\in A$ we set $a(x) = e_x(a)$. Recall that the map $a \mapsto (a(x))_{x\in X}$ from $A$ into $\prod_{x\in X} A_x$ \index{$\EuScript{A}$} is injective and that $x\mapsto \norm{a(x)}$ is upper semicontinuous and vanishes at infinity (see \cite{Rie} or \cite[Proposition C.10]{Will}). 

\begin{rem}\label{rem:USC} The notion of $\cC_0(X)$-algebra is  intimately linked to the notion of  {\it upper semicontinuous $C^*$-bundle} over $X$, as defined in \cite[Definition C.16]{Will}, to which we refer.  Given such a bundle $\EuScript{A}$ we denote by $\Gamma_b(\EuScript{A})$ \index{$\Gamma_b(\EuScript{A})$} the $C^*$-algebra of bounded continuous sections of this bundle, and by $\Gamma_0(\EuScript{A})$ \index{$\Gamma_0(\EuScript{A})$} its ideal of sections that vanish at infinity.  Then $\Gamma_0(\EuScript{A})$ is in a na\-tural way a $\cC_0(X)$-algebra \cite[Proposition C.23]{Will}. Conversely \cite[Theorem C.25]{Will}, given a  $\cC_0(X)$-algebra $A$, there is a unique topology on  $\EuScript{A} = \coprod_{x\in X} A_x$ that makes $\EuScript{A}$ an upper semicontinuous $C^*$-bundle\index{upper semicontinuous $C^*$-bundle}  over $X$ and such that the map sending $a\in A$ to  $x\mapsto a(x)$ is a $\cC_0(X)$-linear isomorphism from $A$ onto $\Gamma_0(\EuScript{A})$.
 \end{rem}
 
\begin{ex}\label{ex-abel-c(x)} Let $(Y,p)$  be a fibre space over $X$, as in Subsection \ref{subsec-Action}. Then $A= \cC_0(Y)$ is a $\cC_0(X)$-algebra via the map $\rho : f\mapsto f\circ p$. We have $A_x = \cC_0(Y^x)$ where $Y^x = p^{-1}(x)$ for $x\in X$. For $f\in \cC_0(Y)$, then $f(x)\in\cC_0(Y^x)$ is the restriction of $f$ to $Y^x$ (see \cite[Example C.4]{Will}). Like in \cite{Will} we do not assume that $p$ is open. The map $x\mapsto \norm{f(x)}$ is continuous if and only if $p$ is open (see \cite[Proposition C.10]{Will}). 
\end{ex}

Let $(Y,p)$  be a fibre space over $X$ and let $A$ be a $\cC_0(X)$-algebra. Then $A\otimes \cC_0(Y)$ is a $\cC_0(X\times Y)$-algebra. We set $F = \set{(p(y),y): y\in Y}$. It is a closed subset of $X\times Y$.  We set  $p^*(A) = \big(A\otimes \cC_0(Y)\big)_F$.   With its natural structure of $\cC_0(Y)$-algebra, $p^*(A)$\index{$p^*(A)$} is called the {\it pull-back of}\index{pull back of a $\cC_0(X)$-algebra}  $A$ {\it via} $p$. Let us observe that $\big(p^*(A)\big)_y = A_{p(y)}$. 

\begin{ex}\label{fibre-product} Let $(Y_i,p_i)$, $i= 1,2$, be two fibre spaces over $X$. Then $\cC_0(Y_i)$, $i= 1,2$, are $\cC_0(X)$-algebras and $\cC_0(Y_1)\otimes \cC_0(Y_2) = \cC_0(Y_1\times Y_2)$ is a $\cC_0(X\times X)$-algebra. Let us denote by $\Delta$ the diagonal of $X\times X$. Let us set $p= p_1\times p_2: Y_1\times Y_2 \to X\times X$ and $F = p^{-1}(\Delta) = Y_1\, _{p_1}\!*_{p_2} Y_2$. We have $\cC_0((X\times X)\setminus \Delta) \cC_0(Y_1\times Y_2) = \cC_0((Y_1\times Y_2)\setminus F)$ and therefore
$$  \cC_0(Y_1\, _{p_1}\!\!*_{p_2} Y_2) = \cC_0(Y_1\times Y_2)_F.$$
Note also that $\cC_0(Y_1\, _{p_1}\!\!*_{p_2} Y_2) = p_2^{*}(\cC_0(Y_1)) = p_1^{*}(\cC_0(Y_2))$. It is a $\cC_0(Y_2)$-algebra as well as a $\cC_0(Y_1)$-algebra. We have
$$\big(p_2^{*}(\cC_0(Y_1)\big)_{y} = \cC_0(Y_1)_{p_2(y)} = \cC_0(Y_1^{p_2(y)})$$
for $y\in Y_2$ and similarly $\big(p_1^{*}(\cC_0(Y_2)\big)_{y} = \cC_0(Y_2)_{p_1(y)} = \cC_0(Y_2^{p_1(y)})$ for $y\in Y_1$.

In particular, if $(Y_1,p_1) = (X,\Id_X)$ and $(Y_2, p_2) = (Y,p)$ we have $p^*(\cC_0(X)) = \cC_0(Z)$ where $Z = \set{(p(y),y): y\in Y}$.
\end{ex}

Let $A$ and $B$ be two $\cC_0(X)$-algebras. A {\it morphism $\alpha: A \to B$ of $\cC_0(X)$-algebras} is a morphism of $C^*$-algebras which is $\cC_0(X)$-linear, that is, $\alpha(fa) = f\alpha(a)$ for $f\in \cC_0(X)$ and $a\in A$. For $x\in X$, in this case $\alpha$ factors through a morphism $\alpha_x : A_x \to B_x$ such that $\alpha_x(a(x)) = \alpha(a)(x)$.

Given a fibre space $(Y,p)$  over $X$, then 
$$\alpha\otimes \Id : A\otimes \cC_0(Y) \to B\otimes \cC_0(Y)$$
 passes to the quotient and defines a morphism of $\cC_0(Y)$-algebras\index{$p^*\alpha$} 
 $$p^*\alpha : p^*(A)\to p^*(B).$$
  We set $\alpha_{p(y)}= (p^*\alpha)_y$.  It is a homomorphism from 
 $(p^*(A))_y = A_{p(y)}$ into $(p^*(B))_y = B_{p(y)}$. 
\begin{defn}(\cite{LeG}) Let $\cG$ be a locally compact groupoid  and set $X = \cG^{(0)}$. An {\it action} of $\cG$ on a $C^*$-algebra $A$ is given by a structure of $\cC_0(X)$-algebra on $A$ and an isomorphism $\alpha : s^*(A) \to r^*(A)$ of $\cC_0(\cG)$-algebras such that for every $(\gamma_1,\gamma_2)\in \cG^{(2)}$ we have $\alpha_{\gamma_1\gamma_2} = \alpha_{\gamma_1}\alpha_{\gamma_2}$, where $\alpha_\gamma : A_{s(\gamma)} \to A_{r(\gamma)}$ is the isomorphism deduced from $\alpha$ by factorization.

When $A$ is equipped with such an action, we say that $A$ is a {\it $\cG$-$C^*$-algebra}.\index{$\cG$-$C^*$-algebra}
\end{defn}

\begin{ex}\label{ex:abelian} Let us consider the particular case where $A= \cC_0(Y)$ and let $p: Y\to X= \cG^{(0)}$ be a continuous  map. Then $A$ is a $\cC_0(X)$-algebra and we have $A_x = \cC_0(Y^x)$ where $Y^x = p^{-1}(x)$ for $x\in X$. We have 
\begin{align*}
&s^*(A) = \cC_0(Y _p\!*_{s} \cG), \quad s^*(A)_{\gamma} = \cC_0(Y^{s(\gamma)}),\\
&r^*(A) = \cC_0(Y _p\!*_{r} \cG), \quad r^*(A)_{\gamma} = \cC_0(Y^{r(\gamma)}).
\end{align*}
Since $\alpha_\gamma^{-1}$ is an isomorphism from $\cC_0(Y^{r(\gamma)})$ onto $\cC_0(Y^{s(\gamma)})$, it induces a   homeomorphism $y\mapsto \gamma y$ from $Y^{s(\gamma)}$ onto $Y^{r(\gamma)}$. This defines a structure of left $\cG$-space on $Y$ in the sense of Definition  \ref{def:Gspace}. Conversely, a left $\cG$-space $Y$ obviously  induces the action $\alpha$ of $\cG$ on $\cC_0(Y)$ defined by $\alpha(f)(y,\gamma) = f(\gamma^{-1}y,\gamma)$ if $f\in \cC_0(Y _p\!*_{s} \cG)$ and $p(y) = r(\gamma)$.

As an example, let us take $A= \cC_0(\cG)$ and $p = r$. We have $s^*(A )=\cC_0(\cG _r\!*_s\cG)$ and $r^*(A) = \cC_0(\cG _r\!*_r\cG)$. Then
$\alpha :   \cC_0(\cG _r\!*_s\cG)  \to  \cC_0(\cG _r\!*_r\cG)$ defined by $\alpha(f)(\gamma_1,\gamma_2) = f(\gamma_2^{-1}\gamma_1, \gamma_2)$ is an isomorphism of $\cC_0(\cG)$-algebra. We have  $A_{s(\gamma)} = \cC_0(\cG^{s(\gamma)})$ and  $A_{r(\gamma)} = \cC_0(\cG^{r(\gamma)})$. Moreover for $f\in \cC_0(\cG^{s(\gamma)})$ we have
$\alpha_\gamma(f)(\gamma_1) = f(\gamma^{-1}\gamma_1)$. Thus $\alpha$ is the canonical left action of $\cG$ on itself.

Similarly, the canonical left action of $\cG$ on a locally compact invariant subset $Y$ of $X=\cG^{(0)}$ corresponds to the action 
$$\alpha: s^*(A)= \cC_0(Y *_s\cG)= \cC_0(\cG(Y))  \to r^*(A)= \cC_0(Y*_r\cG) = \cC_0(\cG(Y))$$ given by $\alpha(f)(r(\gamma),\gamma) = f(s(\gamma),\gamma)$,
so $\alpha= \Id_{\cC_0(\cG(Y))}$.
\end{ex}

\subsubsection{Crossed products $($see \cite{Ren87, LeG,KS02, KS04, MW}$)$} Let $\cG$ be a locally compact groupoid with Haar system $\lambda$ and let $A$ be a $\cG$-$C^*$-algebra. We set $\cC_c(r^*(A)) = \cC_c(\cG)r^*(A)$. It is the space of the continuous sections $\gamma \mapsto f(\gamma)\in A_{r(\gamma)}$, with compact support, of the upper semicontinuous bundle of $C^*$-algebras over $\cG$ defined by the $\cC_0(\cG)$-algebra $r^*(A)$ (see Remark \ref{rem:USC})\footnote{When $A=\cC_0(\cG^{(0)})$ equipped with its canonical left $\cG$-action then $\cC_c(r^*(A)) = \cC_c(\cG)$.}. Then, $ \cC_c(\cG)r^*(A)$ is a $*$-algebra with respect to the following operations:
$$(f*g)(\gamma) = \int f(\gamma_1) \alpha_{\gamma_1}\big(g(\gamma_1^{-1}\gamma)\big) \rd\lambda^{r(\gamma)}(\gamma_1)$$
and
$$f^*(\gamma) = \alpha_\gamma\big(f(\gamma^{-1})^*\big)$$
(see \cite[Proposition 4.4]{MW} for instance). We define a norm on $\cC_c(r^*(A))$ by
$$\norm{f}_I = \max \set{\sup_{x\in X} \int\norm{f(\gamma)}\rd\lambda^x(\gamma),\quad \sup_{x\in X} \int\norm{f(\gamma^{-1})}\rd\lambda^x(\gamma)}.$$
The {\it full crossed product} $C^*(\cG,A)$\index{$C^*(\cG,A)$} is the enveloping $C^*$-algebra of the Banach $*$-algebra obtained by completion of $\cC_c(r^*(A))$ with respect to $\norm{\cdot}_I$.

For $x\in X$, we consider the Hilbert $A_x$-module $L^2_{A_x}(\cG^x,\lambda^x, A_x)$. It is the completion of the space $\cC_c(\cG^x, A_x)$ of  continuous compactly supported functions on $\cG^x$ with values in $A_x$, with respect to the $A_x$-valued inner product
$$\scal{\xi,\eta}= \int \xi(\gamma)^*\eta(\gamma)\rd\lambda^x(\gamma).$$
We observe that $L^2_{A_x}(\cG^x,\lambda^x, A_x)$ is canonically isomorphic to $L^2(\cG^x,\lambda^x)\otimes A_x$.

For $f\in \cC_c(r^*(A))$ and $\xi\in \cC_c(\cG^x, A_x)$, we set
\begin{equation}\label{eqn:red_cross}
\big(\Lambda_x(f)\xi\big)(\gamma) = \int \alpha_{\gamma}\big(f(\gamma^{-1}\gamma_1)\big)\xi(\gamma_1) \rd\lambda^{r(\gamma)}(\gamma_1).
\end{equation}
Then $\Lambda_x(f)$ extends to an element of $\cB_{A_x}\big(L^2_{A_x}(\cG^x,\lambda^x, A_x)\big)$ still denoted by $\Lambda_x(f)$. Moreover, $\Lambda_x$ is a representation of $C^*(\cG,A)$. The {\it reduced crossed product} $C^*_{r}(\cG,A)$\index{$C^*_{r}(\cG,A)$}  is the quotient of $C^*(\cG,A)$ with respect to the family of representations $(\Lambda_x)_{x\in X}$ (see \cite[Section 3.6]{KS04})\footnote{As for the definition of  $C^*_{r}(\cG)$ we have made a different choice than that in \cite{KS04}  for the construction of $C^*_{r}(\cG,A)$. However, both constructions are easily seen to be isomorphic.}.

\begin{rem}As explained in \cite{KS04} this family of representations comes from a representation $\Lambda$ of $C^*(\cG,A)$ in the Hilbert $A$-module $L^2_{A}(\cG,\lambda;A)$ \index{$L^2_{A}(\cG,\lambda; A)$}  which is defined by completion of the right $A$-module $\cC_c(r^*(A)) $ with respect to the $A$-valued inner product
$$x\mapsto \scal{\xi,\eta}(x) =\int_{\cG^x}\xi(\gamma)^*\eta(\gamma) \rd \lambda^x(\gamma) \in A_x,$$
the structure of right $A$-module being given by $(\xi a)(\gamma) = \xi(\gamma) a(r(\gamma))$ (note that we  preferred the notation $L^2_{\cC_0(X)}(\cG,\lambda)$, instead of  $L^2_{\cC_0(X)}(\cG,\lambda, \cC_0(X))$ when $A= \cC_0(X)$).
We let $\cC_c(r^*(A))$ act on $L^2_{A}(\cG,\lambda; A)$ by 
$$\big(\Lambda(f)\xi\big)(\gamma) = \int_{\cG^{r(\gamma)}} \alpha_{\gamma}\big(f(\gamma^{-1}\gamma_1)\big)\xi(\gamma_1) \rd\lambda^{r(\gamma)}(\gamma_1),$$
when $f,\xi \in \cC_c(r^*(A))$.

For $x\in X$, the map sending $\xi\otimes b \in L^2_{A}(\cG,\lambda; A)\otimes_{e_x} A_x$  onto 
$\gamma \in \cG^x\mapsto\xi(\gamma)b$
 induces an isomorphism of Hilbert $A_x$-modules from $L^2_{A}(\cG,\lambda; A)\otimes_{e_x} A_x$ onto $L^2(\cG^x,\lambda^x, A_x)$, Under this identification, $\Lambda(f)\otimes_{e_x} \Id$ becomes $\Lambda_x(f)$. It follows that $\Lambda$ extends to a homomorphism from $C^*(\cG,A)$ into $\cB_A\big(L^2_{A}(\cG,\lambda; A)\big)$ with range isomorphic to $C^*_{r}(\cG,A)$.
\end{rem}

\begin{rem}\label{rem:simple} For simplicity of notation, given $f\in \cC_c(r^*(A))$, we will often write $f$ instead of $\Lambda(f)$. We will also use the notation $A\rtimes \cG$,  $A\rtimes_r \cG$, \index{$A\rtimes \cG$, $A\rtimes_r\cG$} instead of  $C^*(\cG,A)$, $C^*_{r}(\cG,A)$, especially when $A$ is commutative.\end{rem}
\begin{ex}\label{ex:particular} Let us come back to the example \ref{ex:abelian}, that is the case of an action $\alpha: \cG \actson \cC_0(Y)$. Let us consider  the groupoid $Y\rtimes \cG$   associated with the corresponding left action of $\cG$ on $Y$. Then we have 
$$\cC_c(Y\rtimes \cG)\subset \cC_c(r^*(A))\subset  \cC_0(Y _p\!*_{r} \cG).$$
 Recall that $L^2((Y\rtimes \cG)^y)$ is canonically identified with $L^2(\cG^{p(y)},\lambda^{p(y)})$, so that for $f\in \cC_c(Y\rtimes \cG)$ and $(y,\gamma)\in Y\rtimes \cG$ the formula \eqref{eqn:red_cross} becomes
$$\big(\Lambda_{p(y)}(f)\xi\big)(y,\gamma) = \int_{\cG^{p(y)}} f(\gamma^{-1}y, \gamma^{-1}\gamma_1)\xi(y,\gamma_1) \rd \lambda^{p(y)}(\gamma_1).$$
It is the same as the formula given in \eqref{eq:rep} with $x = p(y)$. After having remarked that $A_x = 0$ if $x\notin p(Y)$, we conclude that the $C^*$-algebras $C^*_{r}(Y\rtimes\cG)$ and $C^*_{r}(\cG,\cC_0(Y))$ are canonically identified.

If we apply this observation to the case where $Y$ is an invariant locally compact subset of $\cG^{(0)}$ equipped with its natural left $\cG$-action, we see that $C^*_{r}(\cG(Y))$ and $C^*_{r}(\cG,\cC_0(Y))$ are canonically identified.
\end{ex}

 \begin{lem}\label{lem:smb} Let $(\cG,\lambda)$ be a locally compact groupoid with Haar system. Let $(Z,q_Z)$ and $(Y,q_Y)$ be two left $\cG$-spaces and let $p:Z\to Y$ be a  continuous $\cG$-equivariant morphism.
 Let $A$ be a $(Y\rtimes\cG)$-$C^*$-algebra.
 \begin{itemize}
 \item[(i)] $p^*(A)$ is in a natural way a left $(Z\rtimes \cG)$-$C^*$-algebra.
 \item[(ii)] If $p$ is proper and surjective, then $C_{r}^*(Y\rtimes \cG, A)$ embeds canonically into $C_{r}^*(Z\rtimes \cG, p^*(A))$.
 \end{itemize}
 \end{lem}
 
 \begin{proof}  Let us first observe that $Y$ and $Z$ are $(Y\rtimes \cG)$-spaces in an obvious way, that $p$ is $Y\rtimes\cG$-equivariant,  and that the groupoids $Z\rtimes \cG$ and $Z\rtimes(Y\rtimes \cG)$ are canonically isomorphic (see Lemma \ref{lem:idengr}). By  working with $Y\rtimes \cG$ instead of $\cG$, and denoting now by $\cG$ the former groupoid, we can assume that $Y= \cG^{(0)}$ and so $p=q_Z$. We denote by $\alpha: s^*(A) \to r^*(A)$ the $\cG$-action on $A$.
 
 (i) We have to show that $p^*(A)$ has a natural structure of left $(Z\rtimes \cG)$-space. Let us denote for the moment by $\cH$
the groupoid $Z\rtimes \cG$ and by $\underline s$, $\underline r$ its source and range maps respectively, in order to distinguish them from the source and range maps $s : \cG \to \cG^{(0)}$ and  $r:\cG \to \cG^{(0)}$ respectively. Let $P: \cH\to \cG$
be the groupoid homomorphism defined by $P(z,\gamma) = \gamma$. We observe that $p\circ \underline s = s\circ P$
and that that $p\circ \underline r = r\circ P$. It follows that
$$\underline{s}^*(p^*(A)) = (p\circ \underline s)^* (A )= (s\circ P)^* (A) = P^*(s^*( A))$$
and similarly for $r$ instead of $s$. Now, the structure of $\cH$-algebra on $p^* (A)$ is defined by the isomorphism 
$\beta = P^* \alpha : P^*(s^* (A)) \to P^*(r^* (A))$. Note that $\big(P^*(s^*( A))\big)_{(z,\gamma)} = A_{s(\gamma)}$, that $\big(P^*(r^* (A))\big)_{(z,\gamma)} = A_{r(\gamma)}$, and that $\beta_{(z,\gamma)} = \alpha_\gamma$. For these facts we refer to \cite{LeG}.

(ii) Note that since $p$ is proper, the map $P$ is still proper. We define a $*$-homomorphism $\Phi$ from the $*$-algebra $\cC_c(r^*(A))$ into the $*$-algebra 
$$\cC_c\big(\underline{r}^*(p^*(A))\big) = \cC_c\big(P^*(r^*(A))\big)$$
 by $f\mapsto f\circ P$. We have to show that this map is isometric. We have 
 $$\norm{f}_{C^*_{r}(\cG,A)} = \sup_{x\in \cG^{(0)}}\norm{\Lambda_x(f)}$$
 where $\Lambda_x(f)$ acts on the completion $L^2_{A_x}(\cG^x,\lambda^x, A_x)$ of  $\cC_c(\cG^x, A_x)$. Now, we observe that if $z\in Z$ is such that $p(z) = x$, then $(Z\rtimes \cG)^z$ is canonically identified with $\cG^x$ and that $(p^*(A))_{z} = A_x$. It follows that $L^2(\cG^x,\lambda^x)\otimes A_x$ is canonically identified with $L^2((Z\rtimes\cG)^z,\lambda^{p(z)})\otimes (p^*(A))_{z}$ and that  $\Lambda_{z}(\Phi(f)) = \Lambda_x(f)$. This concludes the proof since $\norm{\Phi(f)}_{C^*_{r}(Z\rtimes\cG,p^*(A))} = \sup_{z\in Z}\norm{\Lambda_{z}(\Phi(f))}$.
  \end{proof}

We will need later, in Proposition \ref{prop:equiv}, the following result\footnote{If $A$ is a $\cC_0(X)$-algebra, recall that $\EuScript{A}$ denotes the corresponding  upper semicontinuous $C^*$-bundle over $X$.}.  
\begin{lem}\label{lem:multiple} Let $A$ be a $\cG$-$C^*$-algebra with $\cG$-action $\alpha$. 
\begin{itemize}
\item[(i)] Given $a\in \Gamma_b(\EuScript{A})$ and $\xi \in \cC_c(r^*(A))$, we set $(\kappa(a)\xi)(\gamma) = \alpha_\gamma\big(a(s(\gamma))\big)\xi(\gamma)$. Then $\kappa(a)$\index{$\kappa(a)$}  extends to an operator, still denoted by $\kappa(a)$, in $\cB_A(L^2_{A}(\cG,\lambda; A))$.
\item[(ii)] $\kappa(a)$ is a two-sided multiplier of $C^*_{r}(\cG,A)$. More precisely, for $f\in \cC_c(r^*(A))$ and $a\in A$ we have
$$\kappa(a) \Lambda(f) = \Lambda(\rho(a)f), \quad\hbox{where}\quad (\rho(a)f)(\gamma) = a(r(\gamma)) f(\gamma)$$
and
$$\Lambda(f) \kappa(a) = \Lambda(f\rho'(a)), \quad\hbox{where}\quad (f\rho'(a)(\gamma) = f(\gamma) \alpha_\gamma(a(s(\gamma)).$$ 
\item[(iii)] Let $(u_k)$ be an approximate unit of $A$. Then, for $f\in C^*_{r}(\cG,A)$, we have
$$\lim_k\norm{\kappa(u_k) \Lambda(f) -\Lambda(f)}_{C^*_{r}(\cG,A)} = 0.$$
\end{itemize}
\end{lem}

\begin{proof} (i) We have for $\xi\in \cC_c(r^*(A))$,
\begin{align*}
&\norm{a}^2 \int_{\cG^x}\xi(\gamma)^*\xi(\gamma) \rd \lambda^x(\gamma) - \int_{\cG^x} \Big(\big(\kappa(a)\xi\big)(\gamma)\Big)^* \big(\kappa(a)\xi\big)(\gamma) \rd \lambda^x(\gamma)\\
& = \int_{\cG^x} \xi(\gamma)^*\Big(\norm{a}^2 - \alpha_\gamma\big(a(s(\gamma))^*a(s(\gamma))\big)\Big) \xi(\gamma) \rd \lambda^x(\gamma) \geq 0.
\end{align*}
It follows that   $\norm{\kappa(a)\xi} \leq \norm{a}\norm{\xi}$. We immediately check that the continuous extension $\kappa(a)$ is in $\cB_A(L^2_{A}(\cG,\lambda; A))$.

(ii) Given $f, \xi \in \cC_c(r^*(A))$, $a\in \Gamma_b(\EuScript{A})$, we have
\begin{align*}
\Big(\kappa(a)\Lambda(f)\xi\big)(\gamma) & = \alpha_\gamma\big(a(s(\gamma))\big)\int \alpha_\gamma(f(\gamma^{-1}\gamma_1))\xi(\gamma_1)\rd\lambda^{r(\gamma)}(\gamma_1)\\
&= \int \alpha_\gamma(\rho(a)f)(\gamma^{-1}\gamma_1))\xi(\gamma_1)\rd\lambda^{r(\gamma)}(\gamma_1)\\
&= \Lambda(\rho(a)f)\xi(\gamma).
\end{align*}
Similarly, we have
\begin{align*}
(\Lambda(f)\kappa(a) \xi)(\gamma) 
&= \int\alpha_\gamma(f(\gamma^{-1}\gamma_1))\alpha_{\gamma_1}\big(a(s(\gamma_1)\big)\xi(\gamma_1) \rd\lambda^{r(\gamma)}(\gamma_1)\\
&= \int\alpha_\gamma \Big(f(\gamma^{-1}\gamma_1)\alpha_{\gamma^{-1}\gamma_1}\big(a(s(\gamma^{-1}\gamma_1))\big)\Big)\xi(\gamma_1)\rd\lambda^{r(\gamma)}(\gamma_1)\\
& = \Lambda(f\rho'(a))\xi(\gamma).
\end{align*}
 It follows that $\kappa(a)$ is a two-sided multiplier of $C^*_{r}(\cG,A)$.

(iii) It suffices to consider the case where $f\in \cC_c(r^*(A))$. We use the notation of Remark \ref{rem:simple}. We have to show that $\lim_k\norm{u_k f - f}_{C^*_{r}(\cG,A)} = 0$ (where we write $u_k f$ instead of $\rho(u_k) f$. By \cite[Section 3.6]{KS04}, we have
$$\norm{u_k f - f}_{C^*_{r}(\cG,A)} \leq \norm{u_k f - f}_I.$$
We are going to show that $\lim_k\norm{u_k f - f}_I = 0$. Let $K$ be a compact subset of $\cG$ which contains the support of $f$ and let $c >0$ be such that 
$$\max(\sup_{x\in X} \lambda^x(K), \sup_{x\in X} \lambda_x(K)) \leq c.$$
We set $c' = \sup_{\gamma\in \cG} \norm{f(\gamma)}$. 

Recall that 
$$(u_k f -f)(\gamma) = u_k(r(\gamma)) f(\gamma) - f(\gamma) \in A_{r(\gamma)}.$$
We fix $\gamma\in K$ and choose $a\in A$ with $a(r(\gamma)) = f(\gamma)$. Given $\varepsilon' >0$, there exists $k_\gamma$ such that $\norm{u_{k_\gamma} a-a}<\varepsilon'$. In particular, we have $\norm{u_{k_\gamma}(r(\gamma)) f(\gamma) - f(\gamma)} <\varepsilon'$. Using the upper semicontinuity of the norm, we see that there is a neighborhood $V_\gamma$ of $\gamma$ such that
$$\norm{u_{k_\gamma}(r(\gamma')) f(\gamma') - f(\gamma')} \leq \varepsilon'$$
for $\gamma'\in V_\gamma$.

The compact space  $K$ is covered by  finitely many  such $V_\gamma$'s, which are denoted $V_{\gamma_i}$, $i=1,\dots n$. Let $k_0$ be such that $\norm{u_ku_{k_{\gamma_i}} - u_{k_{\gamma_i}}}\leq \varepsilon'$ for $i=1,\dots n$ and $k\geq k_0$. Take $\gamma \in K$ and choose $V_{\gamma_i}$ such that $\gamma\in V_{\gamma_i}$. We have
\begin{align*}
&\norm{u_k(r(\gamma))f(\gamma) - f(\gamma)}
\leq\norm{u_k(r(\gamma))\Big(f(\gamma)  - u_{k_{\gamma_i}}(r(\gamma))f(\gamma)\Big)} +\\  
&\quad\quad\quad\norm{\Big(u_k(r(\gamma))u_{k_{\gamma_i}}(r(\gamma)) - u_{k_{\gamma_i}}(r(\gamma))\Big)f(\gamma)} + \norm{u_{k_{\gamma_i}}(r(\gamma))f(\gamma) -f(\gamma)}\\
&\leq 2 \varepsilon' + \varepsilon' c'.
\end{align*}

It follows that $\sup_{\gamma\in \cG}\norm{(u_k f -f)(\gamma)}\leq 2 \varepsilon' + \varepsilon' c'$
 and therefore $\norm{(u_k f -f)}_I\leq (2 \varepsilon' + \varepsilon' c')c$.
\end{proof}

 \begin{rem}\label{rem:propRed1} Let us apply this result to the canonical action of $\cG$ on $\cC_0(X)$, where $X= \cG^{(0)}$. We have $C^*_{r}(\cG, \cC_0(X))= C^*_{r}(X\rtimes \cG) = C^*_{r}(\cG)$, where the identifications are canonical. Then we get an embedding $\kappa$ from $\cC_b(X)$ into  the multiplier algebra of $C^*_{r}(\cG))$  such that $\kappa(a) \Lambda(f) = \Lambda(\rho(a)f)$ and $\Lambda(f) \kappa(a) = \Lambda(f\rho'(a))$ where  $(\rho(a)f)(\gamma) = a(r(\gamma)) f(\gamma)$ and $(f\rho'(a)(\gamma) = f(\gamma) a(s(\gamma))$ for $a\in \cC_b(X)$ and $f\in \cC_c(\cG)$.
 
 The operator $\kappa(a)$ acts fibrewise on $L^2_{\cC_0(X)}(\cG,\lambda)$, and for $x\in X$ it gives an ope\-rator $\kappa_x(a)\in \cB(L^2(\cG^x,\lambda^x))$ which is a multiplier of $\Lambda_x(C^*(\cG))$ with $\kappa_x(a) \Lambda_x(f) = \Lambda_x(\rho(a)f)$ and $\Lambda_x(f) \kappa_x(a) = \Lambda_x(f\rho'(a))$. We have $\big(\kappa_x(a)\xi\big)(\gamma) = a(s(\gamma))\xi(\gamma)$ for $a\in \cC_b(X)$ and  $\xi\in L^2(\cG^x,\lambda^x))$.
 
Note that when $\cG$ is a group bundle groupoid, we have $\kappa_x(a) \Lambda_x(f) = \Lambda_x(f) \kappa_x(a)$ for all $x\in X, a\in \cC_0(X), f\in \cC_c(\cG)$. Then $\kappa$ is a non-degenerate homomorphism from $\cC_0(X)$ into the center of the multiplier algebra of $C^*_{r}(\cG)$, which makes $C^*_{r}(\cG)$ a $\cC_0(X)$-algebra.
 \end{rem}

 \section{\textbf{\textsc{Groupoid exactness}}}\label{sect:ext}
 
  \subsection{$C^*$-exactness and KW-exactness}\label{subsec:exact} In this section, we start by extending to the case of groupoids the notions of exactness for locally compact groups introduced and studied by Kirchberg and Wassermann in \cite{KW99, KW99bis}.
 \begin{defn}\label{def:exact} We say that a $C^*$-algebra $A$ is {\it exact} if for every short exact sequence 
$$0\rightarrow J \rightarrow B \rightarrow B/J \rightarrow 0$$
of $C^*$-algebras, the following sequence 
$$0\rightarrow A\otimes J \rightarrow A\otimes B \rightarrow A\otimes (B/J) \rightarrow 0$$
is exact, where $\otimes$ denotes the minimal (or spatial) tensor product.
\end{defn}

 We recall that nuclear $C^*$-algebras are exact, and that exactness is preserved by passing to quotient and   to sub-$C^*$-algebras.
 
\begin{defn} Let $\cG$ be a locally compact groupoid and $A$, $B$ two $\cG$-$C^*$-algebras. A $\cG$-{\it equivariant morphism} $\rho: A\to B$ is a morphism of $\cC_0(X)$-algebras such that $\beta\circ (s^*\rho) = (r^*\rho)\circ \alpha$,
where $\alpha$ and $\beta$ denote the $\cG$-actions on $A$ and $B$ respectively. For every $\gamma\in \cG$ we have $\beta_\gamma\circ \rho_{s(\gamma)} = \rho_{r(\gamma)}\circ \alpha_\gamma$.
\end{defn}

\begin{defn}\label{def:exact1} We say that a {\it locally compact groupoid} $\cG$ with Haar system is {\it $C^*$-exact}\index{C*-exact groupoid}  if $C^*_{r}(\cG)$ is exact. 
We say that  it is {\it exact in the sense of Kirchberg and Wassermann} (or KW{\it-exact})\index{KW-exact groupoid} if for  every $\cG$-equivariant exact sequence
$$0\to I \to A \to B \to 0$$
of $\cG$-$C^*$-algebras, the corresponding sequence
$$0 \to C^*_{r}(\cG,I) \to C^*_{r}(\cG,A) \to C^*_{r}(\cG,B)\to 0$$
of reduced crossed products is exact\footnote{We warn the reader that these two definitions differ  from  those introduced in \cite{Kir78, KW99bis}.}.
\end{defn}

Note that second countable amenable groupoids are KW-exact, because for them reduced crossed products coincide with full crossed products \cite[Theorem 3.6]{Ren91} and  because the analogous sequence for full crossed products is exact \cite[Lemma 6.3.2]{AD-R}.

\begin{rem} KW-exactness of groupoids has been studied by Lalonde in \cite{Lal15, Lal14, Lal17} under  our standing separability assumptions. In \cite[Theorem 4.9]{Lal14}, it is proved that equivalence of groupoids preserves KW-exactness. In \cite[Theorem 6.14]{Lal15}, it is proved that the reduced crossed product $C^*_{r}(\cG,A)$ is exact whenever $\cG$ is a KW-exact locally compact groupoid acting on an exact $C^*$-algebra $A$.  In \cite{Lal17}, among other permanence properties, it is proved that if $\cG$ is a KW-exact groupoid acting on a locally compact space  $Y$, then the semi-direct product groupoid $Y\rtimes \cG$ is KW-exact. Conversely, if $Y\rtimes \cG$ is KW-exact and if the moment map $p: Y\mapsto \cG^{(0)}$ is proper, then $\cG$ is KW-exact.

We have seen that amenability at infinity has similar properties (Propositions \ref{prop:inv_sd}, \ref{prop: stab-amen-infi}), and it is known that $C^*$-exactness is preserved under equivalence of groupoids \cite[Theorem 4.1]{SW12}.
\end{rem}

We will compare in Section \ref{sec:comparison} these notions of exactness, but for now we introduce another one.

\subsection{Inner exactness}\label{subsec:inex} While it is difficult to produce examples of groups that are not exact, to the author's knowledge the first example of a locally compact groupoid that is not KW-exact has been given in 1991 \cite[Remark 4.10]{Ren91}. Another class of nice examples is given by some HLS-groupoids (see \cite{HLS, Wil15} and Proposition \ref{prop:HLS} below).  

These examples are all of the same type, which we will describe now. We first recall a few facts.
Let $\cG$ be a locally  compact groupoid with a Haar system. Let $F$ be a closed invariant subset of $X= \cG^{(0)}$ and set $U = X\setminus F$. It is well-known that the inclusion
$\iota :  \cC_c(\cG(U) )\to \cC_c(\cG)$ extends to an injective homomorphism from $C^*(\cG(U))$ into $C^*(\cG)$ and from
$C^*_{r}(\cG(U))$ into $C^*_{r}(\cG)$. Similarly, the restriction map $\pi: \cC_c(\cG) \to \cC_c(\cG(F))$ extends to a surjective homomorphism from $C^*(\cG)$ onto $C^*(\cG(F))$ and from
$C^*_{r}(\cG)$ onto $C^*_{r}(\cG(F))$. Moreover the sequence 
$$0 \rightarrow C^*(\cG(U)) \rightarrow C^*(\cG) \rightarrow C^*(\cG(F)) \rightarrow 0$$
is exact. For these facts, we refer to \cite[page 102]{Ren_book}, \cite[Section 2.4]{HS}, or to \cite[Proposition 2.4.2]{Ram} for a detailed proof. On the other hand, the examples given in \cite{HLS, Wil15} are based on the existence for $\cG$ of an  invariant closed set $F$  such that
\begin{equation}\label{eq:ie}
0 \rightarrow C^*_{r}(\cG(U)) \rightarrow C^*_{r}(\cG) \rightarrow C^*_{r}(\cG(F)) \rightarrow 0
\end{equation}
 is not  exact.

But the non-exact sequence \ref{eq:ie} is nothing else than
\begin{equation*}
0 \rightarrow \cC_0(U)\rtimes_r \cG \rightarrow \cC_0(X)\rtimes_r \cG\rightarrow \cC_0(F)\rtimes_r \cG \rightarrow 0
\end{equation*}
which corresponds to  the $\cG$-equivariant exact sequence
$$0\rightarrow \cC_0(U) \rightarrow \cC_0(X) \rightarrow \cC_0(F) \rightarrow 0.$$

\begin{defn}\label{def:inamen} A locally compact groupoid with Haar system such that the sequence  \eqref{eq:ie} is exact for every closed invariant subset $F$ of $X$ called KW-{\it inner exact} or simply {\it inner exact}.\index{inner exact groupoid}
\end{defn}

 The class of inner exact groupoids is also interesting in itself and now plays a role in different contexts (see for instance \cite{BL,BEW, BCS}).  This class is quite large. It includes all locally compact groups and more generally the groupoids that act with dense orbits on their space of units. This class is stable under equivalence of groupoids \cite[Theorem 6.1]{Lal17}. Note that KW-exact groupoids are inner exact.

 \begin{rem} If $\Gamma$ is a discrete exact group acting on a locally compact space $X$, the semi-direct product groupoid $X\rtimes \Gamma$ is amenable at infinity,  hence KW-exact (see Propositions \ref{prop:inv_sd}, \ref{prop:equiv}), but it is not true that $X\rtimes \Gamma$ is always inner exact when $\Gamma$ is not exact (although $\Gamma$ is inner exact, as any group). Indeed,  let $\Gamma$ be a Gromov monster. In \cite[Theorem 7.4]{BGW} the authors construct a  second countable locally compact subset $X$ of $\beta\Gamma$, containing $\Gamma$, which is $\Gamma$-invariant with respect to the left action of $\Gamma$ on $\beta\Gamma$, such that the sequence
 \begin{equation*}
0 \rightarrow \cC_0(\Gamma)\rtimes_r \Gamma \rightarrow \cC_0(X)\rtimes_r \Gamma\rightarrow \cC_0(X\setminus \Gamma)\rtimes_r \Gamma \rightarrow 0
\end{equation*}
is not exact in the middle (even not in $K$-theory).
 \end{rem}

\subsubsection{Inner exactness for groupoid bundles}
\begin{defn}\label{def:bundle_groupoid_cont} A {\it groupoid bundle}\index{groupoid bundle} is a triple $(\cG,T,p)$ where $\cG$ is a locally compact groupoid, $T$ is a locally compact space  and $p:\cG^{(0)}\to T$ is a surjective continuous open map such that $p \circ r = p\circ s$.
\end{defn}
 
 This terminology is that of \cite[Definition 1.16]{Will19}. In \cite{LR},  this notion is called a {\it continuous field of groupoids}. Both terminologies are justified\footnote{but the second one can be misleading since, as we will see, $C^*_{r}(\cG)$ is not in general a continuous field of $C^*$-algebras with fibres $C^*_{r}(\cG_t)$, $t\in T$.} by the following observation:
$t\mapsto (p \circ r)^{-1}(t) = \cG_t$ is a bundle (or field) of groupoids. Indeed  $\cG_t$ is the reduction $\cG(p^{-1}(t))$ of $\cG$ by the invariant set $p^{-1}(t)$. In the case where $T= \cG^{(0)}$ and $p$ is the identity map, then $\cG_t$ is the isotropy group  of $\cG$ at $t$, and we say that $\cG$ is a group bundle groupoid.

Let $(\cG,T,p)$ be a  groupoid bundle.  We assume  that $\cG$ is equipped with a Haar system\footnote{This is automatic in the case of a group bundle groupoid since $s= r$ is open \cite[Lemma 1.3]{Ren91}.}. Observe that  $\cC_c(\cG)$ has a structure of $\cC_0(T)$-module by setting  $(fg)(\gamma) = f\circ p(r(\gamma)) g(\gamma)$ for $f\in \cC_0(T)$ and $g\in \cC_c(\cG)$. The map $g \mapsto fg$ extends continuously in order to turn $C^*(\cG)$ and $C^*_{r}(\cG)$ into $\cC_0(T)$-algebras. This is immediate for the reduced $C^*$-algebra\footnote{When $T = X$ and $p= \Id_X$, this is the result stated in \ref{rem:propRed1}.}. In the case of the full $C^*$-algebra, that we will not need, one uses the same arguments as in the proof of \cite[Lemma 1.13, page 59]{Ren_book} (see also \cite[Lemme 2.4.4]{Ram} for details).

We will see that $C^*_{r}(\cG)$ can be viewed as a bundle (or field) of $C^*$-algebras over $T$ in two different ways, but first, let us clarify what we mean by  this notion.

\begin{defn}\label{def:bundle_C*} A {\it  field} 
{\it of $C^*$-algebras over a locally compact space} $X$ is a triple $\cA = (A, \set{\pi_x: A \to A(x)}_{x\in X},X)$ where $A$, $A(x)$ are $C^*$-algebras, and where $\pi_x$ is a surjective homomorphism such that 
\begin{itemize}
\item[(i)] $\set{\pi_x: x\in X}$ is faithful, that is, $\norm{a} = \sup_{x\in X}\norm{\pi_x(a)}$ for every $a\in A$;
\item[(ii)] for $f\in \cC_0(X)$ and $a\in A$, there is an element $fa\in A$ such that $\pi_x(fa) = f(x)\pi_x(a)$ for $x\in X$;
\item[(iii)] the natural homomorphism from $\cC_0(X)$ into the center of the multiplier algebra of $A$ is non-degenerate.
\end{itemize}
\end{defn}

This definition is related to \cite[Definition 1.1]{KW95}. It differs from the  one defined for instance in \cite[\S C.2]{Will} (see Remark \ref{rem:USC}). 
 Note that $A$ is a $\cC_0(X)$-algebra, and that for every $a\in A$ and $\varepsilon >0$ there exists a compact subset $K$ of $X$ such that $\norm{\pi_x(a)} <\varepsilon$ if $x\notin K$. There is a canonical surjection from  $A/\cC_x(X)A$  onto  $A(x)$, which is not injective in general.
 
 \begin{defn}\label{def:bundlebis} Let $A$ be a field of $C^*$-algebras as in the previous definition.
We say that this field is {\it upper semicontinuous} (resp. {\it lower semicontinuous}, resp. {\it continuous}) if the function $x\mapsto \norm{\pi_x(a)}$ is upper semicontinuous (resp. lower semicontinuous, resp. continuous) for every $a\in A$.
\end{defn}

\begin{lem}\label{lem:usc} Let $\cA$ be a field of $C^*$-algebras on a locally compact space $X$.   Then the function $x\mapsto \norm{\pi_x(a)}$ is upper semicontinuous at $x_0$ for every $a\in A$ if and only if $\ker \pi_{x_0}= \cC_{x_0}(X)A$
\end{lem}

\begin{proof} Suppose first that $\ker \pi_{x_0}= \cC_{x_0}(X)A$.  Given $a\in A$ and $\varepsilon >0$ there exist $f\in \cC_{x_0}(X)$ and $b\in A$ such that $\norm{a-fb}\leq \norm{\pi_{x_0}(a} + \varepsilon/2$. Let $V$ be a neighborhood of $x_0$ such that $\abs{f(x)}\norm{b} \leq \varepsilon/2$ for $x\in V$. Then we have, for $x\in V$,
\begin{align*}
\norm{\pi_x(a)}&\leq \norm{\pi_x(a) - f(x)\pi_x(b)} + \abs{f(x)}\norm{\pi_x(b)}\\
&\leq \norm{a-fb} + \varepsilon/2 \leq \norm{\pi_{x_0}(a)} + \varepsilon.
\end{align*}

Conversely, assume that $x\mapsto \norm{\pi_x(a)}$ is upper semicontinuous at $x_0$ for eve\-ry $a\in A$. Let $a\in A$ be such that $\pi_{x_0}(a) = 0$. Given $\varepsilon >0$, let $K = \set{x\in X: \norm{\pi_x(a)}\geq \varepsilon}$. It is a compact subset of $X$ (since it is closed in a compact set) and $x_0\notin K$. Let $f: X\to [0,1]$ with compact support, such that $f(x_0) = 0$ and $f(x) = 1$ if $x\in K$. We have
$$\norm{a-fa} = \sup_{x\in X}\norm{\pi_x(a) - f(x)\pi_x(a)}=\sup_{x\in X\setminus K}\norm{\pi_x(a) - f(x)\pi_x(a)}\leq 2\varepsilon.$$
It follows that  $\ker \pi_{x_0}\subset \cC_{x_0}(X)A$. Obviously, we also have $ \cC_{x_0}(X)A\subset \ker \pi_{x_0}$.
\end{proof}

 \begin{rem}\label{rem:usc} Let us come back to the data of Definition \ref{def:bundle_C*}. It follows from the previous lemma that $A(x) = A/\cC_x(X)A$ for all $x\in X$ if and only if the given field of $C^*$-algebras is upper semicontinuous.
 \end{rem}
  
 Let $(\cG,T,p)$ be a groupoid bundle as above. We set $X = \cG^{(0)}$.  For $t\in T$, we set $X_t = p^{-1}(t)$ and $U_t =X\setminus X_t$. Note that $X_t$ is an invariant closed subset of $X$. Let us explain now how $C^*_{r}(\cG)$ can be viewed as a field of $C^*$-algebras over $T$ in two different ways. First, since it is a $\cC_0(T)$-algebra, we have the field 
 $$(C^*_{r}(\cG), \set{e_t: C^*_{r}(\cG) \to C^*_{r}(\cG)_t}_{t\in  T},T)$$
  where $e_t$ is the quotient map from $C^*_{r}(\cG)$ onto $C^*_{r}(\cG)_t= C^*_{r}(\cG)/\cC_t(T)\cC_{r}^*(\cG)$. Se\-cond, it is the field 
  $$(C^*_{r}(\cG), \set{\pi_t: C^*_{r}(\cG) \to C^*_{r}(\cG(X_t))}_{t\in T},T).$$ 
  The first field is upper semicontinuous by \cite[Proposition 1.2]{Rie}
  and the second is lower semicontinuous by   \cite[Th\'eor\`eme 2.4.6]{Ram} (see also \cite[Theorem 5.5]{LR}).

  \begin{prop}\label{prop:cont_field} Let $(\cG,T,p)$ be a groupoid bundle.  Let $t_0\in T$. The function $t\mapsto \norm{\pi_t(a)}$ is continuous at $t_0$ for every $a\in C^*_{r}(\cG)$ if and only if  the following sequence
 \begin{equation}\label{eq:exact1}
 0\rightarrow C^*_{r}(\cG(U_{t_0})) \rightarrow C^*_{r}(\cG) \stackrel{\pi_{t_0}}{\rightarrow} C^*_{r}(\cG(X_{t_0})) \rightarrow 0
 \end{equation}
  is exact.
  \end{prop}\
  
  \begin{proof} We have $\cC_c(\cG(U_{t_0})) = \cC_c(T\setminus \set{t_0})\cC_c(\cG)$ and by continuity we get $C^*_{r}(\cG(U_{t_0})) = \cC_{t_0}(T)C^*_{r}(\cG)$. It follows from Lemma \ref{lem:usc}  that the function $t\mapsto \norm{\pi_t(a)}$ is upper semicontinuous at $t_0$ for all $a\in C^*_{r}(\cG)$ if and only if the kernel of $\pi_{t_0}$ is  $C^*_{r}(\cG(U_{t_0}))$. Since $t\mapsto \norm{\pi_t(a)}$ is always lower semicontinuous, this proves the proposition.
 \end{proof}

\begin{cor}\label{cor:cont_field}  Let $(\cG,T,p)$ be an inner exact (in particular a KW-exact) groupoid bundle.     Then
$$(C_{r}^*(\cG), \set{\pi_t: C_{r}^*(\cG) \rightarrow C_{r}^*(\cG(X_t))}_{t\in T}, T)$$
 is a continuous field of $C^*$-algebras over $T$.
\end{cor}

\begin{rem}\label{rem:useful}  We immediately get from Proposition \ref{prop:cont_field} that the function $t\mapsto \norm{\pi_t(a)}$ is continuous at $t_0$ for every $a\in C^*_{r}(\cG)$ whenever the groupoid $\cG(X_{t_0})$ is  amenable, because the sequence \eqref{eq:exact1} is exact\footnote{This continuity result was obtained in   \cite[Corollary 2.4.7]{Ram}. This fact is no longer true if $\cG(X_{t_0})$ is only assumed to be exact.}. This follows from the general  observation below that we will use several times.
Here it is applied with $B = C^*(\cG(X_{t_0}))$ and $B_1 =  C^*_{r}(\cG(X_{t_0}))$.
\end{rem}

\begin{lem}\label{lem:WCPinexact} Let us consider the following commuting diagram
\begin{gather}
\begin{aligned}
\xymatrix{
0\ar[r] &I \ar[d]^{p_I} \ar[r]^{\iota} & A\ar[d]^{p_A} \ar[r]^{q} &B
\ar[r]\ar[d]^{p_B} &0\\
0\ar[r] & I_1 \ar[r]^{\iota_1} & A_1\ar[r]^{q_1}&  B_1
\ar[r] &0}
\end{aligned}
\label{xy-com}
\end{gather}
where 
\begin{itemize}
\item $A$ and $A_1$ are $C^*$-algebras with closed ideals $I$ and $I_1$ respectively,  $B= A/I$ and $B_1 = A_1/I_1$,
\item the top sequence is exact, the vertical arrows are surjective homomorphisms. 
\end{itemize}
\begin{itemize}
\item[(i)] If $p_I$ and $p_B$ are injective then $p_A$ is injective.
\item[(ii)]If $p_B$ is injective, then the bottom sequence is exact in $A_1$.
\item[(iii)] If the bottom sequence is exact and $p_A$ is injective then $p_B$ is injective.
\end{itemize}
\end{lem}

\begin{proof} Immediate by diagram chasing. \end{proof}

Let us consider now the case where $X= T$ and $p=\Id_T$, that is, $\cG$ is a groupoid group bundle, with fibres  $\cG(x)$, $x\in X$. Then $C^*_{r}(\cG)$ is a lower semi-continuous bundle of $C^*$-algebras with fibres $C_{r}^*(\cG(x))$. For $x\in X$, the homomorphism 
$\pi_x : $ is the representation $\Lambda_x$ defined in Section \ref{sec:full}.

\begin{prop}\label{prop:innexact} Let $\cG$ be a  group bundle groupoid on $X$. Then, the following conditions are equivalent:
\begin{itemize}
\item[(i)] $\cG$ is inner exact;
\item[(ii)] for every $x\in X$ the sequence
$$
0 \rightarrow C^*_{r}(\cG(X\setminus \set{x})) \rightarrow C^*_{r}(\cG) \rightarrow C^*_{r}(\cG(x)) \rightarrow 0
$$
is exact.
\item[(iii)] $C^*_{r}(\cG)$ is a continuous field of $C^*$-algebras over $X$ with fibres $C^*_{r}(\cG(x)) $.
\end{itemize}
\end{prop}

\begin{proof} (i) $\Rightarrow$ (ii) is obvious and   (ii) $\Rightarrow$ (iii) is a particular case of Proposition \ref{prop:cont_field}. Assume that (iii) holds true and, given a closed (invariant) subset $F$ of $X$,  let us show that the sequence 
$$
0 \rightarrow C^*_{r}(\cG(X\setminus F)) \rightarrow C^*_{r}(\cG) \rightarrow C^*_{r}(\cG(F)) \rightarrow 0
$$
is exact. 

Let $a\in C^*_{r}(\cG)$ such that $\Lambda_x(a) = 0$ for every $x\in F$. Let $\varepsilon >0$ be given. Then $K = \set{x\in X: \norm{\Lambda_x(a)} \geq \varepsilon}$ is a compact subset of $X$ with $K\cap F = \emptyset$. Take a  function $\varphi\in \cC_c(X)$, with values in $[0,1]$,  such that $\varphi(x) = 1$ for $x\in K$ and $\varphi(x) = 0$ for $x\in F$. Recall that $C^*_{r}(\cG)$ is a $\cC_0(X)$-algebra (Remark \ref{rem:propRed1}). So $\varphi a \in C^*_{r}(\cG)$  with $\Lambda_x(\varphi a) = \varphi(x)\Lambda_x(a)$ for all $x$. We have 
$$\norm{\Lambda(a)-\Lambda(\varphi a)}_r = \sup_x\norm{\Lambda_x(a)-\Lambda_x(\varphi a)} \leq \varepsilon$$
  and $\varphi a\in C^*_{r}(\cG(X\setminus F))$. Therefore $a\in C^*_{r}(\cG(X\setminus F))$.
\end{proof}

 Let $\cG$ be locally compact groupoid and $E\subset \cG^{(0)}$.  We set $[E] = r(s^{-1}(E))$ and call it the {\it saturation} of $E$. It is the smallest invariant subset of $\cG^{(0)}$ that contains $E$. When $\cG$ is no longer assumed to be a group bundle groupoid we have the following result.
 
\begin{prop}\label{prop:continn} Let $\cG$ be a locally compact groupoid. The two conditions below are equivalent:
\begin{itemize}
\item[(i)] for every $a\in C^*_{r}(\cG)$, the function $x\mapsto \norm{\Lambda_x(a)}$ is continuous;
\item[(ii)]   $\cG$ is inner exact and for every open inva\-riant subset $U$ of $X= \cG^{(0)}$,  for every compact subset $C$ of $U$, we have $\overline{[C]}\subset U$.
\end{itemize}
\end{prop}

\begin{proof}  Assume that (i) is satisfied.  Given an invariant closed subset $F$ of $X$,  let us show that the sequence 
$$0 \rightarrow C^*_{r}(\cG(X\setminus F)) \rightarrow C^*_{r}(\cG) \rightarrow C^*_{r}(\cG(F)) \rightarrow 0$$
is exact. 

Let $a\in C^*_{r}(\cG)$ such that $\Lambda_x(a) = 0$ for every $x\in F$ and let $\varepsilon >0$ be given. We choose $f\in \cC_c(\cG)$ such that $\norm{\Lambda(f)-\Lambda(a)}_r < \varepsilon/2$. We have 
$$\sup_{x\in F} \norm{\Lambda_x(f)} = \sup_{x\in F} \norm{\Lambda_x(f-a)}\leq \norm{\Lambda(f)-\Lambda(a)}_r  < \varepsilon/2.$$
We set $K = \set{x\in X: \norm{\Lambda_x(f)} \geq \varepsilon/2}$. It  is a closed subset of $X$ with $K\cap F = \emptyset$, which is invariant since the representations $\Lambda_{s(\gamma)}$ and $\Lambda_{r(\gamma)}$ are equivalent for every $\gamma\in \cG$. Take a  function $\varphi\in \cC_c(X\setminus F)\subset \cC_c(X)$, with values in $[0,1]$,  such that $\varphi(x) = 1$ for $x\in K$. Using the notation of Remark \ref{rem:propRed1}, we have $\rho(\varphi)f = \varphi\circ  r f \in \cC_c(\cG(X\setminus F))$ and $\kappa(\varphi)\Lambda(f) = \Lambda(\rho(\varphi)f)$. Given $\xi\in L^2(\cG^x,\lambda^x)$ we have, for every $\gamma\in \cG^x$,
\begin{align*}
\big(\Lambda_x(\rho(\varphi)f)\xi\big)(\gamma) &= \int_{\cG^x} (\rho(\varphi)f)(\gamma^{-1}\gamma_1)\xi(\gamma_1) \rd \lambda^x(\gamma_1) \\
&= \varphi\circ s(\gamma)\big(\Lambda_x(f)\xi\big)(\gamma).
\end{align*}
If $r(\gamma) = x \in K$ we have $s(\gamma)\in K$ and therefore $\varphi\circ s(\gamma) = 1$. It follows that $\Lambda_x(\rho(\varphi)f) = \Lambda_x(f)$ if $x\in K$. If $x\notin K$  we have 
$$\norm{\Lambda_x(\rho(\varphi)f)-\Lambda_x(f)} = \norm{\big(\kappa_x(\varphi)-1\big)\Lambda_x(f)} \leq\norm{1-\varphi}_\infty\norm{\Lambda_x(f)} \leq \varepsilon/2.$$
Thus, we have
\begin{align*}
\norm{\Lambda(\rho(\varphi)f) -\Lambda(a)}_r&\leq \norm{\Lambda(\rho(\varphi)f) -\Lambda(f)}_r + \norm{\Lambda(f)-\Lambda(a)}_r\\
& \leq \sup_{x\in X} \norm{\Lambda_x(\rho(\varphi)f-f)}+\varepsilon/2 \leq \varepsilon,
\end{align*}
with $\rho(\varphi)f \in \cC_c(\cG(X\setminus F))$. It follows that $a\in C^*_{r}(\cG(X\setminus F))$.

Now let $U$ be an open invariant subset of $X$ and let $C\subset U$ be a compact set. We choose an open  set $\Omega$ whose closure is compact and contained into $\cG(U)$, and such that $C\subset r(\Omega)$.  Let $f\in \cC_c(\cG(U))$, nonnegative, such that $f(\gamma) = 1$ if $\gamma\in \Omega$. Note that $\Lambda_x(f)= 0$ if $x\notin U$ (see Remark  \ref{rem:propRed}). Take $x\in C$ and  a nonnegative element $\xi$ in $\cC_c(\cG^x)$. We have 
$$\big(\Lambda_x(f)\xi\big)(x) = \int f(\gamma_1)\xi(\gamma_1)\rd\lambda^x(\gamma_1)\geq \int_{\Omega^x} \xi(\gamma_1)\rd\lambda^x(\gamma_1).$$
We can choose $\xi$ such that $\big(\Lambda_x(f)\xi\big)(x) >0$. The function $\gamma\in \cG^x \mapsto \big(\Lambda_x(f)\xi\big)(\gamma)$ being continuous, we see that 
$\Lambda_x(f)\xi \not=0$ and so $\norm{\Lambda_x(f)}>0$.  Since the function $x\mapsto \norm{\Lambda_x(f)}$ is continuous, there exists $\alpha >0$ such that $\norm{\Lambda_x(f)} \geq \alpha$ for all $x\in C$. The conclusion follows from the inclusions 
$$\overline{[C]} \subset \set{x\in X: \norm{\Lambda_x(f)} \geq \alpha}\subset U.$$

Let us now proceed to the proof of (ii)$\Rightarrow$ (i). It is inspired from \cite{KM}.  Suppose that there is an element $a \in C^*_{r}(\cG)$ such that the map $x\mapsto \norm{\Lambda_x(a)}$ is not continuous. We can assume that $a$ is positive because  $\norm{\Lambda_x(a^*a)}= \norm{\Lambda_x(a)}^2$.
Since the map $x\mapsto \norm{\Lambda_x(a)}$  is lower semicontinuous, it fails to be upper semicontinuous at some point $y\in X$.  Therefore, there exist $\alpha >0$ and a net $(x_i)$ converging to $y$ such that 
$$\norm{\Lambda_{x_i}(a)} \geq \norm{\Lambda_{y}(a)} + \alpha$$
 for all $i$. Replacing $a$ by $f(a)$ via the functional calculus, with 
 $$f(t) = \max \set{0, t-\norm{\Lambda_{y}(a)}},$$ we can  assume that $\norm{\Lambda_{y}(a)} = 0$.

We set $F = \set{x\in X: \norm{\Lambda_{x}(a)} = 0}$. It is an invariant subset of $X$, which is closed since $x\mapsto \norm{\Lambda_x(a)}$  is lower semicontinuous. Then 
$a$ is in the kernel of the quotient map $C^*_{r}(\cG) \to C^*_{r}(\cG(F))$. Therefore there  exists $f\in \cC_c(\cG(X\setminus F))$  with compact support $K$ such that $\norm{\Lambda(f)-\Lambda(a)}_r\leq \alpha/2$.  We have  $\norm{\Lambda_{x}(f)}= 0$ if $x\notin [r(K)]$ (see the remark \ref{rem:propRed}). It follows that $\norm{\Lambda_x(a)} \leq \alpha/2$ if $x\notin \overline{[r(K)]}$ and therefore $x_i\notin \overline{[r(K)]}$  for all $i$, in contradiction with $\lim_i x_i = y \in F$, since $X \setminus \overline{[r(K)]}$ is a neighborhood of $y$.
\end{proof}

\begin{cor}\label{cor:continn} Let $\cG$ be a locally compact groupoid such that the orbit space $\cG\setminus \cG^{(0)}$ equipped with the quotient topology is Hausdorff. Then $\cG$ is inner exact if and only if the function $x\mapsto \norm{\Lambda_x(a)}$ is continuous for every $a\in C^*_{r}(\cG)$.
\end{cor}

\begin{proof} It suffices to check that  for every compact subset $C$ of any open invariant subset $U$ of $\cG^{(0)}$ we have $\overline{[C]} \subset U$. Denoting by $q:\cG^{(0)} \to\cG\setminus \cG^{(0)}$ the quotient map, the set $q(C)$ is compact, hence closed  since $\cG\setminus \cG^{(0)}$ is Hausdorff, and therefore $[C] = q^{-1}(q(C))$ is closed, and of course contained into $U$.
\end{proof}

\begin{ex}  The fact that $\cG\setminus \cG^{(0)}$ is Hausdorff  is obviously satisfied  by group bundle groupoids. It also applies to any proper groupoid $\cG$ (see \cite[Proposition 2.1.12]{AD-R}). One can also show  directly that the saturation of every compact subset $C$ of $\cG^{(0)}$ is closed in this case. Indeed let $x = \lim_i x_i$ with $x_i\in [C]$ for all $i$. We may assume that the $x_i$ are in a compact subset $C_1$ of $X$. Let $\gamma_i$ be such that $r(\gamma_i) = x_i$ and $s(\gamma_i) \in C$. Since $\cG$ is proper, the net $(\gamma_i)$ belongs to a compact subset of $\cG$ and we may assume that this net converges to some $\gamma$. We have $r(\gamma) = x$ and $s(\gamma)\in C$ and therefore $x\in [C]$.

On the other hand this condition is not satisfied in the following situation: let $G\actson X$ be an action of a locally compact group on a locally compact space, and  assume the existence of  an invariant open and not closed subset $U$ of $X$  on which the action of $G$ is minimal. In particular, if $f$ is a nonzero element of $\cC_c(\cG(U))$, the function $x\mapsto  \norm{\Lambda_x(f)}$ is  nonzero and constant on $U$ whereas $ \norm{\Lambda_x(f)} = 0$ if $x\in X\setminus U$. However the groupoid $\cG = X\rtimes G$ is inner exact when $G$ is amenable,  since $\cG$ is amenable.

For instance one can take for $G$ the multiplicative group $\R_+^{*}$ of positive real numbers acting on $\R$ by multiplication and $U = \R_+^{*}$.
\end{ex}

\subsubsection{The case of HLS-groupoids}  Let us recall the notation of  \ref{ex:HLS}: if $\cG$ denotes the HLS-groupoid associated with an approximated group $(\Gamma, (N_k))$, we have $\cG(k) = \Gamma/N_k$ if $k\in \N$ and $\cG(\infty) = \Gamma$. A basic result of \cite{HLS} is that the sequence
\begin{equation}\label{eq:exactHLS}
0 \longrightarrow C^*_{r}(\cG(\N))\longrightarrow C^*_{r}(\cG) \longrightarrow C^*_{r}(\cG(\infty)) \longrightarrow 0
\end{equation}
is not exact whenever  $\Gamma$ has Kazdhan's property (T) (it is not even exact in $K$-theory!). As an example we can take the exact group $SL(3,\Z)$.

For these HLS-groupoids, the inner exactness is a very strong condition which holds if and only if $\Gamma$ is amenable.

\begin{prop}\label{prop:HLS} Let us keep the above notation. We assume that $\Gamma$ is finitely generated. Then the following conditions are equivalent:
\begin{itemize}
\item[(1)] $\Gamma$ is amenable;
\item[(2)]  $\cG$ is  inner exact;
\item[(3)] $C^*_{r}(\cG)$ is a continuous field of  $C^*$-algebras with fibres $C^*_{r}(\cG(x))$;
\end{itemize}
\end{prop}

\begin{proof} If $\Gamma$ is amenable, then the groupoid $\cG$ is amenable (see \cite[Examples 5.1.3 (1)]{AD-R}), and therefore inner exact. That (2) implies (3) is proven in 
Corollary \ref{cor:cont_field}. 
Let us prove that (3) $\Ra$ (1). Assume by contradiction that $\Gamma$ is not amenable. We fix a symmetric probability measure $\mu$ on $\Gamma$ with a finite support that generates $\Gamma$ and we choose $n_0$ such that the restriction of $q_n$ to the support of $\mu$ is injective for $n\geq n_0$. We define $a\in\cC_c(\cG)\subset C^*_{r}(\cG)$ such that $a(\gamma) = 0$ except for $\gamma = (n,q_n(s))$ with $n\geq n_0$ and $s\in \supp(\mu)$ where $a(\gamma) = \mu(s)$. Then $\pi_n(a) = 0$ if $n<n_0$ and $\pi_n(a) = \lambda_{\Gamma_n}(\mu) \in C_{r}^*(\Gamma_n)= C^*_{r}(\cG(n))$ if $n\geq n_0$, where $\lambda_{\Gamma_n}$ is the quasi-regular representation of $\Gamma$ in $\ell^2(\Gamma_n)$. By Kesten's result \cite{Kes,Kes1} on the spectral radius relative to symmetric random walks, we have $\norm{\lambda_{\Gamma_n}(\mu)}_{C_{r}^*(\Gamma_n)} =1$ for $n_0\leq n\in \N $ and $\norm{\lambda_{\Gamma_\infty}(\mu)}_{C_{r}^*(\Gamma_\infty)} < 1$ since $\Gamma$ is not amenable. It follows that $C^*_{r}(\cG)$ is not a continuous field of $C^*$-algebras with fibres $C^*_{r}(\cG(n))$ on $\N^+$, a contradiction.
\end{proof}
 In this proposition it is only for (3) $\Rightarrow$ (1)  that we need to assume that $\Gamma$ is finitely generated. In fact the implication (2) $\Rightarrow$ (1) follows from Lemma \ref{lem:closed WCP}.

The following result shows  that for HLS-groupoids $C^*$-exactness is also a very strong condition.

\begin{prop}\label{prop:HLSbis} Let $\cG$ be the HLS-groupoid associated with an approximated group $(\Gamma, (N_k))$. Then $\Gamma$ is amenable if and only if $C^*_{r}(\cG)$ is exact.
\end{prop}
\begin{proof} We refer to \cite[Lemma 3.2]{Wil15} for the proof that that the exactness of $C^*_{r}(\cG)$ implies that $\Gamma$ is amenable. In turn, this last property implies that $\cG$ is amenable and therefore $C^*_{r}(\cG)$ is exact since it is nuclear.
\end{proof}

\begin{rem}\label{rem:HLSbis} Let $\cG$ be the HLS-groupoid associated with any approximated group $(\Gamma, (N_k))$, where $\Gamma$ is  a residually finite exact non-amenable group, for instance the free group $\F_2$. It is a group bundle  groupoid,  of which all fibers are exact discrete groups, although it is not $C^*$-exact. 
\end{rem}

\section{\textbf{\textsc{Amenability at infinity and exactness}}}\label{sec:comparison}

\subsection{Relations between various forms of exactness}\label{subsec:comparison} The proofs of the propositions \ref{prop:exact} and \ref{prop:equiv} are adapted from the corresponding ones in case of groups (see \cite[Theorem 7.2]{AD02} for instance), 
\begin{prop}\label{prop:exact} Every KW-exact groupoid is $C^*$-exact.
\end{prop}
\begin{proof}
 Let $\cG$ be a locally compact KW-exact groupoid. We set $X = \cG^{(0)}$. Let
$$0 \to I \to A \to B \to 0$$
be an exact sequence of $C^*$-algebras.
Then
$$0 \to I\otimes \cC_0(X) \to A\otimes \cC_0(X) \to B\otimes \cC_0(X) \to 0$$
is a $\cG$-equivariant exact sequence of $\cG$-$C^*$-algebras, where the $\cG$-actions are induced by the left action of $\cG$ on $X$. Then
$$0 \to C_{r}^*(\cG, I\otimes \cC_0(X))  \to C_{r}^*(\cG, A\otimes \cC_0(X))
 \to C_{r}^*(\cG, B\otimes \cC_0(X)) \to 0$$
is an exact sequence and  we have just to observe that it coincides with
the sequence
$$0 \to I \otimes C_{r}^*(\cG) \to A \otimes C_{r}^*(\cG) \to B \otimes C_{r}^*(\cG)
\to 0.$$
\end{proof}

\begin{prop}\label{prop:equiv} Let $\cG$ be a locally compact groupoid with Haar system which acts amenably on a second countable fibrewise compact fibre space. Then $\cG$ is $KW$-exact.
\end{prop}

\begin{proof} Let $Y$ be an amenable second countable $\cG$-space with $p : Y \to X= \cG^{(0)}$ proper and surjective, 
and let 
$$0 \to I \to A \to B \to 0$$
be an equivariant exact sequence  of $\cG$-$C^*$-algebras.
Then
$$0 \to p^*(I )\to p^* (A) \to p^*( B) \to 0$$
is an equivariant exact sequence of $(Y\rtimes \cG)$-$C^*$-algebras. 
Taking the reduced crossed products we obtain the commutative diagram

$$\begin{array}{ccccccccc}
0&\rightarrow&C_{r}^*(\cG,I)
&\rightarrow&C_{r}^*(\cG,A)&\rightarrow&
C_{r}^*(\cG,B)&\rightarrow&0\\
&&&&&&&&\\
&&{ i_I}\downarrow{}&&{i_A}\downarrow{}&&{i_B}\downarrow{}&&\\
&&&&&&&&\\
0&\rightarrow&C_{r}^*(Y\rtimes \cG,p^* (I))
&\rightarrow&C_{r}^*(Y\rtimes \cG,p^*( A))&\rightarrow&
C_{r}^*(Y\rtimes \cG,p^* (B))&\rightarrow&0
\end{array}$$

The vertical arrows were introduced in Lemma \ref{lem:smb} and shown to be injective, due to the fact that $p$ is proper.  Since $Y\rtimes \cG$ is second countable and amenable the reduced crossed products of the second line
are also full crossed products \cite[Proposition 6.1.8]{AD-R}, and therefore the second line is exact (see \cite[Lemma 6.3.2]{AD-R}).

Let us show that the first line is exact in the middle. Assume that $a \in C_{r}^*(\cG,A)$ is sent onto $0\in C_{r}^*(\cG,B)$. Then $i_A(a)$  belongs to $C_{r}^*(Y\rtimes \cG,p^* (I))$ due to the exactness of the second line. 

We now prove that this forces $a$ to belong to $C_{r}^*(\cG,I)$.  First we introduce some notation. We denote by $\EuScript I$ the upper semicontinuous $C^*$-bundle over $X=\cG^{(0)}$ associated with the $\cC_0(X)$-algebra $I$ (see Remark \ref{rem:USC}). If $u\in I= \Gamma_0(\EuScript I)$ we denote by $\tilde u\in p^*(I)$ the section $y\mapsto u(p(y)) \in I_{p(y)}$. Then, with the notation of Lemme \ref{lem:multiple}, $\kappa(\tilde u)$ is a two-sided multiplier of $C_{r}^*(Y\rtimes \cG,p^* (I))$.

For $f \in \cC_c(r^*(A))\subset C^*_{r}(\cG,A)$ (that we write  $f: \gamma \mapsto f(\gamma) \in A_{r(\gamma)}$) the map $\gamma\mapsto \big(\rho(u)f\big)(\gamma) = u(r(\gamma))f(\gamma) \in I_{r(\gamma)}$ is in $\cC_c(r^*(I)) \subset C^*_{r}(\cG,I)$.  As in Lemma \ref{lem:multiple}, we see that $f\mapsto \rho(u)f$ extends to a map, denoted by $\kappa_{A,I}(u)$, from $C^*_{r}(\cG,A)$ into $C^*_{r}(\cG,I)$. Similarly  we define a map $\kappa_{p^*(A),p*(I)}(u): C_{r}^*(Y\rtimes \cG,p^* (A))\to C_{r}^*(Y\rtimes \cG,p^* (I))$, which sends every $f\in \cC_c(\underline{r}^*(p^*(A))\subset C_{r}^*(Y\rtimes \cG,p^* (A))$ to $(y,\gamma) \mapsto u(r(\gamma))f(y,\gamma)$. 

We check that the restriction of $\kappa_{p^*(A),p*(I)}(u)$ to $C_{r}^*(Y\rtimes \cG,p^* (I))$ is $\kappa(\tilde u)$ and that for $b\in C_{r}^*(\cG,A)$  we have
$$\kappa_{p^*(A),p^*(I)}(u)i_A(b) = i_A(\kappa_{A,I}(u) b) = i_I(\kappa_{A,I}(u) b).$$
 
Next, let $(u_j)$ be a bounded approximate
unit of $I$. Then $(\widetilde{u_j})$ is  a bounded approximate unit of $p^*(I)$. Using Lemma \ref{lem:multiple} we see that  $i_A(a) = \lim_j \kappa(\tilde{u_j})i_A(a)$ since $i_A(a) \in C_{r}^*(Y\rtimes \cG,p^* I)$. It follows that we have
$$ i_A(a) = \hbox{lim}_j\kappa_{p^*(A),p^*(I)}(u_j)i_A(a))= \hbox{lim}_j i_I(\kappa_{A,I}(u_j) a)\in i_I\big(C_{r}^*(\cG,I)\big),$$
 and we conclude that $a\in C^*_{r}(\cG,I)$ since $i_A$ is injective.
\end{proof}

 We denote by $\mathscr {E}$ (resp. $\mathscr {E}'$) the family of locally compact groupoids for which $C^*$-exactness (resp. KW-exactness) implies amenability at infinity.

\begin{lem}\label{lem:C*exactAmen} The families $\mathscr {E}$ and  $\mathscr {E}'$ are preserved under equivalence of groupoids.
\end{lem}
\begin{proof} 

By \cite[Theorem 4.1]{SW12}, we know that if two locally compact groupoids are equivalent, then their reduced $C^*$-algebras are Morita equivalent.  Recall that two $C^*$-algebras $A$, $B$ are Morita equivalent if there exists an $A$-$B$ imprimitivity bimodule, and that  the Brown-Green-Rieffel theorem states that two separable $C^*$-algebras $A$ and $B$ are Morita-equivalent if and only if $A\otimes \cK$ and $B\otimes \cK$ are isomorphic, where $\cK$ is the $C^*$-algebra of compact operators on a separable Hilbert space \cite{BGR}. In particular, in this case, it is obvious that exactness is preserved under Morita equivalence. The lemma follows from this observation as well as from the facts that KW-exactness and amenability at infinity are preserved under equivalence of groupoids (see Theorem \cite[Theorem 4.9]{Lal14} and Proposition \ref{prop: stab-amen-infi} respectively). 
\end{proof}

\subsection{The case of \'etale groupoids}\label{subsec:compar-etale}
One strategy in order to show that a $C^*$-algebra is exact is to embed it into a nuclear algebra. One of our goals in this section is to show that if $\cG$ is a second countable,\footnote{We insist on our standing assumption since it is crucial here.} amenable at infinity, \'etale groupoid, then $C^{*}_r(\beta_r\cG\rtimes \cG)$ (or $C^*_{u}(\cG)$) is nuclear. This will imply that $C^*_{r}(\cG)$ is exact since it is contained into $C^{*}_r(\beta_r\cG\rtimes \cG)$. It is known  that if $\cH$ is a second countable  amenable groupoid, then $C^*_{r}(\cH)$ is nuclear and $C^*_{r}(\cH)= C^*(\cH)$ (see \cite[Corollary 6.2.14, Proposition 6.1.8]{AD-R}. Unfortunately, even if $\cG$ is second countable, the groupoid $\cH = \beta_r\cG\rtimes \cG$ is not second countable. 

Our  preliminary task   is to show that the results just mentioned under separability assumptions remain true in the following context.

\begin{prop}\label{prop:nonsep_nuc} Let $\cG$ be an \'etale second countable locally compact  groupoid with Haar system and let $(Z,q)$ be a fibrewise compact amenable $\cG$-space (not assumed to be second countable) with $q : Z\to \cG^{(0)}$ surjective. Then $C^*_{r}(Z\rtimes \cG)$  is nuclear and $C^*_{r}(Z\rtimes \cG) = C^*(Z\rtimes \cG)$. \end{prop}

\begin{proof}  As already said, the only difficulty is that we have no separability assumption on $Z$. We denote by $\cF$ the set of finite subsets of $\cC_0(Z)$, ordered by inclusion. We will construct a family $(A_F)_{F\in \cF}$ of nuclear $C^*$-subalgebras of $C_{r}^*(Z\rtimes \cG)$ such that $A_{F_1} \subset A_{F_2}$ when $F_1\subset F_2$ and $C_{r}^*(Z\rtimes \cG) = \overline{\cup_{F\in \cF} A_F}$. Then $C_{r}^*(Z\rtimes \cG)$ will be nuclear by \cite{Lance73}.

{\it Construction of $A_F$}. We choose a second countable fibrewise compact amenable $\cG$-space $(Y,p)$ and $q_Y:Z\to Y$ as in Lemma \ref{lem:scf}. Recall that $p$ and $q_Y$ are surjective and proper. We view $Y$ and $Z$ as $(Y\rtimes\cG$)-spaces in an obvious way (see Lemma \ref{lem:idengr}). We denote by $B_F$ the smallest $C^*$-subalgebra of $\cC_0(Z)$ which contains  $q_{Y}^*\cC_0(Y) = \set{f\circ q_Y: f \in \cC_0(Y)}$ and $F$, and is stable under convolution by the elements of $\cC_c(Y\rtimes\cG)$. It is a sepa\-rable abelian $C^*$-algebra and therefore its spectrum $Y_F$ is second countable. The inclusions $B_F\subset \cC_0(Z)$ and $q^*_{Y} : \cC_0(Y) \to B_F$ are essential since $q_Y$ is proper. Let $q_F: Z \to Y_F$ and $p_F: Y_F\to Y$ be the continuous surjective maps corresponding to these inclusions.  Note that $p_F\circ q_F = q_Y$ is a proper   map and therefore $p_F$ and $q_F$ are proper. Since $B_F$ is stable under convolution by the elements of $\cC_c(Y\rtimes\cG)$, using Proposition \ref{prop:smaller} we see that $(Y_F,  p_F)$ has a unique structure of $(Y\rtimes\cG)$-space (or equivalently of $\cG$-space) which makes $q_F$ equivariant. Since the groupoids $Y_F\rtimes\cG$ and $Y_F\rtimes (Y\rtimes \cG)$ are isomorphic (see again Lemma \ref{lem:idengr}) and since $Y\rtimes\cG$ is amenable, by Proposition \ref{cor:locpropre}, the $\cG$-space $Y_F$ is amenable and so, by \cite[Corollary 5.2.14]{AD-R}, the $C^*$-algebra $C_{r}^*(Y_F\rtimes \cG)$ is nuclear. 

 Since $q_F$ is proper,  $A_F = C_{r}^*(Y_F\rtimes \cG)$ is canonically embedded in $C_{r}^*(Z\rtimes \cG)$ by Lemma \ref{lem:smb}. For $f\in \cC_c(Y_F\rtimes \cG)$  the embedding is given by the surjection $(z,\gamma) \mapsto (q_F(z),\gamma)$. Note that if $F_1\subset F_2$ we have  $A_{F_1} \subset A_{F_2}$.

{\it Proof of $C_{r}^*(Z\rtimes \cG) = \overline{\cup_{F\in \cF} A_F}$}. It suffices to show that every $f\in \cC_c(Z\rtimes \cG)$ belongs to $\cup_{F\in \cF} A_F$. There exists a compact subset $K$ of $\cG$ such that the support of $f$ is contained in $q^{-1}(r(K))*K$. Using a finite covering of $K$ by open bisections and a corresponding partition of units, it suffices to consider the case where $K$ is contained in an open bisection $S$. For $z\in Z$, we set $\wt{f}(z) =f(z, r_{S}^{-1}(q(z))$ if $z\in q^{-1}(r(S))$ and $\wt f(z) = 0$ otherwise. Then $\wt f\in \cC_c(Z)$ and therefore $\wt f = g\circ q_F$ where $g$ belongs to some $\cC_c(Y_F)$. Let $\varphi \in \cC_c(\cG)$, with support contained in $S$ and equals to $1$ on $K$. Then for $(z,\gamma)\in Z\rtimes \cG$, we have $f(z,\gamma) = g\circ q_F(z)\varphi(\gamma)$ where $(y,\gamma) \mapsto  g(y)\varphi(\gamma)$ belongs to $\cC_c(Y_F\rtimes \cG)$. It follows that $f\in A_F$.

Finally, we similarly show that $C^*(Z\rtimes \cG) = \overline{\cup_{F\in \cF}C^*(Y_F\rtimes \cG)}$ and since 
$$C_{r}^*(Y_F\rtimes \cG)  = C^*(Y_F\rtimes \cG)$$
 we see that
$C^*(Z\rtimes \cG)= C^*_{r}(Z\rtimes \cG)$.
\end{proof}

\begin{thm}\label{cor:ai_ex} Let $\cG$ be an \'etale groupoid and consider the following conditions:
\begin{itemize}
\item[(1)] $\cG$ is strongly amenable to infinity. 
\item[(2)] $\cG$ is  amenable to infinity. 
\item[(3)] $C^*_{r}(\beta_r \cG\rtimes \cG)$ is nuclear.
\item[(4)] $C^*_{r}(\beta_r\cG \rtimes \cG)$ is exact.
\item[(5)] $\cG$ is KW-exact;
\item[(6)] $C^*_{r}(\cG)$ is exact.
\end{itemize}
Then 
 \vspace{-5mm}\hspace{-3mm}$$\xymatrix{(1) \ar@{=>}[r] \ar@{=>}[d] &  (2)\ar@{=>}[r] &(5)\ar@{=>}[d]\\
(3) \ar@{=>}[r] & (4) \ar@{=>}[r] & (6)} $$
\end{thm} 

\begin{proof}
 (1) $\Rightarrow$ (2) and (3) $\Rightarrow$ (4) are immediate.  (2) $\Rightarrow$ (5) is Proposition \ref{prop:equiv}, (5) $\Rightarrow$ (6) is Proposition \ref{prop:exact}, and (4) $\Rightarrow$ (6) follows from the fact that $C^*_{r}(\cG)$ is embedded into $C^*_{r}(\beta_r \rtimes \cG)$. Finally (1) $\Rightarrow$ (3) follows from Propositions \ref{prop:SC} and \ref{prop:nonsep_nuc}.
 \end{proof}

What is missing in order to get the equivalence of all these properties is whether (6) implies (1).

\subsection{Examples}\label{subsec:examples}

\subsubsection{Locally compact groups} When $\cG$ is a locally compact group, it is shown in \cite{BCL, OS20}  that (5) $\Rightarrow$ (1). If in addition $G$ is inner amenable we have (6) $\Rightarrow$ (1)(see \cite{Osa} when $\cG$ is discrete and \cite[Theorem 7.3] {AD02} for all locally compact groups).

\subsubsection{Transitive groupoids}  Let $(\cG,\lambda)$ be a transitive  locally compact groupoid.  All its isotropy groups  are isomorphic and, fixing $x\in X$, the  space $\cG^x$ is a $\cG(x)$-$\cG$-equivalence (see Subsection \ref{sub:tg}). Lemma \ref{lem:C*exactAmen} has the following consequences: the groupoid  $\cG$ is  KW-exact if and only if  it  is amenable at infinity since this is the case for locally compact groups (\cite{BCL, OS20}); moreover these two properties are equivalent to the fact that $\cG$ is $C^*$-exact when $\cG(x)$ is discrete, for instance when $\cG$ is \'etale. 
  
  \subsubsection{Groupoids related to transformation  groupoids}\label{sect:tg} Let $\alpha : \Gamma\actson X$ be a partial action of a discrete group $\Gamma$ on a locally compact space $X$. Assume that $\Gamma$ is exact. We have seen in Proposition \ref{prop:partial} that the groupoid $\Gamma\ltimes X$ is strongly amenable at infinity. It follows from Proposition \ref{prop:equiv} that it is KW-exact and $C^*$-exact. The fact that $C^*_{r}(\Gamma\ltimes X)$ is exact  also follows from \cite[Corollary 5.3]{AEK},   after having observed that  the $C^*$-algebra $C^*_{r}(\Gamma\ltimes X)$ is isomorphic to the reduced crossed product $\cC_0(X)\rtimes_r \Gamma$ with respect to the partial action of $\Gamma$ (see \cite[Proposition 2.2]{Li16}).

\subsubsection{Group bundle groupoids}\label{subsub:exact} There exist continuous fields 
$$\cA = (A, \set{\pi_x: A \to A(x)}_{x\in X},X)$$
 of exact $C^*$-algebras such that $A$ is not exact \cite{KW95}. 
  Let us consider the case of an inner exact \'etale group bundle groupoid  $\cG$ whose isotropy groups are exact. Then $C^*_{r}(\cG)$ is a continuous is a continuous field of exact $C^*$-algebras. Is $C^*_{r}(\cG)$ an exact $C^*$-algebra?
 We have seen that in case of an HLS-groupoid, its inner exactness implies its amenability. 
 
\subsubsection {\it Inverse semigroups.}\label{par:invers} We  keep the notation of the subsection \ref{subsec:ISA}. Let $S$ be an inverse semigroup. One defines an equivalence relation $\sigma$ on $S$ by saying that $s \,{\sim}_\sigma\, t$ whenever there exists an idempotent $e\in E$  such that $ se = te$. The quotient $S/\sigma$ is a group, denoted by $\Gamma_S$, and called the {\it maximal group homomorphic image of} $S$ (or the minimum group congruence). By an abuse of notation,  $\sigma$ will also denote the quotient map from $S$ onto $\Gamma_S$. Then $\sigma$ is such that $\sigma(st) = \sigma(s)\sigma(t)$ and $\sigma(s^*) = \sigma(s)^{-1}$ for all $s,t\in S$. If $S$ has a zero (denoted by $0$), the group $\Gamma_S$ is trivial. We denote by $S^\times$ the set $S\setminus\set{0}$. When $S$ does not have a zero, we set $S^\times = S$. The map $\sigma : S\to \Gamma_S$ has the following property (which justifies the terminology): every homomorphism from $S$ onto a group $G$ factors through $\Gamma_S$.

To each inverse semigroup $S$ is associated in an explicit way a groupoid $\cG_S$, not Hausdorff in general (see \cite{Exel, Li17}). As in \cite{Li17} we denote by $\wh E$ the space of non-zero maps $\chi$ from $E$ into $\set{0,1}$ such that $\chi(ef) =\chi(e)\chi(f)$ and $\chi(0) = 0$ whenever $S$ has a zero.  Equipped with the  topology induced from the product space  $\set{0,1}^E$, the space  $\wh E$,  called the {\it spectrum} of $S$, is  locally compact and  totally disconnected. Note that when $S$ is a monoid ({\it i.e.,} has a unit element $1$) then $\chi$ is non-zero if and only if $\chi(1) = 1$, and therefore $\wh E$ is compact.  The semigroup $S$  has a canonical action $\theta$ on $\wh E$ as follows. The domain (open and closed) of $\theta_t$ is $D_{t^*t} = \set{\chi\in \wh E: \chi(t^*t) = 1}$ and we set $\theta_t(\chi)(e) = \chi(t^*et)$. Then $\cG_S$ is the groupoid $\cG(S,\theta)$ defined in the subsection \ref{subsec:ISA}.  

\begin{defn}\label{def:idpure} Let $S$ be an inverse semigroup. A {\it partial homomorphism}  is an application $\psi$ from $S^\times$ into a discrete group $\Gamma$ such that $\psi(st) = \psi(s)\psi(t)$ if $st\not=0$. If in addition  $\psi^{-1}(e_\Gamma) = E^\times$, we say that $\psi$ is  an {\it idempotent pure morphism} (where $e_\Gamma$ is the unit of $\Gamma$).  When such an application $\psi$ from $S^\times$ into a group $\Gamma$ exists, the inverse semigroup $S$ is called {\it strongly $E^*$-unitary}. 
\end{defn}

If $S$ is $E$-unitary, then $\sigma: S \to \Gamma_S$ is idempotent pure. Of course, any strongly $E^*$-unitary semigroup is $E^*$-unitary. Note that when $S$ is without zero, $S$ is strongly $E^*$-unitary if and only if it is $E$-unitary. 

Let us recall the following result, that relates $\cG_S$ to a partial transformation groupoid when $S$ is strongly $E^*$-unitary (see \cite[Lemma 5.22]{Li17}).

\begin{lem}\label{lem: Li} Let $S$ be a strongly $E^*$-unitary semigroup with an idempotent pure morphism $\psi: S^\times \to \Gamma$. Then there is a partial action $\beta : \Gamma \actson \wh{E}$ such that the topological groupoids $\cG_S$ and $\Gamma\ltimes_\beta \wh{E}$ are canonically isomorphic. In particular $\cG_S$ is Hausdorff\footnote{In fact,  $\cG_S$ is Hausdorff for every $E^*$-unitary semigroup (see Subsection \ref{subsec:ISA})}.
\end{lem}
The domain  of $\beta_g$ is $ \wh E_{g^{-1}} = \cup_{t\in \psi^{-1}(g)} D_{t^*t}$,
and if $\chi\in \wh E_{g^{-1}}$ and $e\in \wh E$ we have $\beta_g(\chi)(e) = \chi(t^*et)$ where $t$ is any element of $S^\times$ with $\psi(t) = g$ and $\chi(t^*t) = 1$.
The isomorphism from $\cG_S$ onto  $\Gamma\ltimes_\beta \wh{E}$ sends $[s,\chi]$ onto $(\psi(s), \chi)$ (see \cite{Li17}).

\begin{rem}\label{rem:strong-F} The homomorphism $(g,\chi) \mapsto g$ is locally proper if and only if the graph of $\beta_g$ is closed in $\wh{E}\times \wh{E}$ for all $g\in \Gamma$ (see Remark \ref{exs:loc_proper} (d)). This holds in particular if $\wh E_{g}$ is closed for all $g\in \Gamma$, for instance if $\psi^{-1}(g)$ has  a maximum element $m_g$ for every $g$ in the image of $\psi$. Indeed, in this case we have $\wh E_{g^{-1}} = D_{m_{g}^*m_g}$, which is compact. A strongly $E^*$-unitary with such a property is called a strongly $F^*$-unitary semigroup. Under this assumption the groupoid $\cG_S$ is inner amenable and is  equivalent to some transformation groupoid $Y\rtimes \Gamma$ for some  action of $\Gamma$ on a locally compact space (see Example \ref{ex:Abd}). This includes many examples of inverse semigroups \cite{Law02}. In fact this  equivalence  between the two groupoids holds for even more general semigroups (see \cite{KS02}).
\end{rem}

Let $S$ be an inverse semigroup. Its reduced $C^*$-algebra $C^*_{r}(S)$ and its full $C^*$-algebra $C^*(S)$ are defined in \cite{Pat}. They are canonically isomorphic to the reduced $C^*$-algebra  and to the full $C^*$-algebra, respectively, of the groupoid $\cG_S$  (see \cite[Theorem 4.4.2]{Pat}, \cite[Theorems 3.3 and  3.5]{KS02}, \cite[Theorems 5.17 and 5.18]{Li17}).

\begin{prop}\label{prop:inverse_E*} Let $S$ be a strongly $E^*$-unitary semigroup as in the previous lemma. The reduced $C^*$-algebra $C^*_{r}(S)$ is exact when $\Gamma$ is exact.
\end{prop}

\begin{proof} Immediate, by Lemma \ref{lem: Li}, Proposition \ref{prop:partial} and Propositions \ref{prop:equiv}, \ref{prop:exact}.
\end{proof}

When $S$ is $E$-unitary the converse holds.

\begin{prop}\label{prop:inverse_E} Let $S$ be a $E$-unitary inverse semigroup. Then $C^*_{r}(S)$ is exact if and only if the maximal group homomorphic image $\Gamma_S$ is an exact group.
\end{prop}

\begin{proof} We have to show that $\Gamma_S$ is exact when $C^*_{r}(S)$ is exact. It suffices to consider the case where $S$ does not have a $0$, since  $S= E$  otherwise. Let $\chi_\infty$ be the character such that $\chi_\infty(e) = 1$ for every $e\in E$. Then $\chi_\infty$ is $\cG_S$-invariant in $\wh E = \cG_S^{(0)}$ and  $[t,\chi_\infty] \mapsto \sigma(t)$ is an isomorphism from the isotropy group $\cG_S(\chi_\infty)$ onto $\Gamma_S$. But the $C^*$-algebra $C^*_{r}(\cG_S(\chi_\infty))$ is a quotient of $C^*_{r}(\cG_S)$. This latter $C^*$-algebra is exact since it is canonically isomorphic to $C^*_{r}(S)$. It follows that the group $\Gamma_S$ is exact.
\end{proof}

\begin{rem}\label{rem:inverse_E} Let $S$ be a $E$-unitary inverse semigroup. Then $C^*_{r}(S)$  is nuclear if and only if the \'etale groupoid $\cG_S$ is amenable. This latter condition implies that the isotropy group $\Gamma_S$ at $\chi_\infty$ is amenable. Conversely, assuming that $\Gamma_S$ is amenable, then the groupoid $\cG_S$ is amenable (by Lemma \ref{lem: Li} and Proposition  \ref{cor:amen-pa}) and therefore $C^*_{r}(S)$  is nuclear.
\end{rem}

\begin{ex}\label{ex:Willett} Let $(\Gamma, (N_k))$ be an approximated group. We keep the notation of Example \ref{ex:HLS}. Let $S = \set{q_k(t) : k\in \N^+, t\in \Gamma}$. Formally, $S = \cG$ is  the HLS-groupoid defined in Example \ref{ex:HLS} but we view $S$ as an inverse semigroup in the following way. The product is given by $q_m(t).q_n(u) = q_{m\wedge n}(tu)$ where $m\wedge n$ is the smallest of the two elements $m,n$. We set $q_m(t)^* = q_m(t^{-1})$. The set $E_S$ of idempotents is $\set{q_m(e) : m \in \N^+}$ that we identify with $\N^+$. The product of two idempotents is given by $m.n = m\wedge n$. The spectrum $\wh E$ is the set $\set{\chi_m : m \in \N^+}$ where $\chi_m(n) = 1$ if and only if $m\leq n$. It is homeomorphic to the compact space $\N^+$.  The groupoid $\cG_S$ associated with $S$ is the space of equivalence classes of pairs $\big(q_m(t),\chi_k\big)$ with $k\leq m$,  where $\big(q_m(t),\chi_k\big) \sim \big(q_n(u),\chi_{k'}\big)$ if and only if $k = k'$ and $q_k(t) = q_k(u)$. The map sending the class of $\big(q_m(t),\chi_k\big)$ to $q_k(t)$ is an isomorphism of topological groupoids from $\cG_S$ onto the HLS-groupoid $\cG$. 
 
  Assuming $N_0 \varsubsetneqq \Gamma$,  the semigroup $S$ is without zero. The maximal group homomorphism image $\Gamma_S$ of $S$ is the finite group $\Gamma/N_0$ and $\sigma : S \to \Gamma/N_0$  is $q_m(t) \mapsto q_0(t)$. It is not idempotent pure and $S$ is not $E$-unitary.   If $\Gamma$ is not amenable, we know that $C^*_{r}(\cG_S)$ is not exact (see Proposition \ref{prop:HLSbis}). So Proposition \ref{prop:inverse_E*}  is not true in this case.
\end{ex}

\subsubsection{\it Sub-semigroups of a group.}\label{subpar_sg} In this subsection, we consider a discrete group $\Gamma$ and a sub-semigroup $P$ which contains the unit $e$ of $\Gamma$. For $p\in P$, let $V_p$ be the isometry in $\cB(\ell^2(P))$ defined by
$$V_p\delta_q = \delta_{pq}.$$
The {\it reduced $C^*$-algebra} or {\it Toeplitz algebra} of $P$ is the sub-$C^*$-algebra $C^*_{r}(P)$ of $\cB(\ell^2(P))$ generated by these isometries. 

An inverse semigroup $S(P)$, called the {\it inverse hull of} $P$, is attached to $P$. One of its definitions  is 
$$S(P) = \set{V_{p_1}^*V_{q_1} \cdots V_{p_n}^*V_{q_n} : n\in \N, p_i,q_i\in P}.$$
It is an inverse semigroup of partial isometries in $\cB(\ell^2(P))$ (see \cite[\S 3.2]{Nor14}).  An important property of $S(P)$ is that the map $\psi$ from $S(P)^\times$ into $\Gamma$ such that
$$\psi\big(V_{p_1}^*V_{q_1} \cdots V_{p_n}^*V_{q_n}\big)= p^{-1}_1 q_1 \cdots p^{-1}_n q_n$$
 is well defined and is an idempotent pure morphism (see \cite[Proposition 3.2.11]{Nor14}). Moreover, by \cite[Lemma 3.4.1]{Nor14}, the semigroup $S(P)$ does not have a zero if and only if $PP^{-1}$ is a subgroup of $\Gamma$. 

Let us denote by $\lambda_\Gamma$ the left regular representation of $\Gamma$ and by $E_P$ the orthogonal projection from $\ell^2(\Gamma)$ onto $\ell^2(P)$. We say that $(P,\Gamma)$ satisfies the {\it Toeplitz condition} if for every $g\in \Gamma$ such that $E_P \lambda_g E_P \not= 0$, there exist $p_1,\dots, p_n, q_1,\dots, q_n \in P$ such that $E_P \lambda_g E_P = V_{p_1}^*V_{q_1}\dots V_{p_n}^*V_{q_n}$. For instance the quasi-lattice ordered groups introduced by Nica in \cite{Nica92} satisfy this property (see \cite[\S 8]{Li13}).

\begin{prop}\label{prop:subgroup} Let $(P,\Gamma)$ be as above.
\begin{itemize}
\item[(i)] If $\Gamma$ is exact, then the $C^*$-algebra $C^*_{r}(P)$ is exact.
\item[(ii)] Assume that $G = PP^{-1}$ and that the Toeplitz condition is satisfied. Then $C^*_{r}(P)$ is exact if and only if the group $\Gamma$ is exact.
\end{itemize}
\end{prop}

\begin{proof} (i) follows from Proposition \ref{prop:inverse_E*} and from the fact that $C^*_{r}(P)$ is a quotient of $C^*_{r}(\cG_{S(P)})$ (see \cite[Corollary 3.2.13]{Nor14}).

(ii) Assume now that $\Gamma = PP^{-1}$ and that the Toeplitz condition is satisfied. Then the inverse semigroup $S(P)$ does not have a $0$. Moreover, the map $\tau : \Gamma_S\to \Gamma$ such that $\tau\circ \sigma = \psi$ is an isomorphism. Indeed $\psi$ is surjective, so $\tau$ is also surjective. Assume that $\tau(\sigma(x)) = e$, with $x\in S(P)$. Since $\psi$ is idempotent pure, we see that $x$ is an idempotent and therefore $\sigma(x)$ is the unit of $\Gamma_S$. 

It follows from Proposition \ref{prop:inverse_E} that $C^*_{r}(S(P))$ is exact if and only if $\Gamma$ is exact. Finally, we conclude by using the fact that $C^*_{r}(P) = C^*_{r}(S(P))$ since the Toeplitz condition is satisfied (see \cite[Theorem 3.2.14]{Nor14} and \cite[Lemma 2.28]{Li12}).
\end{proof}

\section{\textbf{\textsc{A sufficient condition for $C^*$-exactness to imply amenability at infinity}}}\label{sect:eiai}
The main result of this section is Theorem \ref{thm:equiv}. Its proof requires some prelimi\-naries.
We first recall a classical Kirchberg's characterization of exact $C^*$-algebras.

\begin{defn}\label{def:nuclear} Let $\Phi:A\to B$ be a completely positive contraction  between two $C^*$-algebras. We say that $\Phi$ is {\it factorable} if there exists an integer $n$ and completely positive contractions $\psi: A \to M_{n}(\C)$, $\varphi :  M_{n}(\C) \to B$ such that $\Phi = \varphi \circ \psi$.

We say that $\Phi$ {\it is  nuclear} if there exists a net of factorable   completely positive contractions $\Phi_i:A\to B$ such that 
$$\forall a\in A, \quad\lim_i\norm{\Phi(a) -\Phi_i(a)} = 0.$$ 
 \end{defn}

Recall that a $C^*$-algebra $A$ is nuclear if  for every $C^*$-algebra $B$, there is only one $C^*$-norm on the algebraic tensor product $A\odot B$. This is equivalent to the fact that $\Id_A : A\to A$ is nuclear  (see \cite[Theorem 3.8.7]{BO}).

\begin{thm} [Kirchberg]  A $C^*$-algebra  $A$ is exact if and only if there exist a Hilbert space $\cH$ and a nuclear embedding of $A$ into $\cB(\cH)$ (or equivalently a nuclear embedding in some $C^*$-algebra).  
\end{thm}
For a proof of this deep result we refer to  \cite[Theorem 3.9.1]{BO}. We will need the following extension of this theorem.

\begin{lem}\label{lem:Kirch} Let $A$, $B$ be two separable $C^*$-algebras, where $B$ is nuclear. Let $\cE$ be a countably generated Hilbert $C^*$-module over $B$.  Let $\iota : A \to \cB_B(\cE)$ be an embedding of $C^*$-algebras. Then $A$ is exact if and only if $\iota$ is nuclear.
\end{lem}

\begin{proof} If $\iota$ is nuclear, it is well-known that $A$ is exact (see \cite[Proposition 7.2]{Wass}). 

Conversely, assume that $A$ is exact and that $A$ is not unital, the case where $A$ has a unit being simpler. We denote by $\widetilde{A}$ the $C^*$-algebra obtained by adjunction of a unit $1$ to $A$. We assume that $\widetilde{A}$ is  embedded into $\cB(H)$ for some separable Hilbert space $H$ and  that $1$ is the unit of $\cB(H)$. We set $H_\infty = \ell^2(\N)\otimes H$ and denote by $i$ the embedding of $\widetilde{A}$ into $\cB(H_\infty)$ sending $a$ to $\Id_{\ell^2(\N)} \otimes a$. Observe that $i(\widetilde{A})\cap \cK(H_\infty) = \set{0}$,  where $\cK(H_\infty)$ is the $C^*$-algebra of compact operators on $H_\infty$. We denote by $\tau$ the embedding of $\widetilde{A}$ into $\cB_B(H_\infty\otimes B)$ obtained by composition of the canonical embeddings $i$  of $\widetilde{A}$ into  $\cB(H_\infty)$ and of $\cB(H_\infty)$ into $\cB_B(H_\infty\otimes B)$.

By Kasparov's absorption theorem \cite[Theorem 6.2]{Lance_book}, the Hilbert $C^*$-modules $\cE\oplus (H_\infty\otimes B)$ and $H_\infty\otimes B$ are isomorphic. Let $U \in \cB_B(\cE\oplus (H_\infty\otimes B), H_\infty\otimes B)$ be a unitary operator and set  $\pi(a)= U(\iota(a)\oplus \tau(a)) U^*$ for $a\in A$. We define a unital representation $\pi_1$ of $\widetilde{A}$ into $\cB_B(B\oplus(H_\infty\otimes B))$ by 
$$\pi_1(\lambda,a)(b,\xi) = (\lambda b, \lambda\xi + \pi(a)\xi),$$
where $\lambda\in \C$, $a\in A$, $b\in B$ and $\xi\in H_\infty\otimes B$.

Next, let $V:B\oplus (H_\infty\otimes B)\to H_\infty\otimes B$ be an isomorphism between these two Hilbert $C^*$-modules
and set $\pi_2 = V\pi_1V^*$.

Let $\varepsilon >0$ and $a_1,\cdots,a_n \in A$ be given. Since $B$ is nuclear, the Kasparov-Voiculescu theorem  \cite[Theorem 6]{Kasp} implies the existence of a unitary operator $W \in \cB_B(H_\infty\otimes B,(H_\infty\otimes B)\oplus (H_\infty\otimes B))$
such that 
$$\hbox{for}\,\, 1\leq k\leq n, \quad \norm{W \tau(a_k) W^* -(\tau \oplus \pi_2)(a_k)} \leq \varepsilon.$$

Since $\widetilde{A}$ is exact and since $\tau$ factors through $\cB(H_\infty)$, there exists a factorable  completely positive contraction $\Phi : \widetilde{A} \to \cB_B(H_\infty\otimes B)$  such that 
$$\hbox{for}\,\, 1\leq k\leq n, \quad \| \Phi(a_k) - \tau(a_k)\| \leq \varepsilon.$$

It follows that 
$$\hbox{for}\,\, 1\leq k\leq n, \quad \| W \Phi(a_k) W^* - \tau(a_k) \oplus \pi_2(a_k)\| \leq 2\varepsilon.$$ 
Denoting by $J : H_\infty\otimes B \to (H_\infty\otimes B)\oplus (H_\infty\otimes B)$ the inclusion into the second component of this direct sum, we get for $1\leq k\leq n,$
$$\| J^* W \Phi(a_k) W^* J -  \pi_2(a_k)\| \leq 2\varepsilon,$$
and therefore
$$\|V^*J^*W \Phi(a_k)W^*JV - \pi_1(a_k)\| \leq 2\varepsilon.$$

Let $J_1$ be the inclusion from   $H_\infty\otimes B$ onto $B\oplus H_\infty\otimes B$. Then we have 
$$\hbox{for}\,\, 1\leq k\leq n, \quad \| J_{1}^*V^*J^*W \Phi(a_k)W^*JVJ_1 - \pi(a_k)\| \leq 2\varepsilon,$$
and therefore
$$\hbox{for}\,\, 1\leq k\leq n, \quad \| U^*J_{1}^*V^*J^*W \Phi(a_k)W^*JVJ_1U - (\iota\oplus \tau)(a_k)\| \leq 2\varepsilon.$$
Finally, if $J_2$ is the inclusion from $\cE$ into $\cE\oplus(H_\infty  \otimes B)$ we have
$$\hbox{for}\,\, 1\leq k\leq n, \quad \| J_{2}^*U^*J_{1}^*V^*J^*W \Phi(a_k)W^*JVJ_1UJ_2 - \iota(a_k)\| \leq 2\varepsilon.$$

Observe that $a\mapsto J_{2}^*U^*J_{1}^*V^*J^*W \Phi(a)W^*JVJ_1UJ_2$ is a factorable completely positive contraction from $A$ into $\cB_B(\cE)$.
This shows that the embedding $\iota$ of $A$ into $\cB_B(\cB)$ is nuclear. 
\end{proof}

The major part of the rest of this section \ref{sect:eiai} is adapted from an unpublished note of Jean Renault that we thank for allowing us to use his ideas.
 
\begin{defn} Let $(\cG,\lambda)$ be a locally compact groupoid with a Haar
system and let $B$ be a $C^*$-algebra. A map $\Psi :C^*_{r}(\cG) \to B$ is said to have a {\it compact support} if there exists a compact subset $K$ of $\cG$ such that $\Psi(f) = 0$ for every $f \in \cC_c(\cG)$ with $\supp(f) \cap K = \emptyset$.
\end{defn}

\begin{lem}[Renault] Let $(\cG,\lambda)$ be a locally compact groupoid with a Haar
system, and let $\Phi :  C^{*}_{r}(\cG) \to M_n({\C})$ be a completely
positive map. Then for every $\varepsilon > 0$ and every finite subset $F$ of  $C^{*}_{r}(\cG)$
there exists a completely positive map $\Psi :  C^{*}_{r}(\cG) \to M_n({\C})$
with compact support such that $\| \Psi(a) - \Phi(a) \| \leq \varepsilon$ for
$a\in F$.
\end{lem}

\begin{proof} Using the Stinespring dilation theorem, we get a representation $\rho$ of $C^*_{r}(\cG)$ into some Hilbert space $H_\rho$ and vectors $e_1,\dots,e_n\in H_\rho$ such that, for $a\in  C^*_{r}(\cG)$, we have
$$\Phi(a) =[ \scal{e_i,\rho(a) e_j}] \in M_n(\C).$$
 Let $\cE$ be the Hilbert module $ L^2_{\cC_0(X)}(\cG,\lambda)$ over $\cC_0(X)$ where $X=\cG^{(0)}$. Let $\mu$ be a probability measure on $X$ with support $X$ and define on $\cG$ the measure $\mu\circ\lambda$ by
 $$\int_{\cG} f\rd\mu\circ\lambda = \int_X \big(\int_{\cG^x} f(\gamma) \rd \lambda^x(\gamma)\big) \rd \mu(x).$$
We consider the faithful representation $f\mapsto\tilde{ \Lambda}(f) =\Lambda(f)\otimes \Id$ in the Hilbert space $\cE\otimes_{\cC_0(X)} L^2(X,\mu)$. Observe that this Hilbert space is canonically isomorphic to the Hilbert space $L^2(\cG,\mu\circ\lambda)$ and that, for $f\in \cC_c(\cG)$ and $\xi\in L^2(\cG,\mu\circ\lambda)$ we have
$$(\tilde{ \Lambda}(f)\xi)(\gamma) = \int f(\gamma^{-1}\gamma_1)\xi(\gamma_1) \rd\lambda^{r(\gamma)}(\gamma_1).$$
Since the representation $\rho$ is weakly contained in $\tilde{ \Lambda}$, given $\varepsilon' >0$, there exist a multiple $\tilde{ \Lambda}_K =\tilde{ \Lambda}\otimes \Id_K$ of $\tilde{ \Lambda}$,
where $K$ is some separable Hilbert space, and vectors $\xi_1,\dots,\xi_n$ in $L^2(\cG,\mu\circ\lambda)\otimes K = L^2(\cG,\mu\circ\lambda, K)$ such that
$$\abs{\scal{e_i,\rho(a) e_j} - \scal{\xi_i, \tilde{ \Lambda}_K(a)\xi_j}} < \varepsilon'$$
for $a\in F$ and $i,j\in \set{1,\dots,n}$. Moreover, we can choose the $\xi_i$'s to have a compact support. For $f\in \cC_c(\cG)$, we have
$$\scal{\xi_i, \tilde{ \Lambda}_K(f)\xi_j} = \int \scal{\xi_i(\gamma),\xi_j(\gamma_1)}f(\gamma^{-1}\gamma_1)\rd\lambda^x(\gamma)\rd\lambda^x(\gamma_1)\rd\mu(x).$$

It follows that the completely positive map
$$a\mapsto \Psi(a) = \Big[\scal{\xi_i, \tilde{ \Lambda}_K(a)\xi_j} \Big]$$
has a compact support and satisfies $\| \Psi(a) - \Phi(a) \| \leq \varepsilon$ for
$a\in F$ if $\varepsilon'$ is small enough.
\end{proof}

\begin{cor}\label{cor:approx} Let $B$ be a $C^*$-algebra and let $\Phi :  C^{*}_{r}(\cG) \to B$ be a nuclear completely positive map.
Then for every $\varepsilon > 0$ and every $a_1,\dots,a_k \in  C^{*}_{r}(\cG)$
there exists a factorable   completely positive map $\Psi :  C^{*}_{r}(\cG) \to B$, with  compact support,
such that $\| \Psi(a_i) - \Phi(a_i) \| \leq \varepsilon$ for
$i = 1,\dots,k$.
\end{cor}
 
The notion of positive definite function extends to the case  of $f: \cG\to B$ in an obvious way: it is {\it positive definite} if for every $x\in X$, $n\in \N$ and $\gamma_1,\dots,\gamma_n \in \cG^x$, then $[f(\gamma_i^{-1}\gamma_j)]$ is a positive element of $M_n(B)$. In the next lemma, we will use the following observation: {\it $f$ is positive definite if and only if for every finite set $I$ and every groupoid homomorphism $\theta: I\times I \to \cG$, where $I\times I$ is the trivial groupoid on $I$, the element $[f\circ \theta(i,j)]$ of the $C^*$-algebra $M_I(B)$ of matrices over $I\times I$ with coefficients in $B$ is positive}. This follows from the fact that every groupoid homomorphism $\theta: I\times I \to \cG$ is such that $\theta(i,j) = \gamma_i^{-1}\gamma_j$ where $i\mapsto \gamma_i$ is a map from $I$ into $\cG^x$ for some $x\in X$. Indeed, it suffices to fix $i_0\in I$, and set $x= \theta^{(0)}(i_0)$ and $\gamma_i = \theta(i_0,i)$. 

Given $f:\cG\times \cG \to \C$, we set $f_\gamma(\gamma') = f(\gamma,\gamma')$.

\begin{lem}[Renault] Let $(\cG,\lambda)$ be a locally compact groupoid with a Haar system.
\begin{itemize}
\item[(a)] Let $f\in \cC_c(\cG)$ be a continuous positive definite function. Then, $f$ viewed as an element of $C^*_{r}(\cG)$ is a positive element.
\item[(b)] Let $f:\cG\times \cG \to \C$ be a properly supported positive definite function. Then $\gamma\mapsto f_\gamma$ is a continuous  positive definite function  from $\cG$ into  $C^*_{r}(\cG)$.
\end{itemize}
\end{lem}

\begin{proof} (a) Let $\cE = L^2_{\cC_0(X)}(\cG,\lambda)$. We have to show that $\Lambda(f)\in \cB_{\cC_0(X)}(\cE)$ is positive, which amounts to prove that for every $x\in \cG^{(0)}$ and $\xi\in L^2(\cG^x,\lambda^x)$, we have $\scal{\xi,\Lambda_x(f)\xi} \geq 0$. It suffices to consider the case where $\xi\in \cC_c(\cG^x)$. Let $K$ be the support of $\xi$. Since
$$\scal{\xi,\Lambda_x(f)\xi} = \int \overline{\xi(\gamma)}f(\gamma^{-1}\gamma_1)\xi(\gamma_1)\rd\lambda^x(\gamma)\rd\lambda^x(\gamma_1),$$
then by approximating the restriction of $\lambda^x$ to $K$  by positive measures with finite support we get the desired positivity result.

(b) Since $f$ is properly supported, the map $\gamma \mapsto f_\gamma$ is continuous from $\cG$ into $\cC_c(\cG)$ endowed with the inductive limit topology, and therefore from $\cG$ into $C^*_{r}(\cG)$ endowed with its norm topology. 
To show that    $\gamma\mapsto F(\gamma) = f_\gamma$ is a continuous   positive definite function, let $I$ be a finite set and let $\theta : I\times I \to \cG$ be a groupoid homomorphism. We have to check that  $[F\circ\theta(i,j)] \in M_I(C^*_{r}(\cG))$ is positive. But 
$M_I(C^*_{r}(\cG)) = C^*_{r}( (I\times I)\times\cG)$. Therefore, by the first part of the lemma applied to the product groupoid $(I\times I)\times \cG$, it suffices to show that 
$$((i,j),\gamma) \mapsto f_{\theta(i,j)}(\gamma) = f(\theta(i,j),\gamma)$$ 
belongs to $\cC_c((I\times I)\times\cG)$ and is positive definite. But this is clear, since this function is obtained by composing $f$ with the homomorphism $\theta\times \Id : (I\times I)\times \cG\to \cG\times \cG$.  
\end{proof}

\begin{thm}\label{thm:equiv} Let $\cG$ be an  inner amenable \'etale groupoid. Then the following condition are equivalent:
\begin{itemize}
\item[(1)] $\cG$ is strongly amenable at infinity.
\item[(2)] $\cG$ is  amenable at infinity.
\item[(3)] $C^*_{r}(\beta_r\cG\rtimes\cG)$ is nuclear.
\item[(4)] $C^*_{r}(\beta_r\cG\rtimes\cG)$ is exact;
\item[(5)] $\cG$ is KW-exact;
\item[(6)] $C^*_{r}(\cG)$ is exact.
\end{itemize}
\end{thm}

\begin{proof} By Theorem \ref{cor:ai_ex}  it suffices to show that (6) implies (1).
Therefore, let us assume that $C^*_{r}(\cG)$ is exact. The proof is adapted from ideas of Jean Renault.

We fix a compact subset $K$ of $\cG$ and $\varepsilon >0$. We want to find a continuous bounded positive definite kernel $k\in \cC_t(\cG*_r\cG)$ such that  $\abs{k(\gamma,\gamma_1) - 1} \leq \varepsilon$ whenever $\gamma^{-1}\gamma_1 \in K$ (see Theorem \ref{prop:amen_inf}).

We set $\cE =  \ell^2_{\cC_0(X)}(\cG)$ with $X = \cG^{(0)}$. Recall that $\lambda^x$ is the counting measure on $\cG^x$. We first choose a bounded, continuous positive definite function $f$ on $\cG\times \cG$, properly supported,
such that $|f(\gamma,\gamma) - 1| \leq \varepsilon/2$ for $\gamma \in K$ and $f(x,y)\leq 1$ for $(x,y)\in X\times X$. By Lemma \ref{lem:Kirch} the regular representation $\Lambda$ is nuclear. Then, using Corollary \ref{cor:approx}, we find  a compactly supported completely positive map $\Psi : C^*_{r}(\cG) \to \cB_{\cC_0(X)}(\cE)$   such that\footnote{Here, we write $f_\gamma$ instead of $\Lambda(f_\gamma)$
for simplicity of notation.}
$$\norm{\Psi(f_\gamma) - f_\gamma}_r \leq \varepsilon/2$$
for $\gamma \in K$. 
We also choose a continuous function $\xi : X\to [0,1]$ with compact support such that $\xi(x) = 1$ if $x\in s(K)\cup r(K)$. 

Let $(\gamma,\gamma_1)\in \cG*_r\cG$. We choose an open bisection $S$  such that $\gamma\in S$ and a continuous function $\varphi: X\to [0,1]$, with compact support in $r(S)$ such that $\varphi (x) = 1$ on a neighborhood of $r(\gamma)$. We denote $\xi_\varphi$ the continuous function on $\cG$ with compact support  (and thus $\xi_\varphi \in \cE$) such that
$$\xi_\varphi(\gamma') = 0 \,\,\hbox{if}\,\, \gamma'\notin S,\quad \xi_\varphi(\gamma') = \varphi\circ r(\gamma')\xi\circ s(\gamma') \,\,\hbox{if}\,\,\gamma'\in S.$$
Note that $\norm{\xi_\varphi}_\cE \leq 1$.
We define $\xi_{\varphi_1}$ similarly with respect to $\gamma_1$. 

Then we set
\begin{align*}
k(\gamma,\gamma_1) &=  \langle \xi_\varphi,
\Psi(f_{\gamma^{-1}\gamma_{1}})\xi_{\varphi_1}\rangle (r(\gamma))
= \xi\circ s(\gamma)\big(\Psi(f_{\gamma^{-1}\gamma_{1}})\xi_{\varphi_1}\big)(\gamma)\\
&= \scal{\big(\Psi(f_{\gamma^{-1}\gamma_1}\big)^*\xi_\varphi, \xi_{\varphi_1}}(r(\gamma))= \overline{\big(\Psi(f_{\gamma^{-1}\gamma_1})^*\xi_\varphi\big)(\gamma_1)} \xi\circ s(\gamma_1).
\end{align*}
We observe that $k(\gamma,\gamma_1)$ does not depend on the choices of $S,\varphi, S_1,\varphi_1$.

We see that $k$ is continuous since $\gamma \mapsto \Psi(f_\gamma)$ is continuous and, since $\Psi$ is completely positive,  we see that $k$ is  positive definite. Moreover, there is a compact subset $K_1$ of $\cG$ such that $\Psi(f_\gamma) = 0$ when $\gamma\notin K_1$, because  $\Psi$ is compactly supported, and  $f$  is properly supported.  If follows that $k$ is  supported in a tube.

We fix $(\gamma,\gamma_1)\in \cG*_r\cG$  such that $\gamma^{-1}\gamma_1 \in K$. Then we have
$$\abs{k(\gamma,\gamma_1) -1}\leq \varepsilon/2 + \abs{ \langle \xi_\varphi,
\big(\Lambda(f_{\gamma^{-1}\gamma_{1}})\xi_{\varphi_1}\big)
\rangle(r(\gamma)) -1},$$
and
$$
 \langle \xi_\varphi,
\big(\Lambda(f_{\gamma^{-1}\gamma_{1}})\xi_{\varphi_1}\big)
\rangle(r(\gamma)\\
=  \xi\circ s(\gamma)\xi\circ s(\gamma_1)f(\gamma^{-1}\gamma_{1},\gamma^{-1}\gamma_{1}).
$$

 Observe that $s(\gamma)\in r(K)$ and $s(\gamma_1)\in s(K)$ and therefore $ \xi\circ s(\gamma) = 1=  \xi\circ s(\gamma_1)$.
It follows that 
$$\abs{k(\gamma,\gamma_1) -1}\leq \varepsilon/2 + \abs{f(\gamma^{-1}\gamma_{1},\gamma^{-1}\gamma_{1})-1}\leq \varepsilon.$$
To end the proof it remains to check that $k$ is a bounded kernel. Since this kernel is positive definite, it suffices to show that $\gamma\mapsto k(\gamma,\gamma)$ is bounded on $\cG$. We have
$$k(\gamma,\gamma) = \scal{\xi_\varphi, \Psi(f_{s(\gamma)})\xi_\varphi}(r(\gamma))\leq \norm{\Psi(f_{s(\gamma)})}_r.$$
Our claim follows since  $\Psi(f_{s(\gamma)}) = 0$ when $s(\gamma)\notin K_1\cap X$ and $x\mapsto f_x$ is continuous from the compact set  $K_1\cap X$ into $C_r^{*}(\cG)$. 
\end{proof}

\begin{cor}\label{cor:main1} Let $\cG$ be an \'etale groupoid such that there exists a locally proper continuous homomorphism $c$ from $\cG$ into a  discrete group $\Gamma$. Then   the six  conditions of Theorem \ref{thm:equiv} are equivalent. Moreover, they hold when $\Gamma$ is exact.
\end{cor}

\begin{proof} We observe that $\Gamma$ is   inner amenable. Then, by Proposition \ref{prop:wai}
the groupoid $\cG$ is inner amenable. 
\end{proof}

\begin{cor} \label{cor:equiv} Let $\cG$ be a  locally compact groupoid with Haar system.
Assume that $\cG$   is  equivalent to an  inner amenable \'etale groupoid
$\cH$. Then the following conditions are equivalent:
\begin{itemize}
\item[(i)] $\cG$ is amenable at infinity;
\item[(ii)] $\cG$  is KW-exact;
\item[(iii)]  $C^*_{r}(\cG)$ is exact.
\end{itemize}
 \end{cor}

\begin{proof} This follows  from the theorem \ref{thm:equiv} together with the lemma \ref{lem:C*exactAmen}.
\end{proof}

\begin{rem}With similar techniques as those used to prove Theorem \ref{thm:equiv} we can prove the following result.\end{rem}

 \begin{prop} Let $\cG$ be a  locally compact groupoid with Haar system. The following conditions are equivalent:
 \begin{itemize}
 \item[(i)] $\cG$ is amenable;
 \item[(ii)] $C^*_{r}(\cG)$ is nuclear and $\cG$ is  inner amenable.
 \end{itemize}
 \end{prop}
 
 For a locally compact group this result was proved in \cite{LP}. The case of a transformation groupoid was dealt with in \cite[Theorem 5.8]{AD02}.

\section{\textbf{\textsc{Weak containment, exactness and amenability}}}\label{sec:WC}

\begin{defn}\label{def:weakamen} Let $(\cG,\lambda)$ be a locally compact groupoid equipped with a Haar system. We say that $\cG$ has the {\it weak containement property}, \index{weak containment property}(WCP) \index{(WCP)} in short, if the canonical surjection from its full $C^*$-algebra $C^*(\cG)$ onto its reduced one $C^*_{r}(\cG)$ is injective, {\it i.e.}, the two completions $C^*(\cG)$ and $C^*_{r}(\cG)$  of $\cC_c(\cG)$ are the same.
\end{defn}

A very useful theorem of Hulanicki \cite{Hul} asserts that a locally compact group $G$ has the (WCP) if and only if it is amenable. While it has long been known that every amenable locally compact groupoid has the (WCP) \cite{Ren91, AD-R}, whether the converse holds was a longstanding open problem (see \cite[Remark 6.1.9]{AD-R} where this problem is related to the behaviour of the (WCP) with respect to reduction). 

The goal of this section is to study the scope of the family  \index{$\mathscr C$} $\mathscr C$ of locally compact groupoids for which the (WCP) implies amenability.  

 Let $\cG$  be a locally compact groupoid, $U$ an invariant open subset  of $\cG^{(0)}$, and $F = \cG^{(0)}\setminus U$.  If the reduced groupoids $\cG(U)$ and $\cG(F)$ have the (WCP), then $\cG$ inherits this property (see Lemma \ref{lem:WCPinexact}). If $\cG$  has the (WCP), then  $\cG(U)$ has the (WCP),  while it is not always true that $\cG(F)$  has the (WCP).  Such an example was given  by Willett \cite{Wil15}, showing at the same time that for groupoids the (WCP) does not always imply amenability.

\begin{thm}\label{thm:Wil} $($\cite{Wil15}$)$ Let $\F_2$ be the free group on two generators. Let us set $N_k = \cap \set{\ker \phi}$ where $\phi$ ranges over all group homomorphisms from $\F_2$ into all groups of cardinality $\leq k$. Then the pair $(\F_2, (N_k))$ is an approximated group whose corresponding HLS-groupoid  $\cG$ has the (WCP), and is not amenable since $\cG(\infty) = \F_2$ is not amenable (and also does not have  the (WCP)).
\end{thm}

Let us remark that this groupoid has the (WCP) but is not inner exact (see Proposition \ref{prop:HLS} or Lemma \ref{lem:closed WCP}  below).
\subsection{The influence of inner exactness}\label{subsec:ia} 

\begin{lem}\label{lem:closed WCP}  Let $\cG$ be a locally compact groupoid that has  the (WCP) and let $F$ be a closed invariant subset of $X =\cG^{(0)}$. Then $\cG(F)$ has the (WCP) if and only if the sequence
\begin{equation}\label{eq:exact}
0\to C^*_{r}\big(\cG(X\setminus F) \big) \to C^*_{r}(\cG) \to C^*_{r}(\cG(F)) \to 0
\end{equation}
is exact.

In particular if $F = \set{x}$ this is equivalent to the amenability of $\cG(x)$.
\end{lem}
\begin{proof} This follows from Lemma \ref{lem:WCPinexact}. \end{proof}

Therefore the {\it a priori} assumption of inner exactness removes an obstacle for the fact that the (WCP) could imply amenability (because the amenability of a groupoid $\cG$ implies the amenability of $\cG(F)$  for every closed invariant subset $F$ of $\cG^{(0)}$). It is indeed sufficient in the following examples.  Let us remark that Hulanicki's theorem plays an important role in their proofs.

\begin{prop}\label{prop:WCP} Let $\beta: \Gamma \actson X$ be a partial action of a discrete group having a fixed point $x\in X$. Then   the semi-direct product $\cG=\Gamma \ltimes_\beta X$  is amenable if and only if it has the (WCP) and the  sequence \eqref{eq:exact}, with $F= \set{x}$, is exact.
\end{prop}

\begin{proof} Indeed, if $\cG$ has the (WCP) and if the sequence  \eqref{eq:exact}, with $F= \set{x}$,  is exact then $\Gamma = \cG(x)$ is amenable, and then we use the proposition \ref{cor:amen-pa}.
\end{proof}

\begin{rem} This result applies in particular to the groupoid $\cG_S$ where $S$ is a $E$-unitary inverse semigroup without zero, because, keeping the notation of the subsection \ref{par:invers}, we know that $\cG_S$ is isomorphic to the semi-direct product of a partial action of $\Gamma_S$ on $\wh E$ and that $\Gamma_S$ is the isotropy group of $\chi_\infty$ which is $\cG_S$-invariant (see Lemma \ref{lem: Li} and Remark \ref{rem:inverse_E}). Therefore, we get the following equivalences: 
\begin{itemize}
\item[(a)] $\cG_S$ has the (WCP) and the sequence below is exact,
\begin{equation}\label{eq:exact2}
0\to C^*_{r}\big(\cG_S(\wh E\setminus \set{\chi_\infty}) \big) \to C^*_{r}(\cG_S) \to \Gamma_S \to 0.
\end{equation}
\item[(b)] the group $\Gamma_S$ is amenable.
\item[(c)] the groupoid $\cG_S$ is amenable.
\end{itemize}

An inverse semigroup is said to be $F$- inverse if $\sigma^{-1}(g)$ has a maximum element for all $g\in \Gamma_S$ (see \cite{Law}). In this case,  we note that the sequence \eqref{eq:exact2}  is always exact (see \cite[Theorem 83]{F-S}).
\end{rem}

\begin{prop}\label{prop:WCPbundle} Let $\cG$ be an inner exact  group bundle groupoid  over a locally compact space $X$. Then $\cG$ has the (WCP) if and only if it is amenable.
\end{prop}

\begin{proof} We observe that all the units of $\cG$ are invariant. Therefore, if $\cG$ has the (WCP), the group $\cG(x)$ is amenable for every unit and then we use the fact that $\cG$ is amenable if all the groups $\cG(x)$ are amenable.
\end{proof}

\begin{prop}\label{prop:trans} \cite{Bun} Let $\cG$ be a transitive locally compact groupoid. The following conditions are equivalent:
\begin{itemize}
\item[(i)] $\cG$ is amenable;
\item[(ii)] any isotropy group is amenable;
\item[(iii)] $\cG$ has the (WCP).
\end{itemize}
\end{prop}

\begin{proof} Let $x\in \cG^{(0)}$. We have already recalled  in Subsection \ref{sub:tg} that the groupoids $\cG$ and $\cG(x)$ are equivalent. Since amenability is invariant under equivalence  of groupoids (see \cite[Theorem 2.2.17]{AD-R}), we see that (i) $\Leftrightarrow$ (ii). We have also already recalled that (i) $\Rightarrow$ (iii). That (iii) $\Rightarrow$ (i) is proved in \cite{Bun}. One can also use  the fact that the (WCP) is preserved under equivalence of groupoids \cite[Theorem 5.3]{SW12}.
\end{proof}
 
 In fact, both previous propositions are particular cases of the  following result, due to Christian B\"onicke \cite{Bon}. We reproduce below his unpublished proof  and we thank him for giving permission to do so.

\begin{thm}\cite{Bon}\label{thm:bon} Let $\cG$ be a  locally compact groupoid\footnote{Recall again that we implicitely assume that our groupoids are second countable.} with Haar system $\lambda$, such that the orbit space $\cG\setminus \cG^{(0)}$ equipped with the quotient topology is $T_0$. Then the following conditions are equivalent:
\begin{itemize}
\item[(i)] $\cG$ is  amenable;
\item[(ii)] $\cG$ has the (WCP) and is inner exact.
\end{itemize}
\end{thm}

\begin{proof}   

Let us  show that (ii) $\Rightarrow$ (i).  We assume that (ii) is satisfied. By \cite[Theorem 9.86]{Will19}, it suffices to show that for every $x\in X= \cG^{(0)}$, the isotropy group $\cG(x)$ is amenable. By \cite[Theorem 2.27]{Will19}, we know that the orbit $\cG x$ of $x$ is locally closed, hence the intersection of an open subset and a closed subset  of $X$. It is therefore an open subset of its closure, and it is a locally compact space.  

Since $\cG$ is inner exact and since $\overline{\cG x}$ is a closed invariant subset of $\cG^{(0)}$, we see, by Lemma \ref{lem:closed WCP}, 
 that the reduced groupoid  $\cG(\overline{\cG x})$ has the (WCP).

Now, having observed that $\cG x$ is an open  $\cG(\overline{\cG x})$-invariant subset of $\overline{\cG x}$, we have a commutative diagram
$$\xymatrix{C^*(\cG({\cG x}))\ar[d]\ar[rr]&& C^* (\cG(\overline{\cG x}))
\ar[d]\\
C^*_{r}(\cG({\cG x}))\ar[rr]&&C^*_{r}(\cG(\overline{\cG x}))}$$
where both horizontal maps are injective and where the right vertical map is injective. It follows that $C^*(\cG({\cG x})) = C^*_{r}(\cG({\cG x}))$. But the groupoid $\cG({\cG x})$ is transitive, and so, since it has the  (WCP) its  isotropy group $\cG(x)$ is amenable.
\end{proof}

We end this subsection with two permanence properties of the family   $\mathscr C$.

\begin{lem}\label{lem:WCPna} The familiy $\mathscr C$ is preserved under equivalence of groupoids.
\end{lem}
\begin{proof} This follows from the fact that the (WCP) is preserved under equivalence of groupoids \cite[Theorem 5.3]{SW12} as well as amenability \cite[Theorem 2.2.17]{AD-R}. 
\end{proof}

\begin{prop}  Let $\cG$ be a  locally compact groupoid. We assume the existence of a countable set $(F_k)$ of closed invariant subsets of $X= \cG^{(0)}$ such that $X = \bigcup_k F_k$ and such that each reduced groupoid $\cG(F_k)$ is in $\mathscr C$. Then $\cG\in \mathscr C$.
\end{prop}

\begin{proof} Using \cite[Theorem 2.14]{Ren_13} its suffices to show that $\cG$ is Borel amenable. Thus, we have to show the existence of a sequence $(g_n)_{n\in \N}$ of nonnegative Borel functions on $\cG$ such that
\begin{itemize}
\item[(a)] $\int g_n \rd\lambda^{x} \leq 1$ for every $x\in \cG^{(0)}$;
\item[(b)] $\lim_n \int g_n \rd\lambda^{x} = 1$ every  $x\in\cG^{(0)}$;
\item[(c)] $\lim_n \int \abs{g_i(\gamma^{-1}\gamma_1) -  g_n(\gamma_1)}\rd\lambda^{r(\gamma)}(\gamma_1) = 0$ for every  $\gamma\in \cG$.
\end{itemize}
For each $k$, we choose a sequence $(g_n^{k})$ satisfying the conditions (a), (b), (c) with respect to the groupoid $\cG(F_k)$. Next, we set $F'_{k} = F_k\setminus \bigcup_{h< k} F_{h}$. The  family $(F'_{k})$ form a partition  of $X$ into Borel invariant subsets. Given $n\in \N$ and $\gamma\in \cG$ we set $g_n(\gamma) = g^{k}_n(\gamma)$ where $k$ is the unique integer such that $\gamma\in \cG(F'_{k})$. This sequence $(g_n)_{n\in \N}$ satisfies the above conditions (a), (b), (c).
\end{proof} 

\subsection{The case of transformation groupoids}\label{subsec:tg} Let us recall first a few facts concerning exactness for transformation groupoids. Let $G\actson X$ be an action of a locally compact group $G$ on a locally compact space $X$. Suppose that $G$ is KW-exact, or equivalently that $G$ is amenable at infinity (see \cite{BCL, OS20}). It follows from Proposition \ref{prop:inv_sd} that the groupoid $X\rtimes G$ is strongly amenable at infinity. Conversely, if $X\rtimes G$ is amenable at infinity and if $X$ is compact, then $G$ is amenable at infinity (see Proposition \ref{prop:inv_sd}).

With an assumption stronger than the inner exactness of the groupoid $X\rtimes G$, namely the fact that $G$ is KW-exact,  the following positive result was proved in  \cite{Mat} when $G$ is discrete and $X$ is compact; later, this result was extended in \cite[Theorem 5.2]{BEW} to the case where $X$ is  locally compact and finally in  \cite[Theorem 5.15]{BEW20} to the case where $G$ and $X$ are locally compact.  

\begin{thm}\label{thm:WCPtg} Let $G\actson X$ be an action of a locally compact group $G$ on a locally compact space $X$. Assume that $G$ is KW-exact and that $X\rtimes G$ has the (WCP). Then the groupoid $X\rtimes G$ is amenable.
\end{thm}

From this, we deduce the more general following result.

\begin{cor}\label{cor:WCPtg}  Let $(Z,\rho,\sigma)$ be a locally proper and faithful generalized morphism from a locally compact groupoid $\cG$  into a locally compact KW-exact group $G$.   If $\cG$ has the (WCP) then it is amenable.
\end{cor}

\begin{proof}  By Proposition \ref{prop:homo1}, there is an action of $G$ on a locally compact space $Y$ such that $\cG$ is equivalent to $Y\rtimes G$.  Since the (WCP) is preserved under equivalence of groupoids we see that $Y\rtimes G$ has the (WCP), and so is amenable by the previous theorem and it follows that $\cG$ is amenable .\end{proof}

This result applies in particular when there is a  locally  proper and faithful continuous homomorphism from $\cG$ into $G$. It is possible that the faithfulness hypothesis is not necessary (see Corollary \ref{cor:exactetale} when $\cG$ is \'etale).

The theorem \ref{thm:WCPtg}  is also generalized as follows.

\begin{cor}\label{cor:WCPtg1} Let $c: \cG \to G$ be a faithful strongly surjective continuous homomorphism into a KW-exact locally compact group such that $\ker(c)$ is a proper groupoid. If $\cG$ has the (WCP) then it is amenable.
\end{cor}

\begin{proof} This follows from the previous observation, since $c$ is locally proper by the proposition \ref{prop:locprop}.
\end{proof}

\subsection{The case of \'etale groupoids} The result of the theorem \ref{thm:WCPtg} relative to the groupoid $X\rtimes G$ in the case where $G$ is discrete has been extended by Julian Kranz  to any \'etale groupoid which is strongly amenable at infinity.

\begin{thm}\cite{Kra}\label{thm:Kra} Let $\cG$ be an \'etale locally compact groupoid. The following conditions are equivalent:
\begin{itemize}
\item[(i)] $\cG$ is amenable;
\item[(ii)]  $\cG$ has the (WCP) and is strongly amenable at infinity. 
\end{itemize}
\end{thm}

\begin{proof} For simplicity, we give a proof when $\cG$ is the transformation groupoid $X\rtimes \Gamma$ defined by a left action of a discrete group $\Gamma$ on a locally compact space $X$. This is a result obtained in \cite{BEW} but with a different approach. The  proof given below is similar to that given in \cite{Kra} for the general case of an \'etale groupoid  where $\Gamma$ is replaced by the inverse semigroup of open bisections of $\cG$,  which is the source of  several technical difficulties. As in \cite{Mat} the key ingredients of the proof are the use of the Haagerup standard representation for  the bidual of a $C^*$-algebra, and the bimodule property of a contractive completely positive  map (see \cite[Proposition 1.5.7]{BO}). The conclusion will follow from the two lemmas below.

\begin{lem}\label{lem:Mat2}$($\cite[Lemma 5.4]{Kra}$)$ Let $\Gamma\actson A$ be an action of a discrete group $\Gamma$ on a $C^*$-algebra $A$. There is a homomorphism from $ A^{**}\rtimes \Gamma$ into $(A\rtimes \Gamma)^{**}$ such that the composition $A\rtimes \Gamma \to A^{**}\rtimes \Gamma \to (A\rtimes \Gamma)^{**}$ is the canonical inclusion of $A\rtimes \Gamma$ into its bidual.
\end{lem}

\begin{lem}\label{lem:Mat1} $($\cite[Lemma 5.3]{Kra}$)$ We keep the above notation and we assume that $\cC_0 (X) \rtimes \Gamma) = \cC_0 (X) \rtimes_r \Gamma))$. Then there exists a $\Gamma$-equivariant contractive completely positive map $\Phi: \cC_0(\beta_r \cG) \to \cC_0(X)^{**}$ that extends the canonical inclusion of $\cC_0(X)$ into  $\cC_0(\beta_r \cG)$.
\end{lem}

These statements require some explanation. First, since $\Gamma$ is discrete, by bi\-duality all its actions on a $C^*$-algebra $A$ extend to its bidual $A^{**}$. Second, to any left action of $\cG = X\rtimes \Gamma$ on a locally compact space $Y$ is canonically associated a left action of $\Gamma$ on $Y$  for which the  groupoids $Y\rtimes \Gamma$ and $Y\rtimes \cG$ are canonically isomorphic. Moreover the moment map  $p$ from $Y$ to $X= \cG^{(0)}$  relative to the $\cG$-action on $Y$ is $\Gamma$-equivariant (see Lemma \ref{lem:idengr}). Recall that $ty = (tp(y), t)y$ for $(y,t) \in Y\times \Gamma$. We will apply this observation to the action of $\cG$ on $Y = \beta_r \cG$.

With the help of these two lemmas, let us proceed to the proof of (ii) $\Rightarrow$ (i) in Theorem \ref{thm:Kra} when $\cG = X\rtimes \Gamma$.   We denote by $\iota$ and $\iota_X$ the canonical $\Gamma$-equivariant inclusions from $\cC_0(X)$ into $\cC_0(\beta_r\cG)$ and into $\cC_0(X)^{**}$ respectively. For $f \in \cC_0(X)$, we have $\iota(f) = f\circ p_\beta$, where $p_\beta: \beta_r\cG \to X$ is the moment map of the $\cG$-action. By functoriality, $\iota$ and $\iota_X$  induce homomorphisms 
$$\tilde\iota :  \cC_0(X)\rtimes \Gamma \to \cC_0(\beta_r\cG)\rtimes \Gamma$$ 
and 
$$\widetilde{\iota_X} : \cC_0(X) \rtimes \Gamma \to \cC_0(X)^{**} \rtimes \Gamma.$$ Thanks to Lemma \ref{lem:Mat1} and to the functoriality of full crossed products with respect to contractive completely positive  maps we get from $\Phi$ a contractive completely positive  map $\tilde\Phi :  \cC_0(\beta_r\cG)\rtimes \Gamma$ into $\cC_0(X)^{**}\rtimes \Gamma$ such that $\widetilde{\iota_X} = \tilde\Phi\circ \tilde\iota$. Moreover by Lemma \ref{lem:Mat2} we get a homomorphism 
$$\rho : ,\cC_0(X)^{**} \rtimes \Gamma \to  (\cC_0(X)\rtimes \Gamma)^{**}$$
 such that $\rho\circ\widetilde{\iota_X}$ is the inclusion of  of $\cC_0(X)\rtimes \Gamma$ in its bidual. This inclusion is $\rho\circ\tilde\Phi\circ \tilde\iota$. Now, using the isomorphisms
$$ \cC_0(\beta_r\cG)\rtimes \Gamma= C^*(\beta_r\cG\rtimes \Gamma) \simeq C^*(\beta_r\cG\rtimes \cG )\simeq  C^*_{r}(\beta_r\cG\rtimes \cG )$$
and the fact that $C^*_{r}(\beta_r\cG\rtimes \cG )$ is nuclear, since $\cG$ is amenable at infinity (see Proposition \ref{prop:nonsep_nuc} for these observations), we see that $\tilde\iota$ is nuclear, and so is the inclusion from $\cC_0(X)\rtimes \Gamma$ in its bidual. Then it follows from \cite[Proposition 2.3.8]{BO} that the $C^*$-algebra $\cC_0(X)\rtimes \Gamma$ is nuclear, and so $C^*_{r}(X\rtimes \Gamma)$ is also nuclear. This implies that the action of $\Gamma$ on $X$ is amenable by \cite[Theorem 5.3]{AD02}.
\end{proof}

\noindent{\it Proof of Lemma \ref{lem:Mat2}.}  We choose a faithful, normal, unital representation of \\$(A\rtimes \Gamma)^{**}$ in some Hilbert space $H$. Its restriction to $A\rtimes \Gamma$ is the integrated form of a covariant representation $(\pi,u)$. Denote by $\pi^{**}$ the normal extension of $\pi$ to $A^{**}$. Then $(\pi^{**},u)$ is a covariant representation of  $\Gamma \actson A^{**}$ and $\pi^{**}\rtimes u$ sends $A^{**}\rtimes \Gamma$ into $(A\rtimes \Gamma)^{**}$. \qed

\noindent{\it Proof of Lemma \ref{lem:Mat1}.}  We represent $\cC_0(X)^{**}$ in standard form on a Hilbert space $H$. We denote by $s\mapsto u_s$ the unitary representation of $\Gamma$ that implements the action of $\Gamma$ on $\cC_0(X)^{**}$. By restriction to  $\cC_0(X)$ we obtain a covariant representation $(\pi,u)$ of $\Gamma \actson \cC_0(X)$. By integration, we get a representation 
$$\pi\rtimes u :\cC_0(X)\rtimes_r \Gamma  = \cC_0(X)\rtimes \Gamma\to \cB(H).$$
 Denote by $\iota$ the inclusion of $\cC_0(X)$ into $\cC_0(\beta_r\cG)$ and by $\tilde\iota$ the corresponding inclusion of $\cC_0(X)\rtimes_r \Gamma$ into $\cC_0(\beta_r\cG)\rtimes_r \Gamma$. By the Arveson extension theorem, there is a completely positive contraction $\phi$ such that the following diagram commutes:
$$\xymatrix{\cC_0(X)\rtimes_r \Gamma \ar@{^{(}->}[d]_{\tilde\iota}\ar[rd]^{\pi\rtimes u}  &\\
\cC_0(\beta_r\cG)\rtimes_r \Gamma\ar@{.>}[r]^-{\phi} &\cB(H).}$$
Since $\iota(\cC_0(X))$ is in the multiplicative domain of $\phi$ and since $\cC_0(\beta_r\cG)$ is commutative we see that $\phi(\cC_0(\beta_r\cG))$ is contained into $\pi(\cC_0(X))' = \cC_0(X)^{**}$.

It remains to see that $\phi$ restricted to $\cC_0(\beta_r\cG)$ is $\Gamma$-equivariant. We choose an approximate unit $(e_\kappa)$  for $\cC_0(X)$. We observe that $\iota(e_\kappa)$ is an approximate unit for $\cC_0(\beta_r\cG)$ and also for $\cC_0(\beta_r\cG) \rtimes_r \cG$. Let $f\in \cC_0(\beta_r\cG)$ and let $s\in \Gamma$. Since $\iota(e_\kappa) u_s$ is in the multiplicative domain of $\phi$ we  have
\begin{align*}
\phi(u_s f u_{s}^*)& = \lim_\kappa\phi(\iota(e_\kappa) u_s f u_{s}^*)= \lim_\kappa \phi(\iota(e_\kappa) u_s)\phi(f u_{s}^*)\\
 &= \lim_\kappa \pi(e_\kappa) u_s \phi(f u_{s}^*) = u_s \phi(f u_{s}^*).
 \end{align*}
Similarly, we see that $\phi(f u_{s}^*) = \phi(f) u_{s}^*$ and so $\phi(u_s f u_{s}^*) = u_s\phi(f)u_{s}^*$. \qed

\begin{rem} When $\Gamma$ is exact and  $X$ is compact, there is a more direct proof of the fact that the (WCP) of $\cG = X\rtimes \Gamma$ implies  its amenability. Indeed in this case $\beta_r \cG = \beta(X\times \Gamma)$ and  $\ell^\infty \Gamma$ is a $\Gamma$-equivariant  unital sub-$C^*$-algebra of $\cC(\beta_r\cG) = \cC_b(X\times \Gamma)$. It follows from Lemma \ref{lem:Mat1} that there an unital $\Gamma$-equivariant completely positive map $\Psi$ from $\ell^\infty\Gamma$ into $\cC_0(X)^{**}$. Since the action of $\Gamma$ on $\ell^\infty\Gamma$ is amenable, there is a net $(h_i)$ of positive definite (with respect to the $\Gamma$-action on $\ell^\infty\Gamma$\footnote{If $\alpha : \Gamma \actson A$, recall that a function $h:\Gamma \to A$ is said to be positive definite with respect to $\alpha$ if for every finite subset $F$ of $\Gamma$ the matrix $[\alpha_s(h(s^{-1}t))]_{s,t\in F}$ is nonnegative.}),  finitely supported, functions  from  $\Gamma$ to $ \ell^\infty\Gamma$, such that $h_i(e)\leq 1$ for all $i$ and $\lim_i h_i(s) = 1$ in norm, for all $s\in \Gamma$ (see \cite{AD87}).  Then $(\Psi\circ h_i)$ is a net of positive definite  (with respect to the $\Gamma$-action on $\cC_0(X)^{**}$), finitely supported, functions from $\Gamma$ to $\cC_0(X)^{**}$ that converges pointwise and in norm to $1$. By \cite{AD87} this implies that 
the action of $\Gamma$ on $X$ is amenable.
\end{rem}

As  applications one gets the following results. The first one was also obtained in \cite[Theorem 4.12]{BFS}.

\begin{cor}\label{cor:exactpart}  Let $\Gamma\actson X$ be a partial action of an exact  discrete group on a locally compact space $X$. Then the partial transformation groupoid $\Gamma\ltimes X$ has the (WCP) if and only if it is amenable.  
\end{cor}

\begin{proof} It is an immediate consequence of Theorem \ref{thm:Kra} and of Proposition \ref{prop:partial}.
\end{proof}

\begin{cor}\label{cor:exactetale} Let $\cG$ be  an \'etale groupoid.  We assume that there exists a locally proper homomorphism from $\cG$ into a KW-exact locally compact group $G$. Then $\cG$ has the (WCP) if and only if it is amenable.  
\end{cor}

\begin{proof} Since the group $G$ is amenable at infinity, using Proposition \ref{prop:locpropinv} we see that $\cG$ is strongly amenable at infinity, and we conclude thanks to  Theorem \ref{thm:Kra}.
\end{proof}

\section{\textbf{\textsc{Open questions and comments}}}\label{sec:open} 
\subsection{About exactness} Recall that in Subsection \ref{subsec:comparison} we have denoted by $\mathscr {E}$ (resp. $\mathscr {E}'$) the family of locally compact groupoids for which $C^*$-exactness (resp. KW-exactness) implies amenability at infinity.

The family $\mathscr {E}$ contains all inner amenable locally compact groups \cite[Theorem 7.3]{AD02} and also all locally compact groups $G$ such that $C^*_{r}(G)$ has a tracial state \cite{Man}. None of these two classes of locally compact groups is included into the other one \cite{Man, AD23}. Note that many locally compact groups are not inner amenable, for instance all connected non-amenable ones \cite[Remark 5.10]{AD02}. All locally compact groups are in the  family $\mathscr {E}'$  \cite{BCL, OS20}.
To the author's knowledge, whether $C^*$-exactness implies KW-exactness  is still an open problem for locally compact groups. By \cite{CZ},  it suffices to consider the case of unimodular totally disconnected groups.

The family $\mathscr {E}$ also contains all inner amenable \'etale groupoids (Theorem \ref{thm:equiv}). While inner amenability admits several equivalent characterizations in the case of locally compact groups \cite{LR87, LP88, LP, CT}, almost nothing is known for groupoids. In particular it is not known whether all \'etale locally compact groupoids  (the analogue of discrete groups!) are inner amenable. It would be interesting to study  the case of HLS-groupoids and semi-direct products relative to partial actions of discrete groups.

The family $\mathscr {E}'$ contains, besides all locally compact groups, all transitive locally compact groupoids (by invariance under equivalence). We don't know whether all locally compact groupoids are in $\mathscr {E}'$. This is certainly a fairly technical issue.

\subsection{(WCP) vs amenability}
Recall that \index{$\mathscr C$} $\mathscr C$ is the family of locally compact groupoids for which the (WCP) implies amenability. 

This family  includes all locally compact groups and more generally all inner exact locally compact groupoids $\cG$ such that $\cG\setminus \cG^{(0)}$ is a $T_0$ space (Theorem \ref{thm:bon}). In particular, all transitive locally compact  groupoids and all inner exact group bundle groupoids are in $\mathscr C$.

\'Etale strongly amenable at infinity groupoids are also in $\mathscr C$  (Theorem \ref{thm:Kra}). Does Theorem \ref{thm:Kra} remain true for all locally compact strongly amenable at infinity groupoids, and even for all inner exact locally compact groupoids? 

A first approach would be to study the case of  a semi-direct product groupoid $X\rtimes G$ where $X$ is a locally compact space and $G$ is a locally compact group. One has some indications of a positive response in this direction:
 if $X$ is compact or if $G$ is discrete we have $X\rtimes G$ amenable if and only if $X\rtimes G$ has the (WCP) and is strongly amenable at infinity. The second case follows from Theorem \ref{thm:Kra}. As for the first case, the strong amenability at infinity implies that $G$ is KW-exact by Proposition \ref{prop:inv_sd} and the conclusion follows from \cite[Theorem 5.15]{BEW20}. 
 
 In the more general case where $X$ is not necessarily compact, the strong amenabi\-lity at infinity of $X\rtimes G$ does not imply that $G$ is KW-exact.    As already said, by  \cite[Theorem 5.15]{BEW20},  under the  assumption that $G$ is KW-exact, we have $X\rtimes G\in \mathscr C$. It would be interesting to see whether we can replace the KW-exactness of $G$ by the weaker assumption that $X\rtimes G$ is strongly amenable at infinity.

\newpage
\part{}
\begin{appendix}

\section{\textbf{\textsc{Fibrewise compactifications}}}
 In this appendix, we are given  a {\it fibre space} $(Y,p)$ over a locally compact space $X$ (as defined in Section \ref{subsec-Action}). For $x\in X$ we denote by $Y^x$ the {\it fibre} $p^{-1}(x)$. We say  $(Y,p)$ is  {\it fibrewise compact } if $p$ is proper. Our goal is the construction of fibrewise compactifications of fibre spaces. It extends the usual study of compactifications of locally compact spaces, corresponding to the case where $X$ is reduced to a point.

\begin{defn}\label{def:fibComp}
 A {\it fibrewise compactification} \index{fibrewise compactification}
of $(Y,p)$ is a triple $(Z, \varphi, q)$
where $Z$ is a locally compact space, $q: Z \rightarrow X$ is a continuous
{\it proper} map and $\varphi : Y \rightarrow Z$ is a homeomorphism onto an open
dense subset of $Z$ such that $p = q \circ \varphi$. Then  $q(Z)$ is closed and $p(Y)$ is dense into $q(Z)$.

If $(Z,\varphi,q)$ and $(Z_1,\varphi_1,q_1)$ are
two fibrewise compactifications of $(Y,p)$,  we say that  $(Z,\varphi,q)$ is {\it smaller than} $(Z_1,\varphi_1, q_1)$ if there exists a continuous map  $\psi : Z_1 \to Z$  is a  such that  $\psi\circ \varphi_1 = \varphi$. 
\end{defn}

Note that $q\circ\psi \circ\varphi_1 = q\circ \varphi = p =q_1\circ\varphi_1$ and so $q\circ\psi = q_1$, that is, $\psi$ is a morphism of fibre spaces.  Moreover, it is easily seen that $\psi$ is proper and surjective.

\subsection{Construction of fibrewise compactifications}\label{subsec:fc}

Recall that $\cC_b(Y)$ (resp. $\cC_0(Y)$) is the $C^*$-algebra of continuous bounded (resp. vanishing  at infinity) functions from $Y$ to $\C$, and  that $\cC_c(Y)$ is the involutive subalgebra of continuous functions with compact support. We set\footnote{We warn the reader that $p^*{\mathcal C}_0(X)$  is not the same as $p^*(\cC_c(X))$ introduced in  Subsection \ref{subsect:Gaction}. We think it will not be confusing.}
$$p^*{\mathcal C}_0(X) = \set{f\circ p:f\in \cC_0(X)}.$$

 Given a fibrewise compactification $(Z, \varphi, q)$ of $(Y,p)$, we usually identify $Y$ with the open subset $\varphi(Y)$ of $Z$.  Then
$\varphi^* : g \mapsto g|_Y$ is an isomorphism from ${\mathcal C}_0(Z)$ onto
a $C^*$-subalgebra of ${\mathcal C}_b(Y)$ which obviously contains ${\mathcal C}_0(Y)$
and also $p^*{\mathcal C}_0(X)$ since $q$ is proper. We will often identify ${\mathcal C}_0(Z)$ to $\varphi^*\cC_0(Z)$. In particular, for $f\in \cC_0(X)$, $f\circ p$ is viewed as defined on $Y$ or on $Z$.

Let us denote by $\cC_0(Y,p)$ the closure of \index{$\cC_0(Y,p)$} \index{$\cC_c(Y,p)$}
$$\cC_c(Y,p) = \{g \in {\mathcal C}_b(Y) : \exists K\,\, \hbox{compact}\,\,\subset X, \,\supp(g) \subset
p^{-1}(K) \},$$
where $\supp(g) $ is the support of $g$. Note that $\cC_0(Y,p)$ is the $C^*$-algebra of conti\-nuous bounded functions $g$ on $Y$ such that for every $\varepsilon >0$ there exists a compact subset $K$ of $X$ satisfying $\abs{g(y)} \leq \varepsilon$ if $y \notin p^{-1}(K)$.  

If $(Z, \varphi, q)$ is a fibrewise compactification of $(Y,p)$ one  immediately checks that 
$$p^*{\mathcal C}_0(X) + {\mathcal C}_0(Y) \subset \varphi^*\cC_0(Z) \subset \cC_0(Y,p).$$

\begin{lem}\label{lem:C-alg} Let $(Y,p)$ be a fibre space.
\begin{itemize}
\item[(i)] Let $\mathcal K$ be the family of compact subsets of $X$, ordered by inclusion. For $K\in \mathcal K$ we choose $u_K : X\to [0,1]$ in $\cC_c(X)$ such that $u_K(x) =1$ if $x\in K$. Then  $(u_K\circ p)_{K\in \mathcal K}$ is an approximate unit of $\cC_0(Y,p)$.
\item[(ii)] Let $(Z,\varphi,p)$ be a fibrewise compactification of $(Y,p)$. Then $\cC_0(Z)$ has a natural structure of $\cC_0(X)$-algebra in the sense of Definition \ref{def:C-alg}. It is given by the homomorphism $f\mapsto f\circ p$ from $\cC_0(X)$ into $\cC_0(Z)$.
\end{itemize}
\end{lem}

\begin{proof} The first assertion is immediate and it implies that the homomorphism $f\mapsto f\circ p$ is non-degenerate.
\end{proof}

\begin{prop} \label{prop:f_c} The fibrewise compactifications of $(Y,p)$ are in bijective correspondence
with the $C^*$-subalgebras $A$ of  $\cC_0(Y,p)$ containing $p^*{\mathcal C}_0(X) + {\mathcal C}_0(Y)$.
More precisely, to $A$ we associate $(Z, \varphi, q)$, where $Z$ is the Gelfand spectrum of $A$,
$\varphi$ is defined by the essential embedding of ${\mathcal C}_0(Y)$ into $A$, and
$q$ is defined by the embedding of  ${\mathcal C}_0(X)$ into ${\mathcal C}_0(Z)$ via $p^*$. The inverse map
sends $(Z, \varphi, q)$ to $\varphi^*\cC_0(Z)$.
\end{prop}

\begin{proof} Let $A$ be as in the above statement and denote by $Z$ its Gelfand spectrum. Observe that $\cC_0(Y)$ is an ideal  of $A$ and therefore $Y$ is canonically identified to an open subspace of $Z$. Moreover, $Y$ is dense in $Z$ since $\cC_0(Y)$ is an essential ideal of $A$.

Now, let us fix $z\in Z$ and consider the homomorphism $f\mapsto (f\circ p)(z)$ defined on $\cC_0(X)$, where $f\circ p$ is viewed as defined on $Z$. Let us check that it is non-zero, hence a character of $\cC_0(X)$. Indeed,  the previous lemma shows that $A$  has an approximate unit of the form $(u_K\circ p)$.  There is at least one compact subset $K$ of $X$ such that $u_K\circ p(z)\not = 0$. It follows that there exists a unique element in $X$, that we denote by $q(z)$, such that $(f\circ p)(z) = f(q(z))$ for every $f\in \cC_0(X)$. In particular, we get $p(z) = q(z)$ when $z\in Y$. Moreover, since $f\circ q$ is continuous for every $f\in \cC_0(X)$, we see that $q$ is continuous. Finally, $q$ is a proper map because $f\circ q$ vanishes at infinity for every $f\in \cC_0(X)$.

The other assertions of the proposition are immediate.
\end{proof}

\begin{rems}\label{rem:f_c} (a) Let $(Z,q)$ be a fibrewise compactification of $(Y,p)$. When $\overline{p(Y)}$ is compact, then $Z$ is compact since $q(Z) = \overline{p(Y)}$.

 When $X$ is reduced to a point, the fibrewise compactifications are the usual compactifications. They are in bijective correspondence, via the Gelfand duality, with the unital $C^*$-subalgebras of $\cC_b(Y)$ containing $\cC_0(Y)$. 

(b) Observe that $p^*{\mathcal C}_0(X) + {\mathcal C}_0(Y)$ is a $C^*$-algebra (see \cite[Corollary 1.8.4]{Dix}). So, there is a smallest fibrewise compactification, which is  given by the spectrum
of $p^*{\mathcal C}_0(X) + {\mathcal C}_0(Y)$. We denote it by $(Y_{p}^+,p^+)$ \index{$(Y_{p}^+,p^+)$} and call it the {\it  Alexandroff fibrewise compactification}. \index{Alexandroff fibrewise compactification} The largest  fibrewise compactification is given by the spectrum of
 $\cC_0(Y,p)$. We denote it by $(\beta_p Y,p_\beta)$ \index{$(\beta_p Y, p_\beta)$} and call it the {\it  Stone-\v Cech fibrewise compactification}. \index{Stone-\v Cech fibrewise  compactification}

When $\overline{p(Y)}$ is compact, we have $\cC_0(Y,p) = \cC_b(Y)$. Therefore $\beta_p Y$ is the usual Stone-\v Cech
compactification $\beta Y$. When $Y = X\times V$ is a product with $X$ compact and $V$  locally compact,  the Alexandroff fibrewise compactification of $Y$  with respect to the first projection is $X\times V^+$. On the other hand, when $X$ is infinite, compact, we have $\beta_p (X\times V) = \beta(X\times V)$, which can differ from $X\times \beta V$ if $V$ is not compact (see \cite{Glic}). Finally, let us observe that whenever $p$ is proper we have $\cC_0(Y) = \cC_0(Y,p)$. So, in this case $(Y,p)$ has only one fibrewise compactification, namely itself.
\end{rems}

\begin{rem}\label{rem:notopen} The $C^*$-algebra $\cC_0(Y,p)$ has be introduced independently in \cite{Las} where it  is implicitely considered that (with our notation) the map $p_\beta: \beta_p Y \to X$ is open. However, this is not always true, as shown by the following example. We take  $X= [0,1]$ and
$Y =  (X\times \set{0} )\sqcup (]1/2,1]\times\set{1}) \subset \R^2$
 with the topology induced by $\R^2$, and $p$ is the projection onto $X$. Then 
 $$\beta_p Y =  \beta Y = (X\times\set{0}) \sqcup (\beta ]1/2,1])\times \set{1})$$
  since $X$ is compact, and $p_\beta$ is not open since it  sends the open subset $\beta ]1/2,1]\times \set{1}$ of $\beta_p Y$ onto $[1/2,1]$. Note that the fibre of $\beta_p Y$ over $1/2$ is 
  $$(\beta_p Y)^{1/2} =\set{(1/2,0)} \sqcup (\beta]1/2,1]\setminus ]1/2,1])\times\set{1}.$$
   So $Y^{1/2}=\set{(1/2,0)}$ is not dense into $(\beta_p Y)^{1/2}$.
 
 The Alexandroff  fibrewise compactification of $Y$ is 
 $$Y_p^{+} = (X\times \set{0}) \sqcup ([1/2,1]\times\set{1}).$$
  Here also $p^+$ is not open, and the fibre of $Y^+$ over $1/2$ is $\set{(1/2,i): i=0, 1}$.
\end{rem}

 Let $(Z,\varphi,q)$ be a fibrewise compactification of $(Y,p)$. Although $Y^x$ is not dense into $Z^x$ in general, we have the following useful easy result.

\begin{lem}\label{lem:up}  Let $(Z,\varphi,q)$ and $(Y,p)$ as above. Let $V$ be an open subset of $X$. Then $p^{-1}(V)$ is dense into $q^{-1}(V)$.
\end{lem}

\begin{proof} This follows immediately from the density of $Y$ in $Z$ and  the fact that $q^{-1}(V)$ is open.
\end{proof}

\subsection{Some properties of fibrewise compactifications}\label{subsec:pfc}

 Let $(Y,p)$ be a fibre space over $X$. We  observe that {\it there is a greatest open subset $U$ of $X$ such that the restriction of $p$ to $p^{-1}(U)$ is proper}. Indeed, let $(U_i)_{i\in I}$ be a family of open subsets with this property. Then $\cup_{i\in I} U_i$ still has this property. To show this fact, let $K$ be a compact subset of $\cup_{i\in I} U_i$. Every $x\in K$ has a compact neighborhood which is contained in some $U_i$. Then we can cover $K$ by  finitely many compact sets $C_1,\dots, C_n$ with $C_k$ contained in some $U_{i_k}$, for $k=1,\dots,n$. It follows that $p^{-1}(K) = \cup_{k=1}^n p^{-1}(K\cap C_k)$ is compact.

Note that $U$ is described as the set of elements of $X$ having a compact neighborhood $K$ such that $p^{-1}(K)$ is compact. It can be empty.

\begin{prop}\label{prop:proper} Let $(Y,p)$ be a fibre space over $X$ and $U$ as above. Let $(Z,\varphi,q)$ be a fibrewise compactification of $(Y,p)$. Then $q^{-1}(U) = p^{-1}(U)$. So, $Z$ and $Y$ have the same fibres over $x\in U$.
\end{prop}

\begin{proof} Let $V$ be an open set contained in a compact subset $K$ of $U$. Then we have $p^{-1}(V) \subset p^{-1}(K)$. Observe that $p^{-1}(K)\subset p^{-1}(U)$ is  compact  and that $p^{-1}(V)$ is dense into $q^{-1}(V)$. It follows that $q^{-1}(V) \subset p^{-1}(K)$ and so $q^{-1}(U) = p^{-1}(U)$ since $U$ is the union of such open sets $V$.
\end{proof}

 The  most important example for us is that of the fibre space $(Y,p) = (\cG,r)$ where $\cG$ is an  \'etale groupoid and $r: \cG\to\cG^{(0)}$ is its range map, which is a local homeomorphism. In the  general case of an \'etale fibre space, we give in Proposition \ref{lem:U} a concrete description the above open subset $U$ of $X$.

\begin{defn}\label{def:etale} We say that a fibre space $(Y,p)$ over $X$ is {\it \'etale}\footnote{In \cite[Definition 1.19]{James} such fibre spaces are called fibrewise discrete.} if every $y\in Y$ has an open neighborhood  $S$ such that $p(S)$ is open and the restriction of $p$ to $S$ is a homeomorphism onto $p(S)$.  We denote by $p_S^{-1}$ its inverse map, which is defined on $p(S)$. Such a set $S$ is called an {\it open  section}.
\end{defn}

We collect  several consequences of the definition. First, $p$ is an open map and the fibres $Y^x$ with their induced topology are discrete. Moreover, the topology of $Y$ has a basis of open sections. Finally, for every $g\in \cC_c(Y)$, let us set $h(x) = \sum_{y\in Y^x} g(y)$. Then $h\in \cC_c(X)$. Indeed, using a partition of unity, it suffices to consider the case where $g$ has its compact support in some open section $S$. Then we have $h(x) = g\circ p_S^{-1}(x)$ if $x\in p(S)$ and $h(x) = 0$ otherwise. The conclusion follows immediately.

\begin{prop}\label{lem:U} Let $(Y,p)$ be an \'etale fibre space over $X$. Let $W$ be the open subset of $X$ formed by the elements $x$ such that the cardinal of the fibres of $Y$ is finite and constant in a neighborhood  of $x$.
 Then,  $W$ is the greatest open subset $U$ of $X$ such that the restriction of $p$ to $p^{-1}(U)$ is proper.
\end{prop}

\begin{proof} Assume that $x\in W$ and write $Y^x = \set{y_1,\dots, y_n}$. We choose disjoint open sections $S_1,\dots, S_n$ with $y_i\in S_i$ for $i= 1,\dots,n$.
We can choose $S_1,\dots, S_n$ small enough so that they have the same image $V$ under $p$ and that $p^{-1}(V) = \cup_{i=1}^n S_i$ since the fibres of $Y$ have the same cardinal $n$ in a neighborhood of $x$. Let $K$ be a compact neighborhood of $x$ contained in $V$. Then $p^{-1}(K) = \cup_{i=1}^n p^{-1}(K) \cap S_i$ is a finite union of compact sets and therefore is compact. Therefore $x\in U$. Conversely, assume that $x\in U$. It has a compact neighborhood $K$ such that $p^{-1}(K)$ is compact. Let $f\in \cC_c(X)$ be with support in $K$ and equal to $1$ in a neighborhood of $x$. Then $f\circ p$ belongs to $\cC_c(Y)$. For $x_1\in X$, let us set $h(x_1) = \sum_{y\in Y^{x_1}} f\circ p(y)$. Then $h : X\to [0,+\infty[$ is continuous and is equal to the cardinal of $Y^{x_1}$ in a neighborhood of $x$. Thus, we have $x\in W$.
\end{proof}

\subsection{The Alexandroff fibrewise compactification}\label{subset:afc} Its fibres have a simple description, in contrast with the Stone-\v Cech  situation.

\begin{prop}\label{prop:alex} Let $(Y,p)$ be a  fibre space over $X$ and let $U$ be the greatest open subset of $X$ such that the restriction of $p$ to $p^{-1}(U)$ is proper.  We assume that $p$ is surjective. Then:
\begin{itemize}
\item[(i)] the fibre $(Y^{+}_p)^x$ of the Alexandroff fibrewise compactification of $(Y,p)$ is of the following form:
\begin{itemize}
\item $(Y^{+}_p)^x = Y^x$ if $x\in U$;
\item $(Y^{+}_p)^x = (Y^{x})^+ = Y^x\cup \set{\omega_x}$, the Alexandroff compactification of $Y^x$, if $Y^x$ is not compact;
\item $(Y^{+}_p)^x$ is the disjoint union of $Y^x$ and a singleton $\set{\omega_x}$ when $x\notin U$ with $Y^x$ compact. \end{itemize}
\item[(ii)] $Y^{+}_p$ is the disjoint union of $Y$ and $F = \set{\omega_x : x\in X\setminus U}$. Moreover, $Y$ is a dense open subset of $Y^{+}_p$ and $F$, with the induced topology, is canonically homeomorphic to $X\setminus U$.
\end{itemize}
\end{prop}

\begin{proof} Recall that $Y^{+}_p$ is the spectrum of the the abelian $C^*$-algebra   $A = I + B$ where $I$ is the ideal $\cC_0(Y)$ and $B = p^*\cC_0(X)$. It follows that the spectrum $Y$ of $I$ is identified with the open subset of characters of $A$ that are non-zero on $I$, and $F = Y^{+}_p\setminus Y$ is the set of characters whose kernel contains $I$, that is, the set of characters of $A/I$. Since $A/I$ is isomorphic to $B/(B\cap I)$, we have first to determine  $B\cap I$.

Let $g = f\circ p$ with $f\in \cC_0(X)$ and assume that $g\in \in \cC_0(Y)$. We claim that $f(x_0) = 0$ when $x_0\notin U$. Indeed, assume on the contrary that $f(x_0) \not=0$. Set $\varepsilon = \abs{f(x_0)}/2$ and $V= \set{x\in X: \abs{f(x)}\geq \varepsilon}$. Then $p^{-1}(V)$ is not compact, in contradiction with the fact that $f\circ p\in \cC_0(Y)$.  It follows that $B\cap I$ is contained into $p^*\cC_0(U)$.

On the other hand, since the restriction of $p$ to $p^{-1}(U)$ is proper we have  $p^*\cC_0(U)\subset \cC_0(Y)$. It follows that $B\cap I = p^*\cC_0(U)$. 

Then $B/(B\cap I) = p^*\cC_0(X)/ p^*\cC_0(U)$ is isomorphic to $p^*\cC_0(X\setminus U)$ and thus to $\cC_0(X\setminus U)$. The spectrum $F$ of $B/(B\cap I)$ is thus homeomorphic to $X\setminus U$.  The inverse homeomorphism sends  $x \in X\setminus U$ onto the  character $\omega_x$ defined by  
$$\omega_x(h) = f(x)$$
for any decomposition $h = g+ f\circ p$ as a sum of an element  $g\in I$ and an element $f\circ p \in B$. This is not ambiguous since $B\cap I = p^*\cC_0(U)$. 

The rest of the proposition is now immediate.
\end{proof}

\subsection{The Stone-\v Cech  fibrewise compactification}\label{subset:SC}
The next proposition shows that $(\beta_p Y, p_\beta)$ is the solution of an universal problem.

\begin{prop}\label{prop:universal} Let $(Y,p)$ and $(Y_1,p_1)$ be two fibre spaces over $X$, where $(Y_1,p_1)$ is fibrewise compact. Let $\varphi_1: (Y,p) \to (Y_1,p_1)$ be a morphism of fibre spaces, that is $p_1\circ\varphi_1= p$. There exists 
a unique continuous map  $\Phi_1 : \beta_p Y \to Y_1$ that extends $\varphi_1$. Moreover, $\Phi_1$ is  a proper morphism of fibre spaces (that is, $\Phi_1$ is proper and $p_\beta = p_1\circ \Phi_1$) and we have $\Phi_1(\beta_p Y) = \overline{\varphi_1(Y)}$. 
\end{prop}

\begin{proof} Since $p_1$ is proper, we have $\quad p_{1}^*\cC_0(X) \subset \cC_0(Y_1)$, and therefore 
$$p^*\cC_0(X) = \varphi_{1}^*\big(p_{1}^*\cC_0(X)\big) \subset \varphi_{1}^*\cC_0(Y_1).$$ 
Moreover, we have $\varphi_{1}^* \cC_0(Y_1) \subset  \cC_0(Y,p)= \cC_0(\beta_p Y)$.

 For $z\in \beta_p Y$, we consider the homomorphism $f\mapsto (f\circ \varphi_1)(z)$ defined on  $\cC_0(Y_1)$ (where we view $f\circ \varphi_1$ as defined on $\beta_p Y$). 
 
Since $p^*\cC_0(X)$ contains an approximate unit for $\cC_0(Y,p)$, it follows that $f\mapsto (f\circ \varphi_1)(z)$  is a non-zero homomorphism, thus is of the form $f\mapsto f(\Phi_1(z))$ for a unique $\Phi_1(z)\in Y_1$.  Of course, $\Phi_1$ extends $\varphi_1$ and is continuous. Next, observe that for $y\in Y$  and $f\in \cC_0(X)$, we have 
$$f\circ p_\beta(y) = f\circ p(y) = f\circ p_1\circ \varphi_1(y) = f\circ p_1\circ\Phi_1(y)$$
and so $p_\beta(y) = p_1\circ\Phi_1(y)$. Since $Y$ is dense in $\beta_p Y$ we get $p_\beta = p_1\circ \Phi_1$. Finally, $\Phi_1$ is proper because $p_\beta$ is proper.
\end{proof}

\begin{rem}\label{subset:fibres} Let $(Y,p)$ be a fibre space over $X$ and let us consider its fibrewise Stone-\v Cech compactification $\beta_p Y$ that we also denote here by $Z$.  We assume that $Y$ is second countable, hence normal.
For $x\in X$, we denote by $\theta$ the homomorphism from $\cC_0(Z^x)$ into $\cC_b(Y^x)$ sending $g$ to its restriction to $Y^x$. This map is surjective. Indeed take $h\in\cC_b(Y^x)$ and extend it to a bounded function on Y. If $f\in \cC_c(X)$ is such that $f(x)= 1$ we see that $g=(f\circ p) h$ is in $\cC_0(Z)$. If $k$ denotes the restriction of  $g$ to $Z^x$  we have  $\theta(k) = h$. It follows that the inclusion map  from $Y^x$ into $Z^x$ extends continuously to an inclusion from $\beta(Y^x)$ into $Z^x =(\beta_p Y)^x$. This inclusion can be strict as seen in the remark \ref{rem:notopen}.

As already mentioned, even when $p:Y = X\times V\to X$ with $X$ compact infinite, the situation can be  complex.  Indeed in this case we have $\beta_p Y = \beta(X\times Z)$ which is different from $X\times \beta Z$ when  $V$ is a second countable non-compact locally compact space (see \cite[Theorem 1]{Glic}).
\end{rem}
\end{appendix}
\printindex

\bibliographystyle{plain}

\end{document}